\documentclass[11pt]{amsart}
\usepackage[utf8]{inputenc}
\usepackage{txfonts}
\usepackage[T1]{fontenc}

\usepackage{amsmath,amsthm,amssymb,hyperref, mdframed, eucal}
\usepackage{mwe}
\usepackage{mathrsfs} 
\usepackage{tikz}
\usepackage[all]{xy}
\usetikzlibrary{decorations.markings}
\usetikzlibrary{backgrounds,shapes}
\usepackage{subfig}
\usepackage{tikz}
\usetikzlibrary{shapes,arrows}
\usetikzlibrary{shapes.geometric, arrows}
\usetikzlibrary{decorations.markings}
\usepackage{subfig}
\usepackage{tabularx}
\usepackage{enumitem}
\usepackage{amsmath,amsthm,amscd}
\usepackage{amssymb,amsfonts}
\usepackage{fancyhdr} 
\usepackage{pgf,tikz}
\usepackage{caption}
\usepackage{tikz-cd}
\usepackage{tikzscale}
\usetikzlibrary{patterns}
\usepackage{rotating}
\usepackage{xcolor}
\usepackage{lscape}
\usepackage{array} 
\usepackage{tabularx}
\usepackage{multirow}
\usepackage{multicol}
\usepackage{booktabs} 
\usepackage{caption}
\usepackage{xspace}
\usepackage{xfrac}
\usepackage{blindtext}
\usepackage[ruled,vlined]{algorithm2e}
\usepackage{imakeidx}
\makeindex

\usetikzlibrary{decorations.markings}
\usetikzlibrary{patterns}
\usetikzlibrary{calc}
\usetikzlibrary{shapes.geometric}

\usepackage{color}
\usepackage{hyperref}
\hypersetup{
    colorlinks=true, 
    linktoc=all,     
    linkcolor=blue,  
    citecolor=blue,
    filecolor=blue,
    urlcolor=blue
}

\definecolor{aliceblue}{RGB}{0.9, 0.95, 1.0}
\definecolor{smoked}{RGB}{216, 212, 204}
\definecolor{mauve}{RGB}{200, 55, 171}
\definecolor{apricot}{RGB}{250, 144, 4}
\definecolor{sky}{RGB}{66, 169, 244}
\definecolor{plum}{RGB}{76, 0, 102}
\definecolor{forest}{RGB}{90,145,120}
\definecolor{sand}{RGB}{180,160,120}

\usepackage{geometry} 
\numberwithin{equation}{section}

\newcommand\Z{{\mathbb Z}}
\newcommand\ziz{{\Z\oplus i\,\Z}}

\newcommand\R{{\mathbb R}}

\newcommand{\C}{{\mathbb C}}

\newcommand{\spz}{\mathrm{Sp}(2g,\Z)}

\newcommand{\pol}{{\textnormal{P}}}

\newcommand{\cp}{\mathbb{C}\mathrm{\mathbf{P}}^1}

\newcommand{\strata}{\mathcal H_g(\,\mu\,)}

\newcommand{\shomolz}{\textnormal{H}_1(\,\Sigma,\mathbb{Z}\,)}

\theoremstyle{plain}                    
\newtheorem{thm}{Theorem}[section]
\newtheorem*{hthm}{Haupt's Theorem}

\newtheorem{lem}[thm]{Lemma}
\newtheorem{prop}[thm]{Proposition}
\newtheorem{cor}[thm]{Corollary}
\newtheorem*{thmnn}{Theorem}

\newtheorem{ru}[thm]{Rule}

\theoremstyle{definition}
\newtheorem{defn}[thm]{Definition}

\theoremstyle{remark}
\newtheorem{ex}[]{Example}

\newtheorem{rmk}[thm]{Remark}

\newcommand{\cm}[1]{\textcolor{blue}{#1}}

\tikzstyle{rb} = [rectangle, rounded corners, minimum width=3cm, minimum height=1cm, text width=3cm, text centered, draw=black, fill=blue!30]
\tikzstyle{sb} = [rectangle, minimum width=2cm, minimum height=1cm, text width=3cm, text centered, draw=black, fill=violet!30]

\title[The dark side of translation surfaces: It is not all about dynamics]{The dark side of translation surfaces:\\ It is not all about dynamics}

\author{Gianluca Faraco}
\address[Gianluca Faraco]{School of Mathematics, Monash University, Clayton, VIC 3800, Australia}
\email{gianluca.faraco@monash.edu}
\date{\today}

\begin{document}

\dedicatory{}

\begin{abstract}
This survey offers a geometric and topological perspective on translation surfaces, with a focus on aspects that are often overshadowed by the dynamical viewpoint. After presenting several equivalent definitions of translation surfaces, particularly those based on polygonal models and fundamental domains, we explore the structure of their moduli spaces in the spirit of Thurston, including natural stratifications. The final part is devoted to the study of the period map, the realisation of representations as period characters and the fibres of the holonomy map. This approach highlights a foundational, yet sometimes underrated, facet of the theory.
\end{abstract}

\maketitle
\tableofcontents

\section{Introduction}

\noindent Translation surfaces appear naturally at the intersection of geometry, topology, and dynamics, and they have attracted increasing attention in recent decades. A large portion of the research on translation surfaces has focused on their dynamical aspects. This includes the study of interval exchange transformations, billiards in rational polygons, and Teichm\"uller dynamics. These directions have led to deep and beautiful results, and the field has flourished around them.

\smallskip

\noindent This survey, however, explores a different and perhaps less illuminated side of the theory, a side whose existence is well known but often left in the background. Borrowing from the title of the famous Pink Floyd album, we might call it \textit{the dark side of the translation surface}. The word "dark" is not meant as a negative term, but rather as an invitation to shed light on aspects of translation surfaces that, while fundamental, often remain in the shadow of the more visible dynamics-orientated research.

\smallskip

\noindent In the spirit of Thurston, this survey begins by presenting various equivalent ways of defining translation surfaces, particularly through polygonal models and fundamental domains. It then turns to the structure of moduli spaces of translation surfaces and their natural stratifications. This chapter focuses on the period map and the problem of realising abelian differentials with prescribed absolute and relative periods. In fact, the main purpose of this survey, however, is to present Haupt's theorem and its recent developments. We finally conclude with the geometry and topology of isoperiodic foliations, a subject that has been significantly developed through the pioneering work of McMullen. While this perspective may not be the most visible part of the theory, it offers a rich and geometrically intuitive approach that complements the more widely known dynamical picture.

\bigskip

\noindent \textit{Scope of the present note:} In \cite{Mass}, Massart writes, referring to his own work: 
\begin{quote}
    \textit{"There are already many wonderful introductions to translation surfaces (for instance, \cite{FoMa}, \cite{HuSc}, \cite{Moe}, \cite{AW} or \cite{Zo}), and this one is in no way meant as a substitute to any of them [...]"} 
\end{quote}
\noindent I fully agree with this statement. The cited works are difficult to match and encompass a remarkable wealth of material. 
To these, I would add the recent book by Athreya–Masur \cite{AtMa}, a recent survey by Filip \cite{FS}, as well as Massart’s own paper \cite{Mass}, which I found particularly helpful and thoughtfully written, and which also serves as a gentle introduction to Teichm\"uller curves. For the reasons mentioned above, these notes will not address the dynamical aspects in any way. To semi-quote Massart, this work does not intend in any way to replace any of the previous ones, and I hope it may be helpful to those who wish to explore the "dark side" of translation surfaces.

\subsection*{Acknowledgments} The author wishes to thank Athanase Papadopoulos and Ken'ichi Ohshika for inviting him to write the present survey. The author also wishes to thank Gabriel Calsamiglia, Dawei Chen, Stefano Francaviglia, Subhojoy Gupta, Ursula Hamenst\"adt, Thomas Le Fils, Martin M\"oller, Guillaume Tahar and Yongquan Zhang for all discussions he had with them about translation surfaces and moduli spaces of differentials. I would also like to thank Gye-Seon Lee for inviting me to Seoul National University to give a mini-course on translation structures. Part of these notes were written during that visit, and the course itself was inspired by the material presented here. Finally, I thank an anonymous referee for comments which led to improvements in the present note.

\section{Translation surfaces}\label{sec:trans}

\noindent The present chapter is devoted to introducing the reader to \textit{translation surfaces} from a topological point of view in the spirit of Thurston. These can be defined in several ways depending on whether we adopt the language of geometric structures, or a complex-analytic language or a topological one.

\medskip

\textit{Convention.} In what follows we shall agree with the usual identification \(\mathbb C\cong\mathbb R^2\) determined by the well-known association \(z=x+iy\longleftrightarrow(x,y)\) and hence \(dz^2=dx^2\,+\,dy^2\). In particular \(\mathbb C\) is endowed with its standard complex structure which determines a well defined orientation.  

\medskip

\subsection{Translation surfaces as tiled surfaces}\label{ssec:tiled} We begin with a topological definition because this is certainly the most direct and concrete one, useful when we need to visualise the objects. We first recall that a polygon is said to be \textit{finite and simple} if its boundary is a simple curve, \textit{i.e.} without self-intersections and made of finitely many segments. Let \(\mathcal {P}=\{\,\pol_1,\dots,\pol_n\,\}\) be a finite collection of simple polygons in \(\C\) with disjoint interiors. We shall agree that all polygons in the present chapter are understood in the Euclidean sense.

\subsubsection{Topological definition}\label{sssec:firstdef} We begin by introducing the concept of \textit{pairing via translations}, bearing in mind that the following definition naturally extends to other geometries.

\begin{defn}\label{defn:pairing}
    A finite collection of polygons \(\mathcal P\) in \(\mathbb R^2\) admits a \textit{pairing via translations} if there exists a finite subset \(\tau=\{\,g_s\,\,|\,\, s\in\mathcal S\,\}\) of translations of the plane, indexed by the collection \(\mathcal S\) of all the sides of the polygons in \(\mathcal P\), such that for each side \(s\in \mathcal S\) there is a side \(t\in\mathcal S\) such that
    \begin{itemize}
        \item[1.] \(g_s(\,t\,)=s\),
        \item[2.] the translations \(g_s,\,g_t\) satisfy \(g_s\,g_t(\,z\,)\,=\,g_t\,g_s(\,z\,)=z\), that is \(g_s=g_t^{-1}\),
        \item[3.] if \(s\) is a side of \(\pol_i\) and \(t\) is a side of \(\pol_j\), for some \(1\le i,j\le n\), then \(\pol_i\,\cap\,g_s(\,\pol_j\,)=s\).
    \end{itemize}
\end{defn}

\noindent We observe that it is not necessary to explicitly require that each side in the collection \(\mathcal{S}\) corresponds to a unique side, as this uniqueness is already ensured by condition~(2) in Definition~\ref{defn:pairing}. It is worth noticing that a translation in \(\tau\) is allowed to be trivial, \textit{i.e.} the identity, see Figures in \S\ref{sec:per} for some examples.

\smallskip

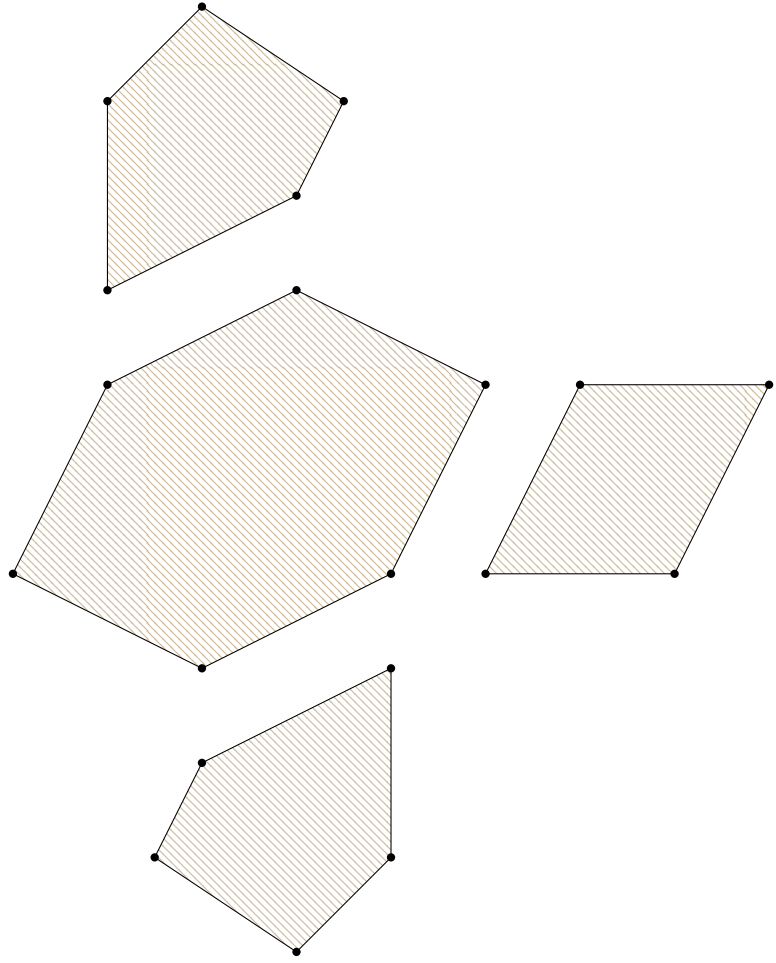
\begin{figure}[!ht]
    \centering
    \begin{tikzpicture}[scale=1.25, every node/.style={scale=0.875}]
    \definecolor{pallido}{RGB}{221,227,227}
    \definecolor{sabbia}{RGB}{207,185,151}


    \pattern [pattern=north west lines, pattern color=sabbia, opacity=1] (7.5,-8)--(11.5,-8)--(9.5,-7)--(7.5,-8);
    \pattern [pattern=north west lines, pattern color=sabbia, opacity=1] (11.5,-8)--(10.5,-10)--(8.5,-11)--(11.5,-8);
    \pattern [pattern=north west lines, pattern color=sabbia, opacity=1] (7.5,-8)--(8.5,-11)--(6.5,-10)--(7.5,-8);
    \pattern [pattern=north west lines, pattern color=sabbia, opacity=1] (11.5,-10)--(13.5,-10)--(14.5,-8 )--(12.5,-8 );
    \pattern [pattern=north west lines, pattern color=sabbia, opacity=1] (7.5,-8)--(11.5,-8)--(8.5,-11)--(7.5,-8);
    \pattern [pattern=north west lines, pattern color=sabbia, opacity=1] (8.5,-12)--(10.5,-11)--(10.5,-13)--( 9.5,-14)--( 8.0,-13)--( 8.5,-12);

    \draw[thin, black] (12.5,-8 )--(11.5,-10);
    \draw[thin, black] (11.5,-10)--(13.5,-10);
    \draw[thin, black] (14.5,-8 )--(13.5,-10);
    \draw[thin, black] (12.5,-8 )--(14.5,-8 );

    \draw[thin, black] (7.5,-8)--(9.5,-7);
    \draw[thin, black] (11.5,-8)--(9.5,-7);
    \draw[thin, black] (11.5,-8)--(10.5,-10);
    \draw[thin, black] (8.5,-11)--(10.5,-10);
    \draw[thin, black] (7.5,-8)--(6.5,-10);
    \draw[thin, black] (8.5,-11)--(6.5,-10);

    \fill[black]       (11.5, -8) circle (1.25pt);
    \fill[black]       ( 7.5, -8) circle (1.25pt);
    \fill[black]       ( 8.5,-11) circle (1.25pt);
    \fill[black]       ( 9.5, -7) circle (1.25pt);
    \fill[black]       (10.5,-10) circle (1.25pt);
    \fill[black]       ( 6.5,-10) circle (1.25pt);
    \fill[black]       (12.5, -8) circle (1.25pt);
    \fill[black]       (14.5, -8) circle (1.25pt);
    \fill[black]       (11.5,-10) circle (1.25pt);
    \fill[black]       (13.5,-10) circle (1.25pt);

    \draw[thin, black] ( 8.5, -12) -- (10.5, -11);
    \draw[thin, black] (10.5, -11) -- (10.5, -13);
    \draw[thin, black] (10.5, -13) -- ( 9.5, -14);
    \draw[thin, black] ( 9.5, -14) -- ( 8.0, -13);
    \draw[thin, black] ( 8.0, -13) -- ( 8.5, -12);

    \fill[black]     ( 8.5, -12)  circle (1.25pt);
    \fill[black]     (10.5, -11)  circle (1.25pt);
    \fill[black]     (10.5, -13)  circle (1.25pt);
    \fill[black]     ( 9.5, -14)  circle (1.25pt);
    \fill[black]     ( 8.0, -13)  circle (1.25pt);

    \pattern [pattern=north west lines, pattern color=sabbia, opacity=1] ( 9.5,-6 )--( 7.5,-7 )--( 7.5,-5)--( 8.5,-4)--(10.0,-5 )--( 9.5,-6 );
    
    \draw[thin, black] (9.5, -6.0) -- (7.5, -7.0);
    \draw[thin, black] (7.5, -7.0) -- (7.5, -5.0);
    \draw[thin, black] (7.5, -5.0) -- (8.5, -4.0);
    \draw[thin, black] (8.5, -4.0) -- (10.0, -5.0);
    \draw[thin, black] (10.0, -5.0) -- (9.5, -6.0);

    \fill[black] (9.5, -6.0) circle (1.25pt);
    \fill[black] (7.5, -7.0) circle (1.25pt);
    \fill[black] (7.5, -5.0) circle (1.25pt);
    \fill[black] (8.5, -4.0) circle (1.25pt);
    \fill[black] (10.0, -5.0) circle (1.25pt);
    
    \end{tikzpicture}
    \caption{A collection of disjoint polygons \textbf{P} that admits a pairing via translations.}
    \label{fig:transurfone}
\end{figure}

\smallskip

\begin{rmk}
    It is worth noting that, according to the definition provided, the polygons in a given collection are not necessarily disjoint, as shown in the Figure~\ref{fig:transurfone}. In the following, many translation surfaces are constructed using a single polygon in the plane subdivided into triangles and/or squares that are not necessarily disjoint. If they are not disjoint, two polygons intersect along a common edge that they share.
\end{rmk}

\noindent We therefore assume that our collection of polygons admits a pairing via translations \(\tau\) and let us denote by \(\mathcal P\) the disjoint union of all polygons of our collection, \textit{i.e.}
\begin{equation}
    \textbf{P}\,=\, \bigcup_{\mathcal P} \pol_i.
\end{equation}
Two points, say \(p,\,q\), are said to be paired if there are two sides, say \(s,\,t\in\mathcal S\), such that \(p\in s\) and \(q\in t\) and a translation \(g_s\) of the plane such that \(g_s(\,q\,)=p\). This pairing of edge points generates an equivalence relation on the set \( \mathcal S \). Let \( \Sigma \) be the topological quotient space of \( \mathcal S \). Then \( \Sigma \) is the space obtained by gluing the polygons in \( \mathcal P \) via the pairing \( \tau \). Its topology depends only on the combinatorial properties of the identification. We have the following:

\begin{prop}\label{prop:mfldstruc}
    \(\Sigma\) is a compact orientable manifold of dimension two, \textit{i.e.} \(\Sigma\) is a compact orientable surface.
\end{prop}

\noindent We postpone the proof to \S\ref{sssec:mfldstruc}. Notice that, in principle, \(\Sigma\) may not be a connected surface and thus be the union of finitely many connected surfaces because the collection \(\mathcal P\) is finite. In what follows we shall require \(\Sigma\) to be connected, just because if it is not, we study its connected components separately. 

\smallskip

\noindent More important is the fact that two distinct collections of polygons that admit a pairing via translations may define the same surface in the following sense: two collections of polygons, say \(\mathcal P\) and \(\mathcal Q\), are considered to define the same surface if one can be cut into pieces along straight lines, and these pieces can be translated and reassembled to form the other collection. Notice that when a polygon is cut into two pieces, the resulting boundary components must be paired, and two polygons can be glued together along a pair of edges only if those edges are paired. We may thus say that \(\mathcal P\) and \(\mathcal Q\) are equivalent collections of polygons. More generally, this criterion yields an equivalence relation on the set of all possible finite collections of polygons in~\(\C\) that admit a pairing via translations.

\begin{defn}\label{defn:firstdef}
    A \textit{translation surface}\index{surface!translation} is an equivalence class of polygons of \(\C\) that admits a pairing via translations. 
\end{defn}

\noindent Before we explore some examples of translation surfaces, let us understand why they are called surfaces in the first place, \textit{i.e.} it is now the right time to show Proposition \ref{prop:mfldstruc}.

\subsubsection{Manifold structure}\label{sssec:mfldstruc} In the present paragraph, we prove that the topological space \(\Sigma\) defined as in \S\ref{sssec:firstdef} is a differentiable surface by providing a differentiable atlas. Let \(\pi\colon\textbf{P} \longrightarrow \Sigma\) be the quotient map, and let \(p \in \Sigma\) be any point. There are three mutually disjoint possibilities as follows.

\smallskip

\noindent In the first place let us assume that \(\pi^{-1}(\,p\,)\) is a singleton. In this case, the preimage of \(p\) is an interior point of some polygon in the collection \(\mathcal{P}\), and as such, it admits an open neighbourhood \(U\) such that the restriction \(\pi_U\colon U \longrightarrow \pi(\,U\,) \subset \Sigma\) is a homeomorphism. Its inverse provides the desired local chart.

\smallskip

\noindent In the case where the preimage of \(p\) is not a singleton, then \(\pi^{-1}(\,p\,)\) must be a finite set because \(\textbf{P}\) is made of finite and simple polygons. See Remark \ref{rmk:chamanara} for a notable counterexample when the polygon has infinitely many sides. Assume that \(\pi^{-1}(\,p\,)\) consists of two points lying in the interior of edges of polygons \(\textnormal{P}_i\) and \(\textnormal{P}_j\). 

\begin{rmk}
    Notice that \(i = j\) is very well admissible. In fact, this is always the case if the collection \(\mathcal P\) is made of a single polygon \textnormal{P}.
\end{rmk}

\noindent In this case, \(p\) has a neighbourhood in \(\Sigma\) whose preimage in \(\textbf{P}\) is the union of two half-disks of equal radius and parallel diameters, each centred at one of the two preimages of \(p\). If the corresponding edges are identified by a translation, we define a chart by mapping one half-disk via the identity and the other via the translation, thus obtaining the desired local homeomorphism onto an open subset of \(\C\).

\smallskip

\noindent The third and the last case is the most subtle to deal with. Assume that \(\pi^{-1}(\,p\,)\) consists of finitely many points, all of which are vertices of polygons in the collection \(\mathcal{P}\), say \(v_1, \dots, v_k\). This happens because the gluing edge identifies only edges of equal length, so vertices are glued only to other vertices. A neighbourhood of \(p\) in \(\Sigma\) lifts to the union of \(k\) circular sectors \(S_1, \dots, S_k\), each of radius \(r\), centred at \(v_i\), and bounded by two edges \(e_i\) and \(f_i\), such that \(f_i\) is parallel to \(e_{i+1}\) for every \(i = 1, \dots, k\) where the indices are taken modulo \(k\).

\smallskip

\noindent Let \(\theta_i\) be the interior angle of the sector \(S_i\) at the vertex \(v_i\). It is a non-trivial fact that the total angle is always an integer multiple of \(2\pi\), that is \(\theta_1+\cdots+\theta_k=2(m+1)\pi\). In fact, this follows directly from the assumption that the polygons are glued together by translations. Without this assumption, such a property is not generally guaranteed. 

\smallskip

\noindent In what follows, we shall call \textit{conical singularity} or \textit{cone point} every point such that \(m\ge1\). Notice that these points must arise from the identification of vertices. Following Massart in his very nice survey \cite{Mass}, we define a chart around \(p\) by mapping each sector \(S_i\) onto a standard circular sector centred at the origin. Setting \(\alpha_i\) as the partial sum \( \alpha_i = \theta_1+\cdots+\theta_{i-1} \), we define the chart by
\begin{equation}\label{eq:chartaroundvertex}
     v_i + \rho\, \exp\Big(\,i\big(\,\theta\, +\, \alpha_i\,\big)\,\Big) \longmapsto \rho\, \exp\left(\,\frac{i}{m+1}\big(\,\theta\, +\, \alpha_i\,\big)\,\right),
\quad \text{for } 0 \leq \rho \leq r,\, 0 \leq \theta \leq \theta_i.
\end{equation}

\noindent This defines a homeomorphism onto an open subset of \(\C\). We thus have a collection of charts, and it is not difficult to see that the transition functions are restrictions of global translations or the identity. As a consequence, our topological space is indeed a manifold, and it is orientable because transition maps, being smooth functions, have Jacobian matrices with positive determinant. Notice that, by design, each local chart is an orientation-preserving map. It remains to show that the surface \(\Sigma\) is compact. This follows from the fact that our collection \(\mathcal{P}\) consists of finitely many compact polygons. Since \(\mathcal{P}\) is compact and the quotient projection is continuous, the claim follows.

\begin{rmk}\label{rmk:chamanara}
    In \cite{Cha}, Chamanara introduced a translation surface of infinite type obtained from a unit square on which each side is subdivided into infinitely many segments. Every side of the square is recursively halved, and edges on opposite segments are identified by translations. This process yields a polygon with infinitely many sides, and the resulting surface has infinitely many handles. Notably, the identification map fails to be finite-to-one: the quotient space includes a single conical singularity arising from the identification of all cutting points and corners in the metric completion.
\end{rmk}

\subsubsection{Tilings}\label{sssec:tilings} Strictly speaking, a tiling is a subdivision of a geometric surface \( \Sigma \) into non-overlapping polygons, \textit{i.e.} polygons with disjoint interiors, called \emph{tiles}, such that the total angle around each vertex is \( 2\pi \).  Without going into the general theory of tilings, which is rich and fascinating enough to have inspired numerous surveys on the topic, for instance see~\cite{GS}, we shall limit ourselves to a few remarks concerning tilings by Euclidean polygons. It is a well-known topological fact that the only closed surface that admits a true tiling (in the strict sense above) is the torus. However, as soon as we relax the rather restrictive condition that vertices must have total angle \( 2\pi \), we arrive at a broader notion of tiling. In this more general setting, the class of surfaces that can be constructed by gluing together Euclidean polygons along their edges by translations becomes much wider. Example~\ref{ex1} below shows, for instance, that any surface of positive genus can be realised by identifying sides of a single \( 4g \)-gon. For our purposes, as we will see, special importance will be given to the class of surfaces known in the literature as \emph{square-tiled surfaces}, see Example~\ref{ex2}.

\smallskip

\subsection{Some examples} We provide here a few explicit examples of translation surfaces. The collection of polygons in Figure \ref{fig:transurfone} determines a translation surface, and the reader is invited to find an explicit pairing via translation to determine it.

\smallskip

\begin{ex}[Regular polygons]\label{ex1}
    Naturally, the simplest example is that of a square or, more generally, a parallelogram embedded in the complex plane. By identifying the opposite edges of the parallelogram by using translations, the resulting surface is a torus. This example can be easily generalised as follows: every regular \(4g\)-gon embedded in \(\mathbb{C}\) has \(2g\) pairs of parallel sides. The edges of each pair can be identified via a translation, and the resulting surface is a closed surface of genus \(g\). 
\end{ex} 


\begin{ex}[Square-tiled surfaces\index{surface!square-tiled}]\label{ex2}
    One of the most celebrated and remarkable families of examples of translation surfaces is that of the square-tiled surfaces. These structures are obtained by gluing together a finite collection of unit squares. In the minimal case, consisting of a single square, one obtains a flat torus \(\mathbb{T}\). For square-tiled surfaces, Definition \ref{defn:firstdef} is equivalent to the following construction. Let \(\mathcal {P}\) be a collection of \(N>0\) squares labelled from from \(1\) to \(N\). Then a square-tiled surface is determined by a pair of permutations \((\sigma_h, \sigma_v)\) in the symmetric group \(\mathfrak S_N\), considered up to equivalence by simultaneous conjugation, where the upper horizontal edge of the \(i\)-th square is glued to the lower horizontal edge of the \(j\)-th square if \(j\) is the successor of \(i\) in \(\sigma_v\). Similarly, the right vertical edge of the \(i\)-th square is glued to the left vertical edge of the \(j\)-th square if \(j\) is the successor of \(i\) in \(\sigma_h\). By construction, on the resulting surface \(\Sigma\) there is a well defined map \(f\colon \Sigma\longrightarrow \mathbb T\) which is a covering of degree \(N\) branched at a single point. As we shall see in \S\ref{sec:per}, in what follows this realisation plays a crucial role.
\end{ex}

\medskip

\subsection{Gauss-Bonnet condition}\label{ssec:gbcond} In the following paragraph, we aim to address the following question: 

\begin{quote}
    \textit{how many conical singularities can a translation surface admit?}
\end{quote}

\noindent It is worth noting that our definition above allows for the presence of singular points, but does not guarantee their existence, nor does it impose any bound on their magnitude. In this section, we therefore set out to count them.

\smallskip

\begin{prop}[Gauss--Bonnet's formula]\label{prop:gbcond}
    Let \(\Sigma\) be a translation surface of genus \(g\) with cone points of total angle \(2(\,m_1+1\,)\pi,\dots,2(\,m_k+1\,)\pi\). Then
    \begin{equation}\label{eq:gbcond}
        2g-2\, -\, \sum_{i=1}^k\,m_i\,=0
    \end{equation}
\end{prop}

\begin{proof}
    There are several ways to prove the Gauss--Bonnet's formula; we provide the most topological one which relies on the well-known formula
    \begin{equation}
        \chi(\,\Sigma\,)\,=\,V\,-\,E\,+\,F,
    \end{equation}
    where \(V,\,E,\,F\) are the numbers of vertices, edges and faces of any triangulation of \(\Sigma\). Since a translation surface is realised by gluing Euclidean polygons by means of translations, it naturally admits a triangulation, say \(\tau\), into Euclidean triangles such that all cone points are vertices of \(\tau\). Recall that, for a Euclidean triangle \(T\) with interior angles \(\alpha, \beta, \gamma\), we have
    \begin{equation}
    \int_{T} k\, dA = \alpha + \beta + \gamma - \pi,
    \end{equation}
    where \(k\) is the Gaussian curvature. In the flat case, \(k= 0\), and the formula shows that the total angle sum equals \(\pi\), unless there is a singularity. Let \(\Delta=\{\,p_1,\dots,p_k\,\}\subseteq\,V\) be the set of singular points--notice that not all vertices are cone points. For every vertex \(v\in\, V\), set \(\theta(\,v\,)\) as
    \begin{equation}\label{eq:weight} 
        \theta(\,v\,)\,=\,
        \begin{cases}
            \,\,2\,(\,m_i\,+\,1\,)\pi \quad \text{ if } v=p_i\\
            \,\,2\pi \qquad \qquad    \quad \text{ if } v \text{ is a regular point}
        \end{cases}.
    \end{equation}
    Since \(\tau\) is a triangulation then \(2E=3F\) because each face is adjacent to three edges, equivalently, each edge is adjacent to two faces. The triangle being Euclidean, then we have the equality 
    \begin{equation}\label{eq:ordersum}
        \sum\theta(\,v\,)\,=\,\pi F.
    \end{equation}
    As a consequence, we get the following
    \begin{equation}
        2\pi\,\chi(\,\Sigma\,)\,=\,2\pi\,V-2\pi\,E+2\pi\,F\,\overset{\eqref{eq:ordersum}}{=}\,2\pi\,V-\sum\theta(\,v\,)\,\overset{\eqref{eq:weight}}{=}\,-2\pi\,\sum_{i=1}^k m_i,
    \end{equation} 
    that is
    \begin{equation}
        2g-2\,-\,\sum_{i=1}^k \,m_i=\,0
    \end{equation}
    as desired.
\end{proof}

\noindent As a direct consequence, we derive a characterisation for low-genus surfaces, \textit{i.e.} \(g\le1\).

\begin{cor}\label{cor:speccases}
    There are no translation surfaces of genus zero. Every translation surface of genus one has no conical singularities. 
\end{cor}

\smallskip

\noindent The result of this paragraph can be summarised as follows: for closed surfaces, there cannot be too many singularities, but neither can there be too few. Moreover, the singularities that do occur cannot have arbitrary order. The Gauss–-Bonnet condition thus provides a balance between the number of conical singularities and their angular magnitudes. What is truly remarkable is that this balance depends solely on the topology of the surface—in particular, on its genus. This is truly extraordinary. 

\smallskip

\subsection{Geometric structures and singular Euclidean metrics}\label{ssec:seconddef} In \S\ref{ssec:tiled}, we introduced translation surfaces as topological surfaces obtained by gluing together finitely many simple polygons in the Euclidean plane via translations. To ensure that the resulting structure is indeed a surface, it was necessary to define a suitable atlas of charts--see \S\ref{sssec:mfldstruc}. This atlas satisfies a key property: the transition maps are either trivial or translations of the plane. In particular, every translation surface admits a unique maximal atlas with this property. In this section, we show that this condition completely characterises translation surfaces in this sense. That is, if a topological surface~\(\Sigma\) is endowed with such an atlas, then it admits a triangulation made of Euclidean triangles. 

\smallskip

\noindent Recall that a generic translation surface may have special points, defined as conical singularities. Although this was not made explicit at the time, around every cone point the definition of charts is essentially given by taking the \((m+1)\)-th root. To better formalise this concept, and thus define translation surfaces in terms of atlases, it is useful to first introduce the following models.

\subsubsection{A small cone}\label{sssec:smallcone} We introduce a cone that serves as a local model. A \textit{cone of angle} \( \pi \) is defined as the quotient space obtained as follows:
\begin{equation}
    \mathcal{C}(\,\pi\,)\,=\,\left\{\,(\,\rho,\,\theta\,) \,\,\big|\,\,\rho\ge0 \text{ and } 0\le \theta\le\pi\,\right\}\,/\, (\,\rho,\,0\,)\sim(\,\rho,\,\pi\,)
\end{equation}
with the metric \(ds^2=d\rho^2+\rho^2\,d\theta^2\), where \((\,\rho,\,\theta\,)\) are the polar coordinates on the closure of the upper half-plane. The \textit{tip of the cone} is the unique point with coordinate \(\rho=0\). Let \(q(\,z\,)\,=\,z^{-1}\,dz^2\) be the quadratic differential on \(\mathbb C\) with a single simple pole at the origin. The following holds, see \cite{Tro86}:

\begin{lem}
    \( \mathcal{C}(\,\pi\,)\) is isometric to \(\mathbb C\) endowed with the flat metric induced by \(q\).
\end{lem}

\begin{proof}
    We first recall that the flat metric induced by \(q\) is \( ds^2\,=\,|\,z^{-1}\,||\,dz\,|^2\). Then consider the function:
    \begin{equation}
        \big(\,\rho,\,\theta\,\big) \longmapsto \left(\,\displaystyle\frac14\,\rho^2\cos(\,2\theta\,)\, ,\displaystyle\frac14\,\rho^2\sin(\,2\theta\,)\,\right)\,=\,\big(\,x,\,y\,\big).
    \end{equation}
    It is a routine exercise to show that the given function realises the desired isometry.
\end{proof}

\noindent We conclude by noting that this construction is a special case of a more general framework, capable of modelling arbitrary cone angles, that is, any angle not greater than \(2\pi\), which we have intentionally omitted so as to avoid unnecessary technical digressions.


\subsubsection{A wider cone}\label{sssec:bigcone} We now proceed to employ the model just defined to introduced bigger cones. Take \( m+1 \) copies of the upper half-plane with boundary, each equipped with the standard flat metric, and similarly \( m+1 \) copies of the lower half-plane with boundary. Label both the upper and lower half-planes from \( 1 \) to \( m+1 \), and glue them as follows: the half-infinite ray \([\,0, \infty\,)\) of the \(i\)-th half-plane is glued to \((\,-\infty, 0\,]\) of the \((i+1)\)-th one, where the indices are taken modulo \( m+1 \). The gluing can be carried out in such a way that the Euclidean metrics on the individual half-planes extend to a singular Euclidean metric on the entire disc, with a unique singularity at the origin having total angle \( 2(m+1)\pi \). Notice that the resulting space is topologically a disc; we define this space as a \textit{Euclidean cone of order}\index{euclidean cone!order} \(m+1\). In this case, the origin plays the role of the \textit{tip of the cone}.


\smallskip

\noindent The terminology is motivated by the fact that there is a well defined branched covering map over the cone \(\mathcal C(\,\pi\,)\) defined as follows:

\begin{equation}\label{eq:projcone}
    \mathbb C\, \overset{g}{\longrightarrow}\,\mathbb C\,\overset{f}{\longrightarrow}\,\mathcal C(\,\pi\,)
\end{equation}
where \(f(\,z\,)=z^2=w\) and \(g(\,u\,)=u^{m+1}=z\). Their composition \(\pi(\,u\,)=u^{2m+2}\) is branched at the tip of the cone \(\mathcal C(\,\pi\,)\). 

\begin{lem}\label{lem:singgeo}
    The quadratic differential \(\pi^*q(\,u\,)\) defines a singular Euclidean metric on \(\mathbb C\) with a singularity of order \(m\) at the origin.
\end{lem}

\begin{proof}
    This is a straightforward and rather simple calculation. In the local coordinate \(u\) we have
    \begin{equation}
        \pi^*q(\,u\,)=\left(\,\frac{d\pi}{du}\,\right)^2\,\frac{du^2}{\pi(\,u\,)}=\,(\,2m+2\,)^2\,u^{2m}\,du^2.
    \end{equation}
    Thus \(\pi^*q\) defines a singular Euclidean metric with a singularity of order \(m\) at the origin.
\end{proof}


\noindent In reference to the concluding remark of \S\ref{sssec:bigcone}, the construction just presented allows for the realisation of cones with arbitrary angle greater than \(2\pi\). These are useful in describing flat structures with more general singularities, which lie beyond the scope of this survey. However, the reader can easily retrace the construction in order to define the singular Euclidean structure on a cone with an arbitrary angle.

\medskip

\subsubsection{Definition via atlases}\label{sssec:sencdef} We are now in a position to provide a second, more intrinsic definition of translation surfaces.

\begin{defn}[Second definition]\label{def:second_translation_surface} 
Let \(\Sigma\) be a closed, topological orientable surface. A \textit{translation structure}\index{structure!translation} on \(\Sigma\) consists of a finite set \(\Delta \subset \Sigma\), possibly empty, and an atlas of \( \mathbb{C}\)-valued charts on \( \Sigma \setminus \Delta \), extended by requiring that for each point \( p \in \Delta \) there exists an integer \( m > 0 \) and a homeomorphism from a neighbourhood of \( p \) to a neighbourhood of the origin in the Euclidean cone of order \( m+1 \), such that transition maps are translations on their overlaps. A \textit{translation surface}\index{surface!translation} is a closed topological surface endowed with a translation structure.
\end{defn}

\noindent A point \( p \in \Delta \) is called a \textit{conical singularity} whereas all other points are called \textit{regular}. It has a cone angle of magnitude \( 2(m+1)\pi \), since it can be obtained by gluing \( 2m + 2 \) half-planes, each of which with angle \( \pi \) at the origin. The term \textit{cone point} is a synonym for \textit{singularity}. In what follows, it will be useful to adopt the following standard terminology: the integer \(m\) is defined as the \textit{order of a cone point}, and it essentially measures how many times the cone angle exceeds \(2\pi\). Some remarks are in order.

\begin{rmk}\label{rmk:orientation}
    A translation structure on a closed topological surface \(\Sigma\) determines a natural orientation that makes the local charts orientation-preserving.
\end{rmk}

\begin{rmk}
    We observe that a Euclidean cone of order one is nothing but the complex plane itself. Therefore, one could define a translation surface as a geometric structure in which every point is modelled on a suitable cone, without explicitly stating in advance that the set of singular points is finite (recall that we are assuming the surfaces to be compact).
\end{rmk}

\begin{rmk}\label{rmk:branchedchart}
    In principle, every chart taking values in a cone of order \(m \geq 1\) can be replaced by its post-composition with the map \(g\), as defined in \S\ref{sssec:bigcone}. Such a post-composition is usually called a \textit{branched chart}, since it fails to be a homeomorphism onto its image. 
    Consequently, one could define translation structures as maximal atlases of charts, possibly branched, whose transition maps are translations on their overlaps. It is in fact not uncommon to find in the literature definitions that favour this second approach over the former, as it allows one to define translation surfaces, and more generally, structures with conical singularities, without first introducing a conical model.
\end{rmk}

\noindent In its simplicity, it is a non-trivial fact that the two definitions are equivalent. We will provide a proof of this fact in \S\ref{sssec:equivalentone}. We first provide a description of the local geometry.

\subsubsection{Local geometry}\label{sssec:locgeo} By using the definitions just given, 
we aim to describe the geometry around a point. We first consider a regular point. A local chart, say \( \varphi_\alpha\colon U_\alpha \longrightarrow\mathbb{C} \), is a homeomorphism onto its image \( \varphi_\alpha(U_\alpha) \subset \mathbb{C} \). As a consequence, the standard Euclidean metric, restricted to the image of the chart, pulls back to a Euclidean metric, say \( ds^2_\alpha \), defined on the open set \( U_\alpha \subset \Sigma \). Each chart centred at a regular point defines a local Euclidean metric, and these local metrics extend to a Riemannian metric on \( \Sigma \setminus \Delta \). Notice that it is locally Euclidean on \( \Sigma \setminus \Delta \), because the transition maps are translations. The Euclidean metric just defined on \(\Sigma\setminus\Delta\) extends over \(\Delta\) to a singular Riemannian metric with conical singularities and Lemma \ref{lem:singgeo} provides an explicit closed form. In what follows we refer to these metrics as singular Euclidean metrics.

\smallskip

\noindent One might be tempted to erroneously define a translation structure as the data of a singular Euclidean metric with conical singularities of angle \(2(m+1)\pi\) on a closed surface which is locally Euclidean away from the singularities. However, for such a metric to define a genuine translation structure, it is essential that the parallel transport it induces on \(\Sigma \setminus \Delta\) be trivial. Example \ref{ex3} in \S\ref{sssec:north} may clarify better this point.

\subsubsection{Showing the equivalence of definitions --- part one}\label{sssec:equivalentone} We now provide a proof of the equivalence between Definitions \ref{defn:firstdef} and \ref{def:second_translation_surface}. One implication is subsumed in \S\ref{sssec:mfldstruc}. In fact, we have already shown the existence of an atlas of charts such that transition functions are translations on their overlaps. We just clarify that the chart defined by Equation~\eqref{eq:chartaroundvertex} around a generic vertex is not a homeomorphism onto a cone of angle \(2(m+1)\pi\), but rather the associated branched covering, as explained in Remark~\ref{rmk:branchedchart}. By lifting this map to the Euclidean cone of order \(m+1\) one gets the desired chart as in Definition \ref{def:second_translation_surface}.

\smallskip

\noindent We only need to show the opposite implication. Let us then consider a surface endowed with a maximal atlas as defined in Definition~\ref{def:second_translation_surface}. We take an arbitrary open cover and then a finite sub-cover. The key fact now is that a surface can be triangulated. We therefore choose a suitable refinement such that the following properties hold:
\begin{itemize}
    \item[1.] each triangle is entirely contained in a chart,
    \smallskip
    \item[2.] and each singular point is a vertex of the triangulation.
\end{itemize}
\noindent We then deform the given triangulation by a homotopy that fixes all the vertices pointwise, so as to obtain triangles that are isometric to Euclidean triangles. The resulting triangulation yields the desired family of polygons as in Definition \ref{defn:firstdef}.

\subsection{Developing-holonomy pair}\label{ssec:devhol} What follows is a specific instance of a much more general construction that applies to an arbitrary geometric structure. As above, we do not dwell on details that are unnecessary for the purposes of this chapter. We begin by defining the developing map. Historically, this concept dates back to Ehresmann \cite{Eh36} and was subsequently revisited by Thurston; see \cite{Thu97}.

\smallskip 

\noindent Let \(\Sigma\) be a closed surface of genus \(g\) and let \(\widetilde\Sigma\) be its universal cover. Since there are no translation structures on the sphere, there is no loss of generality in assuming \(g\ge1\). A translation structure on a surface \(\Sigma\) naturally lifts to a translation structure on \(\widetilde\Sigma\) by making the covering projection a local isometry. A \textit{developing map}\index{developing map} \(\textnormal{dev}\colon \widetilde\Sigma\longrightarrow\mathbb C\) is a local embedding away from the singularities of \(\widetilde\Sigma\) such that 
\begin{itemize}
    \item[1.] its restriction to any sufficiently small open neighbourhood of a regular point is a local chart for the translation structure on \(\widetilde\Sigma\), and
    \smallskip
    \item[2.] its restriction to any sufficiently small open neighbourhood of a conical singularity is a branched chart as in Remark \ref{rmk:branchedchart}. In other words there is a local coordinate \(u\) such that \(\textnormal{dev}(\,u\,)\,=\,u^{m+1}\) for some integer \(m\ge1\). The developing map clearly fails to be a local embedding around a conical singularity because its Jacobian vanishes.
\end{itemize}

\noindent Notice that as a consequence of Remark \ref{rmk:orientation}, a developing map is an orientation-preserving local isometry away from the singularities.

\begin{rmk}\label{rmk:devgenusone}
    Recall that genus one surfaces admit translation structures without conical singularities, see Corollary \ref{cor:speccases}. In this case, a developing map is not only a local embedding, but indeed a global one. This is a matter of Riemannian geometry. A torus endowed with a translation structure is a compact Riemannian manifold, as the local charts define a complete Riemannian metric without conical singularities, and the pull-back metric on its universal cover is also complete. Since the developing map is an orientation-preserving local isometry, a standard result from Riemannian geometry asserts that it must be an orientation-preserving global isometry, in particular, a global embedding.
\end{rmk}

\noindent A developing map for a translation structure can be constructed geometrically by unfolding the surface into the Euclidean plane. As already hinted above, a translation structure naturally determines a triangulation of the surface, which lifts to a triangulation of \( \widetilde{\Sigma} \). 
Moreover, up to passing to a finer triangulation, we may assume that each triangle is entirely contained within a single coordinate chart. This seemingly simple observation leads to an alternative definition of developing maps: they can be described as continuous maps that restrict to isometries on each individual triangle as follows.

\smallskip

\noindent Suppose two adjacent triangles \( T_1 \) and \( T_2 \) share an edge \( e \), and let \( f_1\colon T_1 \to \mathbb{R}^2 \) be an isometric embedding. Then there exists a unique isometric embedding \( f_2\colon T_2 \to \mathbb{R}^2 \) such that \( f_1(\,e\,) = f_2(\,e\,) \) and the images of \( T_1 \) and \( T_2 \) under \( f_1 \) and \( f_2 \), respectively, have disjoint interiors. Gluing \( f_1 \) and \( f_2 \) along the shared edge defines a map \( f_e = f_1 \cup f_2\colon T_1 \cup T_2 \to \mathbb{R}^2 \). By iteratively unfolding adjacent triangles across shared edges, one defines an isometric map from an increasing union of triangles in \( \widetilde{\Sigma} \) to the Euclidean plane. Since the universal cover is simply connected, this process yields a globally defined developing map \( \widetilde{\Sigma} \longrightarrow \mathbb{R}^2 \simeq \mathbb{C} \), uniquely determined up to a translation.

\smallskip

 \noindent The developing map \( f\colon \widetilde{\Sigma} \longrightarrow \mathbb{C} \) of a translation structure on \( \Sigma \) satisfies an equivariance condition with respect to the fundamental group action on the universal cover. Specifically, for each element \( \gamma \in \pi_1(\,\Sigma\,) \), the composition \( f \circ \gamma \) is also a developing map for the same translation structure. Since a developing map is defined up to translation, there exists a translation \( c\in\mathbb C \) such that  
\begin{equation}
f \circ \gamma = f\,+\,c.
\end{equation}
The correspondence  
\begin{equation}
\rho \colon\pi_1(\,\Sigma\,) \longrightarrow\mathbb{C}, \quad \gamma \mapsto c
\end{equation} 
defines a group homomorphism. Notice that post-composition of the developing map with any translation preserves \(\rho\). By adopting the language of geometric structures, such a representation is called the \textit{holonomy}\index{holonomy} representation\index{representation!holonomy} of the translation structure. We may already observe that, since the target is abelian, the representation reduces to a representation on homology, namely \(\chi\colon\shomolz\longrightarrow\mathbb C\), well-known in literature as \textit{period character}\index{period!character}. We will return to this in \S\ref{ssec:perchar}.

\subsection{Complex-analytic perspective}\label{ssec:cadef} The definitions of translation surfaces given above both have a geometric flavour: the first is certainly the most intuitive and also the most convenient for certain applications. There is a third definition, which is arguably the simplest to state but at the same time the most abstract. The next definition we will introduce is of a complex-analytic nature and establishes a non-trivial and very interesting connection between low-dimensional topology, algebraic geometry, and complex analysis.

\smallskip

\subsubsection{Definition via abelian differentials}\label{sssec:thirdef} Let \(X\) be a compact Riemann surface, that is, an irreducible smooth projective algebraic curve over \(\mathbb C\). An \textit{abelian differential}\index{differential!abelian} on a Riemann surface is a holomorphic \(1\)-form (or differential)\index{differential!holomorphic}, that is, a global section of the cotangent bundle of \(X\).  

\begin{defn}[Third definition]\label{defn:thirddef}
    A translation\index{surface!translation} surface is the datum of a non-zero abelian differential, say \(\omega\), on a Riemann surface \(X\). In this setting a translation surface is denoted as a pair \((X,\omega)\).
\end{defn}

\noindent For the time being, we skip details about the moduli space of such pairs and will return to them later on in \S\ref{ssec:hodge}. We want to focus on how this definition is equivalent to the other. The leading question of the present paragraph is 
\begin{quote}
    \textit{why should we call a pair \((X,\omega)\) translation surface when translations do not even appear in the definition?}
\end{quote}

\noindent Translations do appear, though not as explicitly as in the previous cases. Here is the point: a Riemann surface is defined as a maximal atlas of \(\mathbb C\)-charts, with transition maps that are biholomorphisms on their domains of definition. A holomorphic 1-form singles out a sub-atlas that is maximal with respect to the condition that transition maps are translations. In particular, two different holomorphic 1-forms on the same Riemann surface may define different sub-atlases! While a direct equivalence with the first definition is less easy to establish, the equivalence between this third definition and the second one is relatively straightforward.

\subsubsection{Showing the equivalence of definitions --- part two}\label{sssec:equivalenttwo} We now provide a proof of the equivalence between Definitions \ref{def:second_translation_surface} and \ref{defn:thirddef}. 

\smallskip

\noindent One implication is a consequence of our discussion in \S\ref{sssec:locgeo}. Suppose we are given a translation structure on a closed surface \(\Sigma\). Thus we have an atlas of charts to \(\mathbb C\) whose transition functions are translations on their overlaps. Since these are biholomorphisms, we get a well defined complex structure on \(\Sigma\setminus \Delta\). We use these charts to define a holomorphic \(1\)-form on the open surface \(\Sigma \setminus \Delta\) and we set \(\omega = dz\) for any such local coordinate \(z\). 

\smallskip 

\noindent The interesting and non-trivial part is to define the complex structure and the differential at the conical points. The key observation now is the following

\begin{lem}\label{lem:complexconem}
    For every integer \( m \geq 0 \), the Euclidean cone of order \( m+1 \) carries a unique complex structure that makes it biholomorphic to the complex plane. It is, in fact, the Riemann surface that makes the \((m+1)\)-root a single-valued function.
\end{lem}

\noindent By Definition \ref{def:second_translation_surface}, for each singular point, say \(p\), there exists some \( m > 0 \) and a homeomorphism, say \( \varphi \), from a neighbourhood of \( p \) to a neighbourhood of the tip of the Euclidean cone of order \( m+1 \). The complex structure is pulled back by extending the complex structure already defined on \( \Sigma \setminus \Delta \). Even more explicitly, a neighbourhood of the tip of the Euclidean cone is biholomorphic to a neighbourhood of the origin in \( \mathbb{C} \); precomposing this biholomorphism with \( \varphi \) yields a chart for the complex structure. We thus obtain a complex structure \( X \) on the topological surface \( \Sigma \). \\
\noindent By using again the map \( g \) as defined in \eqref{eq:projcone}, the differential \( dz \) on \(\mathbb C\) can be pulled-back to a differential on the Euclidean cone which, in the local coordinate \( u \), takes the form \( \omega = u^m du \); pulling it back via the homeomorphism \( \varphi \) defines the desired holomorphic 1-form in a neighbourhood of the singular point. We have the desired pair \((X,\omega)\) as in Definition \ref{defn:thirddef}.

\medskip

\noindent It remains to show that every pair \((X,\omega)\) defines a translation structure as in Definition \ref{def:second_translation_surface}, thus establishing the equivalence with the complex-analytic definition. The starting point in this case is the following result of complex analysis.

\begin{lem}\label{lem:locformabdiff}
    Let \(X\) be a Riemann surface and let \(\omega\) be a holomorphic \(1\)-form. At any point where \(\omega\) is non-zero, there exists a local coordinate \(z\) for \(X\) in which \(\omega = dz\). At any point where \(\omega\) has a zero\index{differential!zero} of order\index{zero!order} \(m\), there exists a local coordinate \(z\) in which \(\omega = z^m dz\).
\end{lem}

\noindent In a neighbourhood of a point \(p\) which is not a zero of \(\omega\), a local coordinate is defined as
\begin{equation} \label{eq:local_coordinate}
z(\,q\,) = \int_p^q \omega
\end{equation}
in which \(\omega = dz\), and the coordinates of two overlapping neighbourhoods differ by a translation \(z \mapsto z + c\) for some \(c \in \mathbb{C}\). Thus we have an atlas of charts on \(\Sigma\setminus\Delta\) to \(\mathbb C\) whose transition maps are translations. Around a zero, say \(p\), of order \(m \geq 1\), there exists an open neighbourhood \(U\) and a local coordinate \(z\) on \(U\) such that \(\omega = (m+1)\,z^m\,dz\). We may observe that this is the pull-back of the form \(dz\) via the branched covering map \(\varphi\colon U\longrightarrow \mathbb C\) defined as \(\varphi(\,z\,) = z^{m+1}\), since
\begin{equation}\label{eq:branchedchart}
d(\,z^{m+1}\,) = (m+1)\,z^m\,dz.
\end{equation}
By lifting this map to the Euclidean cone of order \(m+1\), we have a homeomorphism from a neighbourhood of \(p\) to a neighbourhood of the tip of the Euclidean cone and hence the desired chart. We have thus defined a translation structure as in Definition \ref{def:second_translation_surface}.

\begin{rmk}\label{rmk:gbanal}
    We observe that the proof also tells us that the set of conical singularities coincides with the set of zeros of \(\omega\). For this reason, in what follows, both will be denoted by \(\Delta\). This should not be too surprising. Indeed, the complex-analytic version of the Gauss--Bonnet theorem asserts that a holomorphic differential on a compact Riemann surface has exactly \(2g - 2\) zeros, counted with multiplicity. 
\end{rmk}

\subsubsection{Where is the north?}\label{sssec:north} We introduce one final concept that will be useful in what follows. In the first place, we observe that in the complex plane there are two main directions given by the real and imaginary axes, which we can visualise as the E-W and N-S directions on a map. In particular, there is a well defined direction, \textit{the north}, given by the positive imaginary axis. These directions determine two transverse foliations of the complex plane, called respectively the \textit{horizontal foliation}\index{foliation!horizontal} and the \textit{vertical foliation}\index{foliation!vertical}.

\smallskip

\noindent Let now \((X,\omega)\) be a translation surface and let \(\Delta\) denote the set of zeros of \(\omega\). By using the associated translation structure, the horizontal foliation pulls back to a well defined horizontal foliation on \(X\setminus\Delta\), since the coordinate changes are translations. In other words, at every regular point of the surface, a north direction is well defined. This foliation extends to \(\Delta\) as a singular foliation: if a singular point has order \(m\), then there are exactly \(m+1\) possible north directions. 

\smallskip

\noindent An interesting consequence is that the parallel transport associated with the locally Euclidean Riemannian metric on \(X\setminus \Delta\) turns out to be trivial. More precisely, the parallel transport induced by the flat connection yields a homomorphism
\begin{equation}
\pi_1(\,X \setminus \Delta, x_o\,) \longrightarrow \mathrm{SO}(2,\mathbb{R}) \cong \mathbb{S}^1
\end{equation}
which acts on the tangent space of \(x_o\). Since \(\mathbb{S}^1\) is abelian, such a homomorphism factors in homology, and hence we have a well defined homomorphism
\begin{equation}
\mathrm{PT}\colon \textnormal{H}_1(\,X \setminus \Delta,\, \mathbb{Z}\,) \longrightarrow \mathrm{SO}(2,\mathbb{R}) \cong \mathbb{S}^1.
\end{equation}
Since the coordinate changes are translations, the parallel transport of a vector pointing north along any closed path in \(X\) avoiding the zeros of \(\omega\) returns the vector to itself. Conversely, if the parallel transport is trivial, then a well defined north direction exists away from the singularities of \(\omega\). In general, a flat metric on a surface does not have trivial parallel transport, which means that the surface is obtained by gluing polygons via isometries of the plane, such as translations and also rotations. For this reason, in \S\ref{sssec:locgeo}, we claimed that a singular Euclidean metric does not necessarily determine a translation structure, see Examples \ref{expoly} and \ref{ex3}. The following criterion holds.

\begin{prop}\label{prop:tripartrans}
    A translation structure is the datum of a singular Euclidean metric on a surface \(\Sigma\) with trivial parallel transport.
\end{prop}

\begin{ex}\label{expoly}
    Every polyhedron can be realised by gluing together a finite collection of Euclidean polygons to form a topological sphere. Famous examples include the Platonic solids. For instance, the cube is topologically a sphere obtained by gluing squares in the familiar way. If a polyhedron has vertices with rational angles, \textit{i.e.} of the form \(2\pi\frac{p_i}{q_i}\), we define \(d\) as the least common multiple of the \(\{\,q_i\,\}\). 
    Let \(\Delta\) be the subset of cone angles of some convex polyhedron \(\mathfrak P\). Notice that \(\mathfrak P\setminus\Delta\) is a punctured sphere. The homology group \(\textnormal{H}_1(\mathfrak P\setminus\Delta,\,\mathbb Z)\) is generated by loops around the punctures and, since there are finitely many punctures, the image of the parallel transport is a finite group of \(\mathbb S^1\) of order \(d\). Let \(\Sigma\) be the smallest cover of \(\mathfrak P\) branched at \(\Delta\) for which the parallel transport homomorphism has trivial image. That is, there is a branched covering map \(f\colon\Sigma\longrightarrow \mathfrak P\) that arises as the completion of the covering of \(\mathfrak P\setminus\Delta\) associated with the kernel of the parallel transport. Hence \(\Sigma\) supports a translation structure as in Proposition \ref{prop:tripartrans}.
    Let \(\mathcal{P}\) be the collection of Euclidean polygons used to construct the original polyhedron, and let \(\mathcal{Q}\) be the collection obtained by taking \(d\) copies of the polygons in \(\mathcal{P}\). The monodromy of \(f\) provides an explicit recipe for how to glue together the polygons in \(\mathcal{Q}\). The resulting surface is homeomorphic to \(\Sigma\) by design.

\end{ex}

\begin{ex}\label{ex3}
    Given a rectangle in the plane with side lengths \(3a\) and \(b\), we subdivide the upper segment into three edges of equal length by adding two vertices. The resulting polygon is thus a hexagon with four right angles and two angles of magnitude \(\pi\). We observe that the sum of the internal angles is \(4\pi\), as expected. We now take four copies of this rectangle: two oriented as the original one, and two rotated by \(\pi\). We glue the sides of the hexagon appropriately, as shown in Figure \ref{fig:halftrans1}, in order to obtain a larger rectangle with twice the original size and with two slits in the middle. 
    
\begin{figure}[!ht]
    \centering
    \begin{tikzpicture}[scale=1.75, every node/.style={scale=0.875}]
    \definecolor{pallido}{RGB}{221,227,227}
    \definecolor{sabbia}{RGB}{207,185,151}

    \pattern [color=sabbia, opacity=0.25] (0,0) rectangle (6,4);
    \draw[thin, sabbia] (3,0) -- (3,4)       ;
    \draw[thin, sabbia] (0,2) -- (6,2)       ;
    \draw[thin, sabbia] (0,0) rectangle (6,4);

    \fill[white] (1,2) to [out= 015, in=165] (2,2) to [out=195, in=-15] (1,2);
    \fill[white] (4,2) to [out= 015, in=165] (5,2) to [out=195, in=-15] (4,2);

    \draw[thin, black ] (1,2) to [out= 015, in=165] (2,2);
    \draw[thin, black ] (1,2) to [out=-015, in=195] (2,2);
    \draw[thin, black ] (4,2) to [out= 015, in=165] (5,2);
    \draw[thin, black ] (4,2) to [out=-015, in=195] (5,2);

    \fill[black] (1,2) circle (1pt);
    \fill[black] (2,2) circle (1pt);
    \fill[black] (4,2) circle (1pt);
    \fill[black] (5,2) circle (1pt);
    
    \foreach \x in {0,3,6} {
        \foreach \y in {0,2,4} {
            \fill[black] (\x,\y) circle (1pt);
        }
    };
    
    \node at (0.5,0.5) {\(\textnormal{R}_1\)};
    \node at (0.5,3.5) {\(\textnormal{R}_4\)};
    \node at (5.5,0.5) {\(\textnormal{R}_2\)};
    \node at (5.5,3.5) {\(\textnormal{R}_3\)};

    \node at (1.5,2.25) {\(s_o^+\)};
    \node at (4.5,2.25) {\(s_1^+\)};
    \node at (1.5,1.75) {\(s_o^-\)};
    \node at (4.5,1.75) {\(s_1^-\)};

    \end{tikzpicture}
    \caption{Realising a singular Euclidean metric which is not a translation surface.}
    \label{fig:halftrans1}
\end{figure}
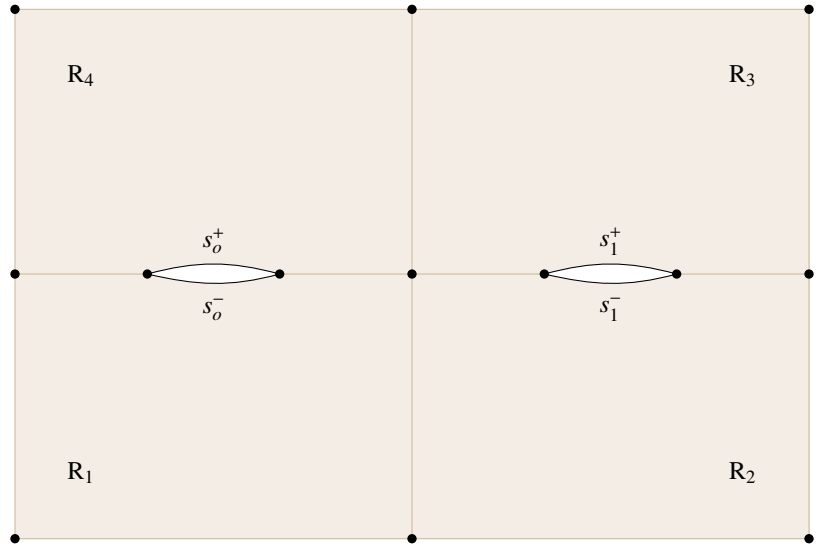
\noindent The opposite sides of the rectangle are identified by translations. The surface obtained by gluing the two slits will be of genus one or two and will carry a singular Riemannian metric with two singularities of cone angle \(4\pi\). The resulting surface thus depends on how the edges \(s_i^{\pm}\) are glued together. There are three cases as follows.
\begin{itemize}
    \item[1.] If the edge \(s_i^+\) is identified with \(s_i^-\), for \(i=0,1\), then the resulting surface is a torus, in particular a translation surface in the sense of Definition \ref{defn:firstdef}. As a consequence, it is endowed with a translation structure in the sense of Definition \ref{def:second_translation_surface}.
    \smallskip
    \item[2.] If the edge \(s_i^+\) is identified with \(s_{i+1}^-\), for \(i=0,1\) with indices taken modulo \(2\), then the resulting space is a surface of genus two in the sense of Definition \ref{defn:firstdef}. In both cases, the identification is by design given by the translation \(z \mapsto z + 3a\). As a consequence, it is endowed with a translation structure in the sense of Definition \ref{def:second_translation_surface}. 
    \smallskip
    \item[3.] Finally, if the edge \(s_i^+\) is identified with \(s_{i+1}^+\), again for \(i=0,1\) modulo 2, then the resulting space is still a surface of genus two. In this case, however, the rectangles \(\textnormal{R}_1\) and \(\textnormal{R}_4\) are glued to \(\textnormal{R}_2\) and \(\textnormal{R}_3\), respectively, via a rotation of order two around the centre of the rectangle. As a consequence, the resulting surface is no longer endowed with a translation structure in the sense of Definition \ref{def:second_translation_surface}. See also Figure \ref{fig:halftrans2}.
\end{itemize}
\end{ex}

\begin{figure}[!ht]
    \centering
    \begin{tikzpicture}[scale=1.75, every node/.style={scale=1}]
    \definecolor{pallido}{RGB}{221,227,227}
    \definecolor{sabbia}{RGB}{207,185,151}

    \pattern [color=sabbia, opacity=0.25] (0,0) rectangle (6,4);
    \draw[thin, sabbia] (3,0) -- (3,4)       ;
    \draw[thin, sabbia] (0,2) -- (6,2)       ;
    \draw[thin, sabbia] (0,0) rectangle (6,4);

    \draw[red, thin] (1.5 ,2.05)--(1.5 ,2.75);
    \draw[red, thin] (1.75,3) .. controls (1.6,3) and (1.5,2.9) .. (1.5,2.75);
    \draw[red, thin] (1.75,3   )--(4.25,3   );
    \draw[red, thin] (4.25,3) .. controls (4.4,3) and (4.5,2.9) .. (4.5,2.75);
    \draw[red, thin] (4.5 ,2.05)--(4.5 ,2.75);

    \draw[blue, ->] (1.5  , 2.05) -- (1.5  , 2.30);
    \draw[blue, ->] (1.5  , 2.50) -- (1.5  , 2.75);
    \draw[blue, ->] (1.55 , 2.91) -- (1.55 , 3.16);
    \foreach \x in {1.75,2.0,...,4.25} {
        \draw[red , ->] (\x, 3) -- (\x, 3.25);
        \draw[blue, ->] (\x, 3) -- (\x, 3.25);
    }
    \draw[blue, ->] (4.45 , 2.91) -- (4.45 , 3.16);
    \draw[blue, ->] (4.5  , 2.05) -- (4.5  , 2.30);
    \draw[blue, ->] (4.5  , 2.50) -- (4.5  , 2.75);

    \fill[white] (1,2) to [out= 015, in=165] (2,2) to [out=195, in=-15] (1,2);
    \fill[white] (4,2) to [out= 015, in=165] (5,2) to [out=195, in=-15] (4,2);

    \draw[thin, black ] (1,2) to [out= 015, in=165] (2,2);
    \draw[thin, black ] (1,2) to [out=-015, in=195] (2,2);
    \draw[thin, black ] (4,2) to [out= 015, in=165] (5,2);
    \draw[thin, black ] (4,2) to [out=-015, in=195] (5,2);

    \fill[black] (1,2) circle (1pt);
    \fill[black] (2,2) circle (1pt);
    \fill[black] (4,2) circle (1pt);
    \fill[black] (5,2) circle (1pt);
    
    \foreach \x in {0,3,6} {
        \foreach \y in {0,2,4} {
            \fill[black] (\x,\y) circle (1pt);
        }
    };
    
    \node at (0.5,0.5) {\(\textnormal{R}_1\)};
    \node at (0.5,3.5) {\(\textnormal{R}_4\)};
    \node at (5.5,0.5) {\(\textnormal{R}_2\)};
    \node at (5.5,3.5) {\(\textnormal{R}_3\)};

    \node at (1.25,2.25) {\(s_o^+\)};
    \node at (4.75,2.25) {\(s_1^+\)};
    \node at (1.25,1.75) {\(s_o^-\)};
    \node at (4.75,1.75) {\(s_1^-\)};

    \end{tikzpicture}
    \caption{Surface endowed with a singular Euclidean metric without trivial parallel transport. Referring to Figure~\ref{fig:halftrans1} of Example~\ref{ex3}, this figure illustrates that if \(s_i^+\) is identified with \(s_{i+1}^+\), then the resulting surface carries a singular Euclidean metric with non-trivial parallel transport. A unit vector field along a non-contractible curve is highlighted in red. After completing a loop, the initial vector is rotated by an angle of \(\pi\).}
    \label{fig:halftrans2}
\end{figure}
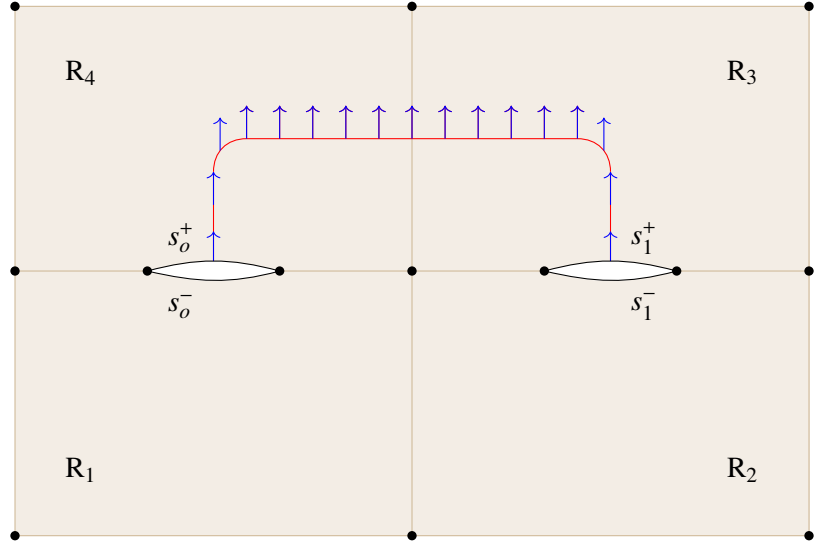

\medskip

\subsection{First-order linear \textnormal{ODEs}}\label{ssec:ode} An even more analytic perspective is given by first-order linear ODEs on Riemann surfaces. Let \(X\) be a closed Riemann surface and let \(\omega\) be a holomorphic \(1-\)form on \(X\). In terms of analytic maps, a translation surface can be seen as a solution of the first-order linear ODE defined as
\begin{equation}\label{eq:ode}
    dF\,=\,\omega
\end{equation}
on a Riemann surface \(X\). On a simply connected open subset, say \(U \subset X\), we can always find a local solution \(F\colon U \longrightarrow \mathbb{C}\) which is a biholomorphism onto its image whenever \(U\) does not contain any cone point. This can be taken as a local chart, and we may easily observe that two solutions, say \(F_1\) and \(F_2\), must differ by some constant \(c \in \mathbb{C}\), \textit{i.e.}, \(F_2 = F_1 + c\). Notice that if \(U\) is a neighbourhood of some cone point of order \(m\), then there is a local coordinate, say \(z\), such that \(\omega(\,z\,) = z^m \, dz\), that is, \(dF = (m+1)\, z^m \, dz\), so that a local solution looks like \(F(\,z\,) = z^{m+1}\). It should not come as a surprise to the reader to encounter the same charts defined in \S\ref{sssec:equivalenttwo}. We thus have defined a translation structure.

\smallskip 

\noindent A local solution generally extends to a multivalued function \(F\colon X \longrightarrow \mathbb{C}\), obtained via the analytic continuation principle, with monodromy\index{representation!monodromy} given by a representation \(\pi_1(\,X\,) \longrightarrow \mathbb{C}\). This function lifts to a genuine single-valued function on the universal cover. The restriction of this lifted function to any sufficiently small open neighbourhood defines a local chart for the underlying translation structure, and therefore must be a developing map. Vice versa, given any translation surface with developing map, say \(\textnormal{dev}\colon \widetilde X\longrightarrow \mathbb C\), then \(\widetilde\omega=\textnormal{dev}^*dz\) is a well defined holomorphic \(1\)-form that projects to a holomorphic \(1\)-form \(\omega\) on \(X\). The key point is that the monodromy of the differential equation \eqref{eq:ode} coincides with the holonomy of the translation structure.

\smallskip

\noindent From a different perspective, equation~\eqref{eq:ode} is equivalent to the first-order differential equation
\begin{equation}\label{eq:ode2}
    du\,-\,\omega\,u=0,
\end{equation}
that means
\begin{equation}
    \frac{du}{u}\,=\,\omega.
\end{equation}
\noindent The equation \eqref{eq:ode2} being linear of the first order admits a unique solution up to some scalar factor. In other words, the ratio of two local solutions is constant. 

\subsection{Isomorphic structures}\label{ssec:equiv} In the previous sections, we have examined translation surfaces from various perspectives. We now wish to provide a notion of \textit{isomorphism} for translation surface. In other words, we wonder when two translation surfaces can be considered equivalent. Depending on the adopted perspective, the term "isomorphism" acquires different meanings: from the viewpoint of translation structures, we shall mean that the structures are isometric and differ by a translation. In terms of Riemann surfaces, we shall mean that the structures are biholomorphic. For our purposes, it is convenient to adopt the complex-analytic perspective as in \S\ref{ssec:cadef}. The leading question here is the following: 
\smallskip
\begin{quote}
    \textit{when are two translation surfaces equivalent?}
\end{quote}
\smallskip
\noindent Two Riemann surfaces, say \(X\) and \(Y\), are declared to be equivalent if and only if there exists a biholomorphism \(f\colon X\longrightarrow Y\). In the same spirit we introduce the following

\begin{defn}
    We define two translation surfaces, say \((X,\omega)\) and \((Y,\xi)\), to be equivalent if and only if there exists a biholomorphism \(f\colon X\longrightarrow Y\) such that \(f^*\xi=\omega\).
\end{defn}

\noindent The following characterisation result holds.

\begin{lem}\label{lem:isolem}
Let \((\,X,\,\omega\,)\) and \((\,Y,\,\xi\,)\) be two equivalent translation surfaces. Then, for every point \(p \in X\), there exist local charts \((U, \varphi)\) around \(p\) and \((V, \psi)\) around \(f(\,p\,)\), where \(f\colon X \longrightarrow Y\) is the biholomorphism satisfying \(f^*\xi = \omega\), such that on \(U\) the equality \(\psi \circ f \,=\,\varphi\,+\,c\) holds for some \(c\in\mathbb C\).
\end{lem}

\begin{proof}
Let \( p \in X \) be any point. Then there exists a neighbourhood \( U \subset X \) of \(p\) and a chart \( \varphi\colon U \to \mathbb{C} \) such that \( \varphi^*(\,dz\,) = \omega \). Similarly, there exists a neighbourhood \( V \subset Y \) of \(f(\,p\,)\) and a chart \( \psi\colon V \to \mathbb{C} \) such that \( \psi^*(\,dz\,) = \xi \). Assume \(p\) is a regular point. In this case \(\varphi\) is a local homeomorphism onto the image and hence \(\psi \circ f \,=\,\varphi\,+\,c\) is equivalent to the condition that \(\psi \circ f \circ \varphi^{-1}\colon \varphi(\,U\,) \to \psi(\,V\,)\) is a translation. Define \(g\,=\, \psi \circ f \circ \varphi^{-1} \colon \varphi(\,U\,) \longrightarrow \psi(\,V\,) \subset \mathbb{C} \). Since \(f\) is a biholomorphism and \(f^*\xi = \omega\), we have
\begin{equation}
g^* (\,dz\,) = dz.
\end{equation}
This implies that \( g \) is a holomorphic function whose derivative is constantly one. Hence, it is of the form
\begin{equation}
g(\,z\,) = z + c
\end{equation}
for some constant \(c \in \mathbb{C}\). Therefore, in local coordinates, \(f\) is a translation. We next assume \(p\) is a cone point of order \(m\). In this case, \(\varphi\), and hence \(\psi\), are branched charts and the local inverse of \(\varphi\) is not well defined. Once again, since \(f^*\xi=\omega\), we have
\begin{equation}
    (\,\psi\circ f\,)^*(\,dz\,)\,=\,\varphi^*dz,
\end{equation}
that means
\begin{equation}
    f^*\Big(\,(\,m+1\,)w^mdw \,\Big)\,=\,(\,m+1\,)z^mdz \quad \Longleftrightarrow \quad (\,m+1\,)\,f(\,z\,)^m\,f'(\,z\,)dz\,=\,(\,m+1\,)z^mdz.
\end{equation}
This latter holds if and only if \(f(\,z\,)^m\,df\,=\,z^m\,dz\). We thus need to solve a non-linear ODE on the open set \( U \subset X \). Since \( f \) is a biholomorphism, we can perform a change of variable and set \( w = f(\,z\,) \). The equation then becomes
\begin{equation}
w^m \, dw = z^m \, dz,
\end{equation}
which is separable. By integrating both sides, substituting \( f(\,z\,) \) for \( w \) on the left side, and performing some manipulations, we obtain the following form:
\begin{equation}
f(\,z\,) = \displaystyle \sqrt[m+1]{\,\,z^{m+1}\,+\,C\,\,}.
\end{equation}
Notice that we do not need to specify any branch of the \(m+1\)-root because \(U\) is a neighbourhood of the tip of the Euclidean cone of order \(m+1\) which is by Lemma \ref{lem:complexconem} the Riemann surface that makes the \((m+1)\)-root a single-valued function. Since the charts are centred, \( p \) has coordinate \( z = 0 \) and \( f(\,p\,) \) has coordinate \( w = 0 \), it follows that \( C = 0 \). Hence \(f(\,z\,)=z\).
\end{proof}

\subsection*{Appendix at \S\ref{sec:trans}: Summarising table}\label{ssec:table} We can summarise the different perspectives discussed so far in the following double-entry table. A translation surface can be understood either in terms of surfaces or in terms of functions, viewed from a geometric or an analytic angle. See Table \ref{tab:sumtab}.

\smallskip

\begin{table}[!ht]
    \centering
    \begin{tikzpicture}[scale=1.125, every node/.style={scale=1}]

    \draw[thick] (-6,0) -- (6,0); 
    \draw[thick] (0,-4) -- (0,4); 

    \node[anchor=center] at (-3, 2 ) {\(\begin{gathered}   \textcolor{red}{\text{translation structure}} \\   \big\{\,\varphi_\alpha\colon U_\alpha\to\mathbb C\,\big\} \\  \text{whose transition functions are} \\ \text{translations on their overlaps} \\ \end{gathered}\)};

    \node[anchor=center] at ( 3, 2 ) {\(\begin{gathered}   \textcolor{red}{\text{Riemann surface}} \\  + \,\,\textcolor{red}{\text{holomorphic \(1\)-form}} \\ (\,X,\,\omega\,) \end{gathered}\)};

    \node[anchor=center] at ( 3,-2) {\(\begin{gathered}   \textcolor{red}{\text{multivalued function}} \\  F\colon X\longrightarrow \mathbb C  \\ \text{with monodromy } \rho\colon\pi_1(\,X\,)\to\mathbb C \end{gathered}\)};
    \node[anchor=center] at (-3,-2) {\(\begin{gathered}   \textcolor{red}{\text{developing map}} \\  \text{smooth map}\,\,\textnormal{dev}\colon \widetilde\Sigma\longrightarrow \mathbb C  \\ \text{equivariant wrt } \rho\colon\pi_1(\,\Sigma\,)\to\mathbb C \end{gathered}\)};

    \node[rotate=90] at (-5.75, 2) {\cm{surfaces}};
    \node[rotate=90] at (-5.75,-2) {\cm{maps on surfaces}};
    \node at (-3,3.75) {\cm{geometric perspective}};
    \node at ( 3,3.75) {\cm{complex-analytic perspective}};
\end{tikzpicture}
    \caption{Summarising table}
    \label{tab:sumtab}
\end{table}

\bigskip

\section{Metric geometry of translation surfaces}\label{sec:metric}

\noindent In the present section we wish to provide some features of translation surfaces seen as metric spaces. We may define a metric on a translation surface \((X, \omega)\) in two equivalent ways. First, we may consider the singular Euclidean metric induced by \(\omega\), which is locally isometric to the Euclidean plane away from the zeros of \(\omega\). 
If we can also describe a neighbourhood of the singular locus via a singular Euclidean metric, then we may define a global metric \(d(x, y)\) as the infimum of the Euclidean lengths of all piecewise differentiable paths joining \(x\) and \(y\) in \((X,\omega)\). As we shall see the expression for the metric is somewhat bizarre.

\medskip

\noindent Alternatively, since translation surfaces can be decomposed into polygons glued along parallel edges \textit{e.g.}, via a triangulation into Euclidean triangles as done in \S\ref{ssec:gbcond} and \S\ref{ssec:devhol}, each piece inherits the standard Euclidean metric. These local metrics can be combined to produce a global metric on \(X\). There are several ways to carry out this construction; for example, by using strings and chains as described in~\cite{BH}. As one may expect, both approaches yield the same resulting metric. 

\smallskip

\subsection{A few notions of metric geometry}\label{ssec:metricgeo} Before proceeding, let us recall some general definitions valid in any metric space. Given a metric space, say \((\Sigma, d)\), the \emph{length} of a continuous curve \(\gamma\colon [a,b] \to \Sigma\) is defined as
\begin{equation}
\sup \sum_{i=0}^{n-1} d\big(\,\gamma(\,t_i\,),\, \gamma(\,t_{i+1}\,)\,\big),
\end{equation}
where the supremum is taken over all partitions \(a = t_0 \leq t_1 \leq \dots \leq t_n = b\) of the interval \([a,b]\). A curve with finite length is called \emph{rectifiable}.

\smallskip

\noindent A curve \(\gamma\colon[a,b] \to \Sigma\) is a \emph{unit-speed geodesic} if, for every \(t \in [a,b]\) and \(s\) close to \(t\), we have
\begin{equation}
d\big(\,\gamma(\,t\,),\, \gamma(\,s\,)\,\big) = |\,t - s\,|.
\end{equation}
Any curve that is a reparametrisation of a unit-speed geodesic is called simply a \emph{geodesic}.

\smallskip

\noindent A metric space \((\Sigma,d)\) is called a \emph{geodesic space} if any two points in \(\Sigma\) can be joined by a geodesic. A metric \(d\) is called a \emph{length metric} if for every pair of points \(x,y \in \Sigma\), the distance \(d(x,y)\) equals the infimum of the lengths of rectifiable curves joining \(x\) and \(y\). A space endowed with a length metric is called a \emph{length space}. It is worth noting that not every metric is a length metric, but every metric space admits an \emph{induced length metric}, which can be constructed by taking the infimum over curve lengths between points. The reader is invited to consult the nice book \cite{BH} for more details.

\smallskip

\subsection{Distance on Euclidean cones}\label{ssec:distcones} In \S\ref{sssec:bigcone} we introduced the Euclidean cone of order \(m\), with \(m \geq 1\) and used it as a model to define the local geometry around conical points. We now describe the metric geometry induced by the singular Euclidean metric. Once again, it is convenient to use polar coordinates to introduce a notion of distance on a Euclidean cone\index{euclidean cone!distance on}. As usual we have a radial coordinate \(\rho \in \mathbb{R}^+\) and an angular coordinate \(\theta \in \big[0,\, 2(m+1)\pi\,\big)\). 
We denote this cone as \(\mathcal C(\,2(m+1)\pi\,)\). 

\smallskip

\noindent The resulting metric has an explicit form. By considering \(S^1 \subset \mathbb{E}^2=\mathcal C(\,2\pi\,)\), the points of the unit circle are exactly those ordered pairs \(x=(x_1,\,x_2)\) such that \(||\,x\,||=1\). The function \(d_{\mathbb S^1}\colon S^1\times S^1\longrightarrow \mathbb R\) that assigns to each pair \((x, y) \in S^1 \times S^1\) the unique real number \(d_{S^1}(\,x,\, y\,) \in [0, \pi]\) such that
\[
\cos d_{S^1}(\,x,\, y\,) = \langle\, x,\, y\,\rangle
\] is a metric and the distance between each pair of points is the length of the shortest arc joining them, see \cite[Proposition 2.1, Chapter I]{BH}. The unit circle, say \(S^1\), centred at the tip of the cone \(\mathcal C(\,2(m+1)\pi\,)\) has length equal to \(2(m+1)\pi\) and it can be endowed with a metric \(d_{S^1}\) that assigns to each pair of points the unique real number in \([\,0, (m+1)\pi\,]\) that represents the length of the shortest arc joining them. In both cases we refer to \(d_{S^1}\) as the circle metric.

\smallskip

\noindent Let \(d_{S^1}\) be the circle metric on \(S^1\). We define the so-called \textit{truncated metric}, and denote it by \(d^\pi\), the function defined as 
\begin{equation}
    d^\pi(\,\theta_1,\,\theta_2\,) = \min\big\{\,\pi,\, d_{S^1}(\,\theta_1, \theta_2\,)\,\big\}.
\end{equation}

\noindent We may observe that \(d^{\pi}=d_{S^1}\) on the unit circle of \(\mathcal C(\,2\pi\,)\). For every \(m\ge1\), for two points \(p = (\,t,\,\theta_1\,)\) and \(q = (\,s,\, \theta_2\,)\) in \(\mathcal C(\,2(m+1)\pi\,)\), we define the distance \(d(\,p,\,q\,)\) via the formula
\begin{equation}
d(\,p,\,q\,)^2 = t^2 + s^2 - 2st \cos\big(\,d^\pi(\,\theta_1,\, \theta_2\,)\,\big).
\end{equation}
One can show that this function \(d\) defines a metric on \(\mathcal C(\,2(m+1)\pi\,)\). While the formula for the distance may seem unusual, it simplifies in the case of small angular separation: for points \(p, q\) with small angular distance in \(S^1\), the expression reduces to the law of cosines in Euclidean geometry.

\smallskip

\noindent A geodesic \(\gamma\) passing through the tip of the cone \(\mathcal C(\,2(m+1)\pi\,)\), after reparametrisation, has the form
\begin{equation}
\gamma\colon (\,-\varepsilon, \varepsilon\,) \to \mathcal C(\,2(m+1)\pi\,), \qquad \gamma(\,t\,) = 
\begin{cases}
\,\,(\,t,\,\theta_1\,), & t < 0, \\
\,\,(\,t,\,\theta_2\,), & t > 0,
\end{cases}
\end{equation}
and for small \(\delta > 0\), we have
\begin{equation}
d(\,\gamma(\,-\delta\,),\, \gamma(\,\delta\,)\,)^2 = 2\delta^2 \,\left(\,1 - \cos\big(\,d^\pi(\,\theta_1,\, \theta_2\,)\,\big)\,\right).
\end{equation}
\noindent For \(\gamma\) to be a geodesic, we must then have \(d(\,\gamma(\,-\delta\,),\, \gamma(\,\delta\,)\,) = 2\delta\), and hence:
\begin{equation}
\cos\Big(\,d^\pi(\,\theta_1,\, \theta_2\,)\,\Big) = -1.
\end{equation}
\noindent This means \(d_{\mathbb S^1}(\,\theta_1,\, \theta_2\,)\ge\pi\). Therefore, \(\gamma\) is a geodesic if and only if it leaves an angle of magnitude at least \(\pi\) on each side of the tip of the cone. Notice that at a regular point, this requirement is classical!  However, when a cone point, say \(p\), has magnitude that exceeds \(2\pi\), say \(2(m+1)\pi\), many geodesics do pass through \(p\): any pair of geodesic segments meeting at \(p\) with an angle of at least \(\pi\) can be joined to form a single geodesic. Thus, unlike in the smooth case, geodesic continuation at a cone point of angle greater than \(2\pi\) is no longer unique.

\smallskip

\subsection{Non-Positive Curvature of Translation Surfaces}\label{ssec:cato} We now analyse the non-positive curvature properties induced by a translation structure on a surface. Let now \(\pi \colon \widetilde{\Sigma} \rightarrow \Sigma\) be the universal covering map. The singular flat metric on \(\Sigma\) lifts to a singular metric on \(\widetilde{\Sigma}\), turning the universal cover into a metric space as well. We show that \(\widetilde{\Sigma}\) is a \(\mathrm{CAT}(\,0\,)\)-space. It follows that \(\Sigma\) is a locally \(\mathrm{CAT}(\,0\,)\)-space. Precisely, we have the following:

\begin{prop}
Let \(\Sigma\) be a closed surface equipped with a translation structure. Then \(\widetilde{\Sigma}\) is a \(\mathrm{CAT}(\,0\,)\)-space. In particular, the geometric boundary of \(\widetilde{\Sigma}\) is an oriented circle.
\end{prop}

\begin{proof}
As already alluded to above, a translation structure determines a singular Euclidean metric with trivial parallel transport; see~\S\ref{sssec:north}. Let \(ds^2\) be its lift to \(\widetilde{\Sigma}\). Equivalently, this metric can be defined by pulling back the standard Euclidean metric via the developing map. The metric is a flat  metric, in fact locally isometric to the Euclidean metric on \(\mathbb C\cong\mathbb E^2\), away from the cone points. 
\noindent Let \(\widetilde{\Sigma}\) denote the universal cover of a translation surface \(\Sigma\), and let \(\widetilde{\Delta} \subset \widetilde{\Sigma}\) be the set of cone points. Consider the regular part \(\widetilde{\Sigma}^\circ := \widetilde{\Sigma} \setminus \widetilde{\Delta}\). At each point of \(\widetilde{\Sigma}^\circ\), the metric is locally isometric to the Euclidean plane \(\mathbb{E}^2\). Since Euclidean space is \(\mathrm{CAT}(\,0\,)\) and the \(\mathrm{CAT}(\,0\,)\) property is local, it follows that \(\widetilde{\Sigma}^\circ\) is locally \(\mathrm{CAT}(\,0\,)\). 
By a standard result in the theory of \(\mathrm{CAT}(\,0\,)\) spaces, the metric completion of a \(\mathrm{CAT}(\,0\,)\) space is \(\mathrm{CAT}(\,0\,)\), see \cite[Prop.~II.3.10]{BH}. Consequently, 
the completed space \(\widetilde{\Sigma}\) is \(\mathrm{CAT}(\,0\,)\). Since \(ds^2\) is smooth outside a sufficiently small closed neighbourhood of the cone points, there exists a smooth flat metric \(ds^2_o\) on \(\widetilde{\Sigma}\) agreeing with \(ds^2\) outside this neighbourhood. Let \(d_o\) denote the distance induced by \(ds^2_o\). The identity map
\begin{equation}
\mathrm{id}:(\widetilde{\Sigma},d)\longrightarrow (\widetilde{\Sigma},d_o)
\end{equation}
is a quasi-isometry, hence \((\widetilde{\Sigma},d)\) and \((\widetilde{\Sigma},d_o)\) have the same ideal boundary, which is homeomorphic to a circle \(S^1\).
\end{proof}


\medskip

\subsection{Developing maps}\label{sssec:devmap} In the present paragraph we want to use the metric properties of translation surfaces for a better understanding of developing maps\index{developing map}. The main result is the following:

\begin{prop}
    Let \(\Sigma\) a translation surface. Then the developing map \(\textnormal{dev}\colon \widetilde\Sigma\longrightarrow\mathbb C\) is uniformly open,\textit{ i.e.} for any \(\varepsilon>0\) there exists \(\delta>0\) such that \(B_\varepsilon(\,\textnormal{dev}(\,\widetilde{q}\,)\,) \subseteq \textnormal{dev}(\,B_\varepsilon(\,\widetilde{q}\,)\,)\) for any \(q\in\Sigma\). Moreover the developing map \(\textnormal{dev}\) is surjective.
\end{prop}

\begin{proof}
    Let \(\Sigma\) be a translation surface and let \(\Delta=\{\, p_1, \dots, p_n\,\} \) be the cone points of the translation surface. Let \(d(\,\cdot,\,\cdot\,)\) be the distance induced by the associated singular Euclidean metric and define
    \begin{equation}
    2\varepsilon = \min\{\, d(\,p_i,\, p_j\,) \mid i \neq j \,\}.
    \end{equation}
    Let \( \widetilde{q} \in \widetilde{\Sigma} \) be a regular point on the universal cover, and suppose that the Euclidean ball \( B_\varepsilon(\,\widetilde{q}\,) \subset \widetilde{\Sigma} \) does not contain any cone point. Then the result follows immediately, since the developing map
    \begin{equation}
    \mathrm{dev} \colon \widetilde{\Sigma} \longrightarrow \mathbb{C}
    \end{equation}
    is a local isometry away from singularities and hence maps \( B_\varepsilon(\,\widetilde{q}\,) \) isometrically to a Euclidean ball in \( \mathbb{C} \). 

    \smallskip
    
    \noindent Now suppose instead that \( B_\varepsilon(\,\widetilde{q}\,) \) contains exactly one cone point, say \( \widetilde{p} \), of order \( m\geq 1 \). Let \( r \) be the geodesic segment joining \( \widetilde{p} \) and \( \widetilde{q} \). Since the angle around \( \widetilde{p} \) exceeds \( 2\pi \), we can consider two additional geodesic segments, say \( \tau_1 \) and \( \tau_2 \), starting at \( \widetilde{p} \), forming angles of \( \pi \) on either side of \( r \). These lie entirely within the ball \( B_\varepsilon(\,\widetilde{q}\,) \), and their union divides the ball into two open regions. The developing map, restricted to the region containing \( r \), is again a local isometry and its image contains the Euclidean ball \( B_\varepsilon(\,\mathrm{dev}(\,\widetilde{q}\,)\,) \). Therefore,
    \begin{equation}
    B_\varepsilon(\,\textnormal{dev}(\,\widetilde{q}\,)\,) \subseteq \textnormal{dev}(\,B_\varepsilon(\,\widetilde{q}\,)\,)
    \end{equation}
    for any \( \widetilde{q} \in \widetilde{\Sigma} \). Suppose now that the developing map is not surjective, \textit{i.e.} \( \mathrm{dev}(\,\widetilde{\Sigma}\,) \subsetneq \mathbb{C} \). Then there exists a point at positive Euclidean distance from the boundary \( \partial \mathrm{dev}(\,\widetilde{\Sigma}\,) \). Fix \( \delta > 0 \) such that \( 0 < \delta < \varepsilon/2 \), and choose \( \widetilde{q} \in \widetilde{\Sigma} \) such that \( \mathrm{dev}(\,\tilde{q}\,) \) has distance \( \delta \) from \( \partial \mathrm{dev}(\,\widetilde{\Sigma}\,) \). Then
    \begin{equation}
    B_\varepsilon(\,\mathrm{dev}(\,\widetilde{q}\,)\,) \subseteq \mathrm{dev}(\,\widetilde{\Sigma}\,),
    \end{equation}
    which contradicts the assumption that \( \mathrm{dev}(\,\widetilde{\Sigma}\,) \) is not the entire plane. Hence, the developing map is surjective.
\end{proof}

\begin{rmk}
    In \cite{Tro07}, the author provides an alternative proof which, however, relies on the same underlying ideas. Unlike our approach, he considers a triangulation of the surface into Euclidean triangles, and uses the key property that a developing is an isometry on each triangle, see \S\ref{ssec:devhol}. It is then shown that a quadrilateral formed by two adjacent triangles is mapped to a quadrilateral. If the developing map were not surjective, one can show that a boundary point of the image must lie on the side of some triangle; however, such a triangle is always contained in some quadrilateral, yielding a contradiction. Our proof does not rely on a triangulation of the surface but instead on properties of the induced distance. It should not be surprising that the two proofs are equivalent.
\end{rmk}
  
\medskip

\subsection{Completeness}\label{ssec:complete} We begin with a general remark: a geometric structure induced by a possibly singular Riemannian metric can be complete in two senses: metric and geometric. We now examine these aspects in the case of translation surfaces.

\subsubsection{Metric sense} In our setting completeness is ensured by the fact that the translation structures are defined on compact surfaces and every Cauchy sequence converges with respect to the distance induced by the singular Euclidean metric. Therefore a translation structure defines a complete metric space. Moreover, a translation structure is also a length space because the length metric coincides with the one induced by the translation structure. By the Hopf--Rinow Theorem, every translation surface is also a geodesic space. As a consequence, any two distinct points can be joined by a geodesic. In the theory of translation surfaces, geodesics joining two conical points (possibly the same) play a particularly important role. These are known as \textit{saddle connections}. It can be shown that every translation surface admits a triangulation made entirely of saddle connections.

\smallskip

\subsubsection{Thurston sense} We have seen in the previous paragraph that translation structures are metrically complete. In the present one we show they are not complete in Thurston's sense. Indeed, this phenomenon occurs precisely because the developing map of a translation surface is not a global diffeomorphism. This failure arises due to the presence of critical points corresponding to the cone singularities on \(\Sigma\).

\smallskip

\noindent Let \( \mathrm{dev} \colon \widetilde{\Sigma} \longrightarrow \mathbb{C} \) be the developing map of a translation structure on \(\Sigma\), and let \( \rho \colon \pi_1(\,\Sigma\,) \longrightarrow \mathbb{C} \) be its holonomy representation. For the readers' convenience we recall the following notion: A subset \( \Omega \) of \(\mathbb C \) is called a \textit{fundamental domain}\index{fundamental!domain} for the action of \( \Gamma=\textnormal{Im}(\,\rho\,)\) if:
\begin{enumerate}
    \item[1.] The union of the translates covers the surface:
    \[
    \mathbb C= \bigcup_{\gamma \,\in\, \Gamma} \,\gamma(\,\Omega\,),
    \]
    \item[2.] The interiors of distinct translates are disjoint:
    \[
    \gamma(\,\Omega\,)^\circ \cap \gamma'(\,\Omega\,)^\circ = \emptyset \quad \text{for all } \gamma \ne \gamma' \in \Gamma.
    \]
\end{enumerate}

\noindent Clearly, this definition is not specific to translation structures, but rather an adaptation of a much more general definition valid for groups acting on topological spaces. In terms of tilings, a fundamental domain is a prototile, and the complex plane can be tessellated using only copies of this tile labelled with the elements of \(\Gamma\). The following equivalence is well known in the literature and some implications dates back to Fricke--Klein.

\begin{equation}
    \begin{split}
        \text{completeness} & \Longleftrightarrow \mathrm{dev} \colon \widetilde{\Sigma} \longrightarrow \mathbb C \text{ is a diffeomorphism} \\
        & \Longleftrightarrow \rho(\,\pi_1(\,\Sigma\,)\,) \,< \,\mathbb C \text{ acts freely and properly discontinuously} \\
        & \Longleftrightarrow \Sigma \cong \rho(\,\pi_1(\,\Sigma\,)\,) \backslash \mathbb C.
    \end{split}
\end{equation}

\medskip

\noindent In general, the developing map fails to be a global diffeomorphism otherwise it would be a global isometry. As a consequence, a fundamental domain for the action does not exist, and the developing map produces an overlapping tiling of \(\mathbb{C}\) in the sense that the interiors of two tiles overlap.  In particular, these structures are not uniformisable in the following sense: there is no open domain \(\Omega \subset \mathbb{C}\) such that the structure is isometric to the quotient \(\rho(\,\pi_1(\,\Sigma\,)\,) \backslash \Omega\) by a proper discontinuous action without fixed points. The following holds:

\begin{lem}\label{lem:torusarecomplete}
    A translation structure on a closed surface \(\Sigma\) of genus \(g>0\) is complete in Thurston's sense if and only if \(g=1\).
\end{lem}

\begin{proof}
    We have already observed in Remark \ref{rmk:devgenusone} that the developing map is a global diffeomorphism, in fact an isometry, if and only if \(g=1\).
\end{proof}

\begin{proof}[Alternative argument for Lemma \ref{lem:torusarecomplete}] 

By adopting a more algebraic perspective, let \(\rho\) be the holonomy representation of the given translation structure. Since \(\rho\) takes values in an abelian group, then the subgroup generated by the commutators is properly contained in the kernel of the representation. If \(g \ge 2\), then \(\rho\) cannot be faithful. Recall that a developing map satisfies an equivariance property with respect to the representation \(\rho\). As a consequence, if \(g \ge 2\), a developing map cannot be injective and thus cannot be a global diffeomorphism. Therefore, the assumption \(g = 1\) is strictly necessary.

\smallskip

\noindent On the other hand, if \(g = 1\), then the image of \(\rho\) is a lattice \(\Lambda < \mathbb{C}\) and the action on \(\mathbb C\) is free and properly discontinuous. Since \(\mathbb{C} / \Lambda \cong \Sigma\), the resulting translation structure is complete in the sense of Thurston, as desired.
\end{proof}

\medskip

\section{Moduli spaces of translation surfaces in genus one}\label{sec:moduligenusone}

\noindent In the present section we shall discuss the moduli space of translation surfaces in genus one. We keep this section separate from the case of surfaces of genus \(g \ge 2\), which will be addressed in \S\ref{sec:modulihighgen}.

\smallskip

\noindent Let us assume \(g=1\). The moduli space \(\mathcal M_1\) of compact Riemann surfaces of genus one is identified with the quotient space \(\mathbb H\,/\,\textnormal{SL}(2,\mathbb Z)\), where \(\mathbb H\) denotes the upper half-plane. This identification is well known and represents a cornerstone result in the theory of Riemann surfaces and elliptic curves. Let us briefly review it from our perspective, with an emphasis on the geometric aspects of translation structures. 

\smallskip

\noindent The main purpose here is to define a moduli space of translation surfaces of genus one. Essentially, there are two approaches: a geometric one, in which we determine fundamental domains, and an analytic one, in which we describe the holomorphic \(1\)-forms supported on a Riemann surface. 

\smallskip

\subsection{Finding fundamental domains}\label{sssec:fundomain} Let \(\mu_1,\,\mu_2\) be non-real collinear elements of \(\mathbb{C}\) such that \(\{\,\mu_1,\,\mu_2\,\}\) is a positive oriented basis for \(\mathbb C\) as a real vector space. The additive subgroup
\begin{equation}\label{eq:lattice}
    \Lambda=\Lambda(\,\mu_1, \mu_2\,)\,=\,\Big\{ \,m\,\mu_1\,+\,n\,\mu_2\,\,|\,\, m,n\in\mathbb Z\,\Big\}
\end{equation}
which they generate is called \textit{lattice}\index{lattice} and we shall call \(\mu_1\) and \(\mu_2\) a basis for \(\Lambda\). We may notice that \(\Lambda\) acts on the complex plane by translations. Since it is discrete, this action is free and properly discontinuous. 

\smallskip

\noindent A \textit{fundamental domain}\index{fundamental!domain} for such an action is the parallelogram\index{parallelogram} \( \textnormal{P} = \left\{\, x\,\mu_1 + y\,\mu_2 \,\middle|\, x,y \in [0, 1] \,\right\} \) in \(\mathbb C\). Every point \(z \in \mathbb{C}\) is equivalent under the action of \(\Lambda\) to some point in \(\textnormal{P}\), and two points in \(\textnormal{P}\) are equivalent if they both lie on its boundary \(\partial \textnormal{P}\) and their difference is \(\mu_1\) or \(\mu_2\). Let \(\sim\) be this equivalence relation. Since each coset in \(\mathbb{C}/\Lambda\) admits a unique representative in \(\textnormal{P}/\sim\,\), the quotient \(\mathbb{C}/\Lambda\) is constructed by identifying the opposite edges of \(\textnormal{P}\) according to the action of \(\Lambda\). The holomorphic structure on \( \mathbb{C} \) yields a holomorphic structure on \( \mathbb{C}/\Lambda \), making it a Riemann surface. This surface is compact, since it is a continuous image of \( \textnormal{P} \), which, by the Heine–Borel theorem, is compact. Vice versa, by the uniformization theorem, all compact Riemann surfaces of genus one arise in this way. It is an easy matter to find a triangulation of \( \mathbb{C}/\Lambda \) (this can be directly drawn in \( \textnormal{P} \)) which shows that this surface has genus one and naturally carries a translation structure.

\smallskip

\begin{rmk}
    Under a topological perspective, we are gluing the opposite sides of \(\textnormal{P}\) by using translations and the resulting surface is a torus. Based on the definition of a translation surface we provided in \S\ref{sssec:firstdef}, we can deduce that such a torus inherits a well defined translation structure.
\end{rmk}  

\begin{rmk}
    In the present section we aim to define the moduli space of translation surfaces in genus one in terms of fundamental domains, namely parallelograms on the complex plane. In the case where two parallelograms \(\mathbb C\) differ by a translation, then the resulting translation surfaces are the same. This is a consequence of Lemma \ref{lem:isolem}. Therefore, in what follows, we shall assume our parallelograms to have one vertex at the origin. This leads us to avoid formulations like "up to translations" in the following.
\end{rmk}

\noindent A translation structure on a torus clearly depends on the parallelogram \(\textnormal P\) which in turn depends on the lattice \(\Lambda(\,\mu_1,\,\mu_2\,)\). We wish to provide a criterion that tells us when two lattices define the same translation surface. The desired criterion is in fact part of a more general result that classifies genus one Riemann surfaces defined via lattices in the complex plane. Based on Lemma \ref{lem:isolem} in \S\ref{ssec:equiv}, given two lattices, say \(\Lambda_1\) and \(\Lambda_2\), the resulting translation surfaces \(\mathbb C/\Lambda_1\) and \(\mathbb C/\Lambda_2\) are equivalent if the underlying Riemann surfaces are biholomorphic.


\begin{prop}\label{prop:samecomplexgone}
Suppose that \(\Lambda_1\) and \(\Lambda_2\) are lattices in \(\mathbb{C}\). Then the Riemann surfaces \(\mathbb{C}/\Lambda_1\) and \(\mathbb{C}/\Lambda_2\) are biholomorphic if and only if \(\Lambda_1\) and \(\Lambda_2\) are similar lattices, that is \(\Lambda_2 = \lambda\, \Lambda_1\) for some \(\lambda \in \mathbb{C} \setminus \{\,0\,\}\).
\end{prop}

\begin{proof}
    If there exists a biholomorphism \( F \colon \mathbb{C} \to \mathbb{C} \), defined as \( z \mapsto \lambda z\,+\,c \), such that \( \lambda \Lambda_1\,+\,c = \Lambda_2 \), then it clearly induces a biholomorphism between the corresponding quotient surfaces. Conversely, any biholomorphism
    \begin{equation}\label{eq:locmapgone}
     f\colon \mathbb{C}/\Lambda_1 \to \mathbb{C}/\Lambda_2
    \end{equation}
    lifts to a biholomorphism \( F \colon \mathbb{C} \to \mathbb{C} \) of the form \( F(\,z\,) = \lambda z + c \) for some \( \lambda \in \mathbb{C}^* \) and \( c \in \mathbb{C} \). We may assume \(c=0\) because translations induce biholomorphisms of all tori. For \( f \) to be well defined, \( F \) must satisfy \( F(\,\Lambda_1\,) = \Lambda_2 \), and thus \( \lambda \Lambda_1 = \Lambda_2 \).
\end{proof}

\noindent The proof is classical and can be found, for instance, in \cite{JW} and we have included it here because 
it contains the key idea for proving the following corollary that provides the aforementioned criterion. 

\begin{cor}\label{cor:sametransone}
    Suppose that \(\Lambda_1\) and \(\Lambda_2\) are lattices in \(\mathbb{C}\). Then they determine the same translation surface if and only if \(\Lambda_1 = \, \Lambda_2\).
\end{cor}

\begin{proof}
    From the proof of Proposition \ref{prop:samecomplexgone} we can easily deduce that the mapping \(f\colon\mathbb{C}/\Lambda_1 \to \mathbb{C}/\Lambda_2 \) is a translation in local charts if and only if \(\lambda=1\). Thence the desired conclusion.
\end{proof}

\noindent The above results show that two lattices define the same translation structure if and only if they coincide as additive subgroups of \( \mathbb{C} \). However, the fundamental domains they determine, \textit{i.e.} the corresponding parallelograms as defined above, need not be isometric. In fact, they generally are not, as their shape depends heavily on the choice of generators. At this point, one may naturally ask whether, given two such parallelograms, it is possible to subdivide one of them into triangles in such a way that, by applying suitable translations, one recovers the other. And if so, how can such a subdivision be determined? 

\begin{figure}[!ht]
    \centering
    \begin{tikzpicture}[scale=1.75, every node/.style={scale=0.875}]
    \definecolor{pallido}{RGB}{221,227,227}
    \definecolor{sabbia}{RGB}{207,185,151}



    \draw[pallido, very thin, dashed] ( 2   , 0   )--( 6.25, 0   );
    \draw[pallido, very thin, dashed] ( 4   , 2   )--( 6.25, 2   );
    \draw[pallido, very thin, dashed] (-0.25, 4   )--( 6.25, 4   );
    \draw[pallido, very thin, dashed] ( 0   , 0   )--(-0.25, 0   );
    \draw[pallido, very thin, dashed] ( 0   , 2   )--(-0.25, 2   );
    
    \draw[pallido, very thin, dashed] ( 0   , 2   )--( 0   , 4.25);
    \draw[pallido, very thin, dashed] ( 0   ,-0.25)--( 0   , 0   );
    \draw[pallido, very thin, dashed] ( 2   , 2   )--( 2   , 4.25);
    \draw[pallido, very thin, dashed] ( 2   ,-0.25)--( 2   , 0   );
    \draw[pallido, very thin, dashed] ( 4   ,-0.25)--( 4   , 2   );
    \draw[pallido, very thin, dashed] ( 4   , 4.25)--( 4   , 3   );
    \draw[pallido, very thin, dashed] ( 6   ,-0.25)--( 6   , 4.25);

    \draw[sabbia, thin] (0,0)--(2,2);
    \draw[sabbia, thin] (6,4)--(2,2);
    \draw[sabbia, thin] (6,4)--(4,2);
    \draw[sabbia, thin] (0,0)--(4,2);
    \draw[sabbia, thin] (0,1)--(2,2);
    
    \draw[sabbia, thin] (0,0)--(2,0)--(2,2)--(0,2)--(0,0);
    \draw[sabbia, thin] (2,2)--(4,2)--(4,3);

    \pattern [color=sabbia, opacity=0.500] (0,0)--(2,1)--(2,2)--(0,0);
    \pattern [color=sabbia, opacity=0.375] (2,2)--(4,2)--(2,1)--(2,2);
    \pattern [color=sabbia, opacity=0.250] (2,2)--(4,2)--(4,3)--(2,2);
    \pattern [color=sabbia, opacity=0.125] (4,2)--(4,3)--(6,4)--(4,2);

    \pattern [color=sabbia, opacity=0.125] (0,0)--(0,1)--(2,2)--(0,0);
    \pattern [color=sabbia, opacity=0.375] (0,1)--(0,2)--(2,2)--(0,1);
    \pattern [color=sabbia, opacity=0.250] (0,0)--(2,0)--(2,1)--(0,0);


    \foreach \x in {0,2,4,6} {
        \foreach \y in {0,2,4} {
            \fill[black] (\x,\y) circle (0.5pt);
        }
    };
    \end{tikzpicture}
    \caption{Different fundamental domains, \textit{i.e.} parallelograms, may define the same translation surface. Here is depicted the case \(\Lambda=\mathbb Z\,+\,i\,\mathbb Z\) with respect to the base \(\{\,1,\,i\,\}\) and the base \(\{\,2+i,\,1+i\,\}\). Triangles with the same shade are isometric and differ by a translation.}
    \label{fig:parallelograms}
\end{figure}
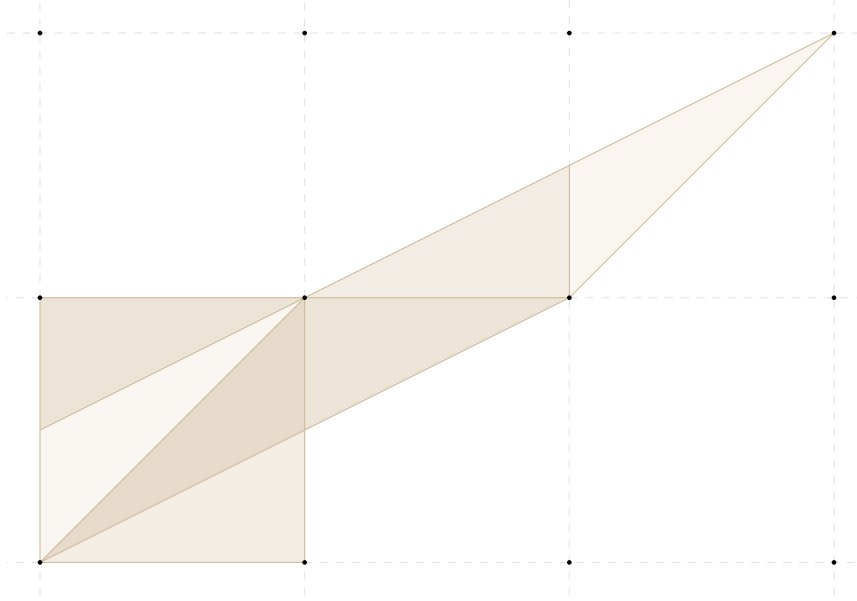

\smallskip

\noindent To address this question, let us begin with some general observations. A lattice \(\Lambda\), as defined in \eqref{eq:lattice}, not only determines a parallelogram \(\textnormal{P}\), as already mentioned, but in fact induces a tiling of the plane, where each tile has vertices in the lattice and is isometric to \(\textnormal{P}\). Moreover, there is a fairly explicit group isomorphism between \(\Lambda\) and \(\mathbb{Z}^2\). We denote by \(\textnormal{SL}(2, \Lambda)\) the group of automorphisms of \(\Lambda\); it can be shown, without too much effort, that this group is isomorphic to \(\textnormal{SL}(2, \mathbb{Z})\), hence the notation. 

\smallskip

\noindent The action of \(\textnormal{SL}(2, \Lambda)\) preserves the lattice, but not the tiling. In fact, the tiling is preserved only by the action of the lattice itself via translations. To determine the effect on the tiling, it is enough to understand the action on the tile \(\textnormal{P}\). Suppose \(\textnormal{P}\) is a fundamental domain of some lattice, say \(\Lambda(\,\mu_1,\mu_2\,)\). Given any element \(A\in\textnormal{SL}(2, \Lambda)\),
the image \(A \cdot \textnormal{P}\) is a tile whose vertices are lattice points determined by the vectors 
\(A(\,\mu_1\,)\) and \(A(\,\mu_2\,)\). If we imagine this tile overlaid on the original tiling, in which \(\textnormal{P}\) served as a fundamental domain, then \(A \cdot \textnormal{P}\) admits a decomposition into polygons determined by the intersection of \(A \cdot \textnormal{P}\) itself with translated tiles of \(\textnormal{P}\). These polygons, once appropriately translated and reassembled, yield the original parallelogram \(\textnormal{P}\). Therefore, two parallelograms in the plane define the same translation surface if and only if, after suitable decomposition and rearrangement, they reconstruct the same parallelogram, see Figure \ref{fig:parallelograms}.

\smallskip

\noindent A fundamental domain is the object carrying the richest amount of information: it defines a translation structure and provides an explicit set of generators for the lattice. Recall from \S\ref{ssec:devhol} that the holonomy of a translation structure on \( \Sigma \) is a representation of \( \pi_1(\,\Sigma\,) \) into \( \mathbb{C} \). A translation structure carries less information because the image of its holonomy only defines the lattice uniformising it, without specifying any generating set nor a fundamental domain. In fact, a generating set for this lattice depends on the choice of generators for the fundamental group or, equivalently, for the first homology group. Thus, a fundamental domain is not uniquely determined by a translation structure. Table~\ref{tab:determination} shows the relationship among fundamental domains, lattices and translation structures on a torus.

\begin{table}[ht!]
    \centering
    \begin{tikzpicture}[baseline=(current bounding box.center), every node/.style={scale=0.875}]
    \node (A) at (0,0) {$\left\{ \begin{array}{c}
    \text{fundamental domains} \\
    \text{\textit{i.e.} parallelograms in } \mathbb{C}
    \end{array} \right\}$};

    \node (B) at (6,0) {$\left\{ \begin{array}{c}
    \text{translation structures} \\
    \text{on a torus}
    \end{array} \right\}$};
    \draw[->, thin] ([xshift=2mm,yshift=1.5mm]A.east) -- ([xshift=-2mm,yshift=1.5mm]B.west);
    \draw[<-, thin, dashed] ([xshift=2mm,yshift=-1.5mm]A.east) -- ([xshift=-2mm,yshift=-1.5mm]B.west);

    \node (C) at (11.5,0) {$\left\{ \begin{array}{c}
    \text{subgroups of }\mathbb C \\
    i.e. \text{ lattices }\Lambda < \mathbb C
    \end{array} \right\}$};
    \draw[<-, thin] ([xshift=2mm,yshift=-1.5mm]B.east) -- ([xshift=-2mm,yshift=-1.5mm]C.west);
    \draw[->, thin] ([xshift=2mm,yshift=1.5mm]B.east) -- ([xshift=-2mm,yshift=1.5mm]C.west);
    \end{tikzpicture}
    \caption{The solid arrow denotes a complete determination of one object by another; a broken arrow denotes that some choice is involved. } 
    \label{tab:determination}
\end{table}

\noindent The holonomy generally does not carry enough information to uniquely determine a translation structure. While this is always the case for surfaces of genus \( g \geq 2 \), in the case of genus one surfaces, if a representation determines a translation structure, then such a structure is unique up to isomorphism, \textit{i.e.}, up to biholomorphisms that act as translations in local charts. A proof of this fact is deferred to Proposition \ref{prop:hauptgenone} in \S\ref{ssec:hologenone}.

\smallskip

\subsection{Moduli spaces}\label{ssec:modulgenone} We are now ready to define a moduli space\index{space!moduli} of translation surfaces in genus one. Let us consider first the set
\smallskip
\begin{equation}
    \mathcal F= \left\{\,\,
        \begin{gathered}
        \textnormal{fundamental domains}\\
        i.e.\textnormal{ parallelograms in }\mathbb C \textnormal{ based at } 0
        \end{gathered}
        \,\,\right\}
\end{equation}

\smallskip

\noindent of parallelograms based at zero topologised as follow. Let \(\{\,\mu_1,\mu_2\,\}\) be a positive oriented basis of \(\mathbb C\). For a parallelogram \(\textnormal{P}=\left\{\, x\,\mu_1 + y\,\mu_2 \,\middle|\, x,y \in [0, 1] \,\right\}\) and \(\varepsilon>0\), we define \(U(\,\textnormal{P},\varepsilon\,)\) to be the subset of fundamental domains \(\textnormal{Q}=\left\{\, x\,\nu_1 + y\,\nu_2 \,\middle|\, x,y \in [0, 1] \,\right\}\) such that \(\nu_1=\mu_1\) and \(\nu_2\in B_{\varepsilon}(\,\mu_2\,)\), where \(B_\varepsilon(\,p\,)\) denotes the Euclidean ball centred at \(p\) of radius \(\varepsilon\). We then consider the topology generated by \(U(\,\textnormal{P},\varepsilon\,)\).

\smallskip

\noindent Each parallelogram in \(\mathcal F\) defines a translation surface of genus one and a lattice in \(\C\) along with a specified set of generators. Two parallelograms define the same translation surface if and only if they are related by \(\textnormal{SL}(2,\mathbb Z)\). One direction is a consequence of Corollary \ref{cor:sametransone} since two translation surfaces are the same if they are determined by the same lattice and hence the respective set of generators must be related by an element in \(\textnormal{SL}(2,\mathbb Z)\). On the other hand, if two parallelograms are related by an element in \(\textnormal{SL}(2,\mathbb Z)\), then there is a way to decompose and reassemble one into another. Thence, in terms of fundamental domains, the set of translation structures on the torus is defined as the quotient space

\begin{equation}\label{eq:omegamone}
    \Omega^*\mathcal M_1=\frac{\left\{\,\,
        \begin{gathered}
        \textnormal{fundamental domains}\\
        i.e.\textnormal{ parallelograms in }\mathbb C \textnormal{ based at } 0
        \end{gathered} \,\,
    \right\}}{\textnormal{SL}(2,\mathbb Z)
    }\,\,\,,
\end{equation}

\smallskip

\noindent equipped with the quotient topology.  In the same fashion, in terms of fundamental domains, the moduli space of compact Riemann surfaces of genus one can be defined as the quotient space

\begin{equation}
        \mathcal M_1=\frac{\left\{\,\,
        \begin{gathered}
        \textnormal{fundamental domains}\\
        i.e.\textnormal{ parallelograms in }\mathbb C \textnormal{ based at } 0
        \end{gathered} \,\,
        \right\}}{
        \textnormal{SL}(2,\mathbb Z)\times\mathbb C^*\,
        }\,\,\,.
\end{equation}

\smallskip

\noindent In this quotient, the factor \(\mathbb C^*\) accounts for the rescaling and rotation of the parallelogram by a complex number. Specifically, this action identifies parallelograms that differ only by a similarity transformation of the complex plane. We next describe \(\mathcal M_1\) by identifying it with a known structure. First, observe that the actions of \(\textnormal{SL}(2, \mathbb Z)\) and \(\mathbb C^*\) on \(\mathcal F\) commute. Let us first consider the rescaling action of \(\mathbb C^*\).


\smallskip

\noindent Let \(\textnormal{P}=\left\{\, x\,\mu_1 + y\,\mu_2 \,\middle|\, x,y \in [0, 1] \,\right\}\) be a parallelogram in \(\mathcal F\). By rescaling \(\textnormal{P}\) by \(\mu_1^{-1}\), we obtain a new parallelogram, say \(\textnormal{P}(\,\tau\,)=\mu_1^{-1}\cdot \textnormal{P} = \left\{\, x + y\,\tau \,\middle|\, x,y \in [0, 1] \,\right\}\), where \(\tau = \mu_2 / \mu_1\). Notice that \(\textnormal{P}(\,\tau\,)\) is the fundamental domain of the lattice \(\Lambda(\,1,\,\tau\,)\). The assumption of positive orientation guarantees that \(\tau \in \mathbb{H}\). We may also notice that for every \(\tau_1,\tau_2\in\mathbb H\), the parallelograms \(\textnormal{P}(\,\tau_1\,)\) and \(\textnormal{P}(\,\tau_2\,)\) are the same up to scaling if and only if \(\tau_1=\tau_2\). Hence, for every element \([\,\textnormal{P}\,]_{\lambda\in\mathbb C^*}\in\mathcal F/\mathbb C^*\), there is a unique \(\tau\in\mathbb H\) such that \(\textnormal{P}(\,\tau\,)\in[\,\textnormal{P}\,]_{\lambda\in\mathbb C^*} \). As a consequence, the map \(\Phi\colon \mathcal F/\mathbb C^*\,\longrightarrow \mathbb H\), that associates to every \([\,\textnormal{P}\,]_{\lambda\in\mathbb C^*}\) the unique \(\tau\in\mathbb H\) such that \(\textnormal{P}(\,\tau\,)\in[\,\textnormal{P}\,]_{\lambda\in\mathbb C^*}\), is well defined. Moreover, \(\Phi\) is also surjective because the parallelogram \(\textnormal{P}(\,\tau\,)\) exists for every \(\tau\in\mathbb H\). Hence \(\Phi\) is a bijection that is more than a mere identification of sets. In fact, the following proposition holds.

\begin{prop}\label{prop:bijfch}
    The map \(\Phi\colon \mathcal F/\mathbb C^*\,\longrightarrow \mathbb H\) is a homeomorphism with respect to the given topologies.
\end{prop}

\begin{proof}
    In order to show this proposition, it is convenient to consider the mapping \(\psi\colon\mathcal F \longrightarrow \mathbb H\) that associates to a parallelogram  \(\textnormal{P}=\left\{\, x\,\mu_1 + y\,\mu_2 \,\middle|\, x,y \in [0, 1] \,\right\}\)  the element \(\mu_1^{-1}\mu_2 \in \mathbb H\). This map is clearly surjective, and it can be seen that it is also continuous and open. Moreover, \(\psi\) is \(\mathbb C^*\)-invariant and thus descends to a homeomorphism \(\Phi\colon \mathcal F/\mathbb C^*\,\longrightarrow \mathbb H\).
\end{proof}

\smallskip

\noindent Next, consider the \(\textnormal{SL}(2,\mathbb Z)\) action on both \(\mathcal F/\mathbb C^*\) and \(\mathbb H\). Since it can be shown that \(\Phi\) is  \(\textnormal{SL}(2,\mathbb Z)\)-equivariant the following consequence holds. 

\begin{cor}
    The mapping \(\phi\colon\mathcal M_1\longrightarrow \mathbb H\,/\textnormal{SL}(\,2,\,\mathbb Z\,)\) is a homeomorphism.
\end{cor}

\medskip

\subsection{Volume function}\label{ssec:volfunc} The parallelogram generated by the vectors \(\mu_1\) and \(\mu_2\) has a well defined two dimensional volume, which can be computed using the formula \( \mathfrak{I}(\,\overline{\mu_1}\,\mu_2\,) \). The volume is clearly invariant under the action of \(\mathbb{C}\) by translations. Moreover, it follows from the above that the action of \(\textnormal{SL}(2, \mathbb{Z})\) preserves the volume of the parallelogram. More formally, this holds because \(\textnormal{SL}(2,\,\mathbb{Z}) = \textnormal{Sp}(2,\mathbb{Z})\), and the latter preserves the volume form. Thus there exists a well defined function called \textit{volume}\index{volume}, \textit{i.e.} a functional
\begin{equation}\label{eq:volumegenone}
    \textnormal{vol}\colon \Omega^*\mathcal{M}_1 \longrightarrow \mathbb R
\end{equation}
that assigns to each element of \(\Omega^*\mathcal{M}_1\) the volume of the associated translation structure which is defined as the 2-dimensional volume of the parallelogram. Its image lies in \(\mathbb R^+\) because \(\{\,\mu_1,\mu_2\,\}\) is a positively oriented basis of the complex plane. 


\medskip

\subsection{Marked tori}\label{ssec:markinggenone} In order to treat translation surfaces under a Teichm\"uller-like perspective, it is needed to equip a translation surface with an additional datum. A \textit{marking}\index{marking!homology} on a Riemann surface \(X\) is the choice of a system of generators of the homology group \(\textnormal{H}_1(\,X,\,\mathbb{Z}\,)\). 

\begin{rmk}
    For a general Riemann surface, a marking is defined as a choice of a generating set for the first homotopy group. However, for surfaces of genus one, the isomorphism \(\pi_1(\,S_1\,) \cong \textnormal{H}_1(\,S_1, \mathbb{Z}\,)\) is well known; hence, choosing a marking in homology is equivalent to choosing a marking in homotopy\index{marking!homotopy}. Although a marking in homology is generally a weaker notion for higher-genus surfaces, it is the most convenient framework to adopt when dealing with translation structures on Riemann surfaces.
\end{rmk}

\noindent Equivalently, given a reference topological surface \(\Sigma\), a marking is an identification
\begin{equation}
m\colon \textnormal{H}_1(\,\Sigma,\,\mathbb{Z}\,) \longrightarrow \textnormal{H}_1(\,X,\,\mathbb{Z}\,).
\end{equation}
Two marked Riemann surfaces\index{marking!Riemann surface} of genus one, say \((X, m_X)\) and \((Y, m_Y)\), are said to be equivalent if there exists a biholomorphism \(f\colon X \longrightarrow Y\) such that \(m_X = f^* m_Y\). The moduli space of marked Riemann surfaces, commonly referred to as the \textit{Teichm\"uller space of genus one}\index{space!Teichm\"uller}, is denoted by \(\mathcal{T}_1\) and the forgetful map that assigns to each marked surface its underlying complex structure is an \(\textnormal{SL}(2,\mathbb{Z})\)-invariant mapping \(\mathcal{T}_1 \longrightarrow \mathcal{M}_1\). We will not give further details here and the interested reader may refer to \cite{IT92}. The following characterisation holds.

\begin{prop}\label{prop:teichtrans}
    The Teichm\"uller space \(\mathcal T_1\) can be identified with \(\mathbb H\,\), the upper half-plane. In particular, there is a unique complex structure on \(\mathcal T_1\) that makes the identification a biholomorphism. 
\end{prop}

\smallskip

\begin{proof}
    In the light of Proposition \ref{prop:bijfch}, it is sufficient to show the existence of a bijection \(\mathcal F/\mathbb C^*\longrightarrow \mathcal T_1\). One direction is as follows. In the first place let us fix a standard ordered generating set for \(\textnormal{H}_1(\,\Sigma,\,\mathbb{Z}\,)\). Then, for an element \([\,\textnormal{P}\,]_{\lambda\in\mathbb C^*}\in \mathcal F/\mathbb C^*\), let \(\textnormal{P}\) be any representative and let \(X\) be complex structure arising by gluing the opposite sides of \(\textnormal{P}\) by translations. Notice that this is well defined because any other parallelogram in the same equivalence class differs from \(\textnormal{P}\) by a dilation. Recall that \(\textnormal{P}\) yields a lattice, say \(\Lambda\), along with an explicit set of generators. It is possible to find a homeomorphism \(\varphi\colon\Sigma\to\mathbb C/\Lambda\) that sends the first and second generators of \(\textnormal{H}_1(\,\Sigma,\,\mathbb{Z}\,)\) to the first and second generators of \(\Lambda\). The induced map in homology provides the desired marking. 

    \noindent Conversely, let \(X\) be a Riemann surface of genus one and recall that its universal cover is \(\mathbb C\). Endow the latter with the standard Euclidean metric and thus endow \(X\) with the unique flat metric that makes the covering projection a local isometry. A marking \(m_X\) on \(X\) yields a set of generators in homology, that is, a system of generators for the fundamental group. Consider their geodesic representatives and lift them to the universal cover. They provide a tiling of \(\mathbb C\) and any tile provides the desired fundamental domain. We note that the parallelogram depends on the chosen metric on \(\mathbb C\), but any rescaling of the metric would produce a parallelogram in the same equivalence class as the one just defined.
\end{proof}

\medskip

\noindent In terms of parallelograms, Proposition \ref{prop:teichtrans} shows that the Teichm\"uller space can be identified with the space

\begin{equation}
        \mathcal T_1=\frac{\left\{\,\,
        \begin{gathered}
        \textnormal{fundamental domains}\\
        i.e.\textnormal{ parallelograms in }\mathbb C \textnormal{ based at } 0
        \end{gathered} \,\,
        \right\}}{\mathbb C^*}\,\,\,.
\end{equation}

\smallskip

\noindent As a consequence, a \textit{marked} Riemann surface of genus one can be equivalently defined as the choice of a class of parallelograms \(\big[\,\textnormal{P}\,\big]_{\lambda\,\in\,\mathbb C^*}\), where the parallelogram \(\textnormal P\) is the fundamental domain of some lattice, say \(\Lambda\), such that \(\mathbb C/\Lambda = X\).  


\smallskip

\noindent 
Given a Riemann surface \(X\) and a marking \(\big[\,\textnormal{P}\,\big]_{\lambda\,\in\,\mathbb C^*}\), the choice of parallelogram yields a well defined translation surface. Conversely, every parallelogram yields a marked translation surface. Therefore, in terms of fundamental domains the set of \textit{marked translation structures} on the torus is simply defined as the space

\smallskip
\begin{equation}
    \Omega^*\mathcal T_1=\left\{\,\,
        \begin{gathered}
        \textnormal{fundamental domains}\\
        i.e.\textnormal{ parallelograms in }\mathbb C \textnormal{ based at } 0
        \end{gathered} \,\,
    \right\}=\mathcal F.
\end{equation}

\bigskip

\subsection{A more analytic approach}\label{sssec:analyticapproac} 
In this section, we aim to provide a more analytical approach to the moduli space of translation structures. 
Once again, we proceed almost as if we were defining the space of complex structures on the torus, but we will place greater emphasis on translation structures. Given any lattice, let us ignore for the moment the associated translation structure as defined in \S\ref{sssec:fundomain}. 

\smallskip


\noindent 

\noindent 
In \S\ref{ssec:modulgenone} we showed the existence of a homeomorphism \(\Phi\colon\mathcal F/\mathbb C^*\longrightarrow \mathbb H\). For every \(\big[\,\textnormal{P}\,\big]_{\lambda\,\in\,\mathbb C^*}\in\mathcal F/\mathbb C^*\), there is a unique representative parallelogram of the form \(\textnormal{P}(\,\tau\,)= \left\{\, x + y\,\tau \,\middle|\, x,y \in [0, 1] \,\right\}\). In turns, \(\textnormal{P}(\,\tau\,)\) determines the lattice \(\Lambda(1,\tau)\). Thus \(\mathbb{H}\) can also be seen as the moduli space of normalised lattices of the form \(\Lambda(1,\tau)\), where \(\tau\) is a parameter that distinguishes between marked complex structures on the torus. By normalising a lattice \(\Lambda\) into the form \(\Lambda(1,\tau)\) a rescaling is involved, see \S\ref{ssec:modulgenone}. This normalisation preserves the complex structures, but it does not preserve the translation structure. As a consequence, the parameter \(\tau\) is not sufficient to account for all possible translation structures defined on \(\mathbb{C} / \Lambda(1,\tau)\); hence, it is necessary to introduce an additional parameter to characterise them.

\smallskip

\noindent For every \(\tau \in \mathbb{H}\), let \(\Lambda(1,\tau)\) be the lattice generated by \(1\) and \(\tau\) and let \(\textnormal{P} = \big\{\,x + y\tau \,\mid\, x, y \in [0, 1]\,\big\}\) be the fundamental domain for the action of \(\Lambda(1,\tau)\) on \(\mathbb C\). The holomorphic structure on \(\mathbb{C}\) naturally descends to a holomorphic structure on \(\mathbb{C}/\Lambda\), turning it into a Riemann surface. We next wish to introduce a translation structure on \(X\). Under the complex-analytic point of view, this amounts to choosing an abelian differential on \(\mathbb C\) that is invariant under translations. The question we now ask is the following: how many translation-invariant differentials exist on \(\mathbb C\)? The following result answers our question.

\begin{lem}
The only holomorphic \(1\)-forms on the complex plane \(\mathbb{C}\) that are invariant under translations are the constant multiples of \(dz\). That is, every holomorphic \(1\)-form on \(\mathbb{C}\) invariant under all translations is of the form \(\omega = \lambda\,dz\) for some \(\lambda \in \mathbb{C}\).
\end{lem}

\begin{proof}
    Let \(\omega = f(\,z\,)\,dz\) be a holomorphic \(1\)-form on \(\mathbb{C}\). The pull back of \(\omega\) by a translation \(t_a(\,z\,) = z + a\), for some \(a \in \mathbb{C}\), is given by
    \begin{equation}
    t_a^* \omega = f(\,z+a\,)\,dz.
    \end{equation}
    Requiring \(\omega\) to be invariant under translations means that \(t_a^* \omega = \omega\) for all \(a \in \mathbb{C}\). Hence, for all \(z, a \in \mathbb{C}\),
    \begin{equation}
    f(\,z+a\,) = f(\,z\,).
    \end{equation}
    This implies that \(f\) is invariant under all translations of \(\mathbb{C}\), and therefore must be constant. Indeed, the only entire functions invariant under all translations are the constant ones. It follows that \(\omega = \lambda\,dz\) for some constant \(\lambda \in \mathbb{C}\), concluding the proof.
\end{proof}

\begin{rmk}
    A Euclidean metric on \(\mathbb{C}\) is defined by a non-vanishing holomorphic differential \(\omega\). For the metric to be the standard flat one, \(\omega\) must take the form \(\omega = \lambda\, dz\), where \(\lambda\in \mathbb{C}^*\) is a non-zero constant. The space of such differentials is the space of non-zero translation-invariant differentials on \(\C\).
\end{rmk}

\noindent Let \(\omega = \lambda\,dz\) be a non-trivial holomorphic \(1\)-form on \(\mathbb{C}\), \textit{i.e.} with \(\lambda \neq 0\). Since \(\omega\) is invariant under translations, it descends to a well defined holomorphic \(1\)-form on \(\mathbb{C}/\Lambda\), thereby defining a translation structure. Conversely, a holomorphic \(1\)-form on \(\mathbb C/\Lambda\) lifts to a doubly periodic \(1\)-form on the complex plane. Therefore, it must be of the form \(\omega=\lambda\,dz\) for some \(\lambda \neq 0\). We observe that here we have tacitly identified \(\mathbb{C}\) with the universal cover of \(X\). This is possible only because the developing map is a global diffeomorphism; see Lemma~\ref{lem:torusarecomplete} and Remark \ref{rmk:devgenusone}. We have the following characterisation result.

\begin{prop}\label{prop:modulimarkedtransgenone}
    Let \(\Sigma\) be a torus. Then for every \(\tau\in\mathbb H\) and for every \(\lambda\in\mathbb C^*\) there is a unique marked translation structure on \(\Sigma\). In other words, there is a holomorphic bijection between the spaces
    \smallskip
    \begin{equation}
        \Omega^*\mathcal T_1 = \left\{\,\,
        \begin{gathered}
        \textnormal{fundamental domains,}\\
        i.e.\textnormal{ parallelograms in }\mathbb C
        \end{gathered} \,\,
        \right\}\,=\,
        \left\{\quad\begin{gathered} \textnormal{marked}\\\textnormal{translation structures}\\ \textnormal{of genus one}\ \end{gathered}\quad\right\} \,\,\,
        \longleftrightarrow \,\,\mathbb H\,\times\,\mathbb C\setminus\{\,0\,\},
    \end{equation}
    obtained by associating a parallelogram \(\textnormal P\) defined by \(\{\,\mu_1,\,\mu_2\,\}\) to \((\,\tau,\,\lambda\,)\) where \(\tau=\mu_1/\mu_2\) and \(\lambda=\mu_2\). 
\end{prop}

\noindent Thus \(\Omega^*\mathcal T_1\) is a \(\mathbb C^*\)-bundle over the Teichm\"uller space \(\mathcal T_1\) and the polar coordinates of \(\mathbb C^*\cong\mathbb R\times\mathbb S^1\) have a very concrete meaning: the radial coordinate represents the volume of a translation structure whereas the angular coordinate in \(\mathbb S^1\) represents the direction of the north. The following holds.

\begin{cor}
    A marked translation surface of genus one is uniquely determined by its underlying marked complex structure, its volume and the direction of the north.
\end{cor}

\begin{proof}
    A translation structure clearly determines an underlying marked complex structure and has a well defined volume and a north direction. Conversely, every parallelogram determines a marked complex structure. Up to rescaling (dilation) by an appropriate real factor, the parallelogram may have any positive volume. Finally, choosing a direction of the north means choosing which foliation of the torus is vertical, namely, induced by the vertical foliation of \(\mathbb C\). Since foliations of the torus are parametrised by \(\mathbb S^1\), we have the desired result.
\end{proof}

    


\medskip

\section{Moduli spaces of translation surfaces}\label{sec:modulihighgen}

\noindent We now discuss the moduli space of translation surfaces of genus \(g\ge2\). Therefore, from this section onwards, we shall assume that \(g \geq 2\) unless otherwise specified. Occasionally, we will refer to translation structures in genus one, but mostly as remarks that allow us to draw comparisons with the higher genus case.

\smallskip

\subsection{Marked translation surfaces}\label{ssec:markedhighgen} The moduli space\index{space!moduli} \(\mathcal{M}_g\) of Riemann surfaces of genus \(g \geq 2\) is a normal complex analytic space of dimension \(3g - 3\) and it is a complex manifold away from the locus of surfaces admitting non-trivial automorphisms\index{automorphism}, see \cite{IT92} and \cite{Nag}. For our approach, however, it is more convenient to work with the Teichm\"uller space\index{space!Teichm\"uller} \(\mathcal{T}_g\), which is a complex manifold everywhere.

\smallskip

\noindent Let \(S\) be a reference surface, namely a closed topological surface of genus \(g\). A \textit{marking}\index{marking!Riemann surface} on a Riemann surface \(X\) is the data of an isomorphism \(m_X \colon \pi_1(\,S\,) \longrightarrow \pi_1(\,\Sigma\,)\). The \textit{Teichm\"uller space} is defined as the moduli space of marked Riemann surfaces, where two marked surfaces \((X,\,m_X)\) and \((Y,\,m_Y)\) are declared to be equivalent if there exists a biholomorphism \(f \colon X \longrightarrow Y\) such that \(m_X = f^* m_Y\). Requiring two marked structures to be equivalent is a stronger condition than merely asking for the underlying complex structures to be biholomorphic. Indeed, it is required that the biholomorphism preserves the marking, that is, the biholomorphism must be isotopic to the identity. \(\mathcal T_g\) is well-known to be a simply connected complex manifold of dimension \(3g-3\). The mapping class group \(\textnormal{Mod}_g\) defined as
\begin{equation}
    \textnormal{Mod}_g\,\cong\,\frac{\textnormal{Homeo}^+(\,S_g\,)}{\textnormal{Homeo}_o(\,S_g\,)}
\end{equation}
\noindent acts properly discontinuously on \(\mathcal T_g\) and the quotient space \(\mathcal T_g\,/\,\textnormal{Mod}_g\cong \mathcal M_g\),
see \cite{IT92} for more details.

\smallskip

\noindent A \textit{marked translation surface}\index{marking!translation surface} is a translation surface along with a marking. In complex-analytic terms, this can be seen as a triple \((X,\omega,m_X)\) and two such triples are equivalent if and only if there exists a biholomorphism isotopic to the identity which is a translation in local charts. 

\subsection{Hodge bundle}\label{ssec:hodge} The Hodge bundle \(\Omega \mathcal{M}_g \to \mathcal{M}_g\) is the holomorphic vector bundle over \(\mathcal M_g\), the moduli space of compact Riemann surfaces of genus \(g\), whose fibre over a point \(X \in \mathcal{M}_g\) is given by the space of holomorphic differentials on \(X\), \textit{i.e.} \(\Omega(\,X\,) = \textnormal{H}^o(X,\, K_X)\), where \(K_X\) denotes the canonical bundle of \(X\). It is a direct application of Riemann--Roch's Theorem to show that \(\Omega(\,X\,)\) has dimension \(g\) and hence the total space has complex dimension \(4g-3\). 

\smallskip

\noindent The Hodge bundle can be pulled back to a holomorphic vector bundle \(\Omega\mathcal T_g\longrightarrow \mathcal T_g\). The strong advantage of working with the Teichm\"uller space is that \(\mathcal{T}_g\) is a \textit{Stein manifold}, see \cite[\S6]{IT92}. 

\begin{rmk}
    There are several characterisations of Stein manifolds in the literature, see \cite{FF} for more details. For our purposes, we recall for the reader's convenience that an open domain in \(\mathbb{C}^n\) is a Stein manifold if and only if it is a domain of holomorphy, that is, there exists a holomorphic function defined on that domain which cannot be extended over the boundary. Bers' embedding ensures that the Teichm\"uller space is biholomorphic to an open subset of \(\mathbb{C}^{3g-3}\). Independent works of Royden and Earle--Kra, combined with classical results in the theory of Stein manifolds, imply that the Teichm\"uller space is a Stein manifold for every genus \(g\). We refer the reader to \cite{Faraco20} for a succinct survey, and to \cite[\S6.4]{IT92} for a more detailed exposition.
\end{rmk}

\noindent Since \(\mathcal{T}_g\) is simply connected, it follows that every vector bundle is \textit{holomorphically} trivial. This is, in fact, a consequence of the Oka--Grauert principle, which asserts that every topological complex vector bundle over a Stein manifold admits an equivalent holomorphic vector bundle structure. Moreover, two holomorphic vector bundles over a Stein base are holomorphically equivalent if and only if they are topologically equivalent. We refer to \cite{FF} and references therein for further details. As a consequence, \(\Omega\mathcal T_g\) is biholomorphic to \(\mathcal T_g\times\mathbb C\) as complex vector bundles.

\smallskip

\noindent The zero section provides a copy of the Teichm\"uller space in \(\Omega\mathcal T_g\) and its complement \(\Omega^*\mathcal T_g\) is the moduli space of marked translation surfaces in genus \(g\). In fact, every point in \(\Omega^*\mathcal T_g\) corresponds to a non-zero holomorphic \(1-\)form on a marked Riemann surface \((X,m_X)\) and thus defines a marked translation structure. By design, the projection map \(\pi\colon\mathcal T_g\longrightarrow \mathcal M_g\) yields a projection map \(\Omega(\,\pi\,)\colon\Omega\mathcal T_g\longrightarrow \Omega\mathcal M_g\) that naturally restricts to \(\Omega^*\mathcal T_g\). The moduli space of (unmarked) translation surfaces is the image of \(\Omega(\,\pi\,)\) and denoted by \(\Omega^*\mathcal M_g\). The image of the zero section in \(\Omega\mathcal M_g\) via the projection \(\Omega(\,\pi\,)\) is the zero section of the Hodge bundle and it is called \textit{trivial stratum}.

\smallskip

\subsection{Strata of differentials}\label{ssec:strata} Following Zorich in \cite{Zo}, the notion of “stratum” originates as follows. As already mentioned in \S\ref{ssec:gbcond}, and further discussed in Remark~\ref{rmk:gbanal}, the sum of the orders of the zeros of a holomorphic \(1\)-form on a compact Riemann surface of genus \(g\) must be equal to \(2g - 2\). Therefore, the subspace \(\Omega^*\mathcal{M}_g\) of the Hodge bundle is naturally stratified into subspaces consisting of holomorphic differentials whose zeros have prescribed orders \(m_1, \dots, m_k\), satisfying \(m_1 + \cdots + m_k = 2g - 2\). In what follows we shall adopt when necessary the following terminology.

\begin{defn}
    We refer to any partition of \(2g-2\), say \(\mu\), as a \textit{signature}\index{signature}, and we shall say that a holomorphic \(1\)-form \(\omega\) has signature \(\mu= (m_1, \dots, m_k)\) if it has exactly \(k\) zeros of orders \(m_1, \dots, m_k\). We shall denote by \(\strata\) the strata\index{stratum} of holomorphic differentials with signature \(\mu\). Finally, for a signature \(\mu\) we denote by \(|\,\mu\,|\) its length, namely the number of entries (counted with multiplicity).
\end{defn}

\smallskip

\noindent It is a fact that every stratum is non-empty. This follows from a classical result in the theory of Riemann surfaces: given a Riemann surface \(X\), there exists a holomorphic differential with a prescribed signature, see \cite[Section III.5]{FK}. In \S\ref{ssec:shiffer} we provide an alternative and geometric argument, see Lemma \ref{lem:stratanotempty}. A stratum \(\strata\) is a complex orbifold whose dimension depends on the signature. The presence of orbifold points is not really surprising and it is due to the existence of translation surfaces with a non-trivial group of symmetries, see \S\ref{ssec:autgroup}. Away from these points, the stratum is a complex manifold. The following result holds.

\begin{prop}\label{prop:stratadim}
    Let \(\mu\) be a signature. Then \(\strata\) has complex dimension \(2g+|\,\mu\,|-1\).
\end{prop}


\noindent There are several proofs of this result. One is based on a topological approach and consists in suitably triangulating the surface. An alternative method is based on the construction of an appropriate map in relative homology. In this exposition, we provide a different proof which relies on the same underlying ideas as the others but does not make explicit use of the periods of holomorphic differentials, as these have not yet been introduced. In order to abridge a little the notation, in the proof it is more convenient to consider a partition \(\mu\) of \(2g-2\) of length \(|\,\mu\,|=k+1\). 

\begin{proof}
    Let \(\mu=(\,m_o,\dots,m_k\,)\) be a signature of holomorphic \(1-\)forms. Let \((X,\omega)\in\strata\) be a translation surface and let \(\Delta=\{\,p_o,\dots,p_k\,\}\) be the set of zeros of \(\omega\).  For the moment, let us regard \(X\) as a topological surface and define a set of simple closed smooth curves, say \( \Gamma = \{\,\gamma_i \mid 1 \leq i \leq s\,\} \), such that the following holds.
    \begin{itemize}
        \item[1.] \( X \setminus \Gamma \) is simply connected. \( X \setminus\Gamma \) is topologically a polygon whose sides occur in pairs, each pair of sides corresponding to a curve \( \gamma_i \in \Gamma \),
        \item[2.] the curves \( \gamma_i \) do not intersect except perhaps at the endpoints, and 
        \item[3.] the set of endpoints of the curves \( \gamma_i \) in \( \Gamma \) contains the set of singular points of \( X \).
    \end{itemize}

    \noindent Notice that such a set of curves \( \Gamma \) can always be found. In fact, we may take a standard set of generators for the fundamental group, say \( \{\alpha_i,\, \beta_i \mid 1 \leq i \leq g\} \), based at \( p_o \), dissecting \( X \) into a \( 4g \)-gon. Next, we consider \(k\) additional curves \( \delta_i \) going from \( p_o \) to each singular point \( p_i \) for \(i=1,\dots,k\). Since there are \( k \) singular points, the construction yields an \textit{embedded} topological \( (4g + 2k) \)-gon in the universal cover, say \(\mathcal F\), whose sides are identified in pairs. The image under the developing map defines, in turn, a polygon in the plane in a very broad sense: the image is a topological disc immersed (but not embedded) in the plane and bounded by a piecewise smooth curve with at most \(4g + 2k\) corner points corresponding to the vertices of the polygon. 
    
    \noindent Now comes the crucial point of the construction: the curves in the set \(\Gamma\) are assumed to be smooth, and it is not necessary to assume that they are the geodesic representatives in their pointed homotopy classes based at \(p_o\). Indeed, what matters to us is \textit{how the vertices of the polygon are positioned in the plane}. Let us now proceed with the details: the developing map determines a finite set, say \(\{\,w_1,\dots,w_{4g+2k}\,\}\subset\mathbb{C}\), which corresponds to the images of the vertices of the polygon \(\mathcal F\) . The labelling is determined by the order in which the vertices of the polygon appear when traversing the boundary in the clockwise direction. Therefore, it is unique up to a cyclic permutation. The differences \(w_{i+1}-w_i\) for \(i=1,\dots,4g+2k\), where indexes are taken modulo \(4g+2k\),
    define \(4g+2k\) parameters. If two edges, say those labelled with \(i\) and \(j\), of \(\mathcal F\) are identified then the corresponding differences, \(w_{i+1}-w_i\) and \(w_{j+1}-w_j\), must be equal. It follows that half of these parameters are redundant. On the other hand, at least \(2g+k\) of these parameter are necessary to determine the image of the boundary \(\mathcal F\). To show that these parameters indeed provide local coordinates, it is necessary to prove that deforming \((X,\omega)\) continuously within the stratum results in the parameters varying continuously as well. However, this follows from the fact that continuous deformations of \((X,\omega)\) induce continuous deformations of any polygonal tessellation determined by the differential, and that the vertices of \(\mathcal{F}\) are always vertices of some Euclidean polygonal tessellation (\textit{e.g.} by taking the geodesic representatives of the curves involved in their pointed homotopy class at \(p_o\)).
\end{proof}

\begin{rmk}\label{rmk:localcoord}
    The proof above implicitly uses periods of holomorphic differentials, which we shall introduce later in \S\ref{sec:per}. A sketch of a more common and better-known proof that explicitly uses periods\index{differential!period} of holomorphic differentials is as follows. 
    Let \((X, \omega)\in\strata\) be a translation surface and let \(\Delta\) be the set of zeros of \(\omega\). Consider the relative homology group\index{homology!relative} \(\textnormal{H}_1(X,\, \Delta,\, \mathbb{Z})\) and observe that this is a \(\mathbb{Z}\)-module of dimension \(n = 2g + k - 1\), where \(k\) is the number of singularities. Choose an appropriate basis \(\{\delta_1, \delta_2, \ldots, \delta_n\}\) of \(\textnormal{H}_1(X,\, \Delta,\, \mathbb{Z})\). Then the local coordinates for \((X, \omega)\) are defined by 
    \begin{equation}
        \left(\,\,\int_{\delta_i}\, \omega\,\,\right) \in \mathbb C^{n}.
    \end{equation}
    Clearly, it is needed to show that such a map is locally one-to-one and is onto an open subset of \(\mathbb C^n\) and, even in this case, the easiest way is to use flat geometry. We refer to \cite[Proposition 1.15]{AW} for more details.
\end{rmk}


\noindent We notice that the projection map \(\strata\longrightarrow \mathcal M_g\) does not define a fibre bundle because the dimension of a stratum may be less than \(3g-3\), the dimension of \(\mathcal M_g\). In the sequel it may be useful to use the following terminology: \(\mathcal H_g(\,2g-2\,)\) is usually called \textit{minimal stratum}\index{stratum!minimal} and \(\mathcal H_g(\,1,\dots,1\,)\) is usually called \textit{principal stratum}\index{stratum!principal}. 

\smallskip

\subsection{Automorphisms group}\label{ssec:autgroup} We already mentioned that the presence of orbifold points in a given stratum is due to the fact that some translation surfaces possess more symmetries than others. In this paragraph, we make a short digression in this direction and ask: 

\smallskip
\begin{quote}
   \textit{ how many complex automorphisms\index{automorphism!complex} preserve the translation structure?}
\end{quote}
\smallskip

\noindent We fix the notation: let \((X,\omega)\) be a Riemann surface endowed with a translation structure. Let \(\textnormal{Aut}(\,X\,)\) be the group\index{automorphism!group} of complex automorphisms of \(X\), that is, homeomorphisms of \(X\) that preserve the complex structure. In the same fashion, let \(\textnormal{Aut}(X,\omega)\) be the group of homeomorphisms that preserve the translation structure determined by \(\omega\), \textit{i.e.} automorphisms of \(X\) such that the equation \(f^*\omega=\omega\) holds. By definition, \(\textnormal{Aut}(X,\omega)\subseteq\textnormal{Aut}(\,X\,)\). A classical theorem by Hurwitz states that, for a Riemann surface \(X\) of genus \(g\ge 2\), its group of complex automorphisms is a finite group of cardinality at most \(84(g-1)\), see \cite{Hur}. Groups attaining this bound are known as \textit{Hurwitz groups}, and the corresponding surfaces are called \textit{Hurwitz surfaces}. It is an interesting fact that Hurwitz surfaces do not exist in every genus. For example, in genus two, the Riemann surface known as the \textit{Bolza curve} is the most symmetric and has exactly 48 complex automorphisms, see \cite{Bol}. The first example of a Hurwitz surface appears in genus \(3\).

\smallskip

\noindent It follows that the group of automorphisms preserving the translation structure is finite and its cardinality cannot exceed the Hurwitz bound for Riemann surfaces. In~\cite{SPWS}, the authors proved the following:

\begin{thm}[ Schlage--Puchta, Weitze--Schmith\"unsen, 2017 ]
A translation surface of genus \(g\ge 2\) has at most \(4g-4\) translation automorphisms.
\end{thm}

\noindent More precisely, the authors showed that this upper bound is attained by a particular class of square-tiled surfaces known as \textit{normal origamis}, that is, square-tiled surfaces for which the ramified covering defined in Example~\ref{ex2} is a normal covering. For more details and further developments in this direction, we refer the reader to~\cite{FrRu} and~\cite{Far24}.

\medskip

\subsection{Additional structures and connected components of strata}\label{ssec:addstruc} In order to state the classification theorem for connected components of the strata \(\strata\), it is first necessary to introduce some additional invariants: the hyperelliptic structure\index{hyperelliptic!structure} and the spin structure\index{structure!spin}. In order to define them, we need to introduce some preliminary notions collected into paragraphs titled as \textit{Interlude}.  

\smallskip

\subsubsection{Interlude: hyperelliptic Riemann surfaces}\label{sssec:hyperriemannsurfaces} Before introducing hyperelliptic translation surfaces\index{hyperelliptic!translation surface}, we briefly recall some basic notions of hyperelliptic\index{hyperelliptic!Riemann surface} Riemann surfaces. Source references for this topic are \cite[Chapter III]{FK} and \cite[Chapter III]{Miranda1995}.

\smallskip

\noindent A compact Riemann surface \(X\) is said to be \textit{hyperelliptic} if there exists a complex automorphism, say \(\tau\), of order two such that \(X/\tau\cong\cp\) and the quotient map \(X\longrightarrow \cp\) is branched at \(2g+2\) points known as \textit{Weierstrass points}. According to \cite[Proposition III.4.11]{Miranda1995}, a Riemann surface \(X\) is hyperelliptic if and only if there exists a polynomial \(p(\,z\,) \in \mathbb{C}\big[\,z\,\big]\) of degree \(d\) with simple roots such that \(X\) is biholomorphic to the Riemann surface obtained by compactifying the affine algebraic curve
\begin{equation}
\mathcal C=\Bigl\{\,(z,w)\in\mathbb{C}^2 \;\,\big|\,\; w^2 = p(\,z\,)\,\Bigr\}
\end{equation}
inside the total space of the line bundle \(\mathcal{O}(\,D\,)\) over \(\mathbb{CP}^1\), where \(D = d\) or \(2D = d+1\) depending on the parity of \(d\). In the above description, the hyperelliptic involution\index{hyperelliptic!involution} is given by \(\tau(z,w) = (z,-w)\), while the projection onto \(\mathbb{CP}^1\) is the map \(\pi_X(z,w) = z\). The fixed points of \(\tau\) coincide with the Weierstrass points of \(X\). This observation makes it possible to provide an explicit description of the spaces of holomorphic differentials on \(X\). A proof of the following statement can be found in \cite[III.7.5 Corollary 1]{FK}.


\begin{prop}
    Let \(\Bigl\{\,(z,w)\in\mathbb{C}^2 \;\big|\; y^2 = P(\,z\,)\,\Bigr\}\) be a hyperelliptic Riemann surface \(X\). Then 
    \begin{equation}
        \Omega(\,X\,)=\textnormal{H}^o(X,K_X)\,=\left\langle\,\frac{z^k\,dz}{w} \,\,\Big|\,\, k=0,\dots,g-1 \,\right\rangle.
    \end{equation}
    In particular, \(\tau\) acts as \(-\textnormal{id}\) on the space of holomorphic differentials \(\Omega(\,X\,)\), \(\tau^*\omega=-\omega\) for all \(\omega\in\Omega(\,X\,)\).
\end{prop}

\smallskip

\noindent Finally, \cite[Corollary 2]{FK} states that the locus of hyperelliptic Riemann surfaces in \(\mathcal M_g\), for \(g\ge2\) is a subspace of dimension \(2g-1\). In particular \(2g-1=3g-3\) if and only if \(g=2\), \textit{i.e.} all Riemann surfaces in genus two are hyperelliptic. 

\smallskip

\subsubsection{Hyperelliptic translation surfaces}\label{sssec:hyper} We now extend the notion of hyperellipticity to translation surfaces. 

\begin{defn}\label{defn:hyperelliptictrans}
    A translation surface \((X,\omega)\) is said to be \textit{hyperelliptic}\index{hyperelliptic!translation surface} if the Riemann surface \(X\) is hyperelliptic with hyperelliptic involution \(\tau\), and \(\omega\) is an anti-invariant abelian differential such that either \(\omega\) has a single zero of maximal order, or \(\omega\) has exactly two zeros of the same order, which are exchanged by the hyperelliptic involution.
\end{defn}

\noindent This definition may look somewhat bizarre at first glance because of the requirements imposed on the abelian differential. In the present section, we aim to motivate this definition. Following \cite{KZ}, we introduce the moduli spaces of meromorphic quadratic\index{quadratic!differential} differentials\index{differential!quadratic} on a Riemann surface.


\begin{defn}\label{def:quadstratum}
Let \(h \geq 0\) be an integer, and let \((\,\ell_1, \dots, \ell_n\,)\) be a collection of integers with \(n \geq 1\) such that each \(\ell_j \geq -1\), \(\ell_j \neq 0\), and \(\ell_1+\cdots+\ell_n = 4h - 4\). Then we denote by \(\mathcal{Q}_h(\,\ell_1, \dots, \ell_n\,)\) the moduli space of pairs \((Y, \phi)\), where \(Y\) is a compact Riemann surface of genus \(h\), and \(\phi\) is a meromorphic quadratic differential on \(Y\) with zeros and poles of orders prescribed by the signature \((\,\ell_1, \dots, \ell_n\,)\). That is, \(\phi\) has a zero of order \(\ell_j\) if \(\ell_j > 0\), and a simple pole if \(\ell_j = -1\), for each \(j\). Moreover, we require that \(\phi\) is a primitive quadratic\index{quadratic!primitive} differential\index{differential!primitive}, namely it is not the global square of any holomorphic differential on \(Y\).
\end{defn}

\noindent Notice that the condition that \(\phi\) is not a global square of an abelian differential is automatically satisfied if at least one of the integers \(\ell_j\) is odd. It is known that each stratum\index{stratum!quadratic differential} \(\mathcal{Q}_h(\,\ell_1, \dots, \ell_n\,)\) is a complex orbifold of dimension \(2h + n - 2\), see \cite{KZ}.

\smallskip

\noindent With every meromorphic quadratic differential \((Y, \phi)\), one canonically associates a connected Riemann surface, say \(X\), equipped with an abelian differential \(\omega\). More specifically, given a quadratic differential, say \(\phi\) on \(Y\), there exists a canonical double covering \(\pi\colon X\longrightarrow Y\), which is cyclic of degree \(2\), such that the pull back of \(\phi\) is the square of an abelian differential \(\omega\) on \(X\). Moreover, \(\pi\) naturally comes with an involution \(\tau\) such that
\begin{equation}
\tau^* \omega = -\omega.
\end{equation}
\noindent The resulting pair \((X,\omega)\) is called the \textit{canonical double cover}. The reader is referred to \cite{Lan04} for a more concrete construction of this canonical cover. The involution \(\tau\) induces an involution in relative cohomology \(\tau^*\colon\textnormal{H}^1(X,\,\Delta,\,\C)\longrightarrow\textnormal{H}^1(X,\,\Delta,\,\C)\), where \(\Delta\) is the set of zeros of \(\omega\), and defines a decomposition \(\textnormal{H}^1(X,\,\Delta,\,\C)\cong \textnormal{E}_1\oplus\textnormal{E}_{-1}\) into the direct sum of subspaces invariant and anti invariant under the involution \(\tau\). By construction, \(\omega\in\textnormal{E}_{-1}\). Following \S\ref{sssec:hyperriemannsurfaces}, the involution \(\tau\) is hyperelliptic if \(Y\) has genus zero and \(\pi\) is branched at \(2g+2\) points. 

\smallskip

\noindent By means of this canonical cover, one can define a canonical mapping from the stratum \(\mathcal{Q}_h(\,\ell_1, \dots, \ell_n\,)\) of quadratic differentials to a stratum \(\mathcal{H}_g(\,\mu\,)\) of abelian differentials, where the list \(\mu = (\,m_1, \dots, m_k\,)\) is determined from the list \((\,\ell_1, \dots, \ell_n\,)\) according to the following:

\begin{ru}\label{ru:partition}
    \(\textnormal{ }\)
    \begin{itemize}
        \item for each even \(2\ell_j > 0\), include a pair \((\,\ell_j,\, \ell_j\,)\) in the list \((\,m_i\,)\);
        \item for each odd \(\ell_j > 0\), include a single term \(\ell_j + 1\);
        \item simple poles, \textit{i.e.} \(\ell_j = -1\), do not contribute to the list \((\,m_i\,)\).
    \end{itemize}
\end{ru}

\smallskip

\noindent Observe that \(\mu\) is a signature of the form

\begin{equation}
    \mu=\Big(\,m_1^{\gcd(2,\ell_1)},\dots,m_n^{\gcd(2,\ell_n)}\,\Big),\,\, \textnormal{ where }\,\, m_i=\frac{\ell_i+2}{\gcd(2,\ell_i)}-1,
\end{equation}
\noindent and then \(g\) is explicitly determined by \(\mu\). By recalling that \(\mu\) is a signature of \(2g-2\), and that \((\ell_1,\dots,\ell_n)\) is a partition of \(4h-4\), we can infer the following formula:

\begin{equation}
    g=1+\frac12\,\sum_{i=1}^n\,\gcd(2,\ell_i)\,m_i \,=\,2h-1+n-\frac12\left(\,\sum_{i=1}^n\,\gcd(2,\ell_i)\,\right).
\end{equation}

\medskip

\begin{lem}\label{lem:canmap}
The canonical map \(\mathcal{Q}_h(\,\ell_1, \dots, \ell_n\,) \longrightarrow \strata \) just defined is an immersion.
\end{lem}

\noindent The construction just presented works in the general setting and the image of the canonical map in \(\strata\) consists of pairs (\(X,q)\) such that \(q\) is a non-primitive quadratic differential on \(X\), that is, \(q\) is the global square of some abelian differential \(\omega\in\Omega(\,X\,)\). 
\smallskip

\noindent Recall that the aim is to provide a motivation for Definition \ref{defn:hyperelliptictrans}. For this purpose we need to work with translation surfaces whose underlying Riemann surface is hyperelliptic, therefore we set \(h=0\). Two specific series of such mappings will play a crucial role for our purposes, namely,
\begin{equation}\label{eq:specialmaps}
    \mathcal{Q}_o(\,-1^{2g+1}, 2g-3\,) \to \mathcal{H}_g(\,2g-2\,), \qquad 
    \mathcal{Q}_o(\,-1^{2g+2}, 2g-2\,) \to \mathcal{H}_g(\,g-1, g-1\,),
\end{equation}
where \(g \geq 2\). In both examples, the base Riemann surface \(Y\) has genus \(h = 0\) and hence its canonical double cover \(X\) is a Riemann surface of genus \(g\) that comes with a hyperelliptic involution by construction. Let us compare the dimensions of the domain and target strata. From Proposition \ref{prop:stratadim}, we compute:
\begin{align}
    \dim_{\mathbb{C}} \mathcal{Q}_o(\,-1^{2g+1}, 2g-3\,) &= 2 \cdot 0 + (2g+2) - 2 = 2g + 1 - 1 = \dim_{\mathbb{C}} \mathcal{H}_g(\,2g-2\,) , \\
    \dim_{\mathbb{C}} \mathcal{Q}_o(\,-1^{2g+2}, 2g-2\,) &= 2 \cdot 0 + (2g+3) - 2 =  2g + 2 - 1 = \dim_{\mathbb{C}} \mathcal{H}_g(\,g-1, g-1\,).
\end{align}

\noindent In both cases the dimensions match, and the canonical maps in \eqref{eq:specialmaps} are local homeomorphisms of complex orbifolds. For the reader's convenience, we show that these are the only cases in which the dimensions match.

\begin{lem}\label{lem:detectingpartitions}
    Let \(\nu\) be a signature of quadratic differentials on \(\cp\) and let \(\mu\) be the signature of abelian differential realised from \(\nu\) by using Rule \eqref{ru:partition}. Then the dimensions of the strata \(\mathcal Q_o(\,\nu\,)\) and \(\mathcal H_g(\,\mu\,)\) coincide if and only if 
    \begin{itemize}
        \item[1.] \(\nu=(\,-1^{2g+1}, 2g-3\,)\) and \(\mu=(2g-2)\), or
        \item[2.] \(\nu=(\,-1^{2g+2}, 2g-2\,)\) and \(\mu=(g-1,g-1)\).
    \end{itemize}
\end{lem}

\begin{proof}
    Let \(\nu=(\,\ell_1,\dots,\ell_n\,)\) be a partition of \(-4\) such that \(\ell_i\ge-1\) and \(\ell_i\neq0\) for all \(i=1,\dots,n\). Let us set the following parameters
    \begin{itemize}
        \item[1.] \(r\) is the number of indices with \(\ell_i=-1\),
        \smallskip
        \item[2.] \(s=n-r\) is the number of indices with \(\ell_i\ge1\), and 
        \smallskip
        \item[3.] \(o\) is the number of indices such that \(\ell_i\ge1\) is odd.
    \end{itemize}
    We first provide a lower and upper bound for the possible values of \(r\). Note that, since \(\ell_1+\cdots+\ell_n=-4\), it follows that \(r\ge4\). We order the entries of \(\nu\) so that \(\ell_1=\cdots=\ell_r=-1\) and \(\ell_{r+i}\ge1\) for \(1=1,\dots,s\). Let \(\mu=(\,m_1,\dots,m_k\,)\) be the signature of abelian differentials realised from \(\nu\) by using Rule \eqref{ru:partition}. Note that \(k\ge s=n-r\) because each \(\ell_{r+i}\ge1\) contributes in the realisation of \(\mu\). We now compare the sum of the entries of the signatures \(\mu\) and \(\nu\), that is,
    \begin{equation}\label{eq:comparisonone}
        \sum_{j=1}^k m_j\,=\,o\,+\,\sum_{i=1}^s\,\ell_{r+i}
    \end{equation}
    because, as an artifact of Rule \eqref{ru:partition}, every \(\ell_i\) contributes on the left-hand sum as \(\ell_i\) if it is even or it contributes as \(\ell_i+1\) if it is odd. Moreover,
    \begin{equation}\label{eq:comparisontwo}
        \sum_{i=1}^n\,\ell_i\,=\,\sum_{i=1}^s\,\ell_{r+i}\,+\,\sum_{i=1}^r\,\ell_i\,=\,-4 \,\,\Longrightarrow\,\, \sum_{i=1}^s\,\ell_{r+i}\,=\,r-4.
    \end{equation}
    By combining equations \eqref{eq:comparisonone} and \eqref{eq:comparisontwo} and by recalling the equality \(m_1\,+\cdots+\,m_k=2g-2\), we deduce that \(2g-2\,=\,o+r-4\), in particular
    \begin{equation}
        2g-2\ge r-4 \,\Longrightarrow r\le 2g+2.
    \end{equation}
    We now determine a lower bound. The dimensions of the strata \(\mathcal Q_o(\,\nu\,)\) and \(\mathcal H_g(\,\mu\,)\) coincide if and only if \(n-2=2g+k-1\), that is if and only if \(n=2g+k+1\). Since \(n=r+s\) and \(k\ge s\), we get
    \begin{equation}
        r+s \ge 2g + s + 1\,\,\Longrightarrow \,\, r\ge 2g+1.
    \end{equation}
    Therefore \(r\in\{\,2g+1,\,2g+2\,\}\). We now analyse these cases separately.
    
    \smallskip
    
    \noindent \textit{Case \(r=2g+1\)}. In this case \(2g+k+1=n=r+s=2g+1+s\) and hence \(k=s\). Equation \eqref{eq:comparisonone} implies that \(\ell_{r+i}\) is odd for all \(i=1,\dots,s\). Moreover, since \(\mu\) is a partition of \(2g-2\) and \(\ell_{r+1}+\cdots+\ell_n=2g-3\) because \(\nu\) is a partition of \(-4\), it follows that \(o=1\) that forces \(k=s=1\). Hence \(\nu=(\,-1^{2g+1}, 2g-3\,)\) and \(\mu=(2g-2)\).

    \smallskip
    
    \noindent \textit{Case \(r=2g+2\)}. Similarly to the case above, since \(\mu\) is a partition of \(2g-2\) and \(\ell_{r+1}+\cdots+\ell_n=2g-2\) because \(\nu\) is a partition of \(-4\), it follows that \(o=0\) and hence \(k=2s\). Since \(2g+k+1=n=r+s=2g+2+s\) it follows that \(k=s+1\) and hence \(k=2\) and \(s=1\). Thus \(\nu=(\,-1^{2g+2}, 2g-2\,)\) and \(\mu=(g-1,g-1)\) as desired.
\end{proof}

\begin{rmk}\label{rmk:actionoftheinvolution}
    Let \((Y,\phi)\) be any structure in \(\mathcal{Q}_o(\,-1^{2g+1}, 2g-3\,)\) and let \((X,\omega)\in\mathcal H_g(\,2g-2\,)\) be its canonical double cover. As already mentioned above, \((X,\omega)\) admits an involution \(\tau\) such that \((X,\omega)/\langle\,\tau\,\rangle\cong(Y,\phi)\). We may observe that the single zero is a fixed point for the involution \(\tau\). On the other hand, if \((Y,\phi)\) is any structure in \(\mathcal{Q}_o(\,-1^{2g+2}, 2g-2\,)\), then its canonical double cover \((X,\omega)\in\mathcal H_g(\,g-1,g-1\,)\) admits an involution \(\tau\) that swaps the zeros of \(\omega\) (otherwise \(\tau\) would have more than \(2g+2\) fixed points and \(X\) would not be hyperelliptic). 
\end{rmk}

\noindent Before returning to the special cases in \eqref{eq:specialmaps}, we state the following fundamental property of the strata of quadratic differentials in genus zero.

\begin{prop}\label{prop:genuszero}
Let \(g = 0\). Then every stratum \(\mathcal{Q}_o(\,\ell_1, \dots, \ell_n\,)\) of meromorphic quadratic differentials is non-empty and connected.
\end{prop}

\noindent The proof can be found in \cite[Proposition 1]{KZ}. Lemmata \ref{lem:canmap} and \ref{lem:detectingpartitions} and Proposition \ref{prop:genuszero} combined together motivate the following definition.

\begin{defn}
    By \textit{hyperelliptic\index{hyperelliptic!component} component}\index{component!hyperelliptic} we define the following connected component\index{component!connected} of the following strata of abelian differentials on compact Riemann surfaces of genus \(g\ge2\): the connected component \(\mathcal H_{g}^{\textnormal{hyp}}(\,2g-2\,)\) of the stratum \(\mathcal H_{g}(\,2g-2\,)\) consisting of abelian differentials on hyperelliptic curves of genus \(g\) arising from the orbifold \(\mathcal{Q}_o(\,-1^{2g+1}, 2g-3\,)\) and the connected component \(\mathcal H_{g}^{\textnormal{hyp}}(\,g-1,g-1\,)\) of \(\mathcal H_{g}(\,g-1,g-1\,)\) of abelian differentials arising from the orbifold \(\mathcal{Q}_o(\,-1^{2g+2}, 2g-2\,)\). A translation surface of genus \(g\) is said to be \textit{hyperelliptic} if it belongs either to \(\mathcal H_{g}^{\textnormal{hyp}}(\,2g-2\,)\) or \(\mathcal H_{g}^{\textnormal{hyp}}(\,g-1,g-1\,)\).
\end{defn}

\begin{rmk}
    Following Remark \ref{rmk:actionoftheinvolution}, we also note that if \(X\) is any hyperelliptic Riemann surface with hyperelliptic involution \(\tau\) and \((X,\omega)\in\mathcal H_g(\,g-1,g-1\,)\) then \(\tau\) may fix the zeros of \(\omega\). In this case the structure \((X,\omega)\) does not belong to a hyperlliptic component. 
\end{rmk}

\begin{rmk}
    As a consequence, we have the following first rough classification of connected components of strata. For every genus \(g \ge 2\), each of the strata \(\mathcal{H}_g(\,2g-2\,)\) and \(\mathcal{H}_g(\,g-1,\,g-1\,)\) is either connected and entirely consists of hyperelliptic translation surfaces, or it has at least two connected components and only one of which corresponds to hyperelliptic translation surfaces.
\end{rmk}


\smallskip

\subsubsection{Interlude 2: index of curves}\label{sssec:index} Let \((X,\omega)\) be a translation surface of genus \(g\) and let \(\Delta\) be the set of zeros of \(\omega\). Recall that away from the zeros and poles of \(\omega\), there is a well defined horizontal direction and hence a non-singular horizontal foliation on \(X \setminus \Delta\) that extends to a singular foliation over \(\Delta\), see \S\ref{sssec:north}. Let \(\gamma\) be a simple smooth curve parametrized by arc length. Assume in addition that \(\gamma\) avoids all the zeros and poles of \(\omega\). We shall define the index of \(\gamma\) as a numerical invariant given by the comparison of the unit tangent field \(\dot{\gamma}(\,t\,)\) and the unit vector field along \(\gamma\) given by unit vectors tangent to the horizontal direction. More precisely, let us denote by \(u(\,t\,)\) the unit vector at \(\gamma(\,t\,)\) tangent to the leaf of the horizontal foliation through \(\gamma(\,t\,)\). Then the assignment 
\begin{equation}\label{comphorfol}
    t \longmapsto \theta(\,t\,) = \angle\big(\, \dot{\gamma}(\,t\,),\, u(\,t\,)\, \big)
\end{equation}
defines a mapping \(f_\gamma \colon \mathbb{S}^1 \longrightarrow \mathbb{S}^1\).

\begin{defn}\label{index}
The index of \(\gamma\) is defined as \(\deg(\,f_\gamma\,)\) and denoted by \(\textnormal{Ind}(\,\gamma\,)\). 
\end{defn}

\textit{Convention.} We agree that \(\mathbb{S}^1\) is counter-clockwise oriented. As a consequence, we agree that the index \(\textnormal{Ind}(\,\gamma\,) = \deg(\,f_\gamma\,)\) is \textit{positive} if \(\gamma\) is counter-clockwise oriented.

\medskip

\noindent The index \(\textnormal{Ind}(\,\gamma\,)\) of a closed curve \(\gamma\) measures the number of times the unit tangent vector field spins with respect to the direction given by the horizontal foliation. Since for translation surfaces the parallel transport yields a trivial holonomy, see Proposition~\ref{prop:tripartrans}, the total change of the angle is \(2\pi\, \textnormal{Ind}(\,\gamma\,)\).

\smallskip

\subsubsection{Interlude 3: Spin structures}\label{sssec:spinpar} The last geometric invariant we shall introduce is the spin structure. Let \(X\) be a compact Riemann surface and let \(P\) be a principal \(\mathbb S^1\)-bundle over \(X\). 

\begin{defn}
    A \textit{spin structure}\index{structure!spin} on \(X\) is a choice of a linear functional \(\xi\colon\textnormal{H}_1(P,\, \mathbb Z_2)\longrightarrow\mathbb Z_2\) having non-zero value on the cycle representing the fibre \(\mathbb S^1\) of \(P\).
\end{defn}

\noindent Notice that this is equivalent to a choice of a double covering \(Q\longrightarrow P\) whose restriction to each fibre of \(P\) is a \(2\)-fold covering of \(\mathbb S^1\).

\smallskip

\noindent An element of \(\textnormal{H}_1(P,\, \mathbb{Z}_2)\) can be regarded as a \textit{framed closed curve} in \(X\), by which we mean a closed curve in \(X\) and a smooth unit vector field along it. For a simple closed curve \(\gamma\) in \(X\), there is a preferred frame given by its unit tangent vector field. We shall denote such a framed closed curve as \(\vec{\gamma}\). This is well defined in the sense that two homologous closed curves \(\gamma,\,\delta \in \textnormal{H}_1(X,\, \mathbb{Z}_2)\) yield homologous framed closed curves \(\vec{\gamma},\,\vec{\delta}\), see \cite[\S3]{JO}.

\smallskip

\noindent The \textit{canonical lift} of a simple closed curve \(\gamma\) is defined as \(\widetilde{\gamma} = \vec{\gamma}\,+\, \textbf{z}\), where \(\textbf{z} \in \textnormal{H}_1(P,\, \mathbb{Z}_2)\) is the homology class represented by the fibre \(\mathbb{S}^1\). The assignment \(\gamma \longmapsto \widetilde{\gamma}\) defines a map \(\lambda\colon\textnormal{H}_1(X,\, \mathbb{Z}_2) \to \textnormal{H}_1(P,\, \mathbb{Z}_2)\), which fails to be a homomorphism. In fact, \cite[Theorem 1B]{JO} states that for every \(\alpha,\beta\in\textnormal{H}_1(X,\, \mathbb{Z}_2)\) the following identity holds:
\begin{equation}\label{eq:condliftcurves}
    \widetilde{\,\alpha+\beta\,}=\widetilde{\alpha}\,+\,\widetilde{\beta}\,+\,\big(\,\alpha\cdot\beta\,\big)\,\textbf{z}, 
\end{equation}
where \(\alpha\,\cdot\,\beta\) denotes the intersection form on \(\textnormal{H}_1(X,\, \mathbb{Z}_2)\), which is symplectic. By post-composing the assignment \(\lambda\) with the spin structure \(\xi\) one obtains a \(\mathbb Z_2\)-valued function \(\Omega_\xi \colon \textnormal{H}_1(X,\, \mathbb{Z}_2) \longrightarrow \mathbb{Z}_2\) on \(\textnormal{H}_1(X,\, \mathbb{Z}_2)\) defined as
\begin{equation}
    \Omega_\xi(\,\gamma\,) = \xi\big(\,\widetilde{\gamma}\,\big) = \langle \,\xi,\,\widetilde{\gamma}\,\rangle,
\end{equation}
which is a quadratic form. Let us explain why this form is indeed quadratic. Since we work over the field \(\mathbb{Z}_2\), we recall that every symplectic form is automatically symmetric; in particular, the intersection pairing on \(\textnormal{H}_1(X,\, \mathbb{Z}_2)\) yields a symmetric form. By following \cite[\S4]{JO}, the form \(\Omega_\xi\) is a quadratic form because it satisfies the identity
\begin{equation}
    \Omega_\xi(\,\alpha+\beta\,)=\Omega_\xi(\,\alpha\,)+\Omega_\xi(\,\beta\,)+\alpha\cdot\beta,
\end{equation}
which follows from the identity \eqref{eq:condliftcurves} together with the fact that \(\xi(\,\textbf{z}\,)=1\). Following \cite{AC} and \cite[\S5]{JO}, we shall introduce the following:

\begin{defn}\label{parity}
For a symplectic basis \(\{\,\alpha_1,\beta_1,\dots,\alpha_g,\beta_g\,\}\) of \(\textnormal{H}_1(X,\, \mathbb{Z}_2)\), the Arf-invariant of \(\Omega_\xi\) is defined by the formula
\begin{equation}\label{arfinv}
    \sum_{i=1}^g \,\langle \,\xi,\,\widetilde{\alpha}_i\,\rangle\,\langle \,\xi,\,\widetilde{\beta}_i\,\rangle  \,\,(\text{mod}\,2).
\end{equation}
The \textit{parity} of a spin structure \(\xi\) is defined as the Arf-invariant of \(\Omega_\xi\). We may notice that the parity of \(\xi\) does not depend on the choice of the symplectic basis.
\end{defn}

\medskip 

\subsubsection{Spin structure\index{structure!spin} arising from an abelian differential}\label{sssec:spinabel} Let \(X \in \mathcal{M}_g\) be a Riemann surface of positive genus \(g \ge 1\) and let \(\omega \in \Omega(X)\) be a holomorphic differential on \(X\). In this paragraph, we aim to show in detail that if \((X, \omega) \in \mathcal{H}_g(\,2m_1, \dots, 2m_k\,)\), then \(\omega\) gives a well defined spin structure on \(X\).

\smallskip

\noindent Let \(\gamma\) be a simple closed curve parametrised by arc length and let \(\vartheta\) be a unit vector field along \(\gamma\). The pair \((\gamma, \vartheta)\) defines a framed closed curve, \textit{i.e.}, an element of \(\textnormal{H}_1(P,\, \mathbb{Z}_2)\). Recall that \(\omega\) defines a non-singular horizontal foliation away from the zeros, hence the unit vector field \(\vartheta\) can be compared to the unit vector field along \(\gamma\) tangent to the horizontal foliation. In fact, as in \S\ref{sssec:index}, this can be done by means of a mapping \(f_{(\gamma,\vartheta)} \colon \mathbb{S}^1 \to \mathbb{S}^1\) defined as in Equation \eqref{comphorfol}. The spin structure induced by a holomorphic differential \(\omega\) on \(X\) is then defined as the \(\mathbb{Z}_2\)-valued linear functional
\begin{equation}\label{spinstrome}
    \xi \colon \textnormal{H}_1(P,\,\mathbb{Z}_2) \longrightarrow \mathbb{Z}_2, \quad \text{such that }\quad \xi(\gamma,\vartheta) = \deg\big(f_{(\gamma,\vartheta)}\big) \,\,(\text{mod}\,2).
\end{equation}

\noindent A direct application of \eqref{spinstrome} shows that, for the canonical lift \(\widetilde{\gamma}\) of \(\gamma\), the spin structure determined by \(\omega\) satisfies the property
\begin{equation}\label{spindef}
    \xi(\,\widetilde{\gamma}\,) = \langle\,\xi,\,\widetilde{\gamma}\,\rangle = \langle\,\xi,\,\vec{\gamma} + \textbf{z}\,\rangle = \text{Ind}(\,\gamma\,) + 1 \,\,(\text{mod}\,2).
\end{equation}

\noindent As a consequence, we can apply formula \eqref{arfinv} to compute the parity \(\varphi(\,\omega\,)\) of the spin structure determined by \(\omega\). Given a symplectic basis \(\{\,\alpha_1,\beta_1,\dots,\alpha_g,\beta_g\,\}\) of \(\textnormal{H}_1(X,\, \mathbb{Z}_2)\), it follows that
\begin{equation}\label{eq:spinparity}
    \varphi(\,\omega\,) = \sum_{i=1}^g \big(\,\text{Ind}(\,\alpha_i\,) + 1\big)\big(\,\text{Ind}(\,\beta_i\,) + 1\big) \,\,(\text{mod}\,2).
\end{equation}

\noindent Since it can be shown that \(\varphi(\,\omega\,)\) does not depend on the choice of the symplectic basis, see \cite{JO}, the spin structure determined by \(\omega\) is well defined. Finally, it follows from the works of Atiyah in \cite{AM} and Mumford in \cite{MD} that the spin parity is invariant under continuous deformations.

\smallskip

\subsubsection{Connected components\index{component!connected} of strata\index{stratum!connected component}}\label{ssec:conncomp} Let us discuss first the case of holomorphic differentials which has been studied by Kontsevich--Zorich in \cite[Theorems 1 and 2]{KZ}. We split this classification in two theorems; one of which is specific to low-genus surfaces, \textit{i.e.} \(g=2,3\).

\begin{thmnn}[Kontsevich--Zorich]
The moduli space of abelian differentials on a curve of genus $g = 2$ contains two strata: $\mathcal{H}_2(1,1)$ and $\mathcal{H}_2(\,2\,)$. Each of them is connected and coincides with its hyperelliptic component\index{hyperelliptic!component}.
Each of the strata $\mathcal{H}_3(2,2)$, $\mathcal{H}_3(\,4\,)$ of the moduli space of abelian differentials on a curve of genus $g = 3$ has two connected components: the hyperelliptic one, and one having odd spin structure\index{structure!spin}. The other strata are connected for genus $g = 3$.
\end{thmnn}

\noindent For surfaces of genus $g\ge4$, the following classification holds.

\begin{thmnn}[Kontsevich--Zorich]
All connected components of any stratum of abelian differentials on a curve of genus $g\ge 4$ are described by the following list:
\begin{itemize}
    \item The stratum $\mathcal{H}_g(\,2g-2\,)$ has three connected components: the hyperelliptic one, and two other components corresponding to even and odd spin structures;
    \item the stratum $\mathcal{H}_g(\,2m,2m\,)$, $m\ge2$ has three connected components: the hyperelliptic one and two other components distinguished by the spin parity;
    \item all the other strata of the form $\mathcal{H}_g(\,2m_1,\dots,2m_k\,)$, where all $m_i \ge 1$, have two connected components corresponding to even and odd spin structures; 
    \item the strata  $\mathcal{H}_g(\,2m-1, 2m-1\,)$, $m\ge 2$, have two connected components: one comprises hyperelliptic structures and the other does not;  finally
    \item all the other strata of Abelian differentials on the curves of genus $g\ge4$ are non-empty and connected.
    \end{itemize}
\end{thmnn}

\smallskip

\subsection{Topology of strata}\label{ssec:topstrata} A question of general interest is to understand the topology of the strata, which, as of now, is still not fully understood. It is known, however, that strata are generally not connected, and their connected components were completely classified by Kontsevich and Zorich in \cite{KZ} as already discussed in \S\ref{ssec:conncomp}.

\subsubsection{Compactness}\label{sssec:comp} Given an arbitrary signature \(\mu\), the set of unit area translation surfaces forms a closed subset of the stratum \(\strata\). This follows from the fact that the function assigning to each translation surface its area with respect to the associated (singular) Euclidean metric is continuous. It is not difficult to show that, up to rescaling the metric appropriately, each translation surface has a unique representative in the unit area stratum. This is often called \textit{projective stratum}\index{stratum!projective} and denoted by \(\mathbb{P}\strata\). Notice that projectivising \(\strata\) means considering the action of \(\mathbb{C}^*\) on \(\strata\), which is continuous. Hence, the action preserves the connected components.

\smallskip

\noindent It is a well-known fact that strata are not compact, and this can be deduced from the non-compactness of the unit area stratum \(\mathbb{P}\strata\). Proving the latter, however, is far less obvious and significantly more interesting. The proof relies on a criterion, known as Masur’s compactness criterion, which states that a closed subset of the unit area stratum is compact if and only if there exists a positive lower bound on the lengths of all saddle connections on all translation surfaces in the subset, see \cite{MH92}. What is particularly interesting is that this geometric-topological result is proved using the study of a specific dynamical system defined on the unit area stratum, by constructing sequences of translation surfaces for which no uniform lower bound exists. In dynamical terms, one may say that there are orbits of translation structures that eventually leave any compact subset.

\begin{rmk}
Such sequences are therefore divergent in the unit area stratum. Interestingly, the underlying complex structures may still converge. This happens, for instance, when a structure is deformed so as to collapse two zeros of the differential into a single point on the surface. See \cite[\S3.7]{Zo}
\end{rmk}

\subsubsection{Volume\index{volume!stratum} of strata\index{stratum!volume}}\label{sssec:volofstrata} In \S\ref{ssec:strata}, we identified -- without providing a detailed proof -- a neighbourhood of a point \((X,\omega)\) in the corresponding stratum with a neighbourhood of the cohomology class \([\omega]\) in the relative cohomology group \(\textnormal{H}^1(\Sigma,\,\Delta,\,\mathbb{C})\), see Remark \ref{rmk:localcoord}. The space \(\textnormal{H}^1(\Sigma, \Delta,\, \mathbb{C})\) is a real vector space and it contains a natural integer lattice \(\textnormal{H}^1(\Sigma, \Delta,\, \mathbb{Z}+ i\, \mathbb{Z})\). There exists a linear volume element \(d\text{vol}\) in the vector space \(\textnormal{H}^1(\Sigma,\, \Delta,\, \mathbb{C})\) normalised in such a way that the volume of a fundamental domain of the lattice is equal to one. The volume form \(d\text{vol}\) naturally restricts to the projectivised stratum \(\mathbb P\strata\) to a volume form \(d\text{vol}_1\). In \cite{MH82}, \cite{Vee82} and \cite{Vee93} it was shown that
\begin{thm}[Masur, Veech]
    For every stratum \(\mathbb P\strata\), the total volume
    \begin{equation}
        \int_{\mathbb P\strata} \,d\textnormal{vol}_1
    \end{equation}
    is finite.
\end{thm}
\noindent Explicit values of these volumes were later computed by Eskin and Okounkov in \cite{EO}. In a subsequent paper \cite{EOP}, Eskin, Okounkov, and Pandharipande computed the volumes of all individual connected components of the strata, showing that each volume is a rational multiple of \(\pi^{2g}\), where \(g\) denotes the genus of the surface. For clear and accessible surveys on these results, the reader is referred to \cite{E} and \cite[\S7]{Zo}.

\subsubsection{Orbifold fundamental groups}\label{sssec:fundgroup} Strata are quasi-projective varieties and thence their (orbifold) fundamental groups are finitely presented. Kontsevich and Zorich famously conjectured that connected components of strata are classifying spaces for their orbifold fundamental group, see \cite{KZ97}. At the present time only limited progress has been achieved in this direction. 

\smallskip

\noindent A first interesting case concerns the minimal stratum \(\mathcal H_2(\,2\,)\) in genus two. We recall that such a stratum is connected, see \S\ref{ssec:conncomp}. The preimage of \(\mathcal H_2(\,2\,)\) in \(\Omega\mathcal T_2\) via the covering projection consists of exactly six connected components. We briefly explain the statement: every Riemann surface of genus two is hyperelliptic and therefore has exactly six Weierstrass points. If \((X,\omega)\in\mathcal H_2(\,2\,)\), then this unique zero must be a Weierstrass point. Indeed, the hyperelliptic involution acts by \(-1\) on holomorphic differentials, see \S\ref{sssec:hyperriemannsurfaces}, and therefore fixes the zero set of \(\omega\). Since the zero is unique, it must be fixed by the involution and hence it is a Weierstrass point. Consequently, for a fixed complex structure on \(\Sigma\), specifying a differential in the stratum \(\mathcal H_2(\,2\,)\) amounts to choosing one of the six Weierstrass points as the location of the double zero. Since the zero of the differential varies continuously under deformation, the choice of the Weierstrass point cannot change continuously. It follows that the preimage of \(\mathcal H_2(\,2\,)\) in \(\Omega\mathcal T_2\) splits into exactly six connected components, each corresponding to one of the six Weierstrass points. Each of these components determines an odd spin structure on \(\Sigma\), see Remark \ref{rmk:algspin}. Hence the orbifold fundamental group of the projectivised stratum \(\mathbb P\mathcal H_2(\,2\,)\), obtained by quotienting \(\mathcal H_2(\,2\,)\) by the natural \(\mathbb C^*\)-action via complex scalar multiplication, can be identified with the finite-index subgroup of the mapping class group \(\mathrm{Mod}_2\) that preserves an odd spin structure.

\begin{rmk}\label{rmk:algspin}
In this note, the spin structure is defined in topological terms. On a surface of genus \(2\) there are exactly six odd spin structures, see for instance \cite{AM}, and these are in natural bijection with the Weierstrass points of the surface. This correspondence is most easily seen using the algebraic description of spin structures as square roots of the canonical bundle, a viewpoint which lies beyond the scope of the present paper.
\end{rmk}

\noindent In \cite{LM14}, Looijenga and Mondello proved the conjecture for the non-hyperelliptic component of the projectivised stratum \(\mathbb{P}\mathcal{H}_3(\,4\,)\), as well as for the stratum \(\mathbb{P}\mathcal{H}_3(\,1,3\,)\), which is connected, see \S\ref{ssec:conncomp}. In both cases, the corresponding orbifold fundamental groups are shown to be quotients of explicit finite-type Artin groups by their centres. In the same work, they also provide explicit presentations for the orbifold fundamental group of the non-hyperelliptic component of the projectivised stratum \(\mathbb{P}\mathcal{H}_3(\,2,2\,)\), that is, the connected component with odd parity.\\
\noindent In \cite{Ham20}, Hamenst\"adt significantly extends the work of Looijenga and Mondello by treating the case of arbitrary genus for projectivised strata of the form \(\mathbb{P} \mathcal{H}_g(2,\dots,2)\). More precisely, she shows that the orbifold fundamental group of the odd component of this stratum, on a surface of genus \(g\), is a quotient of an explicit Artin group with \(3g - 2\) generators. The reader is referred to \cite{Ham20} for more details about this group.

\smallskip 

\noindent To understand further developments, it is necessary to spend a few more words on the general framework. There is an obvious forgetful map defined on every stratum, taking values in \(\mathcal{M}_{g,k}\), where \(\mathcal{M}_{g,k}\) denotes the moduli space of compact Riemann surfaces of genus \(g\) with \(k\) marked points. Any homotopy class of loops in \(\strata\), based at \((X_o, \omega_o)\), gives rise to an isotopy class of self-homeomorphism \(X \to X\) that preserves the set \(\Delta\) of zeros of \(\omega\). More explicitly, say we have a loop \((X_t,\omega_t)\) of differentials, where the coordinate \(t\) is valued in \(S^1\). To get an isotopy class of maps \(X_o \to X_o\) out of this, mark the original differential by an isotopy class of maps \(m_o: S_{g,k} \to X_o\), where \(S_{g,k}\) is a reference topological surface of the type of \(X_o\), and the \(k\) marked points map to the zeros of \(\omega_o\). Now as \(t\) runs from 0 to 1, we can propagate the marking to nearby fibres in a well defined way. After going all the way around the loop, we have a marking \(m_1: S_{g,k} \to X_1\). But \(X_1 = X_o\), so we can get a self-map of \(S_{g,k}\) by taking \(m_1^{-1} \circ m_o\). This is the monodromy of the loop. This induces the so-called \textit{topological monodromy representation}
\begin{equation}
    \rho \colon \pi_1^{\mathrm{orb}}(\,\mathcal{H},\,(X,\omega)\,) \longrightarrow \mathrm{PMod}_{g,k},
\end{equation}
where \(\mathcal{H}\) is the connected component of \(\strata\) containing \((X,\omega)\), and \(\mathrm{PMod}_{g,k}\) denotes the \textit{pure} mapping class group of \(X\), that is the subgroup of diffeomorphisms of \(X\) that fix \(\Delta\). In \cite{Ham18}, Hamenst\"adt studied the image of the monodromy representation by providing an explicit set of generators. Around the same time, further investigations into the monodromy representation were carried out by Calderon and Salter. Through a series of papers \cite{Cal19}, \cite{CS22} and \cite{CS23}, they showed that for \(g \geq 5\), the image of the monodromy representation is isomorphic to certain subgroups of the mapping class group with a well defined geometric meaning, because these subgroups preserve an additional structure determined by \(\omega\).

\smallskip

\noindent We observe that if the monodromy representation \(\rho\) is injective, then each connected component of the preimage of \(\mathcal H\subset\mathcal H_g(\,\mu\,)\) in \(\Omega\mathcal T_g\) would be simply connected. Indeed, \(\widetilde{\mathcal H}\subset \Omega\mathcal T_g\) is a covering space of \(\mathcal H\), and the monodromy representation records how loops in \(\mathcal H\) act on the marked surface. By covering space theory, \(\pi_1^{\textnormal{orb}}(\,\mathcal H,\,(X,\omega)\,)/\ker(\,\rho\,)\) is isomorphic to the group of deck transformations \(\widetilde{\mathcal H}\to \mathcal H\). 
The injectivity of \(\rho\) implies that \(\ker(\,\rho\,)\) is trivial, so that no non-trivial loop in \(\mathcal H\) lifts to a loop in \(\widetilde{\mathcal H}\), and hence \(\widetilde{\mathcal H}\) has trivial fundamental group.  

\smallskip

\noindent We cannot hope the monodromy representation \(\rho\) to be injective in general. For instance, Looijenga and Mondello \cite{LM14} showed that the orbifold fundamental group of the non-hyperelliptic component of the stratum \(\mathcal{H}_3(\,4\,)\), denoted by \(\mathcal{H}^{\textnormal{nh}}_3(\,4\,)\), is isomorphic to the quotient of the Artin group \(\mathcal A\) of finite type \(\textnormal{E}_6\) by its centre. In fact, there exists a natural so-called geometric homomorphism from \(\mathcal A\) into \(\mathrm{Mod}_{3,1}\) and a non-trivial result by Hamenst\"adt \cite[Theorem 2]{Ham18} implies that \(\rho\big(\,\pi_1^{\textnormal{orb}}\big(\,\mathcal{H}^{\textnormal{nh}}_3(\,4\,)\,\big)\) coincides with the image of such a geometric realisation. However, by a result of Wajnryb \cite{Wj99}, the kernel of any geometric homomorphism \(\mathcal A \to \mathrm{Mod}_{3,1}\) cannot be contained in the centre of \(\mathcal A\). Therefore \(\rho\) cannot be injective. 

\smallskip

\subsection{Breaking a zero and Schiffer variations}\label{ssec:shiffer} In the present section, we introduce a surgery that allows us to deform structures within the moduli space, possibly remaining in the same stratum to which they belong, without altering the holonomy of the associated structure. There exists both a geometric-topological and an analytic version of this construction.

\subsubsection{Breaking up a zero\index{zero!break}}\label{sssec:zerobreak} The surgery we are going to describe has been introduced by Eskin--Masur--Zorich in \cite[Section 8.1]{EMZ} and it literally "breaks up" a zero in two, or possibly more, zeros of lower orders. Complex-analytically, this can be thought of as the analogue to the classical Schiffer variation for Riemann surfaces, see \S\ref{sssec:schiffer} and \cite{NS} for more details. This surgery only modifies a translation surface in a contractible neighbourhood of the initial zero. In particular, after the surgery the resulting surface has the same genus as the former one, but the type of zero orders is changed. Moreover, the new translation surface we obtain after the surgery has the same (absolute) period character as the original one. As a consequence, this operation produces small deformations of the original translation structure in the same isoperiodic fibre.

\begin{rmk}
In the context of branched projective structures, such a surgery is also known as the \textit{movement\index{zero!movement} of branched points} and it has been originally introduced by Tan in \cite[\S6]{TA} for showing the existence of a complex one-dimensional continuous family of infinitesimal deformations of a given structure.
\end{rmk}

\noindent Let us now explain this surgery in more detail. Let \((X,\omega)\) be a translation surface. Breaking up a zero is a procedure that takes place in the \(\varepsilon\)-neighbourhood of some zero of order \(m\) of the differential, where it locally resembles the pull-back of the form \(dz\) via a branched covering \(z \longmapsto z^{\,m+1}\). The differential is then modified by a surgery inside this \(\varepsilon\)-neighbourhood. Once this surgery is performed, we obtain a new translation structure with two zeros of orders \(m_1\) and \(m_2\) such that \(m_1 + m_2 = m\). Furthermore, the translation structure remains unchanged outside the \(\varepsilon\)-neighbourhood of such a zero of order \(m\). The idea is to consider the \(\varepsilon\)-neighbourhood of a zero of order \(m\) as \(m+1\) copies of a disc \(D\) of radius \(\varepsilon\) whose diameters are identified in a specified way. We can view this family of discs as a collection of \(m+1\) upper half-discs and \(m+1\) lower half-discs. See Figure~\ref{fig:zerolocmodel}.

\smallskip

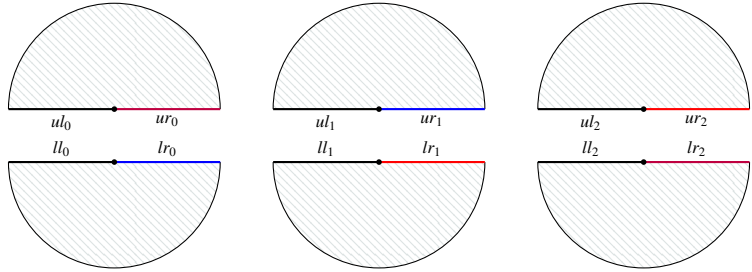
\begin{figure}[ht]
    \centering
    \begin{tikzpicture}[scale=0.7, every node/.style={scale=0.6}]
    \definecolor{pallido}{RGB}{221,227,227}
    \foreach \x [evaluate=\x as \coord using 4 + 5*\x] in {0, 1, 2} 
    {
    \draw [pattern=north west lines, pattern color=pallido] (\coord,0) arc [start angle = 0,end angle = 180,radius = 2];
    \draw [pattern=north west lines, pattern color=pallido] (\coord,-1) arc [start angle = 0,end angle = -180,radius = 2];
    }
    \foreach \x [evaluate=\x as \leftend using 5*\x] [evaluate=\x as \rightend using 2 + 5*\x] in {0,1,2} 
    {
    \draw [thick] (\leftend, 0) -- (\rightend, 0);
    \draw [thick] (\leftend, -1) -- (\rightend, -1);
    }
    \foreach \x [evaluate=\x as \leftlabel using 1 + 5*\x] [evaluate=\x as \rightlabel using 3 + 5*\x] in {0, 1, 2} 
    {
    \node [below] at (\leftlabel, 0) {$ul_{\x}$};
    \node [below] at (\rightlabel, 0) {$ur_{\x}$};
    \node [above] at (\leftlabel, -1) {$ll_{\x}$};
    \node [above] at (\rightlabel, -1) {$lr_{\x}$};
    }
    
\foreach \botindex / \topindex / \colr [evaluate=\botindex as \botleftend using 5*\botindex + 2] [evaluate=\botindex as \botrightend using 5*\botindex + 4] [evaluate=\topindex as \topleftend using 5*\topindex + 4] [evaluate=\topindex as \toprightend using 5*\topindex + 2] in {0/1/blue, 1/2/red, 2/0/purple} 
    {
    \draw [thick, \colr] (\topleftend, 0) -- (\toprightend, 0);
    \draw [thick, \colr] (\botleftend, -1) -- (\botrightend, -1);
    \fill (\toprightend, 0) circle (1.5pt);
    \fill (\botleftend, -1) circle (1.5pt);
    }
 
    \end{tikzpicture}
    \caption{An $\varepsilon$-neighbourhood of a zero of order $2$.}
    \label{fig:zerolocmodel}
\end{figure}

\noindent We now break up a zero of order \(m\) into two zeros of orders \(m_1\) and \(m_2\). This surgery consists of identifying the diameters of the initial \(m+1\) discs in a different way, as follows. First, we modify the labelling on the upper half-disc indexed by \(0\), the lower half-disc indexed by \(m_1\), and all upper and lower half-discs with index greater than \(m_1\) accordingly. The modified labelling is shown in Figure~\ref{fig:splitlocmodel} for the case of splitting a zero of order \(2\) into two zeros of order \(1\). See also Figure \ref{fig:breakazerobranched}.

\smallskip

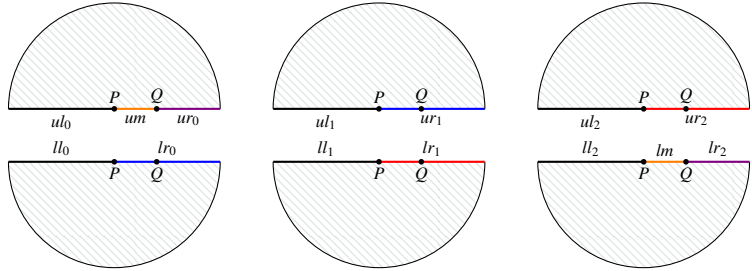
\begin{figure}[ht]
    \centering
    \begin{tikzpicture}[scale=0.7, every node/.style={scale=0.6}]
    \definecolor{pallido}{RGB}{221,227,227}
    \foreach \x [evaluate=\x as \coord using 4 + 5*\x] in {0, 1, 2} 
    {
    \draw [pattern=north west lines, pattern color=pallido] (\coord,0) arc [start angle = 0,end angle = 180,radius = 2];
    \draw [pattern=north west lines, pattern color=pallido] (\coord,-1) arc [start angle = 0,end angle = -180,radius = 2];
    }
    \foreach \x [evaluate=\x as \leftend using 5*\x] [evaluate=\x as \rightend using 2 + 5*\x] in {0, 1, 2} 
    {
    \draw [thick] (\leftend, 0) -- (\rightend, 0);
    \draw [thick] (\leftend, -1) -- (\rightend, -1);
    \node[above] at (\rightend, 0) {$P$};
    \node[below] at (\rightend, -1) {$P$};
    }
    \draw [thick, orange] (2,0) -- (2.8,0);
    \fill (2,0) circle (1.5pt);
    \draw [thick, orange] (12,-1) -- (12.8,-1);
    \fill (12,-1) circle (1.5pt);
    \node [below] at (2.4, 0) {$um$};
    \node [above] at (12.4, -1) {$lm$};
    \foreach \botindex / \topindex / \colr [evaluate=\botindex as \botleftend using 5*\botindex + 2] [evaluate=\botindex as \botrightend using 5*\botindex + 4] [evaluate=\topindex as \topleftend using 5*\topindex + 2] [evaluate=\topindex as \toprightend using 5*\topindex + 4] [evaluate=\topindex as \topsplpt using 5*\topindex + 2.8] [evaluate=\botindex as \botsplpt using 5*\botindex + 2.8] in {0/1/blue, 1/2/red} 
    {
    \draw [thick, \colr] (\topleftend, 0) -- (\toprightend, 0);
    \draw [thick, \colr] (\botleftend, -1) -- (\botrightend, -1);
    \fill (\topleftend, 0) circle (1.5pt);
    \fill (\botleftend, -1) circle (1.5pt);
    \fill (\topsplpt, 0) circle (1.5pt);
    \fill (\botsplpt, -1) circle (1.5pt);
    \node[above] at (\topsplpt, 0) {$Q$};
    \node[below] at (\botsplpt, -1) {$Q$}; 
    }
    \foreach \botindex / \topindex / \colr [evaluate=\botindex as \botleftend using 5*\botindex + 2.8] [evaluate=\botindex as \botrightend using 5*\botindex + 4] [evaluate=\topindex as \topleftend using 5*\topindex + 2.8] [evaluate=\topindex as \toprightend using 5*\topindex + 4] in {2/0/violet} 
    {
    \draw [thick, \colr] (\topleftend, 0) -- (\toprightend, 0);
    \draw [thick, \colr] (\botleftend, -1) -- (\botrightend, -1);
    \fill (\topleftend, 0) circle (1.5pt);
    \fill (\botleftend, -1) circle (1.5pt);
    \node[above] at (\topleftend, 0) {$Q$};
    \node[below] at (\botleftend, -1) {$Q$};
    }
    \foreach \x [evaluate=\x as \leftlabel using 1 + 5*\x] in {0, 1, 2} 
    {
    \node [below] at (\leftlabel, 0) {$ul_{\x}$};
    \node [above] at (\leftlabel, -1) {$ll_{\x}$};
    }
    \foreach \botindex / \topindex [evaluate=\botindex as \botlabel using 3+ 5*\botindex] [evaluate=\topindex as \toplabel using 3+ 5*\topindex] in {0/1, 1/2} 
    {
    \node [below] at (\toplabel, 0) {$ur_{\topindex}$};
    \node [above] at (\botlabel, -1) {$lr_{\botindex}$};
    }
    \foreach \botindex / \topindex [evaluate=\botindex as \botlabel using 3.4+ 5*\botindex] [evaluate=\topindex as \toplabel using 3.4+ 5*\topindex] in {2/0} 
    {
    \node [below] at (\toplabel, 0) {$ur_{\topindex}$};
    \node [above] at (\botlabel, -1) {$lr_{\botindex}$};
    }
    \end{tikzpicture}
    \caption{New labeling for breaking up a zero of order $2$ in two zeros of order $1$.}
    \label{fig:splitlocmodel}
\end{figure}

\noindent Next, we identify \(ul_i\) with \(ll_i\) and \(lr_i\) with \(ur_{i+1}\) as before (where indexes are taken \((\textnormal{mod\,} m+1) \), with the added identification of \(um\) with \(lm\). This identification gives two zeros \(p\) and \(q\), where \(p\) is a zero of the differential of order \(m_1\) and \(q\) is a zero of order \(m_2\). We also obtain a geodesic line segment joining \(p\) and \(q\). Given \(c \in \mathbb{C} \setminus \{0\}\) with length less than \(2\varepsilon\), we can perform the surgery in such a way that the line segment joining \(p\) and \(q\) is \(c\). It is also clear that such a deformation of the translation structure is only local. This procedure can be repeated multiple times to obtain zeros of orders \(m_1, \dots, m_k\) from a single zero of order \(m_1 + \cdots + m_k\).

\smallskip

\noindent Breaking a zero into zeros of smaller orders satisfies the following useful properties.
\begin{itemize}
    \item[1.] Let \((X,\omega)\) be a translation surface in the hyperelliptic component of \(\mathcal H_g(\,2g-2\,)\) and let \((Y,\xi)\) be the translation surface in \(\mathcal H_g(g-1,g-1)\) obtained by breaking the single zero of \(\omega\) into two zeros of order \(g-1\). Then \((Y,\xi)\) is also hyperelliptic.
    \smallskip
    \item[2.] Let \((X,\omega)\) be a translation surface in \(\mathcal H_g(\,2m_1,\dots,2m_k\,)\) and let \(\theta\in\mathbb Z_2\) be its spin structure. Let \((Y,\xi)\) be the translation surface in \(\mathcal H_g(2m_1,\dots,2m_{k-1}, k_1, k_2)\) obtained by breaking the zero of order \(2m_k\) into two zeros of lower orders \(k_1,k_2\) such that \(k_1+k_2=2m_k\). If \(k_1,\,k_2\) are both even, then \((Y,\xi)\) belongs to the connected component with parity \(\theta\). In other words, by breaking a zero into zeros of lower and \textit{even} order does not alter the spin parity.
\end{itemize}

\noindent We finally show how this surgery can be used to show that for every signature \(\mu\) the corresponding stratum \(\mathcal H_g(\,\mu\,)\) is not empty.

\begin{lem}\label{lem:stratanotempty}
    Let \(\mu=(\,m_1,\dots,m_k\,)\) be any partition of \(2g-2\). Then the stratum \(\mathcal H_g(\,\mu\,)\) is not empty.
\end{lem}

\begin{proof}
    Suppose there is a translation surface \((X,\omega)\) such that \(\omega\) has a single zero of order \(2g-2\). By breaking this zero in a judicious way as described above we find a translation surface in \(\mathcal H_g(\,\mu\,)\) for every partition \(\mu=(\,m_1,\dots,m_k\,)\) of \(2g-2\). Thus we only need to find a translation surface in \(\mathcal H_g(2g-2)\). Every regular \(4g\)-gon with \(g\ge2\) does the job and hence we are done.
\end{proof}

\smallskip

\subsubsection{Schiffer variations and movement of zeros}\label{sssec:schiffer} As already alluded above, Schiffer variations are intimately related to the surgery just described in \S\ref{sssec:zerobreak}. Geometrically, Schiffer variations can be seen as a \textit{movement of zeros}\index{zero!movement} along a given direction. 

\begin{rmk}
    By adopting the language of geometric structures on surfaces, the movement of zeros is largely known as the \textit{movement of branch points}, see \cite{CDF}, \cite{CDF2}, and \cite{TA} where, in that case, a \textit{branch point} is understood as a point around which local charts are branched covers onto some open subsets of the model space. In this language, translation surfaces are branched geometric structures, and the zeros thus correspond to the branch points. In the present paper, however, we do not adopt this terminology.
\end{rmk}

\noindent Before stating the movement of zeros, let us introduce the following terminology as in \cite[\S8]{CF}. We remark that this is a different but equivalent formulation compared to the other currently known in literature, see \cite[\S2]{CDF}.

\begin{defn}[Twin paths]\label{def:twins}
On a translation surface \((X,\omega)\), let \(p\) be a zero of \(\omega\) of order \(m\). Consider a collection of \(m+1\) embedded paths \(c_i\colon [\,0,\,1\,]\longrightarrow (X,\omega)\) such that \(c_i(\,0\,)=p\) for \(i=0,\dots,m\), each of which is injectively developed, and all of which overlap once developed, \textit{i.e.} there is a determination of the developing map around \(c_o\,\cup\,\cdots\,\cup\,c_m\) such that
\begin{itemize}
    \item \(c_o,\,c_1,\dots,c_m\) injectively develop to the same arc \(\widehat{c}\subset \C\), and 
    \item \(\text{dev}\big(\,c_i(\,t\,)\,\big)=\text{dev}\big(\,c_j(\,t\,)\,\big)\) for every \(t\in[0,1]\) and \(i,j\in\{0,\dots,m\}\).
\end{itemize}
\noindent For any \(2\le k\le m\), the paths \(c_{i_1},\dots,c_{i_k}\) are called \textit{twin paths}. For any pair \(c_i,\,c_j\), notice that the angle at \(p\) between them is a multiple of \(2\pi\).
\end{defn}

\noindent Recall that around a singularity of order \(m\), the horizontal foliation is singular and determines \(m+1\) sectors of angle \(2\pi\). Informally, we may define twin paths as a collection of at most \(m+1\) segments based at the singularity \(p\), each lying in a different sector, such that they form the same angle with respect to the horizontal foliation. Usually all twin paths \(c_o,\dots,c_m\) at \(p\) as defined above are counter-clockwise ordered if not otherwise specified.

\smallskip

\begin{defn}[Movement of zeros]\label{def:move}
Let \((X,\omega)\) be a translation surface and let \(p\) be a zero of order \(m \ge 1\). Let \(c_o,\dots,c_m\) be a family of embedded paths at \(P\) such that, for any \(i,j\), the pair \(c_i,\,c_j\) forms a pair of embedded twin paths, and \(c_i(1)\) is a regular point for all \(i = 0, \dots, m\). Cut \((X,\omega)\) along \(c_o \cup \dots \cup c_m\) and glue back \(c_i^+\) with \(c_{i+1}^-\) for \(i = 0, \dots, m-1\), and also \(c_o^-\) with \(c_m^+\). The resulting structure is a local deformation of \((X,\omega)\), obtained by a \textit{movement of zeros}\index{zero!movement}. In what follows, whenever necessary, we will denote by \({\rm Move}\big(\,(X,\omega),\,\textbf{c}\,\big)\) such a deformation of \((X,\omega)\), where \(\textbf{c}\) denotes the collection \(\left\{\,c_o,\dots,c_m\,\right\}\) of twin paths.
\end{defn}

\begin{figure}[!h]
\centering
\begin{tikzpicture}[scale=0.7, every node/.style={scale=0.6}]

\definecolor{pallido}{RGB}{187,207,227}

\draw [thick, pallido] (0,5) to (3,5);
\draw [thick, pallido] (3,5) to (6,5);
\draw [thick, pallido] (3,5) to (3,8);
\draw [thick, pallido] (3,5) to (3,2);

\draw [thick, pallido] (12.5,5.5) to (13,8);
\draw [thick, pallido] (12.5,4.5) to (13,2);
\draw [thick, pallido] (13.5,5.5) to (13,8);
\draw [thick, pallido] (13.5,4.5) to (13,2);

\draw [thick, pallido] (10,5) to (12.5,5.5);
\draw [thick, pallido] (10,5) to (12.5,4.5);
\draw [thick, pallido] (16,5) to (13.5,5.5);
\draw [thick, pallido] (16,5) to (13.5,4.5);

 \draw (16,5) ++ (-165:0.2) arc (-165:165:0.2);
 \draw (10,5) ++ (15:0.2)  arc (15:345:0.2);
 \draw (13,8) ++ (-75:0.2)  arc (-75:255:0.2);
 \draw (13,2) ++ (-255:0.2)  arc (-255:75:0.2);

\draw [red]    plot [mark=*, smooth] coordinates {(3,5)};
\draw [black]  plot [mark=*, smooth] coordinates {(0,5)};
\draw [black]  plot [mark=*, smooth] coordinates {(6,5)};
\draw [black]  plot [mark=*, smooth] coordinates {(3,8)};
\draw [black]  plot [mark=*, smooth] coordinates {(3,2)};

\draw [thin] (3,5) circle (3mm);
\draw [thin] (0,5) circle (2mm);
\draw [thin] (6,5) circle (2mm);
\draw [thin] (3,8) circle (2mm);
\draw [thin] (3,2) circle (2mm);

\draw [red] plot [mark=*, smooth] coordinates {(12.5,5.5)};
\draw [red] plot [mark=*, smooth] coordinates {(12.5,4.5)};
\draw [red] plot [mark=*, smooth] coordinates {(13.5,5.5)};
\draw [red] plot [mark=*, smooth] coordinates {(13.5,4.5)};

\draw [black] plot [mark=*, smooth] coordinates {(10,5)};
\draw [black] plot [mark=*, smooth] coordinates {(16,5)};
\draw [black] plot [mark=*, smooth] coordinates {(13,8)};
\draw [black] plot [mark=*, smooth] coordinates {(13,2)};

\node at (3.55,5.5) {$8\pi$};
\node at (3.55,8) {$2\pi$};
\node at (3.55,2) {$2\pi$};
\node at (6,5.5) {$2\pi$};
\node at (0,5.5) {$2\pi$};

\draw [latex- ,line width=0.5pt, blue] (12.5, 5.5) ++(90:6mm) arc (90:180:6mm);
\draw [-latex ,line width=0.5pt, blue] (12.5, 4.5) ++(180:6mm) arc (180:270:6mm);
\draw [latex- ,line width=0.5pt, blue] (13.5, 5.5) ++(90:6mm) arc (90:0:6mm);
\draw [-latex ,line width=0.5pt, blue] (13.5, 4.5) ++(0:6mm) arc (0:-90:6mm);
\end{tikzpicture}
\caption{Movement of a zero of order $3$.}
\label{sm}
\end{figure}
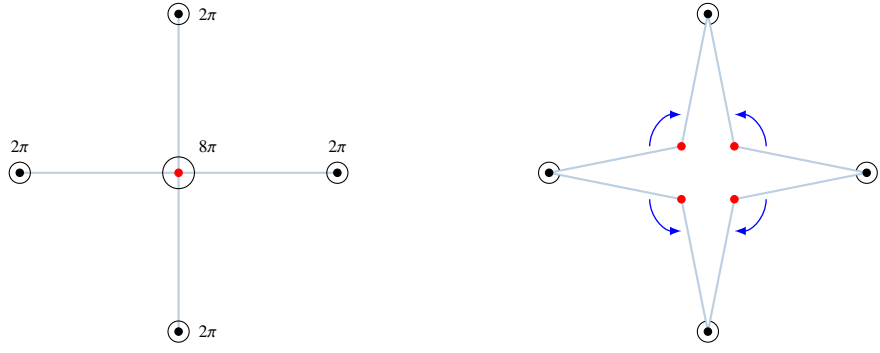

\noindent See Figure \ref{sm} for a visual representation. In the study of branched structures on surfaces, the importance of such a surgery is its key property of preserving the holonomy of the structure. Moreover, it does not change the structure of the branching divisor. In our language, the holonomy being preserved means that the absolute periods remain untouched once the surgery is performed. In particular, the resulting structure lies in the same stratum.

\medskip

\subsubsection{Geometric interpretation of "movement".} We would like to provide an alternative description of the movement of zeros by using a more geometric and dynamical interpretation of the ``movement''. In this paragraph we shall use the definition of translation surface as in \S\ref{ssec:seconddef} and we shall use branched charts as in Remark \ref{rmk:branchedchart}.

\smallskip

\noindent For this purpose, let \((X,\omega)\) be any translation surface. Let \(p\) be a zero of \(\omega\) of order \(m\ge1\) and let \(B = B_\varepsilon(p)\) be an open metric ball embedded in \((X,\omega)\). For \(\varepsilon\) small enough, \(B\) homeomorphically lifts to an open ball in \((\widetilde X, \widetilde\omega)\), and we can assume that the restriction of the developing map is a local branched chart \(\varphi\colon B\longrightarrow \mathbb C\) of degree \(m+1\), branched at \(p\). Let \(\widehat p = \varphi(\,p\,)\) and let \(\widehat B = \varphi(\,B\,)\). Let \(\textbf{c} = \left\{\,c_o,\dots,c_m\,\right\}\) be a collection of twin paths based at \(p\), and let \(q_i\) denote the endpoint of \(c_i\) other than \(p\). Assume that \(\textbf{c}\) is entirely contained in \(B\). Since all paths in \(\textbf{c}\) are assumed to be twins, they all develop onto the same segment, say \(\widehat c \subset \mathbb C\), having \(\widehat p\) and \(\widehat q\) as the extremal points. Notice that \(\varphi(\,q_i\,) = \widehat q\) for every \(i = 0,\dots,m\). In this setting, since \(\varphi\) is a local branched cover, \(\widehat p\) is a branch value and, by assuming all \(q_i\) to be regular points, \(\widehat q\) is a regular value.

\smallskip

\noindent We then perform the surgery described in Definition \ref{def:move}. After the cut-and-paste process, \((X,\omega)\) is thus deformed to a new structure \({\rm Move}\big(\,(X,\omega),\,{\textbf{c}}\,\big)\). In the latter, the local chart \(\varphi\), seen as a branched cover, has been deformed to a new local branched chart \(\psi\colon B\longrightarrow \widehat B \subset \mathbb C\) of degree \(m+1\) in which \(\widehat p\) is a regular value, whereas \(\widehat q\) is a branch value. More precisely, \(\psi^{-1}(\,\widehat q\,)\) has only one preimage, say \(q\), arising from the collapsing of \(\{\,q_0,\dots,q_m\,\}\), and \(\psi^{-1}(\,\widehat p\,)\) has \(m+1\) preimages \(\{p_o,\dots, p_m\}\). The mapping \(\psi\) around every \(p_i\) is now a regular chart.

\smallskip

\noindent Keeping in mind this description, we can say that a zero of \(\omega\) has been moved to another point, and the movement is seen in local charts. For a more dynamical interpretation, we observe that the movement of zeros naturally comes with a one-parameter family. In the above notation, for every \(t\in[0,\,1]\), we can consider the structure \({\rm Move}\big(\,(X,\omega),\,{\textbf{c}^t}\,\big)\), where \(\textbf{c}^t\) is the collection of embedded twin paths made of those sub-arcs of the \(c_i\)'s parametrised by the subinterval \([0,\,t]\). For \(t\) running in the unit interval \([0,\,1]\), we see \(\widehat p\), \textit{i.e.} the developed image of \(p\), \textit{moving} along \(\widehat c\), which motivates the name of this surgery.

\begin{rmk}\label{rmk:breakazerointolowerzeros}
Breaking up a zero, seen as the operation just described in \S\ref{sssec:zerobreak}, turns out to be a special case of the movement of the zero. In fact, let \(\textbf{c} = \{c_1,\dots,c_m\}\) be a collection of twin paths based at \(p\) as above. For \(0 < k < m\), given a sub-collection \(\{c_{i_1},\dots,c_{i_k}\}\) of adjacent twin paths, the surgery described in Definition \ref{def:move} literally \textit{breaks} \(p\) into two lower-order zeros of order \(m-k\) and \(k\), and only the latter is moved away from \(p\). In the local charts, \(\varphi\), seen as a branched covering of degree \(m+1\), is deformed to a new branched chart \(\psi\) of degree \(m+1\) that now has two branch values at \(\widehat p\) and \(\widehat q\). In this case, \(\psi^{-1}(\,\widehat p\,)\) has \(m-k+1\) preimages and \(\psi^{-1}(\,\widehat q\,)\) has \(k+1\) preimages. See Figure \ref{fig:breakazerobranched}.
\end{rmk}

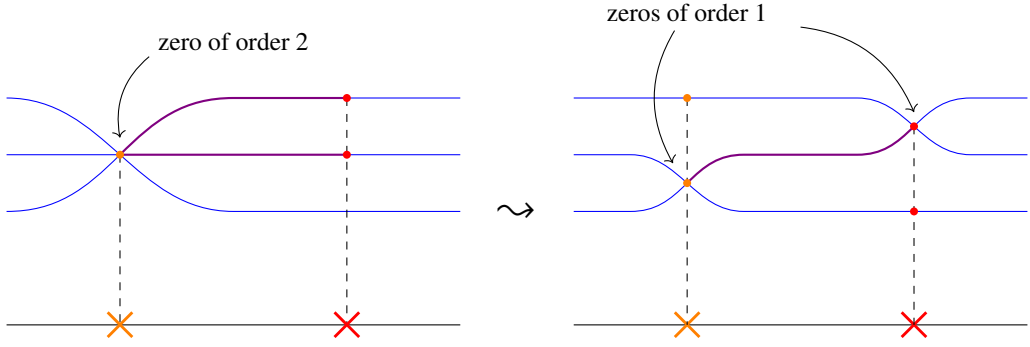
\begin{figure}[h!]
    \centering
    \begin{tikzpicture}[scale=1.5, every node/.style={scale=0.875}]
    \definecolor{pallido}{RGB}{221,227,227}

    \draw[black] (-2,0)--(2,0);
    \draw[blue] (-2, 1.5)--(-1,1.5) (1,1.5)--( 2, 1.5);
    \draw[blue] (-2, 1  ) to[out=0, in=225] (-1,1.5) (1,2)--(2,2);
    \draw[violet, thick] (-1,1.5)--(1,1.5);
    \draw[violet, thick] (-1,1.5) to[out=45, in=180] ( 0,2)--(1,2);
    \draw[blue] (-2, 2) to[out=0, in=180] (0,1)--(2,1);
    \draw[dashed] (-1,1.5)--(-1,0);
    \draw[dashed] (1,2  )--(1,0);
    \fill[orange] (-1,1.5) circle (1pt);
    \fill[red] (1,2) circle (1pt);
    \fill[red] (1,1.5) circle (1pt);
    \node[orange, scale=2] at (-1,0) {\(\times\)};
    \node[red   , scale=2] at (1,0) {\(\times\)};
    \node (A) at (0,2.5) {zero of order \(2\)};
    \draw[black, ->] (A) to[bend right] (-1,1.65);

    \node[scale=1.5] at (2.5,1) {\(\leadsto\)};
    
    \begin{scope}[shift={(5,0)}]
        \draw[black] (-2,0)--(2,0);
        \draw[blue] (-2,1)--(-1.5,1) to[out=0, in=225] (-1,1.25) (1,1.75) to[out=45, in=180] (1.5,2)--(2,2);
        \draw[violet, thick] (-1,1.25) to[out=45, in=180] (-0.5,1.5)--(0.5,1.5) to [out=0, in=225] (1,1.75);
        \draw[blue] (-2,1.5)--(-1.5,1.5) to[out=0, in=180] (-0.5,1)--(2,1); 
        \draw[blue] (-2,2)--(0.5,2) to[out=0, in=180] (1.5,1.5)--(2,1.5);
        \node[orange, scale=2] at (-1,0) {\(\times\)};
        \node[red, scale=2] at ( 1,0) {\(\times\)};
        \draw[dashed] (-1,2)--(-1,0);
        \draw[dashed] ( 1,1.75)--( 1,0);
        \fill[orange] (-1,1.25) circle (1pt);
        \fill[red   ] ( 1,1.75) circle (1pt);
        \fill[orange] (-1,2) circle (1pt);
        \fill[red   ] ( 1,1) circle (1pt);
        \node (C) at (-1,2.75) {zeros of order \(1\)};
        \draw[black, ->] (0,2.625) to[bend left ] ( 1,1.875  );
        \draw[black, ->] (C) to[bend right] (-1.125,1.425);
    \end{scope}
    
    \end{tikzpicture}
    \caption{The figure illustrates how a local chart is modified when a zero is split into zeros of a lower order. In the case shown, a zero of order two is broken into two zeros of order one. The bold purple lines represent the twin paths used to split the zero.}
    \label{fig:breakazerobranched}
\end{figure}

\begin{rmk}
In our setting above, the extremal points \(\{q_0,\dots,q_m\}\) of \(\textbf{c}\) were all supposed to be regular points. This surgery extends \textit{mutatis mutandis} to the case in which some, possibly all, of the \(q_i\)'s are also zeros for \(\omega\). In this extended setting, two or more zeros of \(\omega\) collapse to a new-born zero of higher order. Under this perspective, the movement of zeros is a surgery that not only extends the breaking of a zero but also provides its reverse surgery.
\end{rmk}

\smallskip

\section{Periods of translation surfaces}\label{sec:per}

\noindent In the theory of geometric structures, the holonomy map plays a pivotal role and has been investigated through many different lenses. As we have already seen in \S\ref{ssec:devhol}, every translation structure defines a representation of the fundamental group, known as the \textit{holonomy}. It is worth recalling that this is far from being a feature specific to translation structures; in fact, any geometric structure defines a holonomy representation. The holonomy map assigns to each geometric structure its corresponding holonomy representation. We will not discuss the general case, which requires certain technical considerations, and will instead focus on the structures considered in this chapter. 

\smallskip

\subsection{Period map}\label{ssec:permap} In our setting, the holonomy map is a function defined in the moduli space \(\Omega\mathcal{T}_g\) that takes its values on the space of representations of the fundamental group. The first simplification arises from the fact that the holonomy of any such structure is a representation into \(\mathbb{C}\), seen as an additive group, which descends to homology since the target group is abelian. As a consequence, the holonomy map\index{holonomy!map} is well defined on an intermediate space between the spaces \(\Omega\mathcal{T}_g\) and \(\Omega\mathcal{M}_g\) which is known as the Torelli space.  More specifically, given the map
\begin{equation}
\textnormal{hol}\colon\Omega\mathcal{T}_g \longrightarrow \textnormal{Hom}\Big(\,\pi_1(\,\Sigma\,),\,\mathbb{C}\,\Big),
\end{equation}
there is an equivariant action of the Torelli subgroup \(\mathcal{I}_g\) of the mapping class group \(\mathrm{Mod}_g\) on both spaces. This action is trivial on the representation space, but not on the moduli space \(\Omega\mathcal{T}_g\). Therefore, the holonomy map descends to a map
\begin{equation}\label{eq:permap}
\textnormal{Per}\colon\Omega\mathcal{S}_g \longrightarrow \textnormal{Hom}\Big(\,\mathrm{H}_1(\Sigma,\,\mathbb{Z}),\,\mathbb{C}\,\Big) = \mathrm{H}^1(\,\Sigma,\,\mathbb{C}\,)
\end{equation}
that we shall define as the \textit{Period map}\index{period!map}. The space \(\Omega\mathcal{S}_g\) is the moduli space of translation structures marked in \textit{homology}\index{marking!homology}. Alternatively, this space can be defined directly by introducing a homological marking, in the same spirit as done in \S\ref{ssec:markedhighgen} to define \(\Omega\mathcal{T}_g\). Clearly, both definitions describe the same space.

\smallskip 

\subsection{Period coordinates}\label{ssec:percoordinates} We first wish to connect the local coordinates provided in Remark \ref{rmk:localcoord} with the Period map just defined in \eqref{eq:permap}. Apart from the usual complications arising from the orbifold structure, periods provide local coordinates on any stratum. More precisely, let \((X,\omega)\) be a translation surface with singular locus \(\Delta = \{p_o, \ldots, p_n\}\), and consider the relative homology group \(\textnormal{H}_1(X,\, \Delta,\, \mathbb{Z})\).  Choose a basis, say \(\{\,\alpha_1,\beta_1, \ldots, \alpha_g,\beta_g\,\}\) of \(\textnormal{H}_1(X,\, \mathbb{Z})\), and let \(\delta_1, \ldots, \delta_n\) be arcs connecting \(p_o\) to \(p_1, \ldots, p_n\), respectively. Then, the set \(\{\alpha_1,\beta_1, \ldots, \alpha_g,\beta_g,\,\delta_1, \ldots, \delta_n \}\) forms a basis in the relative homology \(\textnormal{H}_1(\Sigma,\, \Delta,\,\mathbb{Z})\). From the point of view of translation surfaces expressed as pairs \((X, \omega)\), one defines the period map as follows:
\begin{equation}
(X, \omega) \longmapsto \left( \,\int_{\alpha_1} \omega,\,\int_{\beta_1} \omega, \ldots, \int_{\alpha_g} \omega,\,\int_{\beta_g} \omega , \int_{\delta_1} \omega, \ldots, \int_{\delta_n} \omega \,\right).
\end{equation}
This gives a local chart
\begin{equation}
\strata \longrightarrow \mathbb{C}^{2g + n }.
\end{equation}
These charts are known as the \emph{period coordinates}\index{period!coordinate}. The integrals over the \(\{\,\alpha_i,\beta_i\,\}\) are called \emph{absolute periods}\index{period!absolute}, while those over the \(\delta_i\) are called \emph{relative periods}\index{period!relative}. In this perspective, the period map can be seen as 
\begin{equation}
\textnormal{Per}(\,\mu\,) \colon \strata \longrightarrow \mathbb{C}^{2g} = \textnormal{H}^1(\,\Sigma,\, \mathbb{C}\,),
\end{equation}
which assigns to each translation surface its absolute periods:
\begin{equation}
(X, \omega) \longmapsto \left( \,\int_{\alpha_1} \omega,\,\int_{\beta_1} \omega, \ldots, \int_{\alpha_g} \omega,\,\int_{\beta_g} \omega\,\right).
\end{equation}
\noindent By keeping the absolute period fixed, the subspace described by the relative periods is an interesting object to study known as isoperiodic foliation\index{foliation!isoperiodic}, see \S\ref{ssec:isoper}.

\smallskip

\subsection{Period characters}\label{ssec:perchar} A holomorphic \(1\)-form \(\omega\) on a compact Riemann surface naturally defines a representation in homology, say
\begin{equation}
\chi\colon \textnormal{H}_1(\,X,\,\mathbb{Z}\,) \longrightarrow \mathbb{C},\qquad \gamma\longmapsto\int_\gamma \omega,
\end{equation}
whose image is the group of periods of \(\omega\). The representation \(\chi\) is called the \textit{period\index{period!character} character\index{character!period}}. Although the group of periods is intrinsically defined by \(\omega\), the representation is well defined upon choosing a basis in homology. The map \(\textnormal{Per}\) mentioned above assigns to each translation structure marked in homology the period character determined by \(\omega\). Since a marking in homology corresponds to a specific choice of generators of the homology, then the map \(\textnormal{Per}\) is well defined. Notice that there is no way to define a similar map on \(\Omega\mathcal{M}_g\) unless one associates to a translation surface an equivalence class of representations, see \cite[\S4.1.1]{NS}.

\smallskip

\noindent Given a pair \((X,\omega)\), the period representation and the holonomy of the associated geometric structure are deeply related. In \S\ref{ssec:devhol} we already observed that the holonomy representation descends to homology by a purely algebraic argument. Although we have already anticipated that the representation in homology is called the period character (and therefore the reader might expect it to be the one induced by \(\omega\)) it is not a priori obvious that the induced representation actually coincides with the period character. As a consequence we have the following:

\begin{lem}\label{lem:charholo}
    Let \(\Sigma\) be a closed surface of genus \(g\ge1\). A representation \(\chi\colon \textnormal{H}_1(\,\Sigma,\,\mathbb{Z}\,) \longrightarrow \mathbb{C}\) arises as the period character of some homologically marked pair \((X,\omega)\) if and only if it is induced by the holonomy representation of some translation structure on \(\Sigma\).
\end{lem}

\noindent We thus arrive at the central question of this section: 
\begin{quote}
    \textit{under what conditions is a given representation the holonomy of some translation structure?}
\end{quote}

\noindent In the next sections we aim to provide an answer to this question.

\smallskip


\subsection{Algebraic volume}\label{ssec:vol} Before we can address the above question, it is necessary to introduce the notion of volume which plays a pivotal role in the theory. Since we are in real dimension two, perhaps “area” would be a more appropriate term than “volume”, but this is the terminology commonly used. 

\smallskip

\noindent There are several ways to define the volume\index{volume!representation} of a representation\index{representation!volume}. In the present chapter, we provide an algebraic definition that we shall systematically use in the following. Let \(\chi\colon\textnormal{H}_1(\Sigma,\Z)\longrightarrow \C\) be any representation and let \(\{\alpha_i,\,\beta_i\}_{1\le i\le g}\) be a symplectic basis of \(\textnormal{H}_1(\Sigma,\mathbb Z)\). Then, every representation \(\chi\) defines a vector, say \(w(\,\chi\,)\in\C^{2g}\), that is,
\begin{equation}
    w(\,\chi\,)\,=\,\big(\,\chi(\,\alpha_1\,),\,\chi(\,\beta_1\,),\dots,\chi(\,\alpha_g\,),\,\chi(\,\beta_g\,)\,\big).
\end{equation}

\noindent We define \(\Re(\,\chi\,)\) and  \(\Im(\,\chi\,)\) as the real and the imaginary part of \(w(\,\chi\,)\). This approach allows us to regard every representation as a pair of vectors \((u,v)\in\R^{2g}\times\R^{2g}\cong\textnormal{H}^1(\Sigma,\,\R)\times\textnormal{H}^1(\Sigma,\,\R)\). 

\smallskip

\noindent The \textit{volume}\index{volume!algebraic} is the quadratic form \(\textnormal{vol}\colon\textnormal{H}^1(\Sigma,\C)\longrightarrow \R\) defined as \(\textnormal{vol}(\,\chi\,)=\Omega\big(\,\Re(\,\chi\,),\,\Im(\,\chi\,)\,\big)\), where \(\Omega\) is the standard symplectic form on \(\R^{2g}\), induced by the symplectic pairing \(\textnormal{V}\colon\textnormal{H}^1(\Sigma,\,\R)\times\textnormal{H}^1(\Sigma,\,\R) \longrightarrow \R\) determined by Poincar\'e duality.
For a symplectic basis \(\{\,\alpha_1,\beta_1,\dots,\alpha_g,\beta_g\,\}\) of \(\textnormal{H}_1(\,\Sigma,\,\mathbb Z\,)\), it can be shown that the volume of a representation \(\chi\) is equal to
\begin{equation}\label{algvol}
    \text{vol}(\,\chi\,)=\sum_{i=1}^g \Im\Big(\,\,\overline{\chi(\,\alpha_i\,)}\,\chi(\,\beta_i\,)\,\,\Big),
\end{equation} 
where $\Im(\,\overline{z}\,w\,)$ is the usual symplectic form on $\C$. As a consequence, this algebraic definition of volume of a character $\chi$ is invariant under precomposition with any automorphisms in $\textnormal{Aut}\,\big(\,\shomolz\,\big)\cong\spz$. The image of $\chi$, provided that it has rank $2g$, turns out to be a polarised module.

\smallskip

\smallskip

\subsection{Holonomy of translation structures on the torus}\label{ssec:hologenone} We aim to determine the image of the period map \(\textnormal{Per}\) for genus one surfaces. As a preliminary observation, we note that Lemma~\ref{lem:charholo} asserts that, in this specific case, the period character is not only induced by the holonomy of the associated structure, but actually coincides with it as a representation. Indeed, it is a well-known classical fact that for genus one surfaces, the fundamental group coincides with the first homology group. The second key observation in this setting is that, given a representation \(\chi\) and a specified basis \(\{\,\alpha,\,\beta\,\}\) in homology, the algebraic volume as defined in~\S\ref{ssec:vol} coincides, by a mere matter of definitions, with the area of the parallelogram defined by \(\chi(\,\alpha\,)\) and \(\chi(\,\beta\,)\). Since a parallelogram defines a fundamental domain for a translation structure if and only if its two-dimensional volume is positive, it follows that for a representation \(\chi\) to arise as the holonomy, and hence as the period character, of some translation structure, the volume must be positive. The goal of this section is to show that this condition is also sufficient.

\begin{prop}\label{prop:hauptgenone}
    Let \(\Sigma\) be a surface of genus one. A representation \(\chi\colon\textnormal{H}_1(\,\Sigma,\,\mathbb Z\,)\longrightarrow \mathbb C\) arises as the holonomy of some marked translation structure on \(\Sigma\) if and only if \(\textnormal{vol}(\,\chi\,)>0\). Moreover, such a structure is unique up to translations.
\end{prop}

\begin{proof}
    Following the above discussion, it remains to show that the condition \(\textnormal{vol}(\,\chi\,) > 0\) is sufficient. Fix a marking on \(\Sigma\), that is, a choice of a system of generators in homology, say \(\{\,\alpha,\, \beta\,\}\). Let \(\mu_1\) and \(\mu_2\) be the images of the generators under \(\chi\). Since the volume of \(\chi\) is positive, it follows by definition that the two-dimensional volume of the parallelogram determined by \(\mu_1\) and \(\mu_2\) is positive. By gluing together the opposite sides of this parallelogram, one obtains a marked translation structure on the torus whose holonomy is equal to \(\chi\) by design.
\end{proof}

\noindent Once a system of generators for homology is fixed, a representation with positive volume encodes exactly the same data as a parallelogram in the plane: they both define a translation structure and thus a complex structure, and a lattice along with an explicit system of generators. A translation structure, on the other hand, encodes less information: the period group, \textit{i.e.}, the lattice, is well defined, but there is no canonical choice of a system of generators, and hence no uniquely determined fundamental domain. Referring to the diagram in Table~\ref{tab:determination}, the relationship among holonomy representations, parallelograms in \(\mathbb C\) and translation structures on a torus is shown in Table~\ref{tab:determination2}.

\begin{table}[ht!]
    \centering
    \begin{tikzpicture}[baseline=(current bounding box.center), every node/.style={scale=0.875}]
    \node (A) at (0,0) {$\left\{ \begin{array}{c}
    \text{holonomy representations} \\
    \rho\colon\pi_1(\,\Sigma\,)\longrightarrow\mathbb C
    \end{array} \right\}$};

    \node (B) at (6,0) {$\left\{ \begin{array}{c}
    \text{fundamental domains} \\
    \text{\textit{i.e.} parallelograms in } \mathbb{C}
    \end{array} \right\}$};
    \draw[->, thin] ([xshift=2mm,yshift=1.5mm]A.east) -- ([xshift=-2mm,yshift=1.5mm]B.west);
    \draw[<-, thin] ([xshift=2mm,yshift=-1.5mm]A.east) -- ([xshift=-2mm,yshift=-1.5mm]B.west);

    \node (C) at (11.5,0) {$\left\{ \begin{array}{c}
    \text{translation structures} \\
    \text{on a torus}
    \end{array} \right\}$};
    \draw[<-, thin, dashed] ([xshift=2mm,yshift=-1.5mm]B.east) -- ([xshift=-2mm,yshift=-1.5mm]C.west);
    \draw[->, thin] ([xshift=2mm,yshift=1.5mm]B.east) -- ([xshift=-2mm,yshift=1.5mm]C.west);
    \end{tikzpicture}
    \caption{The solid arrow denotes a complete determination of one object by another; a broken arrow denotes that some choice is involved.} 
    \label{tab:determination2}
\end{table}

\subsection{Haupt's theorem}\label{ssec:haupt} The case of genus one surfaces is interesting because it provides insight into the general case, which is decidedly less straightforward to handle. We have already observed that not all representations arise as the holonomy of a translation structure. This remains true in higher genus, but it will not be the only obstruction to the realisation of representations. The main result is the following Theorem by Haupt\index{Haupt theorem}, see \cite{OH}.

\begin{hthm}\label{thm:hauptthm}
    Let \(\Sigma\) be a closed surface of genus \(g\ge2\). A representation \(\chi\colon \textnormal{H}_1(\,\Sigma,\,\mathbb{Z}\,) \longrightarrow \mathbb{C}\) arises as the period character of some pair \((X,\omega)\) if and only if the following conditions holds:
    \begin{itemize}
        \item[1.] \(\textnormal{vol}(\,\chi\,) > 0\), and 
        \item[2.] if \(\textnormal{Im}(\,\chi\,)\,<\,\mathbb C\) is a lattice then \(\textnormal{vol}(\,\chi\,)\,\ge\,2\textnormal{area}\big(\,\mathbb C/\Lambda\,\big)\).
    \end{itemize}
\end{hthm}

\noindent Haupt's Theorem has been subsequently rediscovered by Kapovich in \cite{KM2}, by using Ratner’s solution of Raghunathan’s conjecture, thus providing an interesting alternative approach. Kapovich's key observation is that the set of periods of holomorphic one-forms on homologically marked Riemann surfaces is preserved under the action of the integral symplectic group. In \S\ref{sssec:volcond} and \S\ref{sssec:latticecond} we briefly explain how the above conditions arise as obstructions and how to realise a structure when they hold. We notice that the second condition holds if and only if the image of a representation is a discrete subgroup of \(\mathbb C\) of rank two, meaning that it is isomorphic to \(\mathbb Z^2\).

\smallskip

\subsubsection{Condition\index{volume!condition} on the volume}\label{sssec:volcond} The first obstruction is a mere consequence of Riemann's bilinear relations. In fact, given any pair \((X,\omega)\) the quantity
\begin{equation}
    \frac{i}{2}\,\int_{X} \omega\wedge\overline{\omega}
\end{equation}
is the area of the surface \(\Sigma\) with respect to the singular Euclidean metric on \(X\) induced by \(\omega\), which must be positive. On the other hand, given a basis \(\{\,\alpha_i,\beta_i\,\}\) for the homology \(\textnormal{H}_1(X,\,\mathbb{Z})\cong\textnormal{H}_1(\Sigma,\,\mathbb{Z})\), the following chain of equalities hold:
\begin{equation}
    \frac{i}{2}\,\int_{X} \omega\wedge\overline{\omega} \,\,\overset{\textnormal{RBR}}{=}\,\,\frac{i}{2}\sum_{j=1}^g\,\Big(\,\overline{\chi(\,\alpha_j\,)}\chi(\,\beta_j\,)-\chi(\,\alpha_j\,)\overline{\chi(\,\beta_j\,)} \Big)\,\,=\,\,\sum_{j=1}^g \Im\Big(\,\,\overline{\chi(\,\alpha_j\,)}\,\chi(\,\beta_j\,)\,\,\Big)\,=\,\textnormal{vol}(\,\chi\,),
\end{equation}
where \(\chi\) is the period character of \(\omega\) and \text{RBR} stands for \textit{Riemann Bilinear Relations}. Under a different perspective, the developing map of the translation structure determined by \(\omega\), say \(\textnormal{dev}\colon \widetilde\Sigma\longrightarrow \mathbb C\), yields a multivalued function \(F\colon X\longrightarrow \mathbb C\) and 
\begin{equation}
    \int_{X} \omega\wedge\overline{\omega} \,=\,\int_X\,F^*\big(\,dz\wedge d\overline{z}\,\big).
\end{equation}



\smallskip

\subsubsection{Lattice condition\index{lattice!condition}}\label{sssec:latticecond} We now focus on the second obstruction which applies only when the image of \(\chi\) is a lattice, say \(\Lambda\). This obstruction arises for topological reasons. Since \(\Lambda\cong \textnormal{H}_1(\, T,\,\mathbb Z\,)\), where \(T\) is the topological torus, then \(\chi\) may be regarded as a map \(\chi\colon \textnormal{H}_1(\, \Sigma,\,\mathbb Z\,)\longrightarrow \textnormal{H}_1(\, T,\,\mathbb Z\,)\) in homology. The crux of the matter is that, up to homotopy, \(\chi\) is determined by a mapping \(f\colon \Sigma\longrightarrow T\). In the following we may assume \(\big[\,\Lambda\,:\textnormal{H}_1(\,T,\Z\,)\,\big]=1\). If not, we may lift \(f\) to a map \(\overline{f}\) onto the covering of \(T\) corresponding to the subgroup \(\Lambda\).

\smallskip

\noindent Suppose \(\chi\) arises as the period character of some translation structure, say \((X,\omega)\) on \(\Sigma\). Then, its developing map \(\textnormal{dev}\colon\widetilde{X}\longrightarrow \C\) boils down to a map \(X\longrightarrow \C/\Lambda\) that we may assume to be \(f\), up to homotopy, because \(\textnormal{dev}\) is \(\chi-\)invariant. The restriction of \(f\) around any zero of \(\omega\) looks like the the map \(g\) defined in \S\ref{sssec:bigcone}. In particular, if \(p\) is a zero of \(\omega\) of order \(m\ge1\), then the local degree of \(f\) around \(p\) is equal to \(m+1\ge2\). Moreover, \(\omega\) has at least one zero since \(X\) has genus at least two. Hence \(\deg(\,f\,)\ge2\). In \cite{EA}, Edmonds shows that a map \(f\) of non-zero degree between orientable and closed surfaces is homotopic to a pinching map post-composed with a branched covering. However, since \(f_*=\chi\) and the inequality \(\deg(\,f\,)>\big[\,\textnormal{Im}(\,\chi\,)\,:\textnormal{H}_1(\,T,\Z\,)\,\big]=1\) holds, then the pinching map can be absorbed into the branched covering map. Therefore, \(f\) is homotopic to a branched covering map of surfaces of degree at least \(2\).

\begin{rmk}\label{rmk:secondhauptcondanalyticpersp}
    By adopting a more analytic perspective, the same bound on \(\deg(\,f\,)\) arises as follows. Suppose \(\chi\) arises as the monodromy of the first order linear ODE, say \(dF=\omega\) on \(X\). The multivalued solution on \(X\) yields a non-constant holomorphic map \(f\colon\Sigma\longrightarrow T\) that cannot be a homeomorphism for topological reasons. Hence its degree\index{representation!degree} must be at least two.
\end{rmk}

\noindent We now define a topological invariant naturally attached to \(\chi\). The \textit{degree} of \(\chi\) is the positive integer defined as the ratio:
\begin{equation}
\label{eq:refhaup1}\deg(\,\chi\,)=\frac{\textnormal{vol}(\,\chi\,)}{\textnormal{area}(\,\mathbb C/\Lambda\,)}.
\end{equation}
\smallskip
\noindent Since \(\deg(\,f\,)=\deg(\,\chi\,)\), the second obstruction readily follows, that is
\begin{equation}
    \textnormal{vol}(\,\chi\,)\ge 2\,\textnormal{area}(\,\mathbb C/\Lambda\,).
\end{equation}

\begin{rmk}[Maps of degree one]
    We observe that the constraint \(\deg(\,f\,)\ge2\) vanishes if we do not assume \(f\) to be a branched covering map. Given two surface \(\Sigma_g\) and \(\Sigma_h\) of genus \(g,h\) where \(g>h\), there exist maps \(\Sigma_g\longrightarrow \Sigma_h\) of degree one. By \cite{EA}, these maps are homotopic to pinching maps, namely continuous maps that collapse \(g-h\) handles to a point. In the case \(g\ge2\) and \(h=1\), every continuous map of degree one \(f\colon \Sigma\longrightarrow\mathbb{C}/\Lambda\) yields a representation \(\chi\in\textnormal{H}^1(\Sigma,\C)\) of volume equal to \(\textnormal{vol}(\C/\Lambda)\) that cannot be realised as the period character of some translation structure on \(\Sigma\) because the second Haupt condition is violated. Vice versa, every discrete representation \(\chi\in\textnormal{H}^1(\Sigma,\C)\) of degree one is induced by a pinching map. Finally, we may notice that the complex structure on \(\C/\Lambda\) cannot be pulled-back on \(\Sigma\) otherwise \(f\) would be homotopic to a branched covering map of degree one that is a homeomorphism.
\end{rmk}

\smallskip

\subsubsection{Realising representations following Kapovich's approach}\label{ssec:kapapp} In the previous paragraph, we showed why the conditions of Haupt’s theorem are, in fact, necessary. In this paragraph we want to give an idea of why these conditions are also sufficient. Let \(\chi\colon\textnormal{H}_1(\,\Sigma,\,\mathbb{Z}\,) \longrightarrow \mathbb{C}\) be a non-discrete\footnote{A representation in \(\C\) is \textit{non-discrete}\index{representation!non-discrete} if the image if not a discrete additive subgroup of \(\C\).} representation with positive volume and suppose the existence of a basis in homology which is good in the following sense.

\smallskip

\begin{quote}
    \textit{A basis \(\{\alpha_i,\beta_i\}\) in homology is said to be \textit{good} if and only if \(\Im\Big(\,\,\overline{\chi(\,\alpha_i\,)}\,\chi(\,\beta_i\,)\,\,\Big)>0\) for every \(i=1,\dots,g\).}
\end{quote}

\smallskip

\noindent Under this assumption, each pair \(\{\,\chi(\,\alpha_i\,),\,\chi(\,\beta_i\,)\,\}\) yields the fundamental parallelogram of some translation structure on a torus, by Proposition \ref{prop:hauptgenone}, and the desired translation structure on \(\Sigma\) arises by gluing these tori in some appropriate way. More geometrically, we can find a collection of \(g\) parallelograms such that, once decomposed into polygons, \textit{e.g.} triangles, and appropriately reassembled by translations, they determine the desired translation surface. Figure \ref{fig:genkap} depicts such a situation in genus three.

\begin{figure}[h!]
    \centering
    \begin{tikzpicture}[scale=1.5, every node/.style={scale=0.675}]
    \definecolor{pallido}{RGB}{221,227,227}

    \pattern [pattern=north west lines, pattern color=pallido]
    (0,0)--(2,0)--(2,2)--(0,2)--(0,0);

    \pattern [pattern=north west lines, pattern color=pallido]
    (3,0)--(5,1)--(6,3)--(4,2)--(3,0);

    \pattern [pattern=north west lines, pattern color=pallido]
    (7,0)--(10,0)--(10,2)--(7,2)--(7,0);
    
    \draw[thin, orange        ] (0,0)--(2,0);
    \draw[thin, red           ] (2,0)--(2,2);
    \draw[thin, orange        ] (2,2)--(0,2);
    \draw[thin, red           ] (0,2)--(0,0);

    \draw[thin, orange        ] (3,0)--(5,1);
    \draw[thin, red           ] (5,1)--(6,3);
    \draw[thin, orange        ] (6,3)--(4,2);
    \draw[thin, red           ] (4,2)--(3,0);

    \draw[thin, orange        ] (7,0)--(10, 0);
    \draw[thin, red           ] (10,0)--(10,2);
    \draw[thin, orange        ] (10,2)--(7,2);
    \draw[thin, red           ] (7,2)--(7,0);

    \draw[thin, blue         ] (0.5,0.5  )--(1.5,0.5  );
    \draw[thin, blue         ] (0.5,1.5  )--(1.5,1.5  );
    \draw[thin, blue         ] (4  ,1.5)--(5  ,1.5);
    \draw[thin, blue         ] (8  ,1  )--(9  ,1  );

    \draw[thin, dashed       ] (0,0)--(0.5,0.5);
    \draw[thin, dashed       ] (1.5,0.5)--(2,0);
    \draw[thin, dashed       ] (0,2)--(0.5,1.5);
    \draw[thin, dashed       ] (2,2)--(1.5,1.5);
    \draw[thin, dashed       ] (0,0)--(1.5,0.5);
    \draw[thin, dashed       ] (2,2)--(0.5,1.5);
    \draw[thin, dashed       ] (0,2)--(0.5,0.5);
    \draw[thin, dashed       ] (2,0)--(1.5,1.5);
    \draw[thin, dashed       ] (0.5,0.5)--(0.5,1.5);
    \draw[thin, dashed       ] (1.5,0.5)--(1.5,1.5);

    \draw[thin, dashed       ] (3,0)--(4,1.5);
    \draw[thin, dashed       ] (3,0)--(5,1.5);
    \draw[thin, dashed       ] (5,1.5)--(5,1);
    \draw[thin, dashed       ] (5,1.5)--(6,3);
    \draw[thin, dashed       ] (4,1.5)--(4,2);
    \draw[thin, dashed       ] (6,3)--(4,1.5);

    \draw[thin, dashed       ] (7,0)--(8 ,1);
    \draw[thin, dashed       ] (8,1)--(10,0);
    \draw[thin, dashed       ] (10,0)--(9,1);
    \draw[thin, dashed       ] (9,1)--(10,2);
    \draw[thin, dashed       ] (8,1)--(7,2);
    \draw[thin, dashed       ] (7,2)--(9,1);
    
    \fill (0,0) circle (1.25pt);
    \fill (2,0) circle (1.25pt);
    \fill (2,2) circle (1.25pt);
    \fill (0,2) circle (1.25pt);
    \fill (3,0) circle (1.25pt);
    \fill (5,1) circle (1.25pt);
    \fill (6,3) circle (1.25pt);
    \fill (4,2) circle (1.25pt);
    \fill (7,0) circle (1.25pt);
    \fill (10,0) circle (1.25pt);
    \fill (10,2) circle (1.25pt);
    \fill (7,2) circle (1.25pt);

    \fill[blue] (0.5,0.5) circle (1.25pt);
    \fill[blue] (1.5,0.5) circle (1.25pt);
    \fill[blue] (0.5,1.5) circle (1.25pt);
    \fill[blue] (1.5,1.5) circle (1.25pt);
    \fill[blue] (4,1.5) circle (1.25pt);
    \fill[blue] (5,1.5) circle (1.25pt);
    \fill[blue] (8,1  ) circle (1.25pt);
    \fill[blue] (9,1  ) circle (1.25pt);

    \node at (1.25,1.625) {\(1\)};
    \node at (1.25,1.375) {\(2\)};
    \node at (0.75,0.625) {\(3\)};
    \node at (0.75,0.375) {\(4\)};
    \node at (4.5,1.625) {\(2\)};
    \node at (4.5,1.375) {\(1\)};
    \node at (8.5,1.125) {\(4\)};
    \node at (8.5,0.875) {\(3\)};

    \end{tikzpicture}
    \caption{The figure shows a collection \(\mathcal P\) of polygons. Each polygon gives rise to a specific handle. Each dashed segment in every polygon is labelled so that different segments have different labels. We then divide the polygons into a new finite collection, say \(\mathcal Q\), entirely made of triangles. Each dashed segment defines a pair of sides, which we mark with the same label to indicate they are identified (the figure does not show the labels on the dashed lines in order to avoid too much clutter; the reader is invited to add them if desired). We note that for each label there is a unique pair, by design. On the other hand, each blue segment defines two different sides that are labelled differently. The polygons in the collection \(\mathcal Q\) will then be glued back together, identifying sides with the same label. All this is equivalent to splitting the polygons of the former collection \(\mathcal P\) along the blue segments and simply identifying the sides with the same label \(1,2,3,4\). The resulting surface has genus \(g\) and supports a translation structure in \(\mathcal H_3(1,1,1,1)\) that, by design, has the desired period character.}
    \label{fig:genkap}
\end{figure}
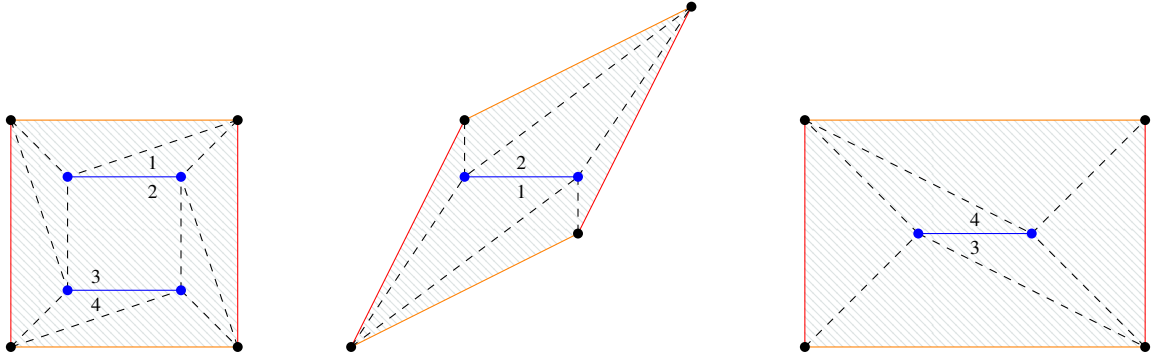

\smallskip

\noindent Determining the existence of a good basis is more of a dynamical problem than a geometric one. In his paper, Kapovich first observed that it is not necessary to consider all representations, but it suffices to consider those of volume one. These representations form a subspace \(\mathcal X=\big\{\,\,\chi\in\textnormal{H}^1(\Sigma,\mathbb C)\,\,|\,\,\textnormal{vol}(\,\chi\,)=1\,\,\big\}\). Recall from \S\ref{ssec:vol} that the standard symplectic form on \(\mathbb{R}^{4g}\) defines a quadratic form \(\textnormal{vol}\colon\textnormal{H}^1(\Sigma,\mathbb C)\longrightarrow\mathbb R\). The space \(\mathcal{X}\) is identified with the level set \(\textnormal{vol} = 1\) and it is invariant under the action of the symplectic group \(\textnormal{Sp}(2g,\mathbb R)\). The action is clearly transitive, and the stabilizer of any point in \(\mathcal X\) is identified with the group \(\textnormal{Sp}(2g-2,\mathbb R)\), see Remark \ref{rmk:stabiliser}, from which we conclude that \(\mathcal X=\textnormal{Sp}(2g,\mathbb R)/\textnormal{Sp}(2g-2,\mathbb R)\). 

\begin{rmk}\label{rmk:stabiliser}
    By regarding a representation as a pair of vectors \(\R^{2g}\times\R^{2g}\), the stabiliser of a representation can be seen as the stabiliser of a pair \((u,v)\) under the action of \(\textnormal{Sp}(2g,\mathbb R)\). The vectors \(u,v\) span a symplectic \(2\)-plane in \(\mathbb{R}^{2g}\), whose symplectic complement therefore has dimension \(2g-2\). The stabiliser acts trivially on this plane and symplectically on its complement. Hence the stabiliser of any pair is isomorphic to \(\mathrm{Sp}(2g-2,\mathbb{R})\).
\end{rmk}

\noindent Let \(\mathbb P\Omega\mathcal S_g\) be the subspace of \(\Omega\mathcal S_g\) of marked translation surfaces with volume one and let \(\mathcal{RX}\) be the image of the map as in \eqref{eq:permap} to \(\mathbb P\Omega\mathcal S_g\). Here \(\mathcal{RX}\) stands for realisable representations in \(\mathcal X\)\footnote{Our notation here is different from that used by Kapovich in his paper.}. The restricted map \(\textnormal{Per}\colon\mathbb P\Omega\mathcal S_g\longrightarrow \mathcal{RX}\) is an open map meaning that \(\mathcal{RX}\) is open in \(\mathcal X\). Moreover, \(\mathcal{RX}\) is invariant under the action of the discrete group \(\textnormal{Sp}(2g,\mathbb Z)\) which is a lattice in \(\textnormal{Sp}(2g,\mathbb R)\). We recall for the readers' convenience that an action is said to be \textit{ergodic}\index{ergodic} if every invariant measurable subset has either full measure or zero measure. Moore's Theorem, see \cite{Zim84}, applies and as a consequence \(\textnormal{Sp}(2g,\mathbb Z)\) acts ergodically on \(\mathcal X\) and \(\mathcal{RX}\subset \mathcal X\) is an open invariant subset of full measure. Let \(\mathcal{GX}\subset \mathcal X\) be the subset of representations for which a good basis as above exist. We have already seen that \(\mathcal{GX}\subset\mathcal{RX}\). Ratner's theory is then applied to show that for every non-discrete representation \(\chi\in\mathcal X\), the orbit \(\textnormal{Sp}(2g,\mathbb Z)\cdot\chi\) intersects \(\mathcal{GX}\), to conclude that \(\chi\in\mathcal{RX}\), see \cite[\S5]{KM2}. However there are representations in \(\mathcal X\) which do not belong to the orbit \(\textnormal{Sp}(2g,\mathbb Z)\cdot\mathcal{GX}\). These correspond to the characters with discrete image.

\begin{figure}[h!]
    \centering
    \begin{tikzpicture}[scale=1.75, every node/.style={scale=0.75}]
    \definecolor{pallido}{RGB}{221,227,227}

    \pattern [pattern=north west lines, pattern color=pallido]
    (0,0)--(8,0)--(8,4)--(0,4)--(0,0);
    
    \draw[thin, orange        ] (0,0)--(8,0);
    \draw[thin, red           ] (8,0)--(8,4);
    \draw[thin, orange        ] (8,4)--(0,4);
    \draw[thin, red           ] (0,4)--(0,0);

    \draw[thin, dashed       ] (1,0)--(1,4);
    \draw[thin, dashed       ] (2,3)--(2,4);
    \draw[thin, dashed       ] (5,3)--(5,4);

    \draw[thin, dashed       ] (2,0)--(2,1);
    \draw[thin, dashed       ] (3,0)--(3,2);
    \draw[thin, dashed       ] (4,0)--(4,2);
    \draw[thin, dashed       ] (6,0)--(6,4);

    \draw[thin, dashed       ] (2,1)--(3,1);
    \draw[thin, dashed       ] (3,2)--(4,2);
    \draw[thin, dashed       ] (2,3)--(5,3);
    
    \draw[thin, blue         ] (1,1)--(2,1);
    \draw[thin, blue         ] (1,2)--(3,2);
    \draw[thin, blue         ] (3,1)--(4,1);
    \draw[thin, blue         ] (4,2)--(6,2);
    \draw[thin, blue         ] (1,3)--(2,3);
    \draw[thin, blue         ] (5,3)--(6,3);
    
    \fill (0,0) circle (1.25pt);
    \fill (8,0) circle (1.25pt);
    \fill (8,4) circle (1.25pt);
    \fill (0,4) circle (1.25pt);

    \fill[blue] (1,1) circle (1.25pt);
    \fill[blue] (2,1) circle (1.25pt);
    \fill[blue] (1,2) circle (1.25pt);
    \fill[blue] (3,2) circle (1.25pt);

    \fill[blue] (3,1) circle (1.25pt);
    \fill[blue] (4,1) circle (1.25pt);
    \fill[blue] (4,2) circle (1.25pt);
    \fill[blue] (6,2) circle (1.25pt);

    \fill[blue] (1,3) circle (1.25pt);
    \fill[blue] (2,3) circle (1.25pt);
    \fill[blue] (5,3) circle (1.25pt);
    \fill[blue] (6,3) circle (1.25pt);

    \node at (1.5,0.875) {\(a_+\)};
    \node at (1.5,1.125) {\(a_-\)};
    \node at (3.5,0.875) {\(a_-\)};
    \node at (3.5,1.125) {\(a_+\)};

    \node at (2  ,1.875) {\(b_+\)};
    \node at (2  ,2.125) {\(b_-\)};
    \node at (5  ,1.875) {\(b_-\)};
    \node at (5  ,2.125) {\(b_+\)};

    \node at (1.5,2.875) {\(c_+\)};
    \node at (1.5,3.125) {\(c_-\)};
    \node at (5.5,2.875) {\(c_-\)};
    \node at (5.5,3.125) {\(c_+\)};

    \draw[black, <->] (1,0.5)--(3,0.5);
    \fill[white] (1.75,0.4) rectangle (2.25, 0.6);
    \node[scale=0.75] at (2,0.5) {\(a_2\)};

    \draw[black, <->] (1,1.5)--(4,1.5);
    \fill[white] (2.25,1.4) rectangle (2.75, 1.6);
    \node[scale=0.75] at (2.5,1.5) {\(a_3\)};

    \draw[black, <->] (1,3.5)--(5,3.5);
    \fill[white] (2.75,3.4) rectangle (3.25, 3.6);
    \node[scale=0.75] at (3,3.5) {\(a_4\)};

    \end{tikzpicture}
    \caption{The figure shows a rectangle with area equal to the volume of the representation. By identifying the opposite sides by means of translations the resulting space in a torus. To create the desired handles, for each \(i\), consider two segments of the same arbitrary length separated by a segment of length \(\chi(\,\alpha_i\,)=a_i\) (drawn in blue), \textit{i.e.}, they differ by a translation \(z\mapsto z+a_i\). Since \(\chi(\,\beta_i\,)=0\) for \(i=2,\dots,g\), the blue segments are at the same level, that is \(a_i\in\Z\), for \(i=1,\dots,g\). The rectangle is then divided into polygons (squares and rectangles). Each dashed segment in the rectangle is labelled so that different segments have different labels. We then divide the rectangle into a finite collection of polygons. Each dashed segment defines a pair of sides with the same label. We note that for each label there is a unique pair, by design. Each of the blue segments defines two different sides that are now labelled differently. The polygons in the collection will then be glued back together, identifying sides with the same label. All this is equivalent to splitting the rectangle along the blue segments and simply identifying the sides with the same label. The resulting surface will now have \(g-1\) more handles and supports a translation structure in \(\mathcal H_4(1,1,1,1,1,1)\) that, by design, has the desired period character.}
    \label{fig:kapdiscr}
\end{figure}
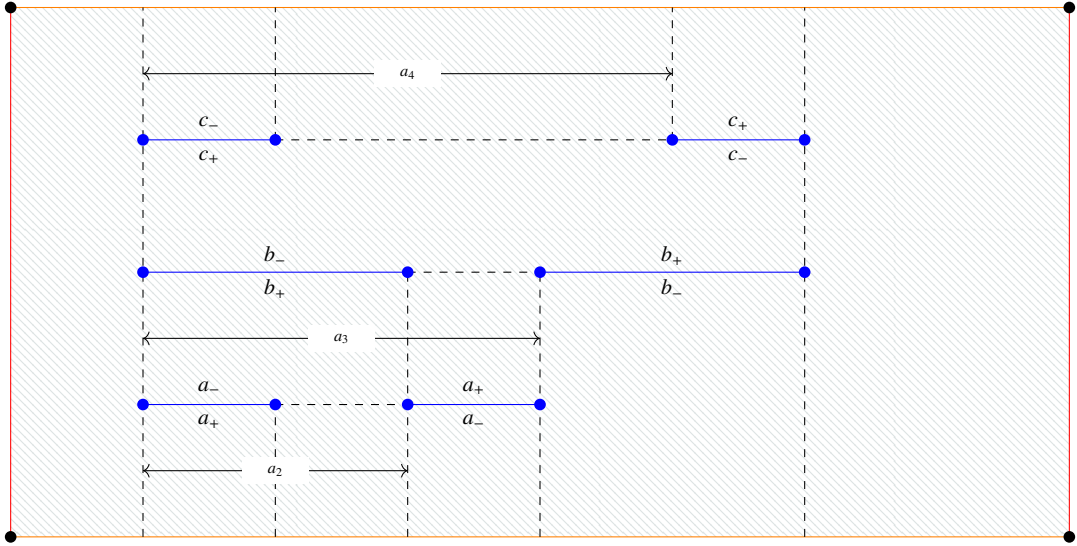

\smallskip

\noindent In case a representation \(\chi\) with positive volume is discrete\index{representation!discrete}, the image of \(\chi\) is a lattice and there is no loss of generality in assuming that \(\textnormal{Im}(\,\chi\,)=\ziz\). This normalisation can be easily explained as follows. Let us suppose for a moment that a discrete representation with image \( \ziz \) and satisfying Haupt's conditions is realisable. Let \(\chi\) be any such representation. Then there exists a collection of polygons \(\mathcal P\) and an explicit pairing via translations, say \(\tau\), see \S\ref{ssec:tiled}, such that the resulting translation surface has period character \(\chi\). Let \(\chi'\) be any other discrete representation with positive volume and let \(\Lambda=\textnormal{Im}(\,\chi'\,)\). We may notice that there is an element \(A\in\textnormal{GL}^+(2,\mathbb{R})\) such that \(A\cdot \ziz =\Lambda\) and, as a consequence, \(A\cdot\mathbb \chi=\chi'\). This element \(A\) can be used to define a new collection of polygons \(\mathcal P'\) and a pairing via translations \(\tau'\) which can be used to realise a translation surface with period character \(\chi'\). Therefore, we may assume that \(\textnormal{Im}(\,\chi\,)=\ziz\). Next, even in this case, we assume the existence of a \textit{good basis} in homology. For discrete representations, a good basis is a basis \(\{\alpha_i,\beta_i\}\) in homology such that 

\smallskip

\begin{itemize}
    \item[1.] \(\chi(\,\alpha_i\,)\in\mathbb Z\) and \(0<\chi(\,\alpha_i\,)<\chi(\,\alpha_1\,)=\textnormal{vol}(\,\chi\,)\in\mathbb Z\) for all \(i=1,\dots,g\),
    \smallskip
    \item[2.] \(\chi(\,\beta_1\,)=i\) whereas \(\chi(\,\beta_i\,)=0\) for all \(i=1,\dots,g\).
\end{itemize}

\smallskip

\noindent See \cite[\S4]{KM2} for details. Under this assumption, the desired structure can be realised by decomposing a rectangle of size \(\textnormal{vol}(\,\chi\,)\times1\) into smaller polygons and then reassembling them properly, see Figure \ref{fig:kapdiscr}.

\smallskip

\noindent In summary, in order to realise a given representation, in both cases Kapovich made use of a good basis adapted to purpose. The bulk of the work in \cite{KM2} consisted in determining such a basis for an arbitrary representation and relies on the full strength of Ratner’s theory.

\begin{rmk}
    It is worth mentioning that Kapovich's approach does not cover the case \(g=2\). In fact, in \cite[\S4]{KM2} it makes use of the condition \(g\ge 3\).
\end{rmk}

\smallskip

\subsection{Recent refined versions}\label{ssec:refinements} Haupt's Theorem provides a full description of the image of the Period map \eqref{eq:permap}. In recent times, several refinements of Haupt's Theorem have been provided by restricting the period mapping \(\mathrm{Per}\) to strata or to connected components of strata of holomorphic differentials. Realising a representation as a character in a prescribed stratum turns out to be a more subtle problem, because in the realisation process the orders of the zeros can no longer be ignored. In \cite{BJJP}, Bainbridge--Johnson--Judge--Park provided necessary and sufficient conditions for a representation to be realised in a connected component of a given stratum. Shortly thereafter Le Fils provided in \cite{fils} necessary and sufficient conditions for a representation \(\chi\) to be a character in a given stratum, with an alternative proof in the same spirit as Kapovich. They showed the following:

\begin{thm}[Bainbridge--Johnson--Judge--Park, Le Fils]
    Let \(\Sigma\) be a closed surface of genus \(g\ge2\) and let \(\mu=(m_1,\dots,m_k)\) be any signature. A representation \(\chi\colon \textnormal{H}_1(\,\Sigma,\,\mathbb{Z}\,) \longrightarrow \mathbb{C}\) arises as the period character of some pair \((X,\omega)\) in the stratum \(\strata\) if and only if the following conditions hold:
    \begin{itemize}
        \item[1.] 
        \(\textnormal{vol}(\,\chi\,)>0\), and 
        \item[2.] 
        if \(\textnormal{Im}(\,\chi\,)\,<\,\mathbb C\) is a lattice, then \[\textnormal{vol}(\,\chi\,)\,\ge\,\displaystyle\max_{1\le i\le k}\big\{\,m_i+1\,\big\}\,\textnormal{area}\big(\,\mathbb C/\Lambda\,\big).\]
    \end{itemize}
\end{thm}

\begin{thm}[Bainbridge--Johnson--Judge--Park]
    Let \(\Sigma\) be a closed surface of genus \(g\ge2\) and let \(\mu\) be any signature. If \(\chi\) can be realised in \(\strata\), then it can be realised in every connected component of the same stratum.
\end{thm}

\smallskip

\noindent For a non-discrete representation, there are no obstruction to its realisation in any stratum other than the one given by the volume. For discrete representations, however, the signature plays a fundamental role. We observe that for the principal stratum \(\mathcal{H}_g(1,\dots,1)\), the conditions of this refined version are identical to those of Haupt's theorem. In fact, for every signature \(\mu=(\,m_1,\dots,m_k\,)\),
\begin{equation}
 2 \le \max_{1\le i\le k}\big\{\,m_i+1\,\big\} 
\end{equation}
and the equality holds if and only if \(m_i=1\) for all \(i=1,\dots,k\). 

\smallskip

\subsubsection{Refined lattice\index{lattice!condition} condition} Let us see where this condition arises when one wishes to realise a discrete representation in a stratum other than the principal one. Next, we recall from \S\ref{ssec:haupt} that every discrete representation  yields a branched covering \(f\colon \Sigma\longrightarrow T\), up to homotopy, of degree at least \(2\). 
\noindent From the classical theory of branched coverings we recall that, since the local degree of \(f\) around a ramification point of order \(m\ge1\) cannot exceed the total degree, it must follows that \(\deg(\,f\,)\ge m+1\). In particular, the total degree must be at least equal to the greatest local degree around every ramification point. 

\smallskip

\noindent Let \((X,\omega)\in\mathcal H_g(\,\mu\,)\) and let \(\Delta=\{\,p_1,\dots,p_k\,\}\) be the set of zeros of \(\omega\) of orders \(\{m_1,\dots,m_k\}\) respectively. Let \(\chi\) be the period character of \((X,\omega)\) and let \(\Lambda\) be its image. We also assume \(\chi\) to be a discrete representation, meaning that \(\Lambda\) is a lattice in \(\mathbb C\). Let \(\textnormal{dev}\colon\widetilde X\longrightarrow \mathbb C\) be the developing map. Since \(\textnormal{dev}\) is \(\chi\)-equivariant, it descends to a branched covering map \(f\colon X\longrightarrow \mathbb C/\Lambda\) where the set of ramification points corresponds to the set of zeros of \(\omega\). The local degree around \(p_i\) is equal to \(m_i+1\) because\(f\) locally looks like the map \(g\) defined in \S\ref{sssec:bigcone}. It readily follows that
\begin{equation}\label{eq:refhaup2}
    \deg(\,\chi\,)=\deg(\,f\,) \ge \max_{1\le i\le k}\big\{\,m_i+1\,\big\},
\end{equation}
where \(m_1,\dots,m_k\) are the orders of the zeros of \(\omega\). Equations \eqref{eq:refhaup1} and \eqref{eq:refhaup2} combined together imply
\begin{equation}
    \textnormal{vol}(\,\chi\,)\,\ge\,\displaystyle\max_{1\le i\le k}\big\{\,m_i+1\,\big\}\,\textnormal{area}\big(\,\mathbb C/\Lambda\,\big).
\end{equation}

\smallskip

\noindent Recall that given a lattice \(\Lambda\) yields a fundamental parallelogram \(\textnormal{P}\), and \(\mathbb C/\Lambda\), as a translation surface, can be realised by gluing the opposite sides of \(\textnormal{P}\) by using translations. Once a signature \(\mu=(\,m_1,\dots,m_k\,)\) is fixed, then \(d=\max_{1\le i\le k}\big\{\,m_i+1\,\big\}\) defines an invariant, namely the least amount of copies of \(\textnormal{P}\) needed to realised a translation surface with discrete period character in the stratum \(\mathcal H_g(\,\mu\,)\), see Figure \ref{fig:degree}. For an example of realisation in genus \(3\), see Figures \ref{fig:discreterealisationexamplepartone} and \ref{fig:discreterealisationexampleparttwo}.

\begin{figure}[h!]
    \centering
    \begin{tikzpicture}[scale=1.35, every node/.style={scale=0.875}]
    \definecolor{pallido}{RGB}{221,227,227}

    \draw[black] (-2,0)--(2,0);
    \draw[blue] (-2, 1.5)--( 2, 1.5);
    \draw[blue] (-2, 1  )--(-1, 1  ) to[out=0, in=180] (1,2)--(2,2);
    \draw[blue] (-2, 2)--(-1, 2) to[out=0, in=180] (1,1)--(2,1);
    \draw[dashed] (0,1.5)--(0,0);
    \fill (0,1.5) circle (1pt);
    \node[red, scale=2] at (0,0) {\(\times\)};
    \node (A) at (1,2.5) {zero of order \(2\)};
    \draw[black, ->] (A) to[bend right] (0,1.65);
    \node at (0,-0.35) {A branched covering map of degree \(3\)};
    \node at (0,-0.7) {with a single ramification point of order two};

    \begin{scope}[shift={(-5,0)}]
        \draw[black] (-2,0)--(2,0);
        \draw[blue] (-2,1)--(-1.5,1) to[out=0, in=180] (-0.5,2)--(0.5,2) to[out=0, in=180] (1.5,1)--(2,1);
        \draw[blue] (-2,2)--(-1.5,2) to[out=0, in=180] (-0.5,1)--(0.5,1) to[out=0, in=180] (1.5,2)--(2,2);
        \fill (-1,1.5) circle (1pt);
        \fill ( 1,1.5) circle (1pt);
        \node[red, scale=2] at (-1,0) {\(\times\)};
        \node[red, scale=2] at ( 1,0) {\(\times\)};
        \draw[dashed] (-1,1.5)--(-1,0);
        \draw[dashed] ( 1,1.5)--( 1,0);
        \node (B) at (0,2.5) {zeros of order \(1\)};
        \draw[black, ->] (B) to[bend right] (-1,1.75);
        \draw[black, ->] (B) to[bend left ] ( 1,1.75);
        \node at (0,-0.35) {A branched covering map of degree \(2\)};
    \node at (0,-0.7) {with a two ramification points of order one};
    \end{scope}

    \begin{scope}[shift={(-2.5,-4.25)}]
        \draw[black] (-2,0)--(2,0);
        \draw[blue] (-2,1)--(-1.5,1) to[out=0, in=180] (-0.5,1.5)--(0.5,1.5) to[out=0, in=180] (1.5,2)--(2,2);
        \draw[blue] (-2,1.5)--(-1.5,1.5) to[out=0, in=180] (-0.5,1)--(2,1); 
        \draw[blue] (-2,2)--(0.5,2) to[out=0, in=180] (1.5,1.5)--(2,1.5);
        \fill (-1,1.25) circle (1pt);
        \fill ( 1,1.75) circle (1pt);
        \node[red, scale=2] at (-1,0) {\(\times\)};
        \node[red, scale=2] at ( 1,0) {\(\times\)};
        \draw[dashed] (-1,1.25)--(-1,0);
        \draw[dashed] ( 1,1.75)--( 1,0);
        \node (C) at (0.5,2.75) {zeros of order \(1\)};
        \draw[black, ->] (C) to[bend left ] ( 1,1.875  );
        \draw[black, ->] (C) to[bend right] (-1,1.5);
        \node at (0,-0.35) {A branched covering map of degree \(3\)};
        \node at (0,-0.7) {with a two ramification points of order one};
    \end{scope}
    \end{tikzpicture}
    \caption{Realisation of a discrete representation in genus \(2\). If the representation has degree \(2\) then it can be realised only in the principal stratum \(\mathcal H_2(1,1)\). If the representation has degree \(3\) (or higher), then it can be realised in both strata \(\mathcal H_2(1,1)\) and \(\mathcal H_2(\,2\,)\).}
    \label{fig:degree}
\end{figure}
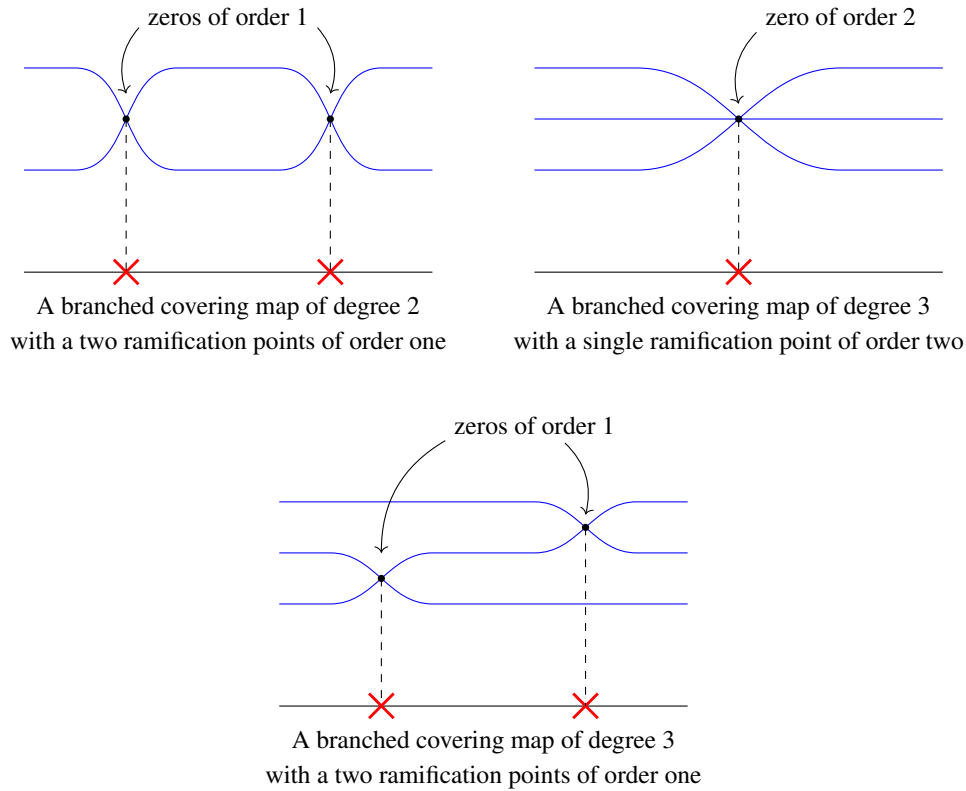

\begin{figure}[h!]
    \centering
    \begin{tikzpicture}[scale=1.75, every node/.style={scale=0.875}]
    \definecolor{pallido}{RGB}{221,227,227}

    \pattern [pattern=north west lines, pattern color=pallido] (0,0)--(2,0)--(2,2)--(0,2)--(0,0);
    
    \draw[orange        ] (0,0)--(2,0);
    \draw[sky           ] (2,0)--(2,2);
    \draw[orange        ] (2,2)--(0,2);
    \draw[red           ] (0,2)--(0,0);

    \draw[blue,   thick] (0.5,0.5)--(1.0,0.5);
    \draw[violet, thick] (1.0,1.5)--(1.5,1.5);

    \fill (0,0) circle (1.00pt);
    \fill (2,0) circle (1.00pt);
    \fill (2,2) circle (1.00pt);
    \fill (0,2) circle (1.00pt);

    \fill[blue  ] (0.5,0.5) circle (0.75pt);
    \fill[blue  ] (1.0,0.5) circle (0.75pt);
    \fill[violet] (1.0,1.5) circle (0.75pt);
    \fill[violet] (1.5,1.5) circle (0.75pt);

    \node[rotate=90] at (-0.125,1) {sheet one};

    \begin{scope}[shift={(4,0)}]
        \pattern [pattern=north west lines, pattern color=pallido]
        (0,0)--(2,0)--(2,2)--(0,2)--(0,0);
    
        \draw[orange        ] (0,0)--(2,0);
        \draw[red           ] (2,0)--(2,2);
        \draw[orange        ] (2,2)--(0,2);
        \draw[sky           ] (0,2)--(0,0);

        \draw[blue,   thick] (0.5,0.5)--(1.0,0.5);
        \draw[violet, thick] (1.0,1.5)--(1.5,1.5);

        \fill (0,0) circle (1.00pt);
        \fill (2,0) circle (1.00pt);
        \fill (2,2) circle (1.00pt);
        \fill (0,2) circle (1.00pt);

        \fill[blue  ] (0.5,0.5) circle (0.75pt);
        \fill[blue  ] (1.0,0.5) circle (0.75pt);
        \fill[violet] (1.0,1.5) circle (0.75pt);
        \fill[violet] (1.5,1.5) circle (0.75pt);

        \node[rotate=270] at (2.125,1) {sheet two};
    \end{scope}

    \draw[black, <->] ( 2.125,1)--(3.875,1);

    \node at (3,0.75) {glue these edges}; 
    \node at (3,0.50) {(coloured in cyan)}; 

    \draw[black, thin, ->] (0.75,-0.125) to[bend right] (1.75,-2);
    \draw[black, thin, ->] (5.25,-0.125) to[bend left ] (4.25,-2);

    \begin{scope}[shift={(-2,3)}]
            \node at (2.75 , -2.625) {\(e_{1,1}\)};
            \node at (3.25 , -1.675) {\(e_{2,1}\)};
    \end{scope}
    
    \begin{scope}[shift={( 2,3)}]
            \node at (2.75 , -2.625) {\(e_{1,2}\)};
            \node at (3.25 , -1.675) {\(e_{2,2}\)};
    \end{scope}

    \begin{scope}[shift={(2,-3)}]
        \pattern [pattern=north west lines, pattern color=pallido]
        (0,0)--(2,0)--(2,2)--(0,2)--(0,0);
    
        \draw[orange        ] (0,0)--(2,0);
        \draw[red           ] (2,0)--(2,2);
        \draw[orange        ] (2,2)--(0,2);
        \draw[red           ] (0,2)--(0,0);

        \draw[blue,   thick] (0.5,0.5)--(1.0,0.5);
        \draw[violet, thick] (1.0,1.5)--(1.5,1.5);

        \fill (0,0) circle (1.00pt);
        \fill (2,0) circle (1.00pt);
        \fill (2,2) circle (1.00pt);
        \fill (0,2) circle (1.00pt);

        \fill[blue  ] (0.5,0.5) circle (0.75pt);
        \fill[blue  ] (1.0,0.5) circle (0.75pt);
        \fill[violet] (1.0,1.5) circle (0.75pt);
        \fill[violet] (1.5,1.5) circle (0.75pt);
    \end{scope}

    \begin{scope}[shift={(1,3)}]
        \pattern [pattern=north west lines, pattern color=pallido] (0,0)--(2,0)--(2,2)--(0,2)--(0,0);
    
        \draw[orange        ] (0,0)--(2,0);
        \draw[sky           ] (2,0)--(2,2);
        \draw[orange        ] (2,2)--(0,2);
        \draw[red           ] (0,2)--(0,0);

        \draw[blue,   thick] (0.5,0.5)--(1.0,0.5);
        \draw[violet, thick] (1.0,1.5)--(1.5,1.5);

        \fill (0,0) circle (1.00pt);
        \fill (2,0) circle (1.00pt);
        \fill (2,2) circle (1.00pt);
        \fill (0,2) circle (1.00pt);

        \fill[blue  ] (0.5,0.5) circle (0.75pt);
        \fill[blue  ] (1.0,0.5) circle (0.75pt);
        \fill[violet] (1.0,1.5) circle (0.75pt);
        \fill[violet] (1.5,1.5) circle (0.75pt);
    \end{scope}
    \begin{scope}[shift={(3,3)}]
        \pattern [pattern=north west lines, pattern color=pallido]
        (0,0)--(2,0)--(2,2)--(0,2)--(0,0);
    
        \draw[orange        ] (0,0)--(2,0);
        \draw[red           ] (2,0)--(2,2);
        \draw[orange        ] (2,2)--(0,2);
        \draw[sky           ] (0,2)--(0,0);

        \draw[blue,   thick] (0.5,0.5)--(1.0,0.5);
        \draw[violet, thick] (1.0,1.5)--(1.5,1.5);

        \fill (0,0) circle (1.00pt);
        \fill (2,0) circle (1.00pt);
        \fill (2,2) circle (1.00pt);
        \fill (0,2) circle (1.00pt);

        \fill[blue  ] (0.5,0.5) circle (0.75pt);
        \fill[blue  ] (1.0,0.5) circle (0.75pt);
        \fill[violet] (1.0,1.5) circle (0.75pt);
        \fill[violet] (1.5,1.5) circle (0.75pt);
    \end{scope}
    \begin{scope}[shift={(-2,3)}]
            \node at (3.75 ,  0.375) {\(e_{1,1}\)};
            \node at (4.25 ,  1.325) {\(e_{2,1}\)};
    \end{scope}
    
    \begin{scope}[shift={( 1,6)}]
            \node at (2.75 , -2.625) {\(e_{1,2}\)};
            \node at (3.25 , -1.675) {\(e_{2,2}\)};
    \end{scope}

        \node at (1.125, -2.125) {Projection of};
        \node at (1.125, -2.275) {the first sheet};
        \node at (1.125, -2.450) {on \(\mathbb C/\Gamma\)};
        \node at (4.875, -2.125) {Projection of};
        \node at (4.875, -2.275) {the second sheet};
        \node at (4.875, -2.450) {on \(\mathbb C/\Gamma\)};
        \node at (3.00 , -0.875) {\(\mathbb C/\Gamma\)};
        \node at (2.75 , -2.625) {\(s_1\)};
        \node at (3.25 , -1.675) {\(s_2\)};

        \draw[black, thin, ->] (1.5,2.125) to[bend left ] (1.5,2.75);
        \draw[black, thin, ->] (4.5,2.125) to[bend right] (4.5,2.75);
    
    \end{tikzpicture}
    \caption{Realisation of a translation surface in \(\mathcal H_3(\,1,1,1,1\,)\) with discrete period character by using two unit squares. The picture shows how to glue two copies of the bottom square with two marked segments (coloured in blue and violet).}
    \label{fig:discreterealisationexamplepartone}
\end{figure}
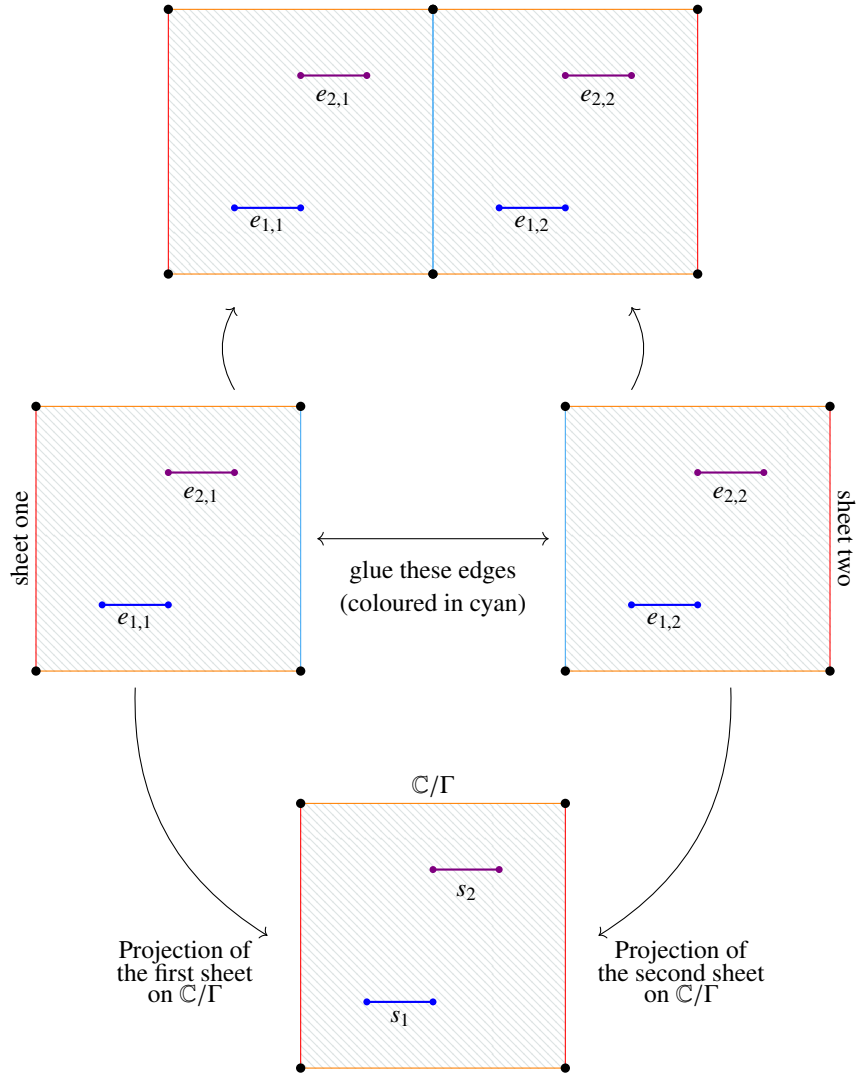

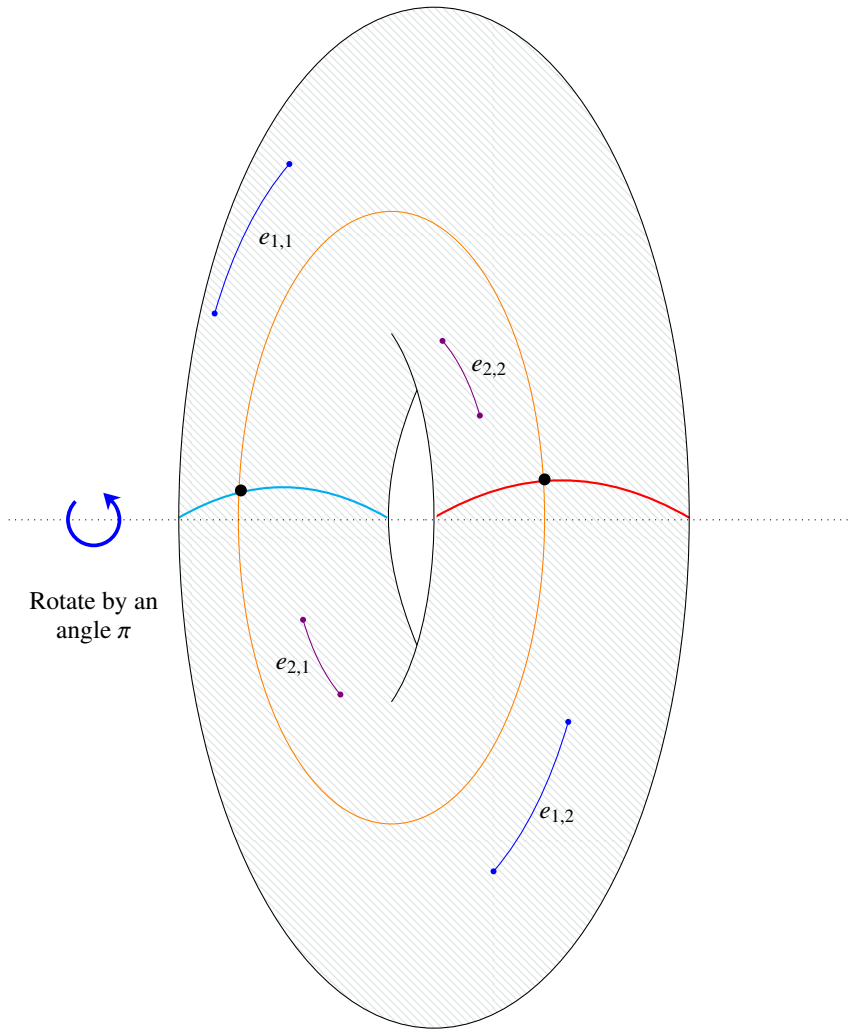
\begin{figure}[h!]
    \centering
    \begin{tikzpicture}[scale=2.25, every node/.style={scale=0.875}]
    \definecolor{pallido}{RGB}{221,227,227}

    \begin{scope}[rotate=90]
    \pattern [pattern=north west lines, pattern color=pallido] (0,0) ellipse (3cm and 1.5cm);
    
    \draw[orange ] (0,0.25 ) ellipse (1.8cm and 0.9cm);
    \draw[red, thick    ] (0,0    ) to[bend left ] (0,-1.5);
    \draw[cyan, thick   ] (0,0.275) to[bend right] (0, 1.5);
    
    \fill[white] (-0.75,0.1) to[out=022.25, in=157.25] (0.75,0.1) to[out=195.25, in=-15.25] (-0.75,0.1);
    \draw[black] (-0.75,0.1) to[out=022.25, in=157.25] (0.75,0.1);
    \draw[thin, black] (0,0) ellipse (3cm and 1.5cm) (0,0);
    \draw[thin, black] (0,0.5) +(210:1.25 and 0.5) arc [start angle=210, end angle=330, x radius=1.25, y radius=0.5];

    \draw[blue] (0,0.25) +(30:2.4 and 1.2) arc [start angle=30, end angle=60, x radius=2.4, y radius=1.2];
    \draw[blue] (0,0.25) +(210:2.4 and 1.2) arc [start angle=210, end angle=240, x radius=2.4, y radius=1.2];
    
    \draw[violet] (0,0.25) +(120:1.2 and 0.6) arc [start angle=120, end angle=150, x radius=1.2, y radius=0.6];
    \draw[violet] (0,0.25) +(300:1.2 and 0.6) arc [start angle=300, end angle=330, x radius=1.2, y radius=0.6];
    
    \fill[violet] (0,0.25) +(120:1.2 and 0.6) circle (0.5pt);
    \fill[violet] (0,0.25) +(150:1.2 and 0.6) circle (0.5pt);
    \fill[violet] (0,0.25) +(300:1.2 and 0.6) circle (0.5pt);
    \fill[violet] (0,0.25) +(330:1.2 and 0.6) circle (0.5pt);

    \fill[blue] (0,0.25) +(030:2.4 and 1.2) circle (0.5pt);
    \fill[blue] (0,0.25) +(060:2.4 and 1.2) circle (0.5pt);
    \fill[blue] (0,0.25) +(210:2.4 and 1.2) circle (0.5pt);
    \fill[blue] (0,0.25) +(240:2.4 and 1.2) circle (0.5pt);

    \node at ( 1.625 , 0.925 ) {\(e_{1,1}\)};
    \node at (-1.75  ,-0.725 ) {\(e_{1,2}\)};
    \node at (-0.875 , 0.825 ) {\(e_{2,1}\)};
    \node at ( 0.875 ,-0.325 ) {\(e_{2,2}\)};

    \fill ( 0.225 ,-0.65 ) circle (1pt);
    \fill ( 0.160 , 1.135 ) circle (1pt);
    \end{scope}

    \draw[thin, black, dotted] (-2.5,-0.0125)--(-0.05,-0.0125) ( 1.60,-0.0125)--(2.5,-0.0125);
    \draw[thin, black, dotted] (-0.05,-0.0125)--( 1.60,-0.0125);
    \node[scale=3, blue] at (-2,-0.0125) {\(\circlearrowleft\)};
    \node at (-2,-0.5  ) {Rotate by an };
    \node at (-2,-0.675) {angle \(\pi\)};

    \end{tikzpicture}
    \caption{Realisation of a translation surface in \(\mathcal H_3(\,1,1,1,1\,)\) with discrete period character. The picture shows the torus and the resulting surface of genus three. In this perspective, the torus and the resulting genus three surface (not depicted) are invariant with respect to a rotation of angle \(\pi\). This rotation swaps the two sheets of the covering.}
    \label{fig:discreterealisationexampleparttwo}
\end{figure}

\subsubsection{Realisation following --Johnson--Judge--Park's approach} In their paper \cite{BJJP}, the authors distinguished two cases, depending on whether the representation is discrete or non-discrete, and used a specific argument for each of them. In the first case, they opted for a direct approach, whereas in the second case they used an indirect one.

\smallskip

\paragraph{\textit{Non-discrete representations}} The case of non-discrete representations\index{representation!non-discrete} has been handled with an indirect approach. Based on the earlier work \cite{CDF2} of Calsamiglia–Deroin–Francaviglia, the authors in \cite{BJJP} show that the preimage under the period map in \eqref{eq:permap} of every non-discrete representation intersects every connected component of a stratum. This is an indirect approach because the representations are not realised explicitly, yet this approach is highly insightful and contributes significantly to our understanding of the period map. 

\smallskip

\paragraph{\textit{Discrete representations: normalisation and reduction}} Let us discuss the discrete\index{representation!discrete} case. In this case the approach adopted is direct, as translation surfaces have been realised explicitly, and it is divided into three main steps. The first one consists in a normalisation process to reduce the cases of study. We observe that it is enough to restrict the proof of the discrete case to square-tiled surfaces. Indeed, as already noted in \S\ref{ssec:kapapp}, there is no loss of generality in assuming that the image of the representation is the standard lattice \(\Gamma=\ziz\). 

\smallskip

\noindent A second reduction applies. We first recall that \(\Sigma\) is a surface of genus \(g\ge2\). Next, we note that if a representation \(\chi\colon\textnormal{H}_1(\Sigma,\,\mathbb Z)\longrightarrow\mathbb C\) is the period character of a translation structure, then so are all those of the form \(\chi\circ h_*\), where \(h\colon\Sigma\longrightarrow\Sigma\) is any homeomorphism and \(h_*\) is the induced map in homology. In the case where \(\textnormal{Im}(\,\chi\,)\) is a lattice, we have seen in \S\ref{ssec:haupt} that \(\chi\) yields a branched covering map \(f\colon\Sigma\longrightarrow T\), up to homotopy, and \(f\circ h\colon\Sigma\longrightarrow T\) is a branched covering map of the same degree. Consider the following maps
\smallskip
\begin{equation}
        \mathcal B=\left\{\,\,
        \begin{gathered}
        \textnormal{translation surfaces arising from}\\
        \textnormal{primitive branched covers } \\f\colon\Sigma\longrightarrow T=\mathbb C/\Gamma
        \end{gathered}
        \,\,\right\}\overset{\textnormal{Per}}{\longrightarrow}
        \mathcal{DR}=\left\{\,\,
        \begin{gathered}
        \textnormal{(discrete) representations}\\
        \chi\colon\textnormal{H}_1(\Sigma,\,\mathbb Z)\longrightarrow\Gamma<\mathbb C\\
        \textnormal{that satisfy Haupt's conditions}
        \end{gathered}
        \,\,\right\}\overset{\textnormal{deg}}{\longrightarrow} \mathbb Z
\end{equation}

\smallskip

\noindent We recall that by \textit{primitive} we mean a branched covering such that \(f_*\colon\pi_1(\,\Sigma\,)\longrightarrow \Gamma\) is surjective (and hence also the representation in homology \(\chi\colon\textnormal{H}_1(\Sigma,\,\mathbb Z)\longrightarrow\Gamma\) is surjective). 

\smallskip

\noindent The map \(\textnormal{Per}\colon\mathcal{B}\to\mathcal{DR}\) associates to each branched covering the period character of the pull-back structure on \(\Sigma\) and the map \(\deg(\,\cdot\,)\) associates to each representation its degree as defined in \eqref{eq:refhaup1}. The map \(\deg\) defines a stratification of \(\mathcal{DR}\) into strata \(\mathcal{DR}(\,d\,)\) of representations of degree \(d\). In turns, the mapping \(\textnormal{Per}\circ\deg\) defines a stratification of \(\mathcal B\) into strata \(\mathcal B(\,d\,)\) of branched covering of degree \(d\). We now recall that a \textit{simple branched covering} is a branched covering with ramification points of order at most one. Simple branched coverings over the torus exist in every degree \(d\geq 2\), and the pull-back of the standard differential \(dz\) on \(\mathbb{C}/\Gamma\) via any such covering of degree \(d\) defines a square-tiled surface on \(\Sigma\) in the principal stratum \(\mathcal{H}_g(\,1,\dots,1\,)\), with period character of degree \(d\). In other words,  for every \(d\ge2\) the stratum \(\mathcal H_g(\,1,\dots,1\,)\,\cap\,\mathcal B(\,d\,)\) is not empty.

\medskip

\noindent By \cite{GabKaz87}, given two primitive simple branched coverings, say \(p\colon \Sigma \longrightarrow \mathbb{C}/\Gamma\) and \(q\colon\Sigma \longrightarrow \mathbb{C}/\Gamma\) of the same degree \(d\ge2\), there exists a homeomorphism \(h\colon \Sigma \longrightarrow \Sigma\) and a homeomorphism \(k\colon \mathbb{C}/\Gamma \longrightarrow \mathbb{C}/\Gamma\) isotopic to the identity such that \(k \circ p = q \circ h\). Since every representation \(\chi\in\mathcal{DR}(\,d\,)\) arises as the period character of some translation surface in \(\mathcal H_g(\,1,\dots,1\,)\,\cap\,\mathcal B(\,d\,)\)\footnote{There is no circularity in the argument; indeed, this is guaranteed by Haupt's theorem by using the construction shown in Figure \ref{fig:kapdiscr}.}, then for every pair of discrete representations, say \(\chi_1,\,\chi_2\in\mathcal{DR}(\,d\,)\), there exists a homeomorphism \(h\colon\Sigma\longrightarrow\Sigma\) such that \(\chi_2=\chi_1\circ h_*\). 

\medskip

\noindent Let \(\mu=(\,m_1,\dots,m_k\,)\) be any signature and \textit{suppose} there is a representation \(\chi\in\mathcal{DR}(\,d\,)\) arising as the period character of some translation surface in \(\mathcal H_g(\,\mu\,)\,\cap\,\mathcal B(\,d\,)\). Then, there exists a branched covering map \(f\colon\Sigma\to\mathbb C/\Gamma\) of degree \(d\) with \(k\) ramification points of orders \(m_1,\dots,m_k\). Let \(h\colon\Sigma\longrightarrow \Sigma\) be any orientation-preserving homeomorphism, then the mapping \(f\circ h\colon\Sigma\longrightarrow \mathbb C/\Gamma\) is a branched covering such that \(\textnormal{Per}(\,f\circ h\,)=\chi\circ h_*\). As a consequence, every representation \(\chi\in\mathcal{DR}(\,d\,)\) arises as the period character of some translation surface in \(\mathcal H_g(\,\mu\,)\,\cap\,\mathcal B(\,d\,)\). At this point, it remains to show that at least one representation can be realised. This is what is shown in the next steps.

\smallskip

\paragraph{\textit{Discrete representations: minimal strata}} \noindent Let \(\chi\in\textnormal{H}^1(\Sigma,\mathbb C)\) be a discrete representation\index{representation!discrete}. The authors in \cite{BJJP} showed in the first place that \(\chi\) is realisable in the strata with signatures \((2g-2)\) and \((g-1,g-1)\), provided that \(\deg(\,\chi\,) \geq d\), where \(d = 2g-2\) or \(d = g-1\) respectively. 
If \(\textnormal{P}\) is a fundamental domain for \(\mathbb C/\Gamma\), \textit{e.g.} the unit square, the realisation amounts to pick a collection of \(\deg(\,\chi\,)\) copies of \(\textnormal{P}\) and then finding an appropriate pairing via translations, see Definition \ref{defn:firstdef}, to realise the desired translation surface. A pairing via translations is far from unique, and different ones may well produce translation surfaces in different connected components of the same stratum. Recall that, by the Kontsevich--Zorich classification of components of strata see \S\ref{ssec:conncomp}, both strata \(\mathcal H_g(\,2g-2\,)\) and \(\mathcal H_g(\,g-1,g-1\,)\) are non-connected and distinguished by topological invariants such as spin parity and hyperelliptic structure. The former has three connected components and the latter has two or three connected components depending on the parity of \(g\). The authors showed that a  discrete representation \(\chi\), provided that \(\deg(\,\chi\,) \geq d\), can be realised in every component of the strata with signature \((2g-2)\) or \((g-1,g-1)\) by determining explicit pairings via translations that yield a translation structure with prescribed zeros and topological invariants. The following pictures show how the same amount of squares (more generally parallelograms) can be used to realise homeomorphic surfaces with structures having a single zero of maximal order but different topological invariants. See Figures \ref{fig:hypermin}, \ref{fig:evenspin} and \ref{fig:oddspin}.

\begin{figure}[h!]
    \centering
    \begin{tikzpicture}[scale=1.575, every node/.style={scale=0.675}]
    \definecolor{pallido}{RGB}{221,227,227}

    \pattern [pattern=north west lines, pattern color=pallido]
    (0,0)--(0,2)--(1,2)--(1,3)--(2,3)--(2,4)--(3,4)--(3,5)--(5,5)--(5,4)--(4,4)--(4,3)--(3,3)--(3,2)--(2,2)--(2,1)--(1,1)--(1,0)--(0,0);
    
    \draw[thin, black        ] (0,0)--(0,2);
    \draw[thin, black        ] (1,0)--(1,1);
    \draw[thin, black        ] (2,3)--(2,4);
    \draw[thin, black        ] (4,3)--(4,4);
    \draw[thin, black        ] (3,2)--(3,3);
    \draw[thin, black        ] (0,0)--(1,0);
    \draw[thin, black        ] (0,2)--(1,2);
    \draw[thin, black, dashed] (1,2)--(2,2);
    \draw[thin, black        ] (1,1)--(2,1);
    \draw[thin, black        ] (2,4)--(3,4);
    \draw[thin, black        ] (3,3)--(4,3);
    \draw[thin, black, dashed] (1,1)--(1,2);
    \draw[thin, black, dashed] (0,1)--(1,1);
    \draw[thin, black        ] (2,2)--(3,2);
    \draw[thin, black, dashed] (2,3)--(3,3);
    \draw[thin, black        ] (2,1)--(2,2);
    \draw[thin, black, dashed] (3,3)--(3,4);
    \draw[thin, black        ] (5,4)--(5,5);
    \draw[thin, black        ] (3,5)--(5,5);
    \draw[thin, black        ] (3,5)--(3,4);
    \draw[thin, black        ] (5,4)--(4,4);
    \draw[thin, black, dashed] (4,4)--(4,5);
    \draw[thin, black, dashed] (3,4)--(4,4);
    \draw[thin, black        ] (1,3)--(2,3);
    \draw[thin, black        ] (1,2)--(1,3);
    \draw[thin, black, dashed] (2,2)--(2,3);

    \fill (0,0) circle (0.75pt);
    \fill (1,0) circle (0.75pt);
    \fill (0,1) circle (0.75pt);
    \fill (1,1) circle (0.75pt);
    \fill (2,4) circle (0.75pt);
    \fill (2,1) circle (0.75pt);
    \fill (2,2) circle (0.75pt);
    \fill (1,2) circle (0.75pt);
    \fill (2,3) circle (0.75pt);
    \fill (3,2) circle (0.75pt);
    \fill (3,3) circle (0.75pt);
    \fill (0,2) circle (0.75pt);
    \fill (1,3) circle (0.75pt);
    \fill (3,4) circle (0.75pt);
    \fill (4,4) circle (0.75pt);
    \fill (4,3) circle (0.75pt);
    \fill (4,5) circle (0.75pt);
    \fill (5,5) circle (0.75pt);
    \fill (5,4) circle (0.75pt);
    \fill (3,5) circle (0.75pt);

    \node at (0.5,0.5) {\(1\)};
    \node at (0.5,1.5) {\(2\)};
    \node at (1.5,1.5) {\(3\)};
    \node at (1.5,2.5) {\(4\)};
    \node at (2.5,2.5) {\(5\)};
    \node at (2.5,3.5) {\(6\)};
    \node at (3.5,3.5) {\(7\)};
    \node at (3.5,4.5) {\(8\)};
    \node at (4.5,4.5) {\(9\)};
        
    \end{tikzpicture}
    \caption{Nine squares are used to realise a translation surface in the stratum \(\mathcal{H}_5(\,10\,)\) with odd hyperelliptic structure. Label the lowest square by \(1\), the topmost by \(9\), and those in between sequentially. The pairing via translations, as in Definition~\ref{defn:firstdef}, is given by the permutations \(\sigma_h = (2\,3)(4\,5)(6\,7)(8\,9)\) and \(\sigma_v = (1\,2)(3\,4)(5\,6)(7\,8)\). We may notice that by rotating all squares by an angle \(\pi\), the labelling remains the same. As a consequence, the resulting structure admits a hyperelliptic involution.}
    \label{fig:hypermin}
\end{figure}
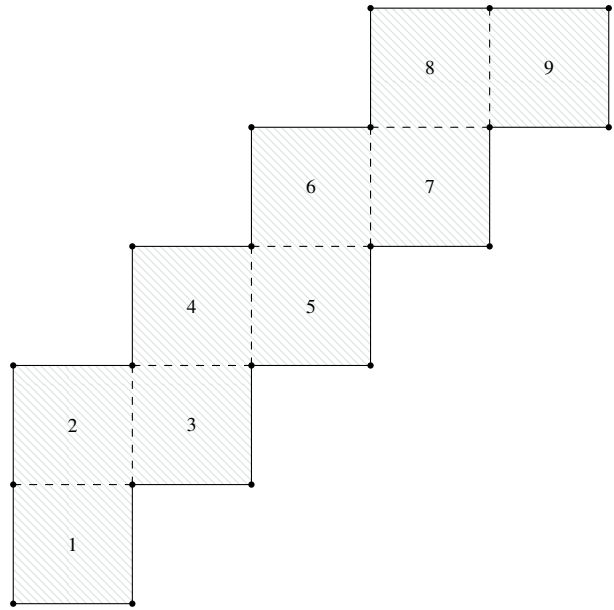

\begin{figure}[h!]
    \centering
    \begin{tikzpicture}[scale=1.575, every node/.style={scale=0.675}]
    \definecolor{pallido}{RGB}{221,227,227}

    \pattern [pattern=north west lines, pattern color=pallido]
    (0,0)--(0,2)--(2,2)--(2,4)--(4,4)--(4,3)--(3,3)--(3,1)--(1,1)--(1,0)--(0,0);

    \pattern [pattern=north west lines, pattern color=pallido](3,4)--(5,4)--(5,5)--(3,5)--(3,4);
    
    \draw[thin, blue]    ( 0    ,  0.5 )--( 1    , 0.5);
    \node at (0.25,0.5675) {$0$};
    
    \draw[thin, violet]  ( 0.75,  0     )--( 0.75, 1    );
    \draw[thin, violet]  ( 1.25,  1) arc [start angle = 0, end angle =180,radius = 0.25];
    \draw[thin, violet]  ( 1.25,  2) arc [start angle = 0, end angle =-180,radius = 0.25];
    \node at (1.225 ,1.25) {$1$};
    
    \draw[thin, blue]    ( 0    ,  1.5  )--( 3     , 1.5);
    \node at (1.875 ,1.4325) {$0$};
    \draw[thin, violet]  ( 1.5,  1     )--( 1.5, 2    );
    \node at (1.625 ,1.875) {$0$};
    
    \draw[thin, blue]    ( 2    ,  2.5 )--( 3    , 2.5);
    \node at (2.75 ,2.625) {$0$};

    \draw[thin, violet]  ( 2.25,  2    )--( 2.25, 4   );
    \draw[thin, violet]  ( 2.25,  2) arc [start angle = 0, end angle =-180,radius = 0.25];
    \draw[thin, violet]  ( 2.25,  1) arc [start angle = 0, end angle =180,radius = 0.25];
    \node at (2.375 ,2.7625) {$1$};
    
    \draw[blue]  ( 2,  3.1875) arc [start angle = 90, end angle =-90,radius = 0.1875];
    \draw[blue]  ( 3,  3.1875) arc [start angle = 90, end angle =270,radius = 0.1875];
    \draw[thin, blue]    ( 3    ,  3.1875 )--( 4    , 3.1875);
    \node at (3.175 ,3.125) {$1$};
    \draw[thin, violet]  ( 3.5,  3    )--( 3.5, 4   );
    \node at (3.425 ,3.5) {$1$};
    \draw[violet]  ( 3.5,  4 ) arc [start angle = 180, end angle =90,radius = 0.25];
    \draw[thin, violet] (3.75,4.25)--(4.25,4.25);
    \draw[violet]  ( 4.25,  4.25 ) arc [start angle = 90, end angle =0,radius = 0.25];
    \draw[violet]  ( 3.5,  5 ) arc [start angle = 180, end angle =270,radius = 0.25];
    \draw[thin, violet] (3.75,4.75)--(4.25,4.75);
    \draw[violet]  ( 4.5,  5 ) arc [start angle = 00, end angle =-90,radius = 0.25];

    \draw[thin, blue]    (3,4.5)--(5,4.5);
    \node at (4.5,4.375) {$0$};
    \draw[thin, violet]  ( 4.75,  4    )--( 4.75, 5   );\node at (4.875 ,4.75) {$0$};

    \draw[thin, black] (0,0)--(0,2);
    \draw[thin, black] (1,0)--(1,1);
    \draw[thin, black] (2,2)--(2,4);
    \draw[thin, black] (4,3)--(4,4);
    \draw[thin, black] (3,1)--(3,3);

    \draw[thin, black] ( 0,0)--(1,0);
    \draw[thin, black] ( 0,2)--(1,2);

    \draw[thin, black] (1,2)--(2,2);
    \draw[thin, black] (1,1)--(3,1);
    \draw[thin, black] (2,4)--(3,4);
    \draw[thin, black] (3,3)--(4,3);
    
    \draw[thin, dashed] (1,1)--(1,2);
    \draw[thin, dashed] (0,1)--(1,1);
    \draw[thin, dashed] (2,2)--(3,2);
    \draw[thin, dashed] (2,3)--(3,3);
    \draw[thin, dashed] (2,1)--(2,2);
    \draw[thin, dashed] (3,3)--(3,4);

    \draw[thin, black]         (3,4)--(3,5);
    \draw[thin, black]         (3,5)--(5,5);
    \draw[thin, black]         (5,5)--(5,4);
    \draw[thin, black]         (5,4)--(4,4);
    \draw[thin, black, dashed] (4,4)--(4,5);
    \draw[thin, black, dashed] (3,4)--(4,4);

    \fill (0,0) circle (0.75pt);
    \fill (1,0) circle (0.75pt);
    \fill (0,1) circle (0.75pt);
    \fill (1,1) circle (0.75pt);
    \fill (2,4) circle (0.75pt);
    \fill (2,1) circle (0.75pt);
    \fill (2,2) circle (0.75pt);
    \fill (1,2) circle (0.75pt);
    \fill (2,3) circle (0.75pt);
    \fill (3,2) circle (0.75pt);
    \fill (3,3) circle (0.75pt);
    \fill (0,2) circle (0.75pt);
    \fill (3,1) circle (0.75pt);
    \fill (3,4) circle (0.75pt);
    \fill (4,4) circle (0.75pt);
    \fill (4,3) circle (0.75pt);
    \fill (5,4) circle (0.75pt);
    \fill (5,5) circle (0.75pt);
    \fill (4,5) circle (0.75pt);
    \fill (3,5) circle (0.75pt);
        
    \end{tikzpicture}
    \caption{In this figure, nine squares are used to realise a translation surface in the stratum \(\mathcal{H}_5(\,10\,)\) with even spin invariant. As before, label the lowest square with \(1\), the topmost square with \(9\), and the ones in between sequentially as shown in Figure \ref{fig:hypermin}. Here, the labels are omitted to avoid confusion in the image. The pairing via translations, as in Definition~\ref{defn:firstdef}, is given by the permutations \(\sigma_h = (2\,3\,4)(6\,7)(8\,9)\) and \(\sigma_v = (1\,2)(4\,5\,6)(7\,8)\). The blue and purple lines, once the sides are glued, define a basis in homology. The numbers next to these lines indicate the index of the curve in the resulting surface, and by applying formula~\eqref{eq:spinparity}, one directly computes that it has even spin structure.}
    \label{fig:evenspin}
\end{figure}
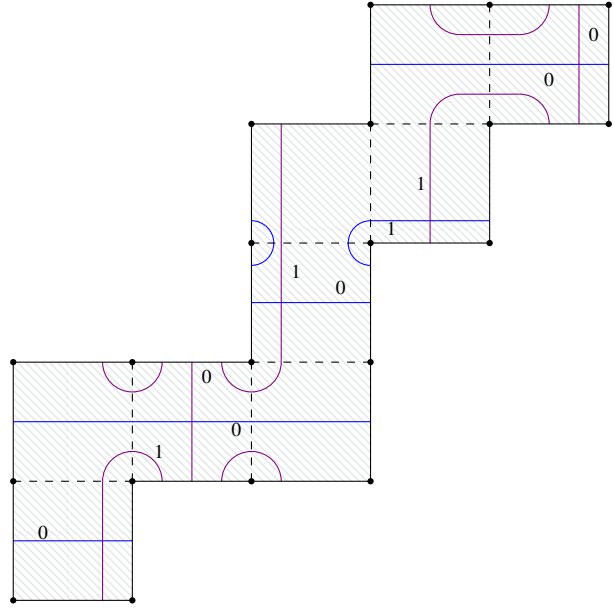

\begin{figure}[h!]
    \centering
    \begin{tikzpicture}[scale=1.575, every node/.style={scale=0.675}]
    \definecolor{pallido}{RGB}{221,227,227}

    \pattern [pattern=north west lines, pattern color=pallido]
    (0,0)--(0,2)--(2,2)--(2,4)--(4,4)--(4,3)--(3,3)--(3,1)--(1,1)--(1,0)--(0,0);

    \pattern [pattern=north west lines, pattern color=pallido](4,3)--(5,3)--(5,5)--(4,5)--(4,3);
    
    \draw[thin, blue]    ( 0    ,  0.5 )--( 1    , 0.5);
    \node at (0.25,0.5675) {$0$};
    
    \draw[thin, violet]  ( 0.75,  0     )--( 0.75, 1    );
    \draw[thin, violet]  ( 1.25,  1) arc [start angle = 0, end angle =180,radius = 0.25];
    \draw[thin, violet]  ( 1.25,  2) arc [start angle = 0, end angle =-180,radius = 0.25];
    \node at (1.225 ,1.25) {$1$};
    
    \draw[thin, blue]    ( 0    ,  1.5  )--( 3     , 1.5);
    \node at (1.875 ,1.4325) {$0$};
    \draw[thin, violet]  ( 1.5,  1     )--( 1.5, 2    );
    \node at (1.625 ,1.875) {$0$};
    
    \draw[thin, blue]    ( 2    ,  2.5 )--( 3    , 2.5);
    \node at (2.75 ,2.625) {$0$};

    \draw[thin, violet]  ( 2.25,  2    )--( 2.25, 4   );
    \draw[thin, violet]  ( 2.25,  2) arc [start angle = 0, end angle =-180,radius = 0.25];
    \draw[thin, violet]  ( 2.25,  1) arc [start angle = 0, end angle =180,radius = 0.25];
    \node at (2.375 ,2.7625) {$1$};
    
    \draw[blue]  ( 2,  3.1875) arc [start angle = 90, end angle =-90,radius = 0.1875];
    \draw[blue]  ( 3,  3.1875) arc [start angle = 90, end angle =270,radius = 0.1875];
    \draw[thin, blue]    ( 3    ,  3.1875 )--( 5    , 3.1875);
    \node at (3.175 ,3.125) {$1$};
    \draw[thin, violet]  ( 3.5,  3    )--( 3.5, 4   );
    \node at (3.425 ,3.5) {$0$};

    \draw[thin, violet]  ( 4.25,  4    )--( 4.25, 5   );
    \draw[violet]  ( 4.25,  4 ) arc [start angle = 0, end angle =-180,radius = 0.25];
    \draw[violet]  ( 4.25,  3 ) arc [start angle = 0, end angle =90,radius = 0.125];
    \draw[violet, thin] (3.875, 3.125)--(4.125, 3.125);
    \draw[violet]  ( 3.75,  3 ) arc [start angle =180, end angle =90,radius = 0.125];
    \node at (4.375 ,4.25) {$1$};
    \draw[thin, blue]    (4,4.5)--(5,4.5);
    \node at (4.5 ,4.625) {$0$};

    \draw[thin, black] (0,0)--(0,2);
    \draw[thin, black] (1,0)--(1,1);
    \draw[thin, black] (2,2)--(2,4);
    \draw[thin, black] (3,1)--(3,3);

    \draw[thin, black] ( 0,0)--(1,0);
    \draw[thin, black] ( 0,2)--(1,2);

    \draw[thin, black] (1,2)--(2,2);
    \draw[thin, black] (1,1)--(3,1);
    \draw[thin, black] (2,4)--(4,4);
    \draw[thin, black] (3,3)--(5,3);
    
    \draw[thin, dashed] (1,1)--(1,2);
    \draw[thin, dashed] (0,1)--(1,1);
    \draw[thin, dashed] (2,2)--(3,2);
    \draw[thin, dashed] (2,3)--(3,3);
    \draw[thin, dashed] (2,1)--(2,2);
    \draw[thin, dashed] (3,3)--(3,4);

    \draw[thin, black]         (5,3)--(5,4);
    \draw[thin, black]         (4,5)--(5,5);
    \draw[thin, black]         (5,5)--(5,4);
    \draw[thin, black, dashed] (5,4)--(4,4);
    \draw[thin, black]         (4,4)--(4,5);
    \draw[thin, black, dashed] (4,3)--(4,4);

    \fill (0,0) circle (0.75pt);
    \fill (1,0) circle (0.75pt);
    \fill (0,1) circle (0.75pt);
    \fill (1,1) circle (0.75pt);
    \fill (2,4) circle (0.75pt);
    \fill (2,1) circle (0.75pt);
    \fill (2,2) circle (0.75pt);
    \fill (1,2) circle (0.75pt);
    \fill (2,3) circle (0.75pt);
    \fill (3,2) circle (0.75pt);
    \fill (3,3) circle (0.75pt);
    \fill (0,2) circle (0.75pt);
    \fill (3,1) circle (0.75pt);
    \fill (3,4) circle (0.75pt);
    \fill (4,4) circle (0.75pt);
    \fill (4,3) circle (0.75pt);
    \fill (4,5) circle (0.75pt);
    \fill (5,5) circle (0.75pt);
    \fill (5,4) circle (0.75pt);
    \fill (5,3) circle (0.75pt);
        
    \end{tikzpicture}
    \caption{Realisation of a translation surface in the stratum \(\mathcal{H}_5(\,10\,)\) with odd spin parity. As before, the lowest square is labelled with \(1\), the topmost one with \(9\), and the ones in between sequentially as shown in Figure \ref{fig:hypermin}. The pairing via translations is now given by the permutations \(\sigma_h = (2\,3\,4)(6\,7\,8)\) and \(\sigma_v = (1\,2)(4\,5\,6)(8\,9)\). The blue and purple lines, once the sides are glued, define a basis of homology and the numbers next to these lines indicate the index of the curve in the resulting surface. By applying formula~\eqref{eq:spinparity}, one directly computes that it has odd spin structure.}
    \label{fig:oddspin}
\end{figure}
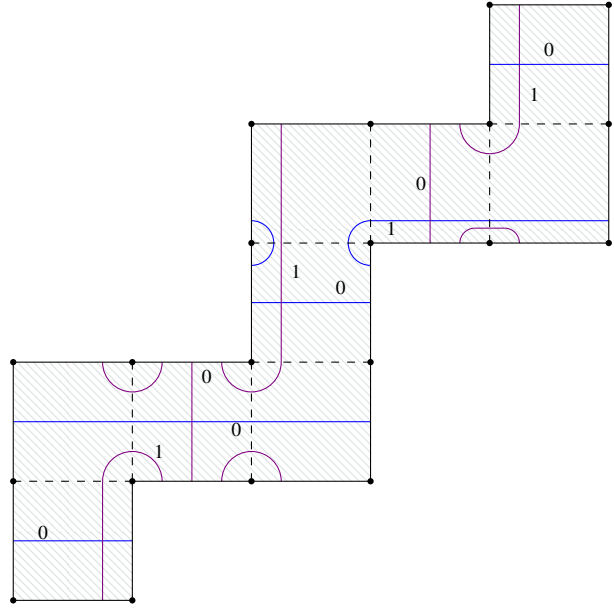

\begin{rmk}\label{rmk:increasingdeg}
    We can observe that in Figures \ref{fig:hypermin}, \ref{fig:evenspin} and \ref{fig:oddspin} the square labelled with \(9\) can be replaced with a rectangle of size \(1\times d\) for every \(d\in\mathbb Z^+\). 
\end{rmk}

\medskip

\paragraph{\textit{Discrete representations: extending to all strata}} In terms of branched coverings over the standard torus, that is \(\mathbb T=\mathbb C/\ziz\), the second step can be summarised as follows. Let \(\mu=(2g-2)\) or \(\mu=(g-1,g-1)\). For any connected component \(\mathcal K\) of a stratum \(\mathcal{H}_g(\,\mu\,)\) and for each integer \(d > 2g-2\) or \(d>g-1\) respectively, there exists a primitive branched covering \(p\colon X \to\mathbb T\) of Riemann surfaces of degree \(d\) such that \((X, p^{*}(dz))\) is a square-tiled surface in \(\mathcal K\). The final step is to extend the above claim to all signatures.

\begin{quote}
    \textit{Given a signature \(\mu=(m_1,\dots,m_k)\), for any connected component \(\mathcal K\) of \(\mathcal H_g(\,\mu\,)\) and for every \(d>\max\{\,m_i\,\}\), there exists a primitive branched covering map \(p\colon X \longrightarrow\mathbb T\) of Riemann surfaces of degree \(d\) such that \((X, p^{*}(dz))\) is a square-tiled surface in \(\mathcal K\).}
\end{quote}

\smallskip 

\noindent In principle, if we require the existence of a branched covering map for \textit{some} \(d>\max\{\,m_i\,\}\), then the result immediately follows from the first step by breaking zeros. In fact, by using the first step, a discrete representation \(\chi\) can be realised in the \(\mathcal H_g(\,2g-2\,)\,\cap\,\mathcal B(\,d\,)\) for every \(d>2g-2\). Since breaking the unique zero into zeros of lower orders does not alter the holonomy representations, see Remark \ref{rmk:breakingpreservesdegree}, for every signature \(\mu=(\,m_1,\dots,m_k\,)\) the representation \(\chi\) can be realised in \(\mathcal H_g(\,\mu\,)\). Therefore a different argument is needed to cover the realisation in strata \(\mathcal{DR}(\,d\,)\) when \(\max\{\,m_i\,\}< d\le 2g-2\).

\begin{rmk}[Breaking a zero does not change the degree]\label{rmk:breakingpreservesdegree}
    We recall, for the reader's convenience, that by breaking a zero, a branched covering map is deformed into another branched covering map of the same degree. This is essentially a consequence of Remark~\ref{rmk:breakazerointolowerzeros} above; see also Figure~\ref{fig:breakazerobranched}.
\end{rmk}

\begin{rmk}
    Recall that a stratum does not have any hyperelliptic component unless the signature is of the form \((2g-2)\) or \((g-1,g-1)\) and these cases have already been handled in the second step. Therefore this step applies to signatures \(\mu=(m_1,\dots,m_k)\) such that \(k\ge3\) or \(k=2\) and \(m_1>m_2\). Moreover, we recall that a stratum has two connected components distinguished by the spin parity if and only if all entries are even.
\end{rmk}

\smallskip

\noindent Let \(\mu=(m_1,\dots,m_k)\) be such that \(m_1\ge m_2 \ge\cdots\ge m_k>0\). The proof in this case is based on a purely algorithmic process. Before describing the algorithm and its efficiency, it is necessary to mention the tools used to implement it. These tools consist of two surgeries that modify a given branched covering map \(f\colon (X,\omega)\longrightarrow \mathbb{C}/\Gamma\) of degree \(d\) into a new branched covering \(g\colon (Y,\xi)\longrightarrow \mathbb{C}/\Gamma\) of the same degree, such that
\begin{itemize}
    \item[\(\star\)] \(Y\) has genus greater than \(X\), and 
    \smallskip
    \item[\(\star\)] if \(\omega\) has signature \((m_1,\dots,m_k)\), then \(\xi\) has signature \((m_1,\dots,m_k,m_{k+1})\) with \(m_{k+1}\) even, or with signature \((m_1,\dots,m_k,m_{k+1},m_{k+2})\) with both \(m_{k+1}\) and \(m_{k+2}\) odd,
\end{itemize}
\noindent where \((X,\omega)\in\mathcal H_g(\,\mu\,)\). Notice that \((X,\omega)\) has genus \(g\ge2\) since \(\mu\) is a partition of \(2g-2>0\) (all \(m_i\)'s are strictly positive).

\begin{rmk}
    Observe that \(d\ge2\) because a branched covering map of degree one is a homeomorphism and two surfaces are homeomorphic if and only if they have the same genus.
\end{rmk}

\noindent These surgeries consist of slitting \((X,\omega)\) along a finite collection of segments and then gluing them with some care to ensure that the resulting structure \((Y,\xi)\) has the desired signature and desired spin parity (if well defined). More specifically, pick any segment, say \(s \subset \mathbb{C}/\Gamma\), with no ramification value in its interior (ramification values are allowed at the endpoints of the segment if needed). Such a segment is a one-simplex, and it extends to a triangulation \(\tau\) of the torus \(\mathbb{C}/\Gamma\). By lifting \(\tau\) to a triangulation of \(X\), we obtain a collection \(\mathcal{P}\) of Euclidean triangles along with a pairing via translations. The gist of the idea is that the same collection of polygons yields different translation surfaces (with different signature and, possibly, parity if well defined) by changing the pairing via translations. Let \(\mathcal E=\big\{\,e_1,\dots,e_d\,\big\}\) be the set of lifts of \(s\), where \(d\) is the degree of \(f\). By fixing an orientation on \(s\), every edge \(e_i\) has a natural induced orientation. We next slit all the edges in \(\mathcal E\) and denote the by \(e_i^+\) and \(e_i^-\) the left and right side respectively. See Figure \ref{fig:slitsurface}.

\begin{figure}[h!]
    \centering
    \begin{tikzpicture}[scale=1, every node/.style={scale=1}]
    \definecolor{pallido}{RGB}{221,227,227}

    \draw[dashed] (0,0) circle (2.5cm);
    \pattern [pattern=north east lines, pattern color=pallido] (0,0) circle (2.5cm);
    \draw[orange]     (-1,1)--( 0, 0);
    \draw[orange, <-] ( 0,0)--( 1,-1);
    \fill[black] (-1, 1)  circle (1.50pt);
    \fill[black] ( 1,-1)  circle (1.50pt);

    \node at (0.25,0.25) {\(e_i\)};

    \begin{scope}[shift={(7,0)}]
        \draw[dashed] (0,0) circle (2.5cm);
        \pattern [pattern=north east lines, pattern color=pallido] (0,0) circle (2.5cm);
        \fill[white] (-1,1) to[out=325, in=125] (1,-1) to[out=145, in=305] (-1,1);
        \draw[orange] (-1,1) to[out=325, in=125] (1,-1);
        \draw[orange] (-1,1) to[out=305, in=145] (1,-1);
        \fill[black] (-1, 1)  circle (1.50pt);
        \fill[black] ( 1,-1)  circle (1.50pt);

        \node at ( 0.35, 0.35) {\(e_i^-\)};
        \node at (-0.35,-0.35) {\(e_i^+\)};
    \end{scope}

    \end{tikzpicture}
    \caption{Comparison of an edge \(e_i\in\mathcal E\) before and after slitting. The picture on the right does not show the orientation of the edges. However these agree with the orientation of the edge \(e_i\) on the left picture.}
    \label{fig:slitsurface}
\end{figure}
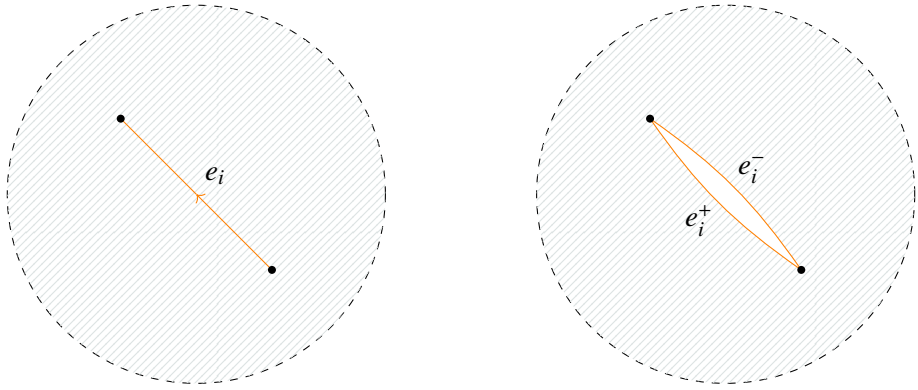


\smallskip

\noindent The way the left and right sides of the edges \(e_i\in\mathcal E\) are glued together determines the resulting surface. One choice, the obvious one, realises \((X,\omega)\) by construction. Other choices yield different topological surfaces. More precisely, every permutation \(\sigma\in\mathfrak{S}_d\) yields a pairing map \(\mathfrak{p}_{\sigma}\colon\mathcal E\longrightarrow \mathcal E\) that associates with the edge \(e_i\) the edge \(e_{\sigma(\,i\,)}\). We define \((X_{\sigma},\omega_{\sigma})\) the translation surface obtained by slitting the former translation surface \((X,\omega)\) along the edges in \(\mathcal E\) and glued according to the permutation \(\sigma\), that is \(e_i^+\) is glued to \(e_{\sigma(\,i\,)}^-\). Notice that \(\sigma=\textnormal{id}\) gives back the original surface \((X,\omega)\).

\smallskip

\noindent In \cite{BJJP}, the authors provide a precise recipe for choosing the segment \(s\) in \(\mathbb{C}/\Gamma\) to lift to \((X,\omega)\) and for choosing an appropriate permutation \(\sigma\in\mathfrak{S}_d\) to obtain the desired differential. More precisely, they provided a recipe to realise a differential with a single additional zero of even order, and a recipe to realise a differential with two additional zeros, both of odd order. By applying these surgeries iteratively in a suitable way, it is possible to realise any signature. Moreover, in the case the starting signature has only even entries, \textit{e.g.} \((2m_1,\dots,2m_k)\), a spin structure is well defined. When performing the surgery that introduces a zero of order \(2m_{k+1}\), the resulting structure also has its own parity. Parity is preserved by the surgery if and only if \(m_{k+1}\) is even and thence reversed if \(m_{k+1}\) is odd. 

\smallskip

\noindent By using these surgeries, the constructive approach adopted by the Bainbridge--Johnson--Jude--Park provides a quasi-algorithmic method to realise a discrete representation within every stratum, systematically handling parity and singularity types through a recursive sequence of surgeries and ad-hoc constructions. Given a signature \(\mu=(m_1,\dots,m_k)\) as above, in the case \(m_i\in2\mathbb Z\) for every \(i=1,\dots,k\), an additional parameter \(\theta\) is introduced for the algorithm. This is defined as 
\begin{equation}  
    \theta=  \mathfrak{spin}+\delta, \text{ where } \delta=\sum_{i=1}^k\frac{m_i}{2} \,\,\,\text{ mod } 2
\end{equation} 
\noindent where \(\mathfrak{spin}\) denotes the desired spin parity and \(\delta\) will count how many times the parity changes when zeros of even order are added. In their work \cite[\S3.4]{BJJP}, the authors provide an algorithm as follows. Let \(\chi\) be a discrete representation such that \(\deg(\,\chi\,)\ge m_1+1\). Then:
\begin{itemize}
    \item[1.] If each \(m_i\) is an even integer, and if
    \begin{itemize}
    \item[1.] \(k = 1\), then \(m_k=2g-2\): realise a square-tiled surface in \(\mathcal H_g(2g-2)\) with prescribed spin structure or hyperelliptic structure . Else
    \smallskip
    \item[2.] \(k = 2\) and \(m_1 = m_2=g-1\): realise a square-tiled surface in \(\mathcal H_g(g-1,g-1)\) with prescribed spin structure or hyperelliptic structure. Else
    \smallskip
    \item[3.] If
        \begin{itemize}
            \item[i.] \(\theta = 0 \mod 2\): realise a square-tiled surface in the component \(\mathcal K\) of \(\mathcal H_g(\,m_1\,)\) with even parity, where \(2g-2=m_1\) and then to add zeros of order \(m_2, m_3, \ldots, m_{n}\). Else
            \smallskip
            \item[ii.]  \(\theta = 1 \mod 2\): realise a square-tiled surface in the component \(\mathcal K\) of \(\mathcal H_g(\,m_1\,)\) with odd parity, where \(2g-2=m_1\) and then to add zeros of order \(m_2, m_3, \ldots, m_{n}\).
            \end{itemize}
            \smallskip
        \end{itemize}
    \item[2.] otherwise, \textit{i.e.} when some \(m_i\) is odd, if
        \begin{itemize}
            \item[1.] \(m_1\) is even: realise a square-tiled surface in the component \(\mathcal K\) of \(\mathcal H_g(\,m_1\,)\) with odd parity and apply both surgeries to the torus cover to add zeros of order \(m_2, m_2, \ldots, m_{n}\).
            \smallskip
            \item[2.] \(m_1\) is odd: then some other zero, say \(m_j\), is odd -- notice that odd zeros come in pairs. An ad-hoc construction is used to realise a structure in \(\mathcal H_g(\,m_1,m_j\,)\) and both surgeries are applied to add zeros of order \(m_i\) for \(i \neq 1, j\).
        \end{itemize}
\end{itemize}

\smallskip

\subsubsection{Realisation following Le Fils' approach} This second approach is markedly different from the first one. In his work \cite{fils}, Le Fils relies on the study carried out by Kapovich in \cite{KM2} and hence on the existence of a good basis as defined in \S\ref{ssec:kapapp}. Even in this case it is necessary to distinguish two cases depending on whether the representation \(\chi\) is discrete or non-discrete.

\smallskip

\paragraph{\textit{A surgery to realise handles}} In \S\ref{ssec:kapapp}, we provided a description of how to construct translation surfaces by gluing handles together along slits. Two handles (or surfaces in general) are cut along segments of the same length such that (in local charts) they differ by a translation, in order to produce a surface of higher genus, \textit{e.g.} see Figure \ref{fig:genkap}. Moreover, this is an iterative process. There is yet another way to attach a handle, more precisely, to glue in a cylinder. Suppose we have a translation surface \((X,\omega)\) and a parallelogram \(\textnormal{P} = \{\, x\mu_1 + y\mu_2 \mid x,y \in [0,1]\, \}\subset \mathbb{C}.\) Suppose further that on \((X,\omega)\) there is a segment \(s\) such that, in local coordinates, it differs from \(\mu_1\) by a translation.  

\noindent Cut \(X\) along \(s\). Topologically, the resulting surface has the same genus as \(X\) but a boundary homeomorphic to \(S^1\). Geometrically, the resulting surface has a piecewise geodesic boundary: more precisely, the boundary consists of two segments of equal length separated by two points around which the angle is a multiple of \(2\pi\). The parallelogram \(\textnormal{P}\) is then glued in a natural way, as shown in Figure \ref{fig:slitcons}, and the resulting surface has one additional handle. Geometrically, we obtain a translation surface whose differential has a zero of even order equal to the sum of the orders of the endpoints of \(s\).

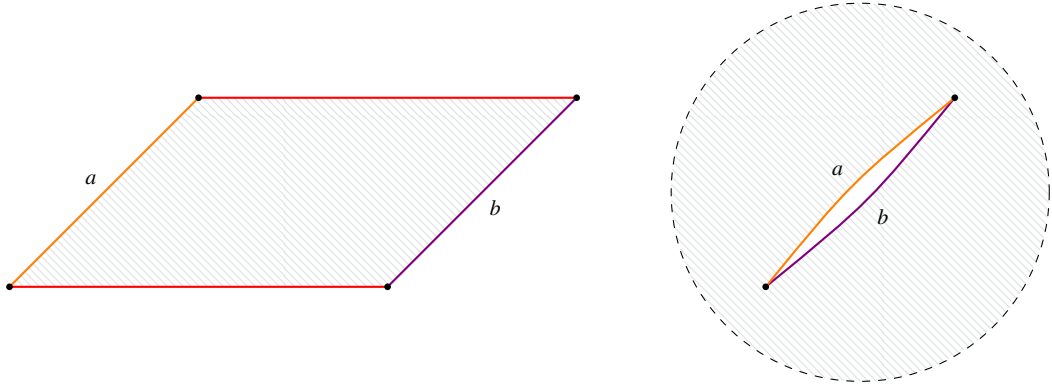
\begin{figure}[!h] \label{fig:seqslithandle}
\centering
\begin{tikzpicture}[scale=1.25, every node/.style={scale=0.75}]
\definecolor{pallido}{RGB}{221,227,227}
\pattern[pattern=north west lines, pattern color=pallido] (-3,1) -- (-5, -1) -- (-9,-1) -- (-7, 1) -- (-3,1);
\node[above left] at (-8,0) {\(a\)};
\node[below right] at (-4,0) {\(b\)};

\draw[thick, violet] (-3,  1)-- (-5,-1);
\draw[thick, orange] (-9,-1) -- (-7, 1);
\draw[thick, red   ] (-5,-1) -- (-9,-1) (-3, 1) -- (-7, 1);

\fill (-3, 1) circle (1pt);
\fill (-5,-1) circle (1pt);
\fill (-9,-1) circle (1pt);
\fill (-7, 1) circle (1pt);

\draw[dashed, black, pattern=north west lines, pattern color=pallido] (0,0) circle (2);
\fill[white] (-1,-1) .. controls (0.1, -0.1) .. (1,1) .. controls (-0.1, 0.1) .. (-1,-1);

\draw[thick, violet] (-1,-1) .. controls (0.1, -0.1) .. (1,1);
\draw[thick, orange] (-1,-1) .. controls (-0.1, 0.1) .. (1,1);
\node[above left] at (-0.1, 0.1) {\(a\)};
\node[below right] at (0.1, -0.1) {\(b\)}; 

\fill (-1,-1) circle (1pt);
\fill ( 1, 1) circle (1pt);

\end{tikzpicture}
\caption{Slit construction with one handle. Edges with the same colour are glued together. }
\label{fig:slitcons}
\end{figure}

\smallskip

\paragraph{\textit{Non-discrete representations}} Let \(\chi\) be a non-discrete representation. In \S\ref{ssec:kapapp} we have mentioned the existence of a basis, say \(\{\alpha_i,\,\beta_i\}\), in homology, such that \(\Im\Big(\,\,\overline{\chi(\,\alpha_i\,)}\,\chi(\,\beta_i\,)\,\,\Big)>0\) for every \(i=1,\dots,g\). In \cite[Proposition 2.5]{fils}, Le Fils showed that such a basis can be chosen so that the following \textit{additional} conditions hold:
\begin{enumerate}
    \item[(A1)] the parallelogram formed by the vectors \(\chi(\,\alpha_1\,)\) and \(\chi(\,\beta_1\,)\) contains in its interior a translate of a rectangle \(\mathcal R = \mathcal I\,\times\,\mathcal J \subset \mathbb{C}\) such that
    \begin{equation}
        \lambda(\,\mathcal I\,) \geq 2g \cdot \max_{2\le i\le g} \big|\,\Re\big(\,\chi(\,\alpha_i\,)\,\big)\,\big|\,\,\text{ and }\,\,\lambda(\,\mathcal J\,) \geq 2g \cdot \max_{2\le i\le g} \big|\,\Im\big(\,\chi(\,\alpha_i\,)\,\big)\,\big|,
    \end{equation}
    where \(\mathcal I,\,\mathcal J\) are open intervals in \(\mathbb R\) and \(\lambda\) denotes the Lebesgue measure on \(\mathbb R\).
    \smallskip
    \item[(A2)] \(\arg(\chi(\,\alpha_i\,)) \neq \arg(\chi(\,\alpha_j\,))\) for every \(2 \leq i < j \leq g\) (see Figure \ref{fig:indfilseven} to understand why this condition is useful).
\end{enumerate}

\smallskip

\begin{rmk}[How non-discreteness plays a role]
    The fact that \(\chi\) is not a discrete representation is a fundamental assumption, as the first property just listed above is not guaranteed for discrete representations. Showing the existence of such a good basis is highly non-trivial, as the argument in~\cite[Proposition 2.5]{fils} relies on a deep result established by Calsamiglia-Deroin-Francaviglia~\cite[Theorem 1.5]{CDF}\footnote{In his paper, Le Fils cites \cite[Theorem 8.1]{CDF2}. This numbering, however, refers to the arXiv version (V2) of \cite{CDF2}, which was the latest version online at the time of Le Fils' publication. Shortly after his paper was published, a new version (V3) of \cite{CDF2} was released, (in which the theorem is renumbered as Theorem 1.5).}. As we shall see briefly, the gist of the idea is to pick a "home" parallelogram and glue the additional \(g-1\) handles onto it. The process in itself is similar to that developed by Kapovich in \S\ref{ssec:kapapp}, but the handles are glued in a different way and the manner in which the additional handles are now glued matters. Without delving too deeply into technical details, condition (A1) ensures the existence of a good basis (in Kapovich's sense for non-discrete representations) such that the absolute periods of the \(\alpha_i\)'s, for \(i=2,\dots,g\), are small enough to largely fit inside the home parallelogram on which we would like to glue additional handles. To better understand this crucial point, we invite the reader to look at Figure~\ref{fig:genkap} once again. We may fix the left-most square as the home parallelogram and the blue slits can be taken sufficiently short so that the remaining \(g-1\) handles can be easily glued. In Le Fils' approach, handles are glued as in Figure~\ref{fig:slitcons} in a way similar, but not identical, to that of Figure~\ref{fig:genkap}. To make sure all handles can be glued, we need to ensure that the absolute periods of the \(\alpha_i\)'s are sufficiently short (by choosing a good set of generators). For discrete representations, we do not have this flexibility. In fact, if \(\Gamma\) is a lattice, then the evaluation \(|\,\cdot\,|\colon\Gamma\longrightarrow\mathbb R\) on \(\Gamma\setminus\{\,0\,\}\) has a minimum (e.g. think of the case \(\Gamma=\mathbb Z[\,i\,]\)). Notice that such a minimum does not exist for non-discrete representations because there is always a sequence converging to zero.
\end{rmk}


\smallskip

\noindent Let \(\{\alpha_i, \beta_i\}_{1 \le i \le g}\) be a good basis for \(\textnormal{H}_1(\Sigma,\,\Z)\) that satisfies the additional conditions (1) and (2) above. As in \S\ref{ssec:kapapp}, each pair \(\{\,\chi(\alpha_i),\, \chi(\beta_i)\,\}\) determines a fundamental parallelogram \(\pol_i\) of some translation structure on a torus, by Proposition~\ref{prop:hauptgenone}. Let \(\mathcal P = \{\pol_1, \dots, \pol_g\}\) be the collection of parallelograms thus defined. It follows from \S\ref{ssec:kapapp} that these parallelograms can be glued along slits to realise a structure in the principal stratum \(\mathcal H_g(1, \dots, 1)\); this is Haupt's theorem. We now summarise the process provided by Le Fils in \cite{fils} to glue these parallelograms to realise higher-order zeros. The approach is rather algorithmic in nature. In fact, the gist of the idea is to first realise structures in strata with even signature, and then extend the construction to arbitrary strata by observing that zeros of odd order necessarily come in pairs.

\smallskip

\noindent A translation surface in the stratum with signature \((2m_1,2m_2,\dots,2m_k)\) can be realised through an iterative process in \(k\) steps, starting from an arbitrary parallelogram in the collection \(\mathcal{P}\). By gluing the opposite sides of this parallelogram, one obtains a translation structure on a torus (this is step 0), and our aim is to glue the remaining \(g-1\) parallelograms in an appropriate way. Let \(\mathcal{P}_1\) denote the collection of remaining polygons. At the \(i\)-th step, \(m_i\) parallelograms are used, and we denote by \(\mathcal{P}_{i+1}\) 
the collection of the remaining polygons. Notice that the chain of sets  
\(
\mathcal{P} \supset \mathcal{P}_1 \supset \mathcal{P}_2 \supset \cdots \supset 
\mathcal{P}_k \supset \mathcal{P}_{k+1} = \emptyset
\)
is descending and terminates after exactly \(k\) steps. At each step, the resulting surface has higher genus than the previous one and supports a translation structure associated with an abelian differential that has one additional zero. 

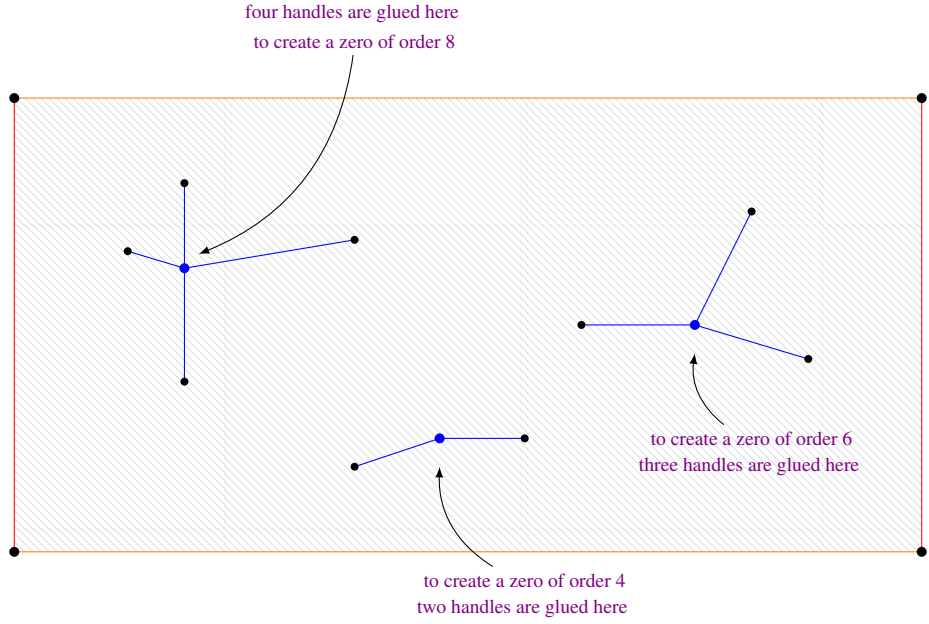
\begin{figure}[h!]
    \centering
    \begin{tikzpicture}[scale=1.5, every node/.style={scale=0.675}]
    \definecolor{pallido}{RGB}{221,227,227}

    \pattern [pattern=north west lines, pattern color=pallido]
    (0,0)--(8,0)--(8,4)--(0,4)--(0,0);
    
    \draw[thin, orange        ] (0,0)--(8,0);
    \draw[thin, red           ] (8,0)--(8,4);
    \draw[thin, orange        ] (8,4)--(0,4);
    \draw[thin, red           ] (0,4)--(0,0);

    \draw[thin, blue] (1.50,2.50)--(1.50,3.25);
    \draw[thin, blue] (1.50,2.50)--(1.50,1.50);
    \draw[thin, blue] (1.50,2.50)--(1.00,2.65);
    \draw[thin, blue] (1.50,2.50)--(3.00,2.75);
    \draw[thin, blue] (6.00,2.00)--(7.00,1.70);
    \draw[thin, blue] (6.00,2.00)--(5.00,2.00);
    \draw[thin, blue] (6.00,2.00)--(6.50,3.00);
    \draw[thin, blue] (3.75,1.00)--(4.50,1.00);
    \draw[thin, blue] (3.75,1.00)--(3.00,0.75);

    \fill[blue ] (1.50,2.50) circle (1.25pt);
    \fill[blue ] (3.75,1.00) circle (1.25pt);
    \fill[blue ] (6.00,2.00) circle (1.25pt);
    
    \fill[black] (0,0)       circle (1.25pt);
    \fill[black] (8,0)       circle (1.25pt);
    \fill[black] (8,4)       circle (1.25pt);
    \fill[black] (0,4)       circle (1.25pt);
    \fill[black] (1.50,3.25) circle (1.00pt);
    \fill[black] (1.50,1.50) circle (1.00pt);
    \fill[black] (1.00,2.65) circle (1.00pt);
    \fill[black] (3.00,2.75) circle (1.00pt);
    \fill[black] (4.50,1.00) circle (1.00pt);
    \fill[black] (3.00,0.75) circle (1.00pt);
    \fill[black] (7.00,1.70) circle (1.00pt);
    \fill[black] (5.00,2.00) circle (1.00pt);
    \fill[black] (6.50,3.00) circle (1.00pt);

    \node     at (3.00, 4.75) {\textcolor{violet}{four handles are glued here }};
    \node (x) at (3.00, 4.500) {\textcolor{violet}{to create a zero of order \(8\)}};
    \draw[-latex, black, bend left] (x) to ( 1.625, 2.625);

    \node     at (6.50, 0.75) {\textcolor{violet}{three handles are glued here }};
    \node (x) at (6.50, 1.000) {\textcolor{violet}{to create a zero of order \(6\)}};
    \draw[-latex, black, bend left] (x) to ( 6.00, 1.75);

    \node     at (4.50, -0.50) {\textcolor{violet}{two handles are glued here }};
    \node (x) at (4.50, -0.25) {\textcolor{violet}{to create a zero of order \(4\)}};
    \draw[-latex, black, bend left] (x) to ( 3.75, 0.75);
    
    \end{tikzpicture}
    \caption{Realising a translation surface in the stratum \(\mathcal H_{10}(\,4,6,8\,)\) with non-discrete period character. Condition (A2) ensures that we can realised \(m_i-\)pods of segments along which we glue \(m_i\) parallelograms as in Figure \ref{fig:slitcons}.}
    \label{fig:indfilseven}
\end{figure}

\noindent We now proceed to the construction, see Figure \ref{fig:indfilseven} for a reference. At the \(i\)-th step, choose a regular point, say \(p_i\), of the surface obtained at the \((i-1)\)-th step, and select \(m_i\) parallelograms, say \(\textnormal{P}_{i\,1}, \dots, \textnormal{P}_{i\,m_i}\) from the collection \(\mathcal{P}_i\). Then one chooses \(m_i\) segments \(s_{i\,1}, \dots, s_{i\,m_i}\), each with endpoints at two regular points, one of which is \(p_i\), and each one parallel to one side of the corresponding parallelogram \(\textnormal{P}_{i\,j}\). Condition (A2) ensures that no pair of segments overlap. The construction described above is then applied \(m_i\) times. The resulting surface has \(m_i\) more handles than the previous one and supports a translation structure in the stratum \(\mathcal{H}_{g_i}(2m_1,\dots,2m_i)\), where \( g_i = m_1 + \cdots + m_i + 1.\) Since \(2m_1 + \cdots + 2m_k = 2g-2\), the process terminates precisely after \(k\) steps, with \(\mathcal{P}_{k+1} = \emptyset\) and conditions \((1)\) and \((2)\) ensure that there is enough room on the first parallelogram to perform these surgeries, see \cite[\S3.1]{fils}. 

\smallskip

\begin{figure}[h!]
    \centering
    \begin{tikzpicture}[scale=1.5, every node/.style={scale=0.675}]
    \definecolor{pallido}{RGB}{221,227,227}

    \pattern [pattern=north west lines, pattern color=pallido]
    (0,0)--(8,0)--(8,4)--(0,4)--(0,0);
    
    \draw[thin, orange        ] (0,0)--(8,0);
    \draw[thin, red           ] (8,0)--(8,4);
    \draw[thin, orange        ] (8,4)--(0,4);
    \draw[thin, red           ] (0,4)--(0,0);

    \draw[thin, violet ] (4.50,1.00)--(5.00,2.00);
    \draw[thin, violet ] (4.00,3.00)--(5.00,3.00);
    \draw[thin, violet ] (1.50,1.50)--(0.50,0.50);
    
    \draw[thin, blue] (1.50,2.50)--(1.50,3.25);
    \draw[thin, blue] (1.50,2.50)--(1.50,1.50);
    \draw[thin, blue] (1.50,2.50)--(1.00,2.65);
    \draw[thin, blue] (1.50,2.50)--(3.00,2.75);
    \draw[thin, blue] (6.00,2.00)--(7.00,1.70);
    \draw[thin, blue] (6.00,2.00)--(5.00,2.00);
    \draw[thin, blue] (6.00,2.00)--(6.50,3.00);
    \draw[thin, blue] (3.75,1.00)--(4.50,1.00);
    \draw[thin, blue] (3.75,1.00)--(3.00,0.75);

    \fill[blue ] (1.50,2.50) circle (1.25pt);
    \fill[blue ] (3.75,1.00) circle (1.25pt);
    \fill[blue ] (6.00,2.00) circle (1.25pt);
    
    \fill[black] (0,0)       circle (1.25pt);
    \fill[black] (8,0)       circle (1.25pt);
    \fill[black] (8,4)       circle (1.25pt);
    \fill[black] (0,4)       circle (1.25pt);
    \fill[black] (1.50,3.25) circle (1.00pt);
    \fill[black] (1.50,1.50) circle (1.00pt);
    \fill[black] (1.00,2.65) circle (1.00pt);
    \fill[black] (3.00,2.75) circle (1.00pt);
    \fill[black] (4.50,1.00) circle (1.00pt);
    \fill[black] (3.00,0.75) circle (1.00pt);
    \fill[black] (7.00,1.70) circle (1.00pt);
    \fill[black] (5.00,2.00) circle (1.00pt);
    \fill[black] (6.50,3.00) circle (1.00pt);
    \fill[black] (4.00,3.00) circle (1.00pt);
    \fill[black] (5.00,3.00) circle (1.00pt);
    \fill[black] (0.50,0.50) circle (1.00pt);


    \pattern [pattern=north west lines, pattern color=pallido]
    (4,-4)--(6,-3)--(7,-1)--(5,-2)--(4,-4);

    \pattern [pattern=north west lines, pattern color=pallido]
    (4,5)--(7,5)--(7,7)--(4,7)--(4,5);

    \draw[thin, orange] (4,5)--(7,5) (7,7)--(4,7);
    \draw[thin, red   ] (7,5)--(7,7) (4,7)--(4,5);

    \fill[black] (4,5) circle (1.00pt);
    \fill[black] (7,5) circle (1.00pt);
    \fill[black] (7,7) circle (1.00pt);
    \fill[black] (4,7) circle (1.00pt);
    
    \draw[thin, violet ] (5   ,6)--(6   ,6);

    \fill[black] (5   ,6) circle (1.00pt);
    \fill[black] (6   ,6) circle (1.00pt);

    \draw[thin, orange] (4,-4)--(6,-3) (7,-1)--(5,-2);
    \draw[thin, red   ] (6,-3)--(7,-1) (5,-2)--(4,-4);
    \fill[black] (4,-4) circle (1.00pt);
    \fill[black] (6,-3) circle (1.00pt);
    \fill[black] (7,-1) circle (1.00pt);
    \fill[black] (5,-2) circle (1.00pt);
    \draw[thin, violet ] (5.25,-3)--(5.75,-2);
    \fill[black] (5.25,-3) circle (1.00pt);
    \fill[black] (5.75,-2) circle (1.00pt);

    \pattern [pattern=north west lines, pattern color=pallido]
    (1,-3)--(3,-3)--(3,-1)--(1,-1)--(1,-3);
    
    \draw[thin, orange] (1,-3)--(3,-3) (3,-1)--(1,-1);
    \draw[thin, red   ] (3,-3)--(3,-1) (1,-1)--(1,-3);

    \fill[black] (1,-3) circle (1.00pt);
    \fill[black] (3,-3) circle (1.00pt);
    \fill[black] (3,-1) circle (1.00pt);
    \fill[black] (1,-1) circle (1.00pt);
    
    \draw[thin, violet ] (2.5,-1.5)--(1.5,-2.5);

    \fill[black] (2.5,-1.5) circle (1.00pt);
    \fill[black] (1.5,-2.5) circle (1.00pt);

    



    \node     at (6.675, 4.50) {\textcolor{red}{slit and glue these edges}};
    \node     at (6.675, 4.25) {\textcolor{red}{to create two zeros of order one}};
    \draw[-latex, black, bend left] (5.50, 5.75) to (4.5, 3.25);

    \node     at (7, -0.25) {\textcolor{red}{slit and glue these edges to}};
    \node     at (7, -0.50) {\textcolor{red}{create two odd-order zeros }};
    \node     at (7, -0.75) {\textcolor{red}{greater than one}};
    \draw[-latex, black] (5.5,-2.25) .. controls (5.75, 0.) .. (4.875, 1.50);

    \node     at (0.5,-0.25) {\textcolor{red}{slit and glue these edges}};
    \node     at (0.5,-0.50) {\textcolor{red}{to create two odd-order zeros }};
    \node     at (0.25,-0.75) {\textcolor{red}{the order of one of which is greater than one}};
    \draw[-latex, black, bend right] (1.75,-2.00) to (1.00, 0.875);
    
    \end{tikzpicture}
    \caption{Realising a translation surface in the stratum \(\mathcal H_{13}(\,1,1,1,5,7,9\,)\) with non-discrete period character.}
    \label{fig:indfilsodd}
\end{figure}
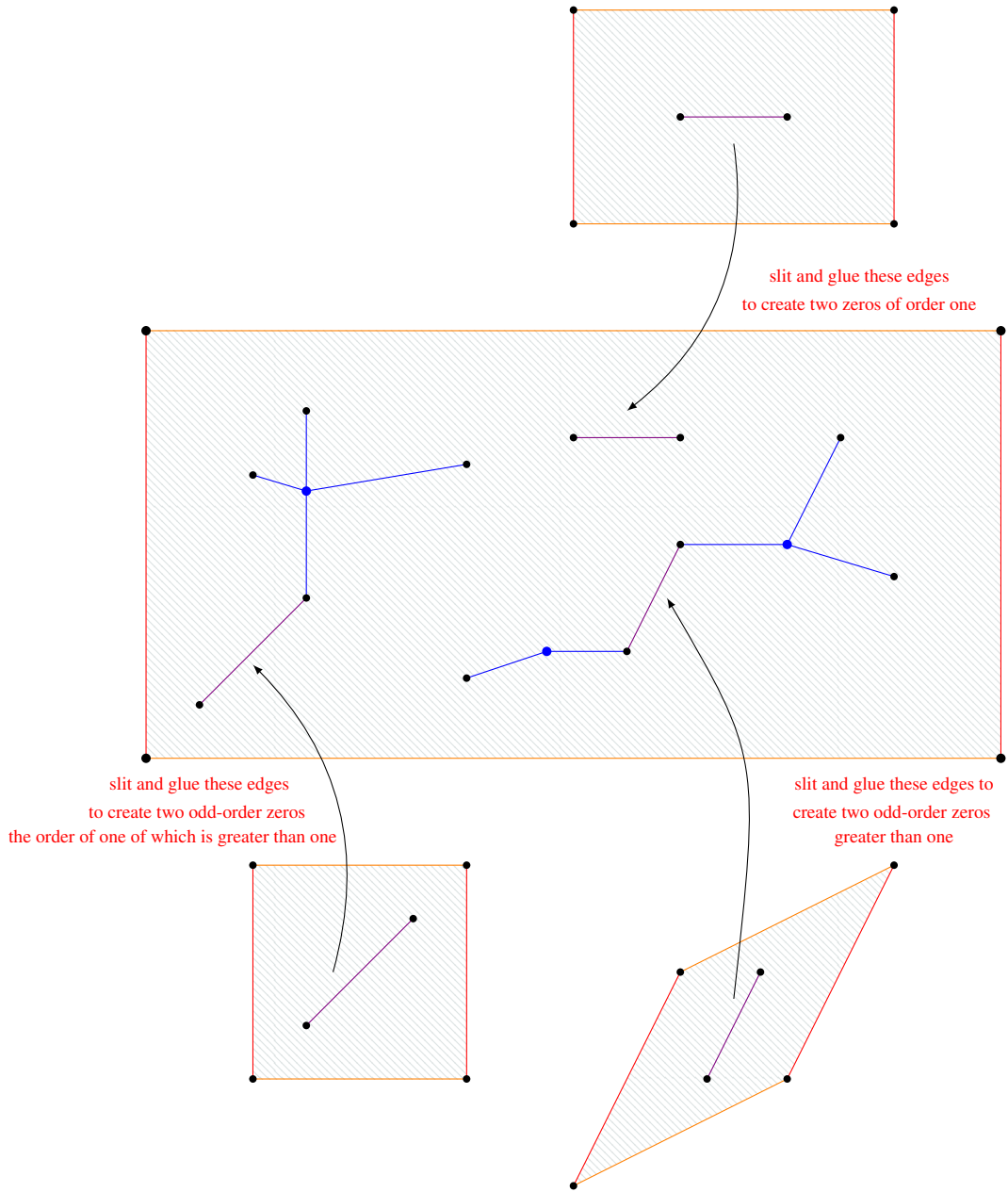

\noindent A generic signature of the form \( \mu = (2m_1,\dots,2m_l, 2m_{l+1}+1,\dots,2m_k+1)\) determines a signature with all even entries \( \mu' = (2m_1,\dots,2m_l, 2m_{l+1},\dots,2m_k). \) Note that \(k \equiv l \pmod{2}\). One first realises a translation structure in the stratum \(\mathcal H_{g'}(\,\mu'\,)\) as described above. In this case, the collection \(\mathcal{P}_{k+1}\) is no longer empty, but it still contains \(\tfrac{k-l}{2}\) parallelograms. These are then used to realise the odd-order zeros. The zeros of order \(2m_{l+1}+1,\dots,2m_k+1\) are grouped into pairs, and for each pair there exists a segment joining the two zeros. A parallelogram from the collection \(\mathcal{P}_{k+1}\) is then glue along a segment parallel to the one joining the two zeros in order to produce an additional handle. After \(\tfrac{k-l}{2}\) such steps, the resulting structure belongs to the stratum \(\mathcal H_g(\,\mu\,)\). See Figure \ref{fig:indfilsodd} for a reference.

\medskip

\paragraph{\textit{Discrete representations}} A discrete representation can be realised via an iterative process on the length of the signature. To apply this procedure, however, a good basis as guaranteed by Kapovich in \cite{KM2} is not sufficient. It is in fact necessary first to determine a homology basis with suitable properties whose existence has been established in \cite[Proposition 2.2]{fils}. More specifically, for every discrete representation \(\chi\) there exists a basis in homology such that
\begin{itemize}
  \item[\(\star\)] \( \chi(\,\alpha_1\,)=\textnormal{vol}(\,\chi\,) \) is a positive integer and \( \chi(\,\beta_1\,) = i \).
  \item[\(\star\)] \(\chi(\,\alpha_i\,)=d \), where \(d\) could be either one or two, for all \( i \ge 2 \) and \(\chi(\,\beta_i\,)=0\).
\end{itemize}

\smallskip

\noindent With this choice of basis we are in the right position to apply the iterative process mentioned above. Let \(\mu=(m_1,\dots,m_k)\) be a signature such that \(m_i \le m_j\) for \(i<j\) and \(k\ge2\). Indeed, the realisation in the minimal stratum is not covered by the process we are about to describe and requires an \emph{ad hoc} argument, which we shall present shortly.

\smallskip

\begin{rmk}
    Along the way, Figures \ref{fig:disfilsone}, \ref{fig:disfilstwo} and \ref{fig:disfilsthree} depict the various steps of the iterative process. For typographical reasons, the imaginary axis has been rescaled by some appropriate factor \(\lambda>1\) to make the pictures easier to consult. Figure \ref{fig:disfilsmin}, instead, has been drawn with the right scale. 
\end{rmk} 

\smallskip

\noindent The process is as follows: Let \(\mathcal R\) be the rectangle of size \(\textnormal{vol}(\,\chi\,)\times1\), then:
\begin{itemize}
    \item[1.] Basic step: suppose \(k=2\), then
    \begin{itemize}
        \smallskip
        \item[1.] if \(\mu=(\,g-1,\,g-1\,)\), then consider \(g\) segments \(s_1,\dots,s_g\) so that two consecutive edges are at distance one from each other. Slit and re-glue them as shown in Figure \ref{fig:disfilsone} to realise the desired structure.
        \smallskip
        \item[2.] if \(\mu=(\,m_1,\,m_2\,)\), then take \(m_2-m_1\) segments as in the step \((1.1)\) and then make a slit of length one joining the rightmost two ends of the last slits, as shown in Figure \ref{fig:disfilstwo}, and some other slits obtained as translations by multiples of~2 of this one, which are glued in a cyclic way to adjust the last conical angle.
    \end{itemize}
    \smallskip
    \item[2.] Inductive step: suppose \(k>2\). In this case \(\mu=(m_1,m_2,\dots,m_k)\). Then 
    \begin{itemize}
        \item[1.] slit \(\mathcal R\) along \(m_1+1\) segments as in the basic step \((1.1)\) to realise a structure in the stratum with signature \((m_1,m_1)\),
        \smallskip
        \item[2.] then make \(m_2 - m_1\) other parallel slits below, such that one of them touches one end of the rightmost preceding slit, which we glue in a cyclic way. In doing so, we obtain a flat surface with conical angles \(2\pi(m_1 + 1)\), \(2\pi(m_2 + 1)\), and \(2\pi(m_2 - m_1 + 1)\). Iterating this process, we get a flat surface with angles \(2\pi(m_1 + 1), \dots, 2\pi(m_{k-1} + 1)\), and \(2\pi(l + 1)\), with \(l \le m_k\). If \(l=m_k\) we are done, see Figure \ref{fig:disfilsthree}. Otherwise
        \smallskip
        \item[3.] apply the step \((1.2)\) to adjust the last conical angle, as shown in Figure \ref{fig:disfilstwo}.
    \end{itemize}
\end{itemize}

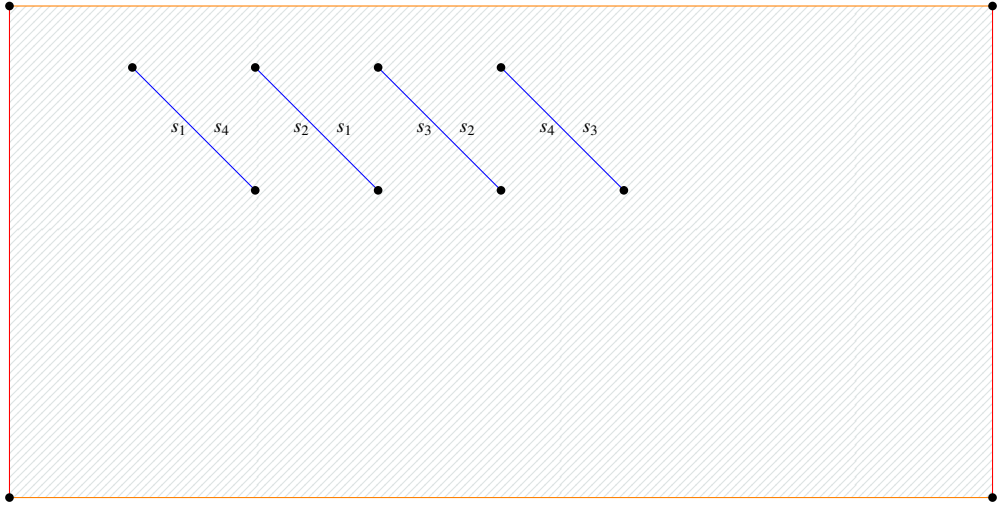
\begin{figure}[h!]
    \centering
    \begin{tikzpicture}[scale=1.625, every node/.style={scale=0.675}]
    \definecolor{pallido}{RGB}{221,227,227}

    \pattern [pattern=north east lines, pattern color=pallido]
    (0,0)--(8,0)--(8,4)--(0,4)--(0,0);
    
    \draw[thin, orange        ] (0,0)--(8,0);
    \draw[thin, red           ] (8,0)--(8,4);
    \draw[thin, orange        ] (8,4)--(0,4);
    \draw[thin, red           ] (0,4)--(0,0);

    \fill[black] (0,0)       circle (1.00pt);
    \fill[black] (8,0)       circle (1.00pt);
    \fill[black] (8,4)       circle (1.00pt);
    \fill[black] (0,4)       circle (1.00pt);

    \foreach \x in {1, ..., 4} 
    {
    \draw[thin, blue] (\x+1,2.5)--(\x,3.5);
    \fill[black] (\x+1,2.5)  circle (1.00pt);
    \fill[black] (\x  ,3.5)  circle (1.00pt);
    }

    \node at (1.375,3) {\(s_1\)};
    \node at (2.375,3) {\(s_2\)};
    \node at (3.375,3) {\(s_3\)};
    \node at (4.375,3) {\(s_4\)};
    \node at (1.725,3) {\(s_4\)};
    \node at (2.725,3) {\(s_1\)};
    \node at (3.725,3) {\(s_2\)};
    \node at (4.725,3) {\(s_3\)};

    \end{tikzpicture}
    \caption{Realising a discrete representation in the stratum \(\mathcal H_4(\,3,\,3\,)\). Slit blue segments and then re-glue the resulting edges with the same label.}
    \label{fig:disfilsone}
\end{figure}

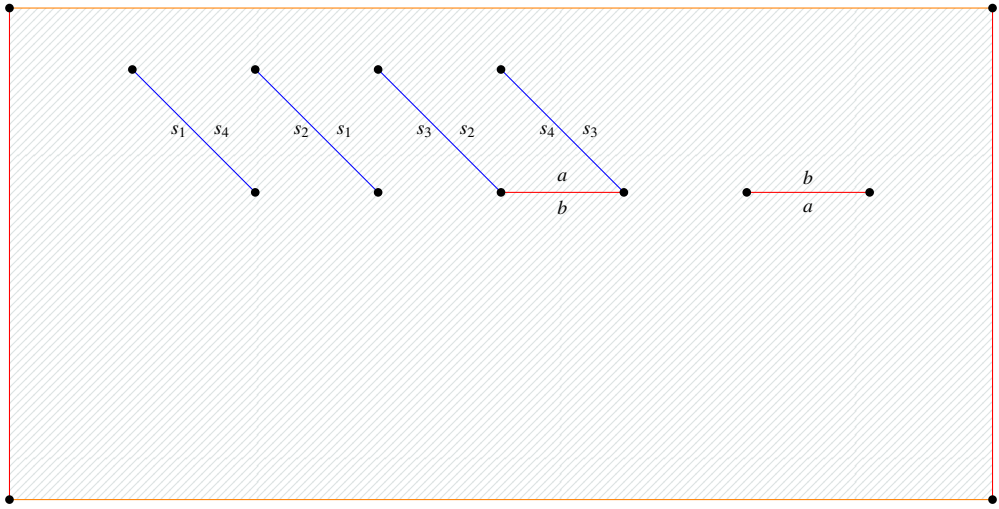
\begin{figure}[h!]
    \centering
    \begin{tikzpicture}[scale=1.625, every node/.style={scale=0.675}]
    \definecolor{pallido}{RGB}{221,227,227}

    \pattern [pattern=north east lines, pattern color=pallido]
    (0,0)--(8,0)--(8,4)--(0,4)--(0,0);
    
    \draw[thin, orange        ] (0,0)--(8,0);
    \draw[thin, red           ] (8,0)--(8,4);
    \draw[thin, orange        ] (8,4)--(0,4);
    \draw[thin, red           ] (0,4)--(0,0);

    \fill[black] (0,0)       circle (1.00pt);
    \fill[black] (8,0)       circle (1.00pt);
    \fill[black] (8,4)       circle (1.00pt);
    \fill[black] (0,4)       circle (1.00pt);

    \draw[red] (4,2.5)--(5,2.5);
    \draw[red] (6,2.5)--(7,2.5);

    \fill[black] (6,2.5) circle (1.00pt);
    \fill[black] (7,2.5) circle (1.00pt);
    
    \foreach \x in {1, ..., 4} 
    {
    \draw[blue ] (\x+1,2.5)--(\x,3.5);
    \fill[black] (\x+1,2.5)  circle (1.00pt);
    \fill[black] (\x  ,3.5)  circle (1.00pt);
    }

    \node at (1.375,3) {\(s_1\)};
    \node at (2.375,3) {\(s_2\)};
    \node at (3.375,3) {\(s_3\)};
    \node at (4.375,3) {\(s_4\)};
    \node at (1.725,3) {\(s_4\)};
    \node at (2.725,3) {\(s_1\)};
    \node at (3.725,3) {\(s_2\)};
    \node at (4.725,3) {\(s_3\)};

    \node at (4.5,2.625) {\(a\)};
    \node at (6.5,2.375) {\(a\)};
    \node at (4.5,2.375) {\(b\)};
    \node at (6.5,2.625) {\(b\)};

    \end{tikzpicture}
    \caption{Realising a discrete representation in the stratum \(\mathcal H_5(\,3,\,5\,)\). Slit both blue and red segments and then re-glue the resulting edges with the same label.}
    \label{fig:disfilstwo}
\end{figure}

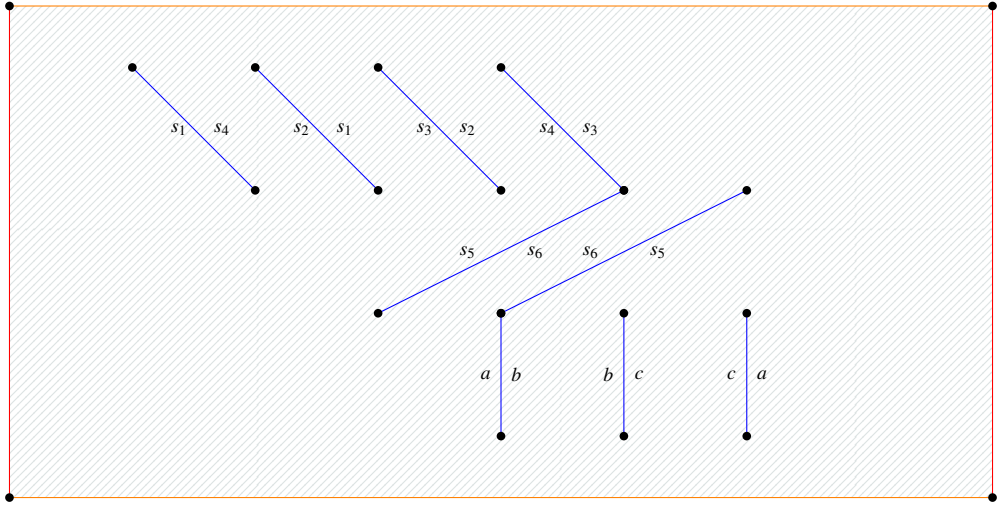
\begin{figure}[h!]
    \centering
    \begin{tikzpicture}[scale=1.625, every node/.style={scale=0.675}]
    \definecolor{pallido}{RGB}{221,227,227}

    \pattern [pattern=north east lines, pattern color=pallido]
    (0,0)--(8,0)--(8,4)--(0,4)--(0,0);
    
    \draw[thin, orange        ] (0,0)--(8,0);
    \draw[thin, red           ] (8,0)--(8,4);
    \draw[thin, orange        ] (8,4)--(0,4);
    \draw[thin, red           ] (0,4)--(0,0);

    \fill[black] (0,0)       circle (1.00pt);
    \fill[black] (8,0)       circle (1.00pt);
    \fill[black] (8,4)       circle (1.00pt);
    \fill[black] (0,4)       circle (1.00pt);

    \foreach \x in {1, ..., 4} 
    {
    \draw[thin, blue] (\x+1,2.5)--(\x,3.5);
    \fill[black] (\x+1,2.5)  circle (1.00pt);
    \fill[black] (\x  ,3.5)  circle (1.00pt);
    }

    \foreach \x in {5,6} 
    {
    \draw[thin, blue] (\x,2.5)--(\x-2,1.5);
    \fill[black] (\x  ,2.5)  circle (1.00pt);
    \fill[black] (\x-2,1.5)  circle (1.00pt);
    }

    \foreach \x in {4,...,6} 
    {
    \draw[thin, blue] (\x,1.5)--(\x,0.5);
    \fill[black] (\x,1.5)  circle (1.00pt);
    \fill[black] (\x,0.5)  circle (1.00pt);
    }

    \node at (1.375,3) {\(s_1\)};
    \node at (2.375,3) {\(s_2\)};
    \node at (3.375,3) {\(s_3\)};
    \node at (4.375,3) {\(s_4\)};
    \node at (1.725,3) {\(s_4\)};
    \node at (2.725,3) {\(s_1\)};
    \node at (3.725,3) {\(s_2\)};
    \node at (4.725,3) {\(s_3\)};

    \node at (4.725,2) {\(s_6\)};
    \node at (4.275,2) {\(s_6\)};
    \node at (3.725,2) {\(s_5\)};
    \node at (5.275,2) {\(s_5\)};

    \node at (3.875,1) {\(a\)};
    \node at (4.875,1) {\(b\)};
    \node at (5.875,1) {\(c\)};
    \node at (6.125,1) {\(a\)};
    \node at (4.125,1) {\(b\)};
    \node at (5.125,1) {\(c\)};

    \end{tikzpicture}
    \caption{Realising a discrete representation in the stratum \(\mathcal H_8(\,2,\,3,\,3,\,4\,)\). Slit blue segments and then re-glue the resulting edges with the same label.}
    \label{fig:disfilsthree}
\end{figure}

\noindent As already alluded to above, the minimal strata is not covered by the above algorithm. A different construction is needed and it is shown in the figure below. We begin, as before, by considering the parallelogram \(\textnormal{P}\) associated with \(\chi(\,\alpha_1\,)\) and \(\chi(\,\beta_1\,)\) in \(\mathbb{C}\), which provides complex charts on the corresponding torus. We first make a horizontal slit of length one in \(\textnormal{P}\). Additional slits are then drawn along segments that are translates of the first one by \(n = 2, 3, 5, \dots, 2g-3\), see Figure \ref{fig:disfilsmin}. Slits with the same label are identified. Since \(\textnormal{vol}(\,\chi\,)\ge2g-1\), there is always enough room to perform this construction.

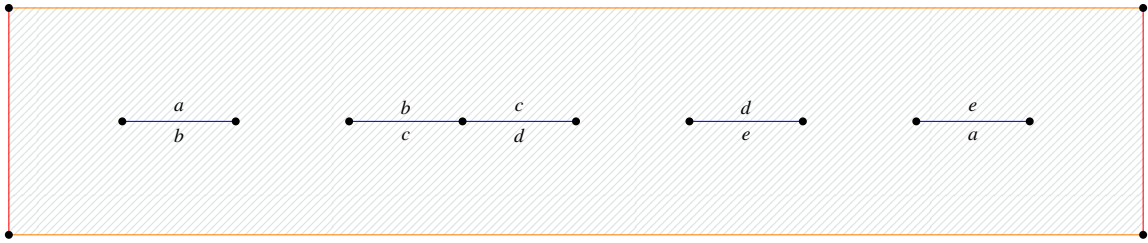
\begin{figure}[h!]
    \centering
    \begin{tikzpicture}[scale=1.5, every node/.style={scale=0.675}]
    \definecolor{pallido}{RGB}{221,227,227}

    \pattern [pattern=north east lines, pattern color=pallido]
    (0,0)--(10,0)--(10,2)--(0,2)--(0,0);
    
    \draw[thin, orange        ] (0, 0)--(10,0);
    \draw[thin, red           ] (10,0)--(10,2);
    \draw[thin, orange        ] (10,2)--( 0,2);
    \draw[thin, red           ] ( 0,2)--( 0,0);

    \fill[black] ( 0,0)       circle (1.00pt);
    \fill[black] (10,0)       circle (1.00pt);
    \fill[black] (10,2)       circle (1.00pt);
    \fill[black] ( 0,2)       circle (1.00pt);

    \draw[blue] (1,1)--(2,1);
    \draw[blue] (3,1)--(5,1);
    \draw[blue] (6,1)--(7,1);
    \draw[blue] (8,1)--(9,1);

    \fill[black] (1,1)       circle (1.00pt);
    \fill[black] (2,1)       circle (1.00pt);
    \fill[black] (3,1)       circle (1.00pt);
    \fill[black] (4,1)       circle (1.00pt);
    \fill[black] (5,1)       circle (1.00pt);
    \fill[black] (6,1)       circle (1.00pt);
    \fill[black] (7,1)       circle (1.00pt);
    \fill[black] (8,1)       circle (1.00pt);
    \fill[black] (9,1)       circle (1.00pt);

    \node at (1.5,1.125) {\(a\)};
    \node at (1.5,0.875) {\(b\)};
    \node at (3.5,1.125) {\(b\)};
    \node at (3.5,0.875) {\(c\)};
    \node at (4.5,1.125) {\(c\)};
    \node at (4.5,0.875) {\(d\)};
    \node at (6.5,1.125) {\(d\)};
    \node at (6.5,0.875) {\(e\)};
    \node at (8.5,1.125) {\(e\)};
    \node at (8.5,0.875) {\(a\)};

    \end{tikzpicture}
    \caption{Realising a discrete representation in the stratum \(\mathcal H_5(\,8\,)\). Slit blue segments and then re-glue the resulting edges with the same label.}
    \label{fig:disfilsmin}
\end{figure}

\medskip

\paragraph{\textit{Genus two case}} The constructions described above rely on the existence of a “good” homology basis, which is guaranteed by Kapovich in \cite{KM2} as soon as \(g \ge 3\). In principle, the genus two case remains uncovered. The techniques employed by Le Fils, however, extend \emph{mutatis mutandis} to genus two in the case of discrete representations, in particular, he shows the existence of a good basis in genus two that allows the same construction to be applied. An ad-hoc argument is instead required for non-discrete representations.  To determine a good basis in this case, a direct approach is used. Le Fils shows that given a representation \(\chi\), if there exists a basis in homology \(\{\alpha_i,\beta_i\}\) such that
\begin{equation}
    0\,<\,\Im\Big(\,\,\overline{\chi(\,\alpha_2\,)}\,\chi(\,\beta_2\,)\Big)\,<\,\frac{\sqrt{3}}{2}\Im\Big(\,\,\overline{\chi(\,\alpha_1\,)}\,\chi(\,\beta_1\,)\Big)
\end{equation}
holds, then \(\chi\) is the period character of a translation surface in \(\mathcal H_2(1,1)\) and of a translation surface in \(\mathcal H_2(\,2\,)\). Once a homology basis is fixed and given a non-discrete representation \(\chi_1\), the proof proceeds iteratively to show the existence of a representation \(\chi_2\) such that the fixed basis is a good basis for \(\chi_2\) and there exists a homeomorphism \(h \colon \Sigma \to \Sigma\) such that \( \chi_2 = \chi_1 \circ h_*\). Since \(\chi_2\) can be realised in both strata, the same must hold for \(\chi_1\).


\smallskip

\subsection{Relative version of Haupt's Theorem}\label{ssec:relative version} A representation \(\chi\) encodes the so-called absolute periods, that is, the periods of all closed curves on the surface. However, the representation also captures information about the so-called \textit{relative}\index{period!relative} periods. A relative period is defined as the evaluation of a differential along a saddle connection. To account for these as well, one needs to consider relative homology. Every pair \((X, \omega)\) determines a set \(\Delta\) as the locus of zeros of the differential, and a representation
\begin{equation}
\chi \colon \textnormal{H}_1(\,\Sigma,\,\Delta,\,\mathbb{Z}\,) \longrightarrow \mathbb{C}
\end{equation}
in relative homology. As in \eqref{eq:permap}, one may define a \textit{relative\index{period!relative map} period map} as \begin{equation}\label{eq:relpermap}
\textnormal{relPer}(\,\mu\,)\colon\Omega\mathcal{S}_{g}(\,\mu\,)\longrightarrow \textnormal{H}^1\big(\,\Sigma,\,\Delta,\,\mathbb C\,\big)\cong\text{Hom}\,\Big(\,\textnormal{H}_1(\,\Sigma,\,\Delta,\,\mathbb Z\,),\,\mathbb{C}\,\Big),
\end{equation}
that associates to every homologically marked translation surface its relative period character. There is an obvious surjection \(\textnormal{abs}\colon\textnormal{H}^1\big(\,S_g,\,Z,\,\mathbb C\,\big)\longrightarrow \textnormal{H}^1\big(\,S_g,\,\mathbb C\,\big)\) which may be regarded as a forgetful map. 

\smallskip

\textit{Convention.} It will be useful to adopt the following terminology. Let \(\chi\colon \textnormal{H}_1(\,\Sigma,\,\Delta,\,\mathbb{Z}\,) \longrightarrow \mathbb{C}\) be a relative period representation. Its absolute part, denoted by \(\chi_{a}\), is defined as the restriction of \(\chi\) to the absolute homology \(\textnormal{H}_1(\,\Sigma,\,\Z\,)\).

\smallskip

\noindent The composition map \(\textnormal{abs}\circ\textnormal{relPer}(\,\mu\,)\) thus yields an association whose image is the set of all period representations that arise as the absolute part of some relative period representation \(\chi\). In principle, the following inclusion holds: \(\textnormal{Im}\big(\,\textnormal{abs}\circ\textnormal{relPer}(\,\mu\,)\,\big)\subseteq\textnormal{Im}\big(\,\textnormal{Per}\,\big)\). 

\begin{rmk}[]
    According to Remark \ref{rmk:localcoord}, the moduli space \(\mathcal H_g(\,\mu\,)\) is locally modelled on the relative homology group \(\textnormal{H}^1(\,\Sigma,\,\Delta,\,\C\,)\). More precisely, for every translation surface \((X,\omega)\in\mathcal H_g(\,\mu\,)\) there exists an open set \(\mathcal U_{(X,\omega)}\) and a biholomorphism \(\varphi\colon\mathcal U_{(X,\omega)}\longrightarrow\varphi(\,\mathcal U_{(X,\omega)}\,)\subseteq\textnormal{H}^1(\,\Sigma,\,\Delta,\,\C\,)\). By the analytic continuation property, \(\varphi\) extends to a multivalued map \(\overline{\varphi}\colon\mathcal H_g(\,\mu\,)\longrightarrow\textnormal{H}^1(\,\Sigma,\,\Delta,\,\C\,)\). The \(\textnormal{relPer}(\,\mu\,)\) map in \eqref{eq:relpermap} is the lift of \(\overline{\varphi}\) on the space \(\Omega\mathcal S_g(\,\mu\,)\), the smallest covering of \(\mathcal H_g(\,\mu\,)\) on which the lift of \(\varphi\) is a genuine map. In \cite{FS}, Filip defines the \textit{Developing map}
    \begin{equation}
        \textnormal{Dev}_{\mu}\colon\widetilde{\mathcal H_g}(\,\mu\,)\longrightarrow \textnormal{H}^1(\,\Sigma,\,\Delta,\,\C\,),
    \end{equation}
    as the lift of \(\varphi\) to the the universal cover of the stratum \(\mathcal H_g(\,\mu\,)\).  It is equivariant with respect to a representation
    \begin{equation}
        \rho_{\mu}\colon\pi_1^{\textnormal{orb}}\Big(\,\mathcal H_g(\,\mu\,),\,(X,\omega)\,\Big)\longrightarrow \textnormal{PMod}_{g,k}
    \end{equation}
    where \(\textnormal{PMod}_{g,k}\) denotes the pure mapping class group, that is the group of homotopy classes of diffeomorphisms of \(X\) that preserves \(\Delta\) pointwise. We also observe that this terminology is not casual; in fact, by adopting the language of geometric structures, \(\mathcal H_g(\,\mu\,)\) carries a geometric structure locally modelled on the vector space \(\textnormal{H}^1(\,\Sigma,\,\Delta,\,\C\,)\) for which \(\big(\,\textnormal{Dev}_{\mu},\,\rho_{\mu}\,\big)\) is the developing-holonomy pair.
\end{rmk}

\noindent In \cite[Question 3.1.15]{FS}, Filip asked to determine the image of the Developing map \(\textnormal{Dev}_{\mu}\). This is equivalent to determine the image of the \(\textnormal{relPer}(\,\mu\,)\) as in \eqref{eq:relpermap}. In fact, since the Torelli subgroup \(\mathcal I_{g,k}\)\footnote{This is the subgroup of \(\textnormal{PMod}_{g,k}\) that fixes every class in absolute homology.} acts trivially in homology, it follows that \(\textnormal{Dev}_{\mu}\) and \(\textnormal{relPer}(\,\mu\,)\) have the same image. Recent independent works by Chen--Faraco \cite{CFrel} and Le Fils in \cite{filsrel} have shown that

\begin{thm}[ Chen-Faraco, Le Fils ]\label{thm:chenfarfils}
    \[
    \textnormal{Im}\big(\,\textnormal{abs} \circ \textnormal{relPer}(\,\mu\,)\,\big) = \textnormal{Im}\big(\,\textnormal{Per}(\,\mu\,)\,\big)
    \]
\end{thm}
\noindent by providing a relative refinement of Haupt's theorem, thereby giving a complete answer to the aforementioned question posed by Filip. We may notice that, in principle, \(\textnormal{Im}\big(\,\textnormal{abs}\circ\textnormal{relPer}(\,\mu\,)\,\big)\subseteq\textnormal{Im}\big(\,\textnormal{Per}(\,\mu\,)\,\big)\). Although it may not be surprising that every non-discrete absolute period representation arises as the absolute part of some relative period representation, in principle, it is not clear whether other obstructions arise for relative period representations with discrete absolute period representation. To understand better this point, we need to make a brief digression into Hurwitz realisation problem in \S\ref{sssec:Hurwitz} and then explain how it is related to Haupt's theorem in \S\ref{sssec:fromrelreptobranching}.

\subsubsection{Hurwitz realisation problem}\label{sssec:Hurwitz} Let \(f\colon \Sigma\longrightarrow S\) be a degree \(d\) branched covering of surfaces. For every \(p\in S\) there is a partition \(\lambda(\,p\,)=\big[\,m_1+1,m_2+1,\dots,m_r+1\,\big]\) of \(d\) such that, over an appropriate neighbourhood of \(p\), the covering map \(f\) is equivalent to a map
\begin{equation}
    \phi\colon\{\,1,\dots,r\,\}\,\times\,\C\longrightarrow \C, \quad \text{ where }\quad \phi(\,i,z\,)=z^{m_i+1},
\end{equation}
and \(p\) correspond to \(0\in\C\), see \cite{EKS}. If \(\lambda(\,p\,)\) is the trivial partition \(\big[\,1,1,\dots,1\,\big]\) for every \(p\in S\), then \(f\) is a genuine covering map. Otherwise, the set of points of \(S\) for which \(\lambda(\,p\,)\) is not trivial forms the \textit{branched set} \(\mathcal B(\,f\,)\) of \(f\). The collection \(\mathcal D(\,f\,)=\big\{\,\lambda(\,p\,)\,|\,p\in\mathcal B(\,f\,)\,\big\}\) is called \textit{branching datum}, where repetitions are allowed. Since \(f\) is a branched covering map, it is subject to the Riemann--Hurwitz theorem. In our notation
\begin{equation}\label{eq:riemannhurwitzid}
    2g-2\,=\,d\,(\,2h-2\,)\,+\sum_{p\in\mathcal B(\,f\,)}\,\sum_{i\in\lambda(\,p\,)}\,m_i,
\end{equation}
where \(g\) and \(h\) denote the genus of \(\Sigma\) and \(S\) respectively. The branching datum \(\mathcal D(\,f\,)\) can be seen as an \textit{algebraic datum} determined by the branched covering map that satisfies the Riemann--Hurwitz identity. 

\medskip

\begin{quote}\textbf{Hurwitz realisation problem}:
    Given two topological surfaces \(\Sigma\) and \(S\), a positive integer \(d\), a finite set \(\mathcal B=\{\,p_1,\dots,p_n\}\subset S\) and an algebraic datum \(\mathcal D=\big\{\,\lambda(\,p\,)\,|\,p\in\mathcal B\,\big\}\), \textit{i.e.} a collection of partitions of \(d\) that satisfies \eqref{eq:riemannhurwitzid}, 
    determine whether there exists a branched covering map \(f\colon\Sigma\longrightarrow S\) such that \(\mathcal D(\,f\,)=\mathcal D\).
\end{quote}

\medskip

\noindent In \cite{EKS}, Edmonds--Kulkarni--Stong show that every algebraic data is realisable as the branching datum of some branched cover \(f:\Sigma\longrightarrow S\), provided \(h\ge1\). A branched covering map \(f\colon\Sigma\longrightarrow S\) is said to be \textit{primitive} if the map \(f_*\) between fundamental groups is surjective. Two decades later, in their works \cite{BGKZ} and \cite{BGKZ2}, Bogatyi--Gonçalves--Kudryavtseva--Zieschang have strengthened Edmonds--Kulkarni--Stong's result by showing that an algebraic datum \(\mathcal D\) can be realised by a primitive branched covering map provided \(h\ge1\).

\smallskip

\noindent We now reconnect the Hurwitz realisation problem with Haupt's theorem. In \S\ref{ssec:haupt} we have seen that every discrete representation \(\chi\) of degree \(d\ge2\) is uniquely determined (up to homotopy) by a branched covering map \(f\colon\Sigma\longrightarrow T\) with algebraic branching datum \(\mathcal D(\,f\,)\). The leading question here is

\smallskip

\begin{quote}
    \textit{How much information of \(\mathcal D(\,f\,)\) does the representation \(\chi\) encode?}
\end{quote}

\smallskip

\noindent It is a consequence of Haupt's theorem in \S\ref{ssec:haupt} that a discrete representation \(\chi\) encodes very little information. In fact, a representation \(\chi\) encodes \textit{only} the degree of the map inducing it. From the point of view of coverings, a representation in relative homology with discrete absolute period part encodes much more information. It encodes the degree \(d\) of the branched covering and the cardinality of the branched set (see \S\ref{sssec:fromrelreptobranching}). A branching datum is specified once a partition of \(d\) is prescribed for every point of the branching set. The key point is that an algebraic datum \(\mathcal D\) determines a well defined a signature \(\mu\), that is, a partition of \(2g-2\) (see formula \eqref{eq:riemannhurwitzid}; if \(h=1 \), then the local degrees determine a partition of \(2g-2\)). Conversely, from a given signature one can always extract a branching datum, although it need not be unique, see Example \ref{ex:differentbranchingdata}. We discuss this in more detail in the following section.

\smallskip

\subsubsection{From relative representations to branching datum}\label{sssec:fromrelreptobranching} We now show how a discrete relative representation yields a well defined branching datum and then state the relative version of the Haupt theorem.

\smallskip

\noindent Let \(\chi \in \textnormal{H}^1(\,\Sigma,\,\Delta,\,\mathbb{C}\,)\) be a representation with discrete absolute period part \(\chi_a\). Without loss of generality we may assume that \(\textnormal{Im}(\,\chi_a\,)=\Gamma=\ziz\). Decompose \(\Delta = \{z_1, \ldots, z_k\}\) into maximum disjoint subsets, say \(\Delta_1, \ldots, \Delta_{\ell}\), such that for any \(\gamma\) joining a pair of zeros \(z_{i_1}, z_{i_2} \in \Delta_i\) the image \(\chi(\,\gamma\,) \in \Gamma\). Then \(\chi\) arises as the period character of some pair \((X,\omega)\) of type \(\mu = (m_1, \dots, m_k)\) if and only if there exists a primitive branched covering map \(\pi\colon X \longrightarrow \mathbb{C}/\Gamma\) of degree \(d\ge \mathrm{vol}(\,\chi\,)\) such that \(\pi\) has \(\ell\) ramification points, say \(y_1, \ldots, y_\ell \in \mathbb{C}/\Gamma\), and over each \(y_i\) the ramification points are the zeros in \(\Delta_i\) with ramification orders given by the zero orders. 

\smallskip

\noindent Let \(\mu^{\vdash} = (\,\mu_1;\, \dots;\, \mu_\ell\,)\) be a partition of a given signature \(\mu\) into \(\ell\) disjoint subsets, where each \(\mu_i\) records the orders of the zeros in \(\Delta_i\). Let \(|\,\mu_i\,|\) be the sum of the entries of \(\mu_i\) and \(|\,\Delta_i\,|\) be the cardinality of the subset. Define 
\begin{equation}\label{eq:neccondtoreal}
    \mathrm{m}(\,\chi,\, \mu^\vdash\,) = \max_{1\le i \le \ell} \Big\{\, |\,\mu_i\,| + |\,\Delta_i\,| \,\Big\}.
\end{equation}

\noindent As a consequence, by considering the degree of \(\pi\), the above description implies that the following condition is necessary for a period character \(\chi\) to be realisable: 
\begin{eqnarray}
\label{eq:hol-necessary}
 \deg(\,\pi\,) = \mathrm{vol}(\,\chi_a\,) \ge \mathrm{m}(\,\chi,\, \mu^\vdash\,) =\max_{1 \le i \le \ell}\,\,\left\{\,\sum_{z_j \in \Delta_i} (m_j + 1)\,\right\}.
\end{eqnarray}

\smallskip

\begin{thm}
Let \(\chi \in \textnormal{H}^1(\,\Sigma,\,\Delta,\,\mathbb{C}\,)\) and \(\mathcal K \subset \mathcal{H}_g(\,m_1, \dots, m_k\,)\) be a connected component of a stratum of abelian differentials. There exists \((X,\omega) \in \mathcal K\) with period map \(\chi\) if and only if the following conditions hold:
\begin{enumerate}
    \item \(\textnormal{vol}(\,\chi_a\,) > 0\),
    \smallskip
    \item if \(\textnormal{Im}(\,\chi\,)\) is a lattice and \(\mu^\vdash\) is a partition of \(\mu\) into \(\ell + 1\) disjoint subsets, then
    \begin{equation}\label{eq:neccondrela}
        \textnormal{vol}(\,\chi\,) \ge \, \max_{1 \le i \le \ell} \,\,\left\{\,\sum_{z_j \in \Delta_i} (m_j + 1)\,\right\}\,\mathrm{Area}(\mathbb{C}/\Lambda) ,
    \end{equation}
    where \(\Delta_1, \dots, \Delta_{\ell}\) is the partition of \(\Delta\) determined by \(\chi\).
\end{enumerate}
\end{thm}

\noindent One may easily notice that the following chain of inequalities holds
\begin{eqnarray}
    2 \,\le\, \max_{1 \le i \le \ell} (\,m_j + 1\,)\,\le\,\max_{1 \le i \le \ell} \,\,\left\{\,\sum_{z_j \in \Delta_i} (m_j + 1)\,\right\}.
\end{eqnarray}
In the first case the equality holds if \(m_i=1\) and the second equality holds if different zeros project to different ramification values. In fact, the last condition reflects the fact that for a branched covering to exist the total degree must be at least the maximum of the local degrees of the branched covering around zeros projecting to the same ramification value on the torus, see Figure~\ref{fig:degree2}.

\smallskip

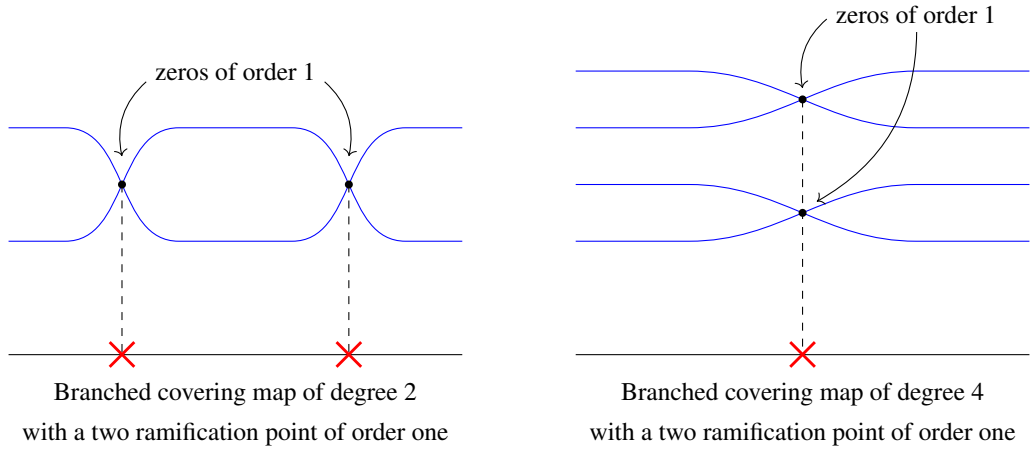
\begin{figure}[h!]
    \centering
    \begin{tikzpicture}[scale=1.5, every node/.style={scale=0.875}]
    \definecolor{pallido}{RGB}{221,227,227}

    \draw[black] (-2,0)--(2,0);
    \draw[blue] (-2, 1  )--(-1, 1  ) to[out=0, in=180] (1,1.5)--(2,1.5);
    \draw[blue] (-2, 1.5)--(-1, 1.5) to[out=0, in=180] (1,1)--(2,1);
    \draw[blue] (-2, 2  )--(-1, 2  ) to[out=0, in=180] (1,2.5)--(2,2.5);
    \draw[blue] (-2, 2.5)--(-1, 2.5) to[out=0, in=180] (1,2)--(2,2);
    \draw[dashed] (0,2.25)--(0,0);
    \fill (0,1.25) circle (1pt);
    \fill (0,2.25) circle (1pt);
    \node[red, scale=2] at (0,0) {\(\times\)};
    \node (A) at (1,3) {zeros of order \(1\)};
    \draw[black, ->] (A) to[bend left] (0.125,1.35);
    \draw[black, ->] (A) to[bend right] (0,2.35);
    \node at (0,-0.35) {Branched covering map of degree \(4\)};
    \node at (0,-0.7) {with a two ramification point of order one};

    \begin{scope}[shift={(-5,0)}]
        \draw[black] (-2,0)--(2,0);
        \draw[blue] (-2,1)--(-1.5,1) to[out=0, in=180] (-0.5,2)--(0.5,2) to[out=0, in=180] (1.5,1)--(2,1);
        \draw[blue] (-2,2)--(-1.5,2) to[out=0, in=180] (-0.5,1)--(0.5,1) to[out=0, in=180] (1.5,2)--(2,2);
        \fill (-1,1.5) circle (1pt);
        \fill ( 1,1.5) circle (1pt);
        \node[red, scale=2] at (-1,0) {\(\times\)};
        \node[red, scale=2] at ( 1,0) {\(\times\)};
        \draw[dashed] (-1,1.5)--(-1,0);
        \draw[dashed] ( 1,1.5)--( 1,0);
        \node (B) at (0,2.5) {zeros of order \(1\)};
        \draw[black, ->] (B) to[bend right] (-1,1.75);
        \draw[black, ->] (B) to[bend left ] ( 1,1.75);
        \node at (0,-0.35) {Branched covering map of degree \(2\)};
    \node at (0,-0.7) {with a two ramification point of order one};
    \end{scope}
    \end{tikzpicture}
    \caption{Realisation of a discrete representation \(\chi\) of a genus two surface in the stratum \(\mathcal H_2(1,1)\). The figure shows the branched covering induced by the developing map. If the degree of \(\chi_a\) is two or three, then it does not arise as the period character of a translation surface whose zeroes project to the same point of the target torus. Compare with Figure \ref{fig:degree}.}
    \label{fig:degree2}
\end{figure}

\smallskip

\noindent In both cases, the approaches adopted by Le Fils and Chen--Faraco are constructive and the representations are realised by explicit translation surfaces with prescribed period characters. We shall not summarise the proofs here. In fact, despite the fact that all constructions are news, the gist of the idea and techniques used rely on the earlier works \cite{KM2}, \cite{fils} and \cite{BJJP}.

\smallskip

\begin{figure}[h!]
    \centering
    \begin{tikzpicture}[scale=1.5, every node/.style={scale=0.875}]
    \definecolor{pallido}{RGB}{221,227,227}

    \draw[black] (-3,0)--(3,0);
    \draw[blue] (-3, 1  )--(-2, 1  ) to[out=0, in=180] (-1,2)--(1,2) to[out=0, in=180] (2,1)--(3,1);
    \draw[blue] (-3, 1.5  )--(3, 1.5  );
    \draw[blue] (-3, 2  )--(-2, 2  ) to[out=0, in=180] (-1,1)--(1,1) to[out=0, in=180] (2,2)--(3,2);
    \draw[blue] (-3, 2.5)--(-2, 2.5) to[out=0, in=180] (-1,4)--(3,4);
    \draw[blue] (-3, 3)--(-2, 3) to[out=0, in=180] (-1,3.5)--(3,3.5);
    \draw[blue] (-3, 3.5)--(-2, 3.5) to[out=0, in=180] (-1,3)--(1,3) to[out=0, in=180] (2,2.5)--(3,2.5);
    \draw[blue] (-3, 4)--(-2, 4) to[out=0, in=180] (-1,2.5)--(1,2.5) to[out=0, in=180] (2,3)--(3,3);
    \draw[dashed] (-1.5,4)--(-1.5,0);
    \draw[dashed] ( 1.5,4)--( 1.5,0);
    \fill (-1.5,1.5 ) circle (1pt);
    \fill (-1.5,3.25) circle (1pt);
    \fill ( 1.5,1.5 ) circle (1pt);
    \fill ( 1.5,2.75) circle (1pt);
    \fill ( 1.5,3.5 ) circle (1pt);
    \fill ( 1.5,4   ) circle (1pt);
    \node[red, scale=2] at (-1.5,0) {\(\times\)};
    \node[red, scale=2] at ( 1.5,0) {\(\times\)};
    \node at (0,-0.35) {Branched covering map \(f\) of degree \(7\)};
    \node at (0,-0.7 ) {with a two ramification point};
    \node at (0,-1.05) {\(\mathcal D(\,f\,)=\big\{\,(1,1,2,3),\,(3,4)\,\big\}\)};

    \node[orange] at (-1.75, 1.725) {\(2\)};
    \node[orange] at (-1.75, 3.25 ) {\(3\)};
    \node[orange] at ( 1.75, 1.725) {\(2\)};
    \node[orange] at ( 1.75, 2.75 ) {\(1\)};
    \node[orange] at ( 1.75, 3.375) {\(0\)};
    \node[orange] at ( 1.75, 3.875) {\(0\)};

    \node[scale=1.25] at (0.00,0.25) {\(\C/\Gamma\)};
    \node[scale=1.25] at (0.00,3.75) {\(\Sigma\)};
    
    \begin{scope}[shift={(0,-6)}]
        \draw[black] (-3,0)--(3,0);
        \draw[blue] (-3, 1.5)--(-2, 1.5) to[out=0, in=180] (-1,2)--(1,2) to[out=0, in=180] (2,1)--(3,1);
        \draw[blue] (-3, 2  )--(-2, 2  ) to[out=0, in=180] (-1,1.5)--(3,1.5);
        \draw[blue] (-3, 1)--( 1, 1) to[out=0, in=180] ( 2,2)--(3,2);
        
        \draw[blue] (-3, 3)--(-2, 3) to[out=0, in=180] (-1,3.5)--(1,3.5) to[out=0, in=180] (2,2.5)--(3,2.5);
        \draw[blue] (-3, 3.5)--(-2, 3.5) to[out=0, in=180] (-1,3)--(3,3); 
        \draw[blue] (-3, 4)--(-2, 4) to[out=0, in=180] (-1,2.5)--(1,2.5) to[out=0, in=180] (2,3.5)--(3,3.5);
        
        \draw[blue] (-3, 2.5)--(-2, 2.5) to[out=0, in=180] (-1,4)--(3,4);
        \draw[dashed] (-1.5,4)--(-1.5,0);
        \draw[dashed] ( 1.5,4)--( 1.5,0);
        \fill (-1.5,1.75) circle (1pt);
        \fill (-1.5,3.25) circle (1pt);
        \fill ( 1.5,1.5 ) circle (1pt);
        \fill ( 1.5,3   ) circle (1pt);
        \fill ( 1.5,4   ) circle (1pt);
        \fill (-1.5,1   ) circle (1pt);
        \node[red, scale=2] at (-1.5,0) {\(\times\)};
        \node[red, scale=2] at ( 1.5,0) {\(\times\)};
        \node at (0,-0.35) {Branched covering map \(f\) of degree \(7\)};
        \node at (0,-0.7 ) {with a two ramification point};
        \node at (0,-1.05) {\(\mathcal D(\,f\,)=\big\{\,(1,2,4),\,(1,3,3)\,\big\}\)};

        \node[orange] at (-1.75, 1.225) {\(0\)};
        \node[orange] at (-1.75, 1.725) {\(1\)};
        \node[orange] at (-1.75, 3.25 ) {\(3\)};
        \node[orange] at ( 1.75, 1.625) {\(2\)};
        \node[orange] at ( 1.75, 2.875) {\(2\)};
        \node[orange] at ( 1.75, 3.875) {\(0\)};

        \node[scale=1.25] at (0.00,0.25) {\(\C/\Gamma\)};
        \node[scale=1.25] at (0.00,3.75) {\(\Sigma\)};
    \end{scope}

    \end{tikzpicture}
    \caption{Branched coverings of degree \(7\) associated with different branching data. The orange labels denote the order of zeros once the translation structure on \(\C/\Gamma\) is pulled-back.}
    \label{fig:examplebranchingdata}
\end{figure}
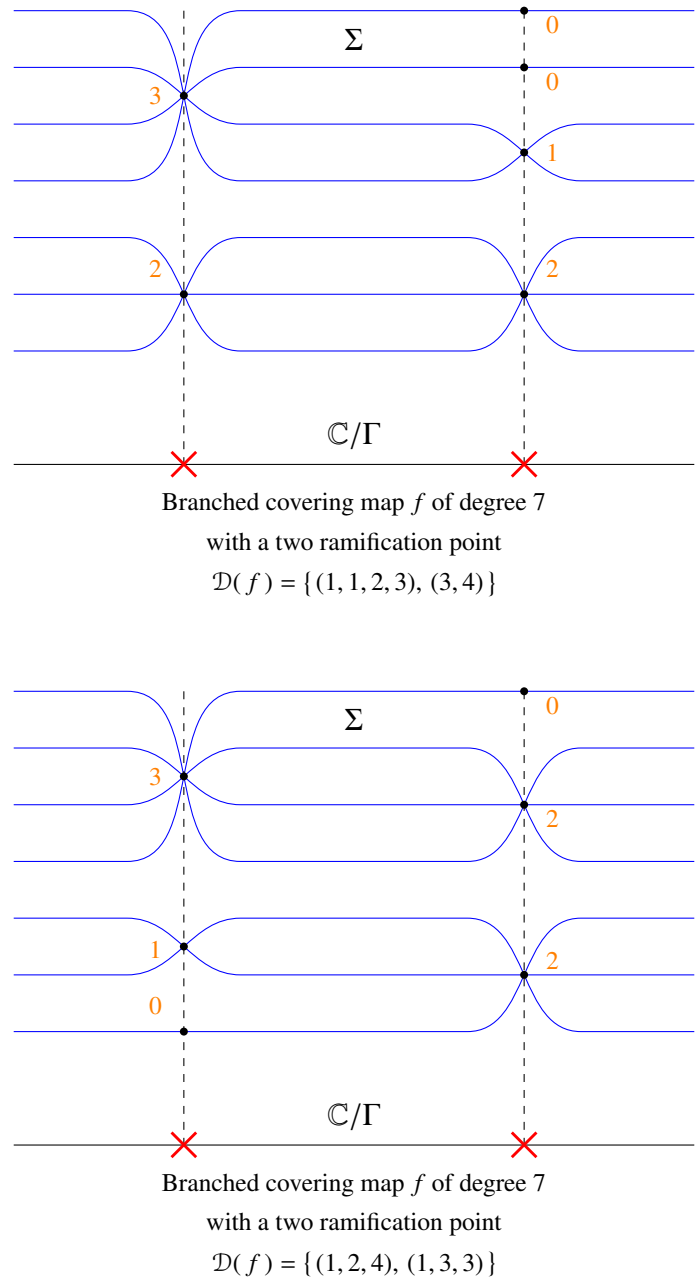

\begin{ex}[Same signature, different branching data \(\mathcal D\)]\label{ex:differentbranchingdata}
    Let \(\mu=(\,1,2,2,3\,)\) be a signature of abelian differentials on a surface of genus \(5\). Let \(\chi\in \textnormal{H}^1(\,\Sigma,\,\Delta,\,\mathbb{C}\,)\) be a representation whose absolute period part has volume \(7\), where \(\Delta\subset\Sigma\) is a finite set of four distinct points. Suppose that \(\chi\) yields a decomposition of \(\Delta\) into two disjoint subsets, say \(\Delta_1\) and \(\Delta_2\), each of which of cardinality two. We ask in how many and in which ways a branching datum \(\mathcal D\) can be prescribed. In other words, we ask in how many and in which ways one can prescribe the local degrees of the covering that we wish to realise. We consider the multi-set \(\{1,2,2,3\}\) and list all its partitions into two subsets of cardinality two (the order of the pairs being irrelevant). There are two possibilities given by \(\mu_1^{\vdash}=\{\,(1,2),(2,3)\,\}\) and \(\mu_2^{\vdash}=\{\,(1,3),(2,2)\,\}\). A direct computation shows that \(\mathrm{m}(\,\chi,\, \mu_1^\vdash\,)=7\) and that \(\mathrm{m}(\,\chi,\, \mu_2^\vdash\,)=6\). In both cases, \(\textnormal{vol}(\,\chi_a\,)\ge\mathrm{m}(\,\chi,\, \mu_i^\vdash\,)\) and hence \(\chi\) can be realised in the stratum \(\mathcal H_5(\,\mu\,)\). By fixing the first partition of \(\mu\), the developing map of the resulting translation surface yields a branched covering \(f_1\colon\Sigma\longrightarrow \C/\Gamma\) with two ramification values, say \(q_1,q_2\), so that \(\Delta_i=f_1^{-1}(\,q_i\,)\) and branching data \(\mathcal D_1=\big\{\,(1,1,2,3),\,(3,4)\big\}\).\footnote{Recall that the order of a zero is one less the local degree, \textit{i.e.} \(d_i=m_i+1\).} In the same fashion, by fixing the second partition of \(\mu\), the developing map of the resulting translation surface yields a branched covering \(f_2\colon\Sigma\longrightarrow \C/\Gamma\) with two ramification values and branching data \(\mathcal D_2=\big\{\,(1,2,4),\,(1,3,3)\big\}\). See Figure \ref{fig:examplebranchingdata}.
\end{ex}

\smallskip

\subsection{Isoperiodic foliations}\label{ssec:isoper} Haupt's theorem, as well as its subsequent refinements, provide only a complete description of the image of the period map. In other words, we know when a representation can be realised as a period character, but, in principle, the resulting structure may not be unique. In this respect, the case of tori is not very interesting, since the fibre over the period map is a single point. Things become truly interesting in genus at least \(g \geq 2\). Translation structures are flexible in the sense that they can be deformed while preserving the absolute periods. The fibre over the holonomy therefore gives a subset of \(\Omega \mathcal{S}_g\) called the \emph{isoperiodic foliation}\index{foliation!isoperiodic}. It is also known as the \emph{absolute period foliation}\index{foliation!absolute period} or the \emph{kernel foliation}\index{foliation!kernel}. Although the terminology is not unique, \emph{isoperiodic foliation} is commonly used. 

\smallskip

\subsubsection{Connectedness of fibres}\label{ssec:conn} A leading question is whether a fibre over a representation is connected. Connectedness means that two structures may be deformed one into another while preserving the absolute periods throughout the deformation. When the absolute periods are preserved, deformations within an isoperiodic leaf may be carried out by varying the relative periods, and this can be achieved through processes such as breaking zeros or moving branch points, as described in~\S\ref{sssec:zerobreak}. 

\smallskip

\noindent Recall that with any period character \(\chi\), one can associate a numerical invariant, called the \textit{degree}, defined as the integer given by the ratio \(\textnormal{vol}(\,\chi\,)/\textnormal{area}(\,\mathbb{C}/\Lambda\,)\) if the image is a lattice, see \eqref{eq:refhaup1}. If the image is not a lattice we set \(\deg(\,\chi\,)=\infty\). A deep result due to Calsamiglia--Deroin--Francaviglia says that the fibre over a representation of degree at least three is connected in the principal stratum \(\mathcal{H}_g(1,\dots,1)\). More precisely,

\begin{thm}[Calsamiglia--Deroin--Francaviglia, \cite{CDF2}]
    For \(g \geq 2\), the fibres of the period map \(\textnormal{Per}\) over points of degree at least three are connected.
\end{thm}

\noindent The strategy these authors adopt is particularly interesting and an outline of the proof is as follows:
\begin{itemize}
    \item[1.] They analyse the isoperiodic degenerations of abelian differentials towards abelian differentials on stable nodal curves and show that it is always possible to find such degenerations.
    \item[2.] They define a bordification of an isoperiodic leaf by adding nodal differentials arising as limits of degenerations, as in step 1. This bordification is naturally stratified by the number of nodes of the underlying curves.
    \item[3.] They show that bordifications are connected and that the connectedness of the fibre is equivalent to that of its bordification.
\end{itemize}

\noindent As for representations of degree two, the situation is quite different since connectedness also depends on the genus. For \(g = 2, 3\), the fact that the fibres of the period map \(\textnormal{Per}\) over points of primitive degree two are connected was already shown in an earlier work by McMullen; see \cite{McM2}. For \(g \geq 5\), Calsamiglia--Deroin--Francaviglia established the converse in \cite{CDF2}, showing that the fibres over representations of degree two are disconnected. It is currently unclear whether, for \(g = 4\), the representations of degree two have connected fibres, although the authors express a favourable opinion in this direction.


\smallskip

\noindent An even earlier contribution to the field was provided by Schmoll in~\cite{Schmoll}, based on prior work with Eskin and Masur in~\cite{Eskin-Masur-Schmoll}. In~\cite[Theorem 2]{Schmoll}, Schmoll shows that in genus \(g=2\), the fibres, denoted by \(\mathcal{F}_d\), over discrete representations of degree \(d\) are connected. The same result was previously established by Fulton in \cite{Fulton}; however, Schmoll independently recovers this theorem, by employing a different methodological framework. 
These fibres enjoy several desirable properties: they are topological surfaces and naturally come with a covering \(\pi_{\mathcal F}\) over the standard torus \(\mathbb{C}/\Gamma\), where \(\Gamma\) is the standard lattice. \(\mathcal F_d\) carries a natural translation structure determined by \(\omega_d = \pi_{\mathcal F}^* dz\), the pull-back of the standard differential \(dz\) via the covering projection \(\pi_{\mathcal F}\). More specifically, \((\mathcal F_d, \omega_d)\) is a square-tiled surface whose conical singularities correspond to translation surfaces in the minimal stratum \(\mathcal H_2(\,2\,)\) arising as branched coverings of degree \(d\) of the standard torus. Interestingly, the cardinality of the set \(\textnormal{Z}(\,\omega_d\,)\) of zeros of \(\omega_d\) in \(\mathcal F_d\) has been explicitly computed:
\begin{equation}
    \left|\,\textnormal{Z}(\,\omega_d\,)\,\right|\,=\,\frac{3}{8}\,d^2(\,d-2\,)\prod_{\,p|d\,}\left(\,1-\frac{1}{p^2}\,\right),
\end{equation}
 see the Remark after Lemma 4.11 in~\cite{Eskin-Masur-Schmoll} and \cite[Formula (3)]{Schmoll}.

\smallskip

\subsubsection{Ergodicity}\label{sssec:ergo} The preimages of the map \(\textnormal{Per}\) as in \eqref{eq:permap} define a stratification of the space \(\Omega \mathcal{S}_g\) into isoperiodic leaves. This foliation has dimension \(2g - 3\) and is also algebraic: its leaves are solutions to certain systems of algebraic equations with respect to the Deligne--Mumford algebraic structure on the moduli space. The covering projection \(\pi\colon \Omega \mathcal{S}_g \longrightarrow \Omega \mathcal{M}_g\) is \(\textnormal{Sp}(2g, \mathbb{Z})\)-equivariant, and since the leaves are invariant under the action of \(\textnormal{Sp}(2g, \mathbb{Z})\), an induced foliation is well defined on the moduli space of translation structures \(\Omega \mathcal{M}_g\). A central problem is to determine the closures of the leaves.

\begin{thm}[Calsamiglia-Deroin-Francaviglia, \cite{CDF2}]
Let \(g \ge3\) and \((X, \omega) \in \Omega\mathcal{M}_g\) be a translation surface with period character \(\chi\) not of degree two with volume \(\textnormal{vol}(\,\chi\,)\). Let \(\Lambda\) denote the closure of the image of its periods. Then the closure of the isoperiodic leaf passing through \((X, \omega)\), up to the action of \(\mathrm{GL}(2, \mathbb{R})\), is:
\begin{itemize}
  \item[1] If \(\Lambda\) is discrete: the 
  space consisting of genus \(g\) primitive branched coverings over \((\mathbb{C}/\Lambda, dz)\) of volume \(\textnormal{vol}(\,\chi\,)\).
  \item[2] If \(\Lambda = \mathbb{R} + i\mathbb{Z}\): the set of abelian differentials with periods contained in \(\Lambda\), with primitive imaginary part, and with volume \(\textnormal{vol}(\,\chi\,)\).
  \item[3] If \(\Lambda = \mathbb{C}\): the subset of \(\Omega \mathcal{M}_g\) consisting of abelian differentials of volume \(\textnormal{vol}(\,\chi\,)\).
\end{itemize}
\end{thm}

\noindent The same result holds even for \(g=2\) with an additional item in the list above, see \cite{CDF2} and \cite{McM2} for more details. As a direct consequence, the isoperiodic foliation is ergodic\index{foliation!ergodicity} where, here, ergodicity\index{ergodic} means that a measurable union of leaves has either zero Lebesgue measure or full Lebesgue measure. In \cite{UH}, the ergodicity of the isoperiodic foliation was proved by Hamenst\"adt with a different methods. 

\begin{thm}[Calsamiglia-Deroin-Francaviglia, Hamenst\"adt]
    The absolute period foliation of the principal stratum is ergodic in every genus \(g\ge2\).
\end{thm}


\smallskip

\subsection{Isoperiodic foliations on strata}\label{ssec:isofolstrata} Every fibre of the period map \eqref{eq:permap} naturally restricts to each stratum, thus defining an isoperiodic foliation on strata. Also in this case, the main questions concern the connectedness of the fibres and whether the isoperiodic foliation is ergodic or not. 

\smallskip

\subsubsection{Fake octagons}\label{sssec:octa} The simplest case, yet less trivial than one might expect, is given by the isoperiodic foliation in the minimal strata. Recall that the minimal stratum consists precisely of those translation structures with a single conical point; since there are no relative periods, its dimension is \(2g\). The isoperiodic foliation in this stratum forms a subspace consisting of isolated points: no infinitesimal deformation of the structure yields a different structure. In other words, applying an infinitesimal movements of points (only one in this case) to a structure in the minimal stratum results in a structure whose developing map differs from the original by a translation, hence the structures are equivalent and no deformation happens. This phenomenon was already shown by McMullen in \cite{McM2}, where he defined the so-called "fake pentagons". A unit-area pentagon appropriately glued to a rotated copy of itself by \(\pi\) produces an octagon with pairwise parallel sides that, when properly glued, determines a genus-two translation surface in the minimal stratum. McMullen, in \cite{McM2}, showed that by suitably deforming the pentagon into a fake pentagon, and gluing it to a rotated-by-\(\pi\)  of itself, one obtains a genus-two translation surface in the minimal stratum with the same periods as the initial one but which is not isomorphic to it. The same procedure actually produces infinitely many such translation surfaces; see \cite[\S 10]{McM}. More recently, Dobrilla–Francaviglia in \cite{DF} provided an alternative proof of the existence of these structures, which they called "fake octagons". Their proof is more down-to-earth and concrete, as it uses an explicit point move that, when iterated, produces pairwise non-isomorphic structures.

\subsubsection{Ergodicity} Chaika and Weiss have recently provided in \cite{CW} a new proof of the ergodicity of the isoperiodic foliation on any connected component of any stratum. Their approach is built upon recent advances in the ongoing research programme led by Eskin, Brown, Filip, Rodriguez--Hertz, and others, which aims to generalise the celebrated "Magic Wand Theorem'' of Eskin–Mirzakhani, see \cite{EM}.

\subsubsection{Connectedness of the fibres in strata}\label{sssec:constrataiso} The restriction of the isoperiodic foliation to a generic stratum is generally not connected. A first example of this phenomenon was observed in \S\ref{sssec:octa} in the case of minimal strata. In \cite{KW}, Winsor provided a number of remarkable results that led to a better understanding of the fibres of the period map on a specific stratum. In some cases, the fibre turns out to be connected, while in others, it is highly disconnected. Among the several results shown by Winsor, one that contrasts with those in \cite{CDF2} is the following:

\begin{thm}[Winsor]
Fix \( g \geq 3 \), and let \( \mathcal K \) be a component of a stratum \( \strata \) with \( |\,\mu\,| > 1 \), such that \( \mathcal K \) is a spin component or a hyperelliptic component. For any representation \( \chi \in \textnormal{H}^1(\Sigma,\, \mathbb{C}) \) with positive volume such that \( \textnormal{Im}(\,\chi\,) \cong \mathbb{Z}^{2g} \), the fibre \( \textnormal{Per}_{\mu}^{-1}(\,\chi\,) \) is disconnected.
\end{thm}

\subsection{Further developments} Following the extensive study of holomorphic differentials on Riemann surfaces, the investigation of \emph{meromorphic differentials}\index{differential!meromorphic} has also attracted considerable attention in recent years. Given a meromorphic differential \(\omega\) on a compact Riemann surface \(X \in \mathcal{M}_g\), one obtains a \emph{period representation} \(\chi\colon \textnormal{H}_1(X \setminus P(\omega),\,\mathbb{Z}) \longrightarrow \mathbb{C},
\) where \(P(\omega)\) denotes the set of poles of \(\omega\). It is not difficult to show that every such representation arises as the \emph{period character} of a meromorphic differential on a prescribed Riemann surface; see \cite{CFG}. In contrast to the holomorphic setting, \emph{there are no obstructions} to realizing a given representation in this broader context. More subtle, however, is the problem of realizing a given representation \emph{within a specific stratum or connected component} of the moduli space. These questions have been addressed in detail in \cite{CFG} and \cite{CF}. An alternative formulation, where all singularities are located at punctures, is presented in \cite{fargup}. Even in the meromorphic case, one can define an \emph{absolute period map}, and the aforementioned works provide a complete description of its image when restricted to a stratum or to one of its connected components. On the other hand, the structure of the \emph{isoperiodic fibration} remains largely unexplored: only a few special cases have been studied so far; see \cite{CD} and \cite{FTZ}.


\bibliographystyle{alpha}
\bibliography{surveytranssurf.bib}

\end{document}